\documentclass[11pt]{article}
\RequirePackage[OT1]{fontenc}
\RequirePackage{amsmath,amssymb,subcaption,graphicx,epstopdf,enumerate,lmodern}
\RequirePackage[round,colon,authoryear]{natbib}
\usepackage{natbib}
\RequirePackage{bigints}
\usepackage{bm,color,fancyvrb,xcolor,mathtools}
\usepackage{amscd}
\usepackage{mathrsfs}
\usepackage{bbm}
\usepackage{booktabs,multirow,placeins}
\usepackage{enumitem}

\font\tencmmib=cmmib10 \skewchar\tencmmib '60
\newfam\cmmibfam
\textfont\cmmibfam=\tencmmib

\def\lessim{\ \lower4pt\hbox{$
		\buildrel{\displaystyle <}\over\sim$}\ }
\def\gessim{\ \lower4pt\hbox{$\buildrel{\displaystyle >}
		\over\sim$}\ }

\usepackage{fancyhdr}
\usepackage{amsmath}
\usepackage{amsthm}
\usepackage{amssymb,amsfonts}
\usepackage[ruled,vlined,linesnumbered]{algorithm2e}
\usepackage{algpseudocode}

\usepackage{graphicx,subcaption,float}
\usepackage{tikz}
\usetikzlibrary{decorations.pathreplacing,calc,patterns}

\usepackage[colorlinks,citecolor=blue,urlcolor=blue, linkcolor=blue]{hyperref}

\newtheorem{theorem}{Theorem}[section]
\newtheorem{proposition}[theorem]{Proposition}
\newtheorem{lemma}{Lemma}
\newtheorem{assumption}{Assumption}
\newtheorem{definition}{Definition}
\newtheorem{corollary}[theorem]{Corollary}
\newtheorem{remark}{Remark}[section]

\DeclarePairedDelimiter\ceil{\lceil}{\rceil}

\def\EE{{\mathbb E}}

\def\PP{{\mathbb P}}

\def\RR{{\mathbb R}}

\newcommand{\bfm}[1]{\ensuremath{\mathbf{#1}}}

   \def\bA{\bfm A}

     \def\EE{\mathbb{E}}

     \def\PP{\mathbb{P}}
     
     \def\RR{\mathbb{R}}

\def\hat{\widehat}

\newcommand{\R}{\mathbb{R}}

\def\lessim{\ \lower4pt\hbox{$\buildrel{\displaystyle <}\over\sim$}\ }
\def\gessim{\ \lower4pt\hbox{$\buildrel{\displaystyle >}\over\sim$}\ }

\usepackage{titling}
\begin{document}

\title{
Differentially Private Nonparametric Modal Learning with Applications to Regression and Clustering}

\author{Arkajyoti Bhattacharjee and Arnab Auddy\\
Department of Statistics \\
The Ohio State University}
\date{\today}

\maketitle
\begin{abstract}
\noindent Density modes provide a localized and interpretable summary of complex, multimodal distributions, but their estimation under rigorous differential privacy constraints remains largely unexplored. We study differentially private recovery of density modes for multivariate distributions under local smoothness, curvature, and separation conditions. We propose DP-GRAMS, a mean-shift inspired method that performs noisy ascent on a differentially private score estimator. Assuming the density belongs locally to a H\"older class with smoothness parameter $\beta > 2$, our score estimator uses bias-reducing higher-order kernels, and then enforces privacy in the gradient ascent steps via gradient clipping and calibrated Gaussian noise. A private initialization scheme combines a density-aware utility with a diversity-inducing suppression rule and, with \(k\asymp M\log n\) draws over a public \(h_{\mathrm{DAP}}\)-grid and suppression radius \(\rho_{\mathrm{init}}\asymp (\log n)^{-1/d}\), achieves high-probability coverage of the modal basins by successively suppressing selected local neighborhoods in competitive regions, while correlated noise across multiple starts enables joint release under a single $(\varepsilon,\delta)$-differential privacy guarantee. We prove that all population modes are recovered with high probability and establish asymptotic error rates of the form $O\!\left((\tfrac{\log n}{n})^{\frac{2(\beta-1)}{d+2\beta}}\right) + O\!\left((\tfrac{\mathrm{polylog}(n,\delta)}{n^2\varepsilon^2})^{\frac{\beta-1}{d+\beta}}\right)$. We also provide minimax lower bounds for private mode estimation, and show that our estimators are nearly optimal, up to a logarithmic factor in the MSE. We present two natural extensions: DP-PMS, a private modal-regression method, and DP-GRAMS-C, a clustering pipeline. Extensive experiments on synthetic and real data demonstrate favorable privacy–utility trade-offs relative to common baselines.
\end{abstract}

\section{Introduction}\label{intro}

Estimating the modes of a probability density -- its local maximizers -- is a central problem in statistics. Unlike global summaries such as the mean or median, modes reveal heterogeneous subpopulations and localized concentrations of probability mass \citep{chacon2015population, chen2016nonparametric}, making them indispensable in multimodal or complex settings. Applications span a wide range of domains, from clustering and classification \citep{avidan2007ensemble, chen2016comprehensive, li2007nonparametric} to computer vision tasks such as object tracking and image segmentation \citep{comaniciu2002mean, comaniciu2003kernel}, as well as nonlinear statistical modeling paradigms including manifold learning and modal regression \citep{einbeck2006fitting, chen2016nonparametric}. Recovering modes is inherently nonregular: small perturbations of the density may create or destroy critical points, and statistical difficulty depends delicately on local smoothness, curvature, and separation \citep{tsybakov2008nonparametric, genovese2014nonparametric}. 

In this paper, we study the problem of mode estimation while maintaining privacy of individual data points. This is motivated from a practical standpoint, where computing distribution summaries from sensitive data that routinely arise in domains such as healthcare or finance presents significant privacy challenges. Recent research has shown that even such summaries can compromise privacy \citep{oberski2020differential, dick2023confidence}. Differential privacy \citep{dwork2006our, dwork2014algorithmic} provides rigorous protection by ensuring that the output of an algorithm is nearly indistinguishable with or without any single individual. Considerable research has developed differentially private methods for means, regression, and clustering \citep[see][]{alabi2020differentially,cai2021cost,oberski2020differential, dankar2013practicing}, but private mode estimation has received little direct attention, with \citet{pacchiano2021robustness} as a notable exception. This motivates the current work, where we pose private mode estimation as a gradient-ascent problem and develop practically implementable techniques with statistical guarantees.

We use the natural characterization of modes as critical points of the 
score function \(\nabla_x\log p(x)\), and identify modes among these critical points through local Hessian conditions. In order to achieve this, we use a suitable estimate of the density, denoted \(\hat{p}(x)\), and the score function induced by it. To decrease the sensitivity of our score estimator from extreme observations, rather than working directly with \(\nabla_x\log \hat p(x)\), we use a stabilized score estimator, obtained by clipping the KDE gradient and flooring the KDE denominator. More precisely, for parameters \(A>0\) and \(p_{\mathrm{floor}}>0\), we define
\begin{equation}\label{eq:def-score-with-floor}
\hat s_{A,p_{\mathrm{floor}}}(x)
=
\frac{\operatorname{clip}_A(\nabla \hat p(x))}
{\max\{\hat p(x),\,p_{\mathrm{floor}}\}},
\qquad
\text{where}
\qquad
\operatorname{clip}_A(v)
:=
v\min\!\left\{1,\frac{A}{\|v\|_2}\right\}.
\end{equation}
To find critical points of the log-density, given an initialization, we use the estimated gradients at each iterate to run the gradient ascent steps:
\[
x_{t+1} = x_t + \eta \, \hat s_{A,p_{\mathrm{floor}}}(x_t),
\]
where \(\eta\) is the stepsize. Our density and score estimation are based on kernel smoothing. Given data \(X_1,\dots,X_n \in \mathbb{R}^d\) and kernel \(K\), we use the kernel density estimator \(\hat p(x)=\frac{1}{nh^d}\sum_{i=1}^nK\!\left(\frac{x-X_i}{h}\right)\) and its gradient $\nabla \hat{p}(x)$ to compute our score estimator. We note here that scaled gradient ascent on \(\log \hat p\) for Gaussian kernels coincides with mean shift \citep{fukunaga1975estimation, cheng1995mean, comaniciu2002mean}. 

Since we allow for multiple modes, we assume that the log density is concave locally, but not necessarily globally. To tackle estimation in this setting, we use a multi-tiered algorithm. First, we pick several initializations that are likely to be candidates for the mode. Second, we refine these initializations via the gradient ascent procedure mentioned above. Finally, we merge sufficiently close refinements to their local centers, thus allowing for a unified mode estimation process. To ensure differential privacy, we use several different mechanisms in the various stages of the algorithm:
\begin{enumerate}
   \item \textbf{Initialization:} we use a density-aware private initialization scheme called DAP, which computes a local empirical-mass utility over a public grid and samples anchor points from the grid via the \emph{exponential mechanism} \citep{mcsherry2007mechanism}. In addition to being density aware due to the exponential mechanism being tuned to the empirical-mass utility, the algorithm uses  a local suppression step after each selected anchor, so that the same neighborhood is not repeatedly selected and new high density regions are visited. With \(k\asymp M\log n\) draws, this scheme yields high-probability coverage of the neighborhood of each mode.
    
    \item \textbf{private stochastic gradient ascent:} we use the stabilized score estimator \(\hat s_{A,p_{\mathrm{floor}}}\) in \eqref{eq:def-score-with-floor} to run a stochastic gradient ascent and at each gradient ascent step, add suitably calibrated Gaussian noise. Similar ideas have been used earlier in parametric problems, via DP-SGD \citep{bassily2014differentially}.
    \item \textbf{correlated noise:} Since iterate sequences from multiple initializations can be close to each other, our privacy preserving noise is correlated across initializations. At iteration \(t\), the correlation structure is induced by an exponential kernel on the current iterate locations \(x_{t,1},\dots,x_{t,k}\). This 
    ensures that multiple initializations do not lead to undue privacy loss.
\end{enumerate}

\noindent We refer to the entire mechanism above as \textsc{DP-GRAMS} (\textbf{D}ifferentially \textbf{P}rivate
\textbf{Gr}adient \textbf{A}scent for \textbf{M}ode \textbf{S}eeking). We ensure that it satisfies \((\varepsilon,\delta)\)-differential privacy \citep[see, e.g.,][]{dwork2006differential}. The privacy budget \(\varepsilon\) is appropriately apportioned into the initialization and gradient ascent steps mentioned above.

On the utility side, we show that the final merged estimator \(\widehat{\mathcal M}\) contains, with high probability, for each true population mode \(\mu_j\), a released point \(\hat\mu_j\) satisfying
\[
\mathbb{E}\|\hat\mu_j-\mu_j\|^2
\;\lesssim\;
\left(\frac{\log n}{n}\right)^{\frac{2(\beta-1)}{d+2\beta}}
+
\left(\frac{\mathrm{polylog}(n,\delta)}{n^2\varepsilon^2}\right)^{\frac{\beta-1}{d+\beta}} .
\] Our results depend crucially on appropriate curvature assumptions and \(\beta\)-H{\"o}lder smoothness of the density function. Most importantly, we assume that modes are sufficiently separated and that the density is locally log-concave around each mode. Next, we use higher order kernels that can leverage advantages due to \(\beta\)-H{\"o}lder smoothness for \(\beta>2\). We complement the above results with the minimax lower bound
\[
\inf_{\hat x\in \mathcal{T}(n,\varepsilon,\delta)}\;
\sup_{p\in\mathcal P_\beta(L),\, \mu\in {\rm modes}(p)}
\EE_p\!\left[
\bigl\|
\hat{x} - \mu
\bigr\|^2
\right]
\;\gtrsim\;
\, n^{-\frac{2(\beta-1)}{d+2\beta}}
+ (n\varepsilon)^{-\frac{2(\beta-1)}{d+\beta}}
\]
where \(\delta=o(n^{-1})\) and \(\mathcal{T}(n,\varepsilon,\delta)\) is the set of all possible estimators based on a sample of size \(n\) and satisfying \((\varepsilon,\delta)\) differential privacy. Thus our estimators are nearly minimax optimal, since the MSE upper bounds match the lower bounds up to logarithmic terms. As expected, the MSE separates into a non-private rate given by the first term \citep[matching Theorem 3 of][]{tsybakov1990recursive} while the second captures the degradation required for privacy. 

The recovered private modes also motivate downstream procedures for modal regression and clustering. We develop:
(i) \textsc{DP-PMS}, which adapts the private ascent mechanism to conditional mode estimation by updating in the response direction, and
(ii) \textsc{DP-GRAMS-C}, which releases private modal centers for clustering and uses deterministic assignments as post-processing or evaluation. These procedures use the same private mode-seeking primitives and are studied empirically in Section~\ref{sec:apps}.

\paragraph{Related work.} 

Our work is closely related to a fast-growing literature on differentially private nonparametric methods. For density estimation, methods based on histograms, orthogonal series \citep{wasserman2010statistical}, and kernels \citep{hall2013differential, wagner2023fast, liu2024differentially} produce private density approximations. The closest prior work on private mode estimation is \citet{pacchiano2021robustness}, who perturb a \(k\)-nearest-neighbor mode estimate and prove a differential privacy guarantee. In contrast, our procedure is a kernel-smoothed score-ascent method for recovering multiple density modes and is accompanied by smoothness-dependent upper and lower error rates. Turning to clustering, several works are based on the \(k\)-means framework \citep{balcan2017differentially,ghazi2020differentially,stemmer2021locally,su2016differentially}, which optimize parametric objectives, but are not suited to nonparametric settings or irregular cluster shapes. Finally for regression tasks, most existing work on differential privacy focuses on approaches \citep{alabi2020differentially,arora2022differentially, cai2021cost,cai2024optimal, wang2018revisiting, sheffet2017differentially} for modeling the conditional mean, leaving nonparametric modal regression unexplored. Finally, while differentially private mode estimation has been studied in \cite{pacchiano2021robustness} in the context of bandits, the authors consider a single mode and nearest neighbor based estimators are studied. Instead, we  develop kernel based methods which lead to strictly improved rates when the density is sufficiently smooth. Note that nearest neighbor estimators are not equipped to take advantage of higher order smoothness. More importantly, we consider multiple modes, a feature that leads to significant complexity due to the inherent tension with privacy. While identifying each mode requires sufficiently granular data distribution, privacy requirements dictate that individual datum are still not distinguished. We therefore develop differentially private algorithms in a setting inspired by the study of mixture models: where the density modes are assumed to sufficiently separated and strongly identified.

We also note the connection of our work with score-based denoising methods
\citep[see, for e.g.,][]{ghosh2025stein,wibisono2024optimal}. The need to preserve differential privacy in our case necessitates a combination between denoising step (score gradient ascent) and adding noise. While we use a specific kernel based score estimator, our approach can potentially be extended to the use of deep learning based score estimators and their differentially private versions.

\medskip

\noindent\textbf{Organization.}
The remainder of the paper is organized as follows. Section~\ref{background} reviews some relevant background. Section~\ref{sec:dp-grams} introduces the \textsc{DP-GRAMS} algorithm. Section~\ref{sec:theory} states the main theoretical results. Section~\ref{sec:apps} contains comprehensive empirical evaluation and implementation details. We conclude with a discussion of future work in Section~\ref{sec:discussion}. Proofs are deferred to Appendix~\ref{sec:main-proofs}; additional experimental results are collected in Appendix~\ref{app:additional-sim}; and downstream pseudocode is given in Appendix~\ref{sec:apndx-algo}.


\section{Background}\label{background}

This section reviews the ingredients underlying \textsc{DP-GRAMS}: population
density modes, score-based mode characterization, kernel density estimation, and
differential privacy.

\subsection{Modes as Statistical Objects}

Our inferential target is the set of modes of a population density \(p:\mathbb{R}^d\to[0,\infty)\). A point \(\mu\in\mathbb{R}^d\) is a \emph{mode} if it is a strict local maximizer of \(p\). A standard sufficient second-order condition is
\[
\left.\nabla_x p(x)\right|_{x=\mu}=0
\qquad\text{and}\qquad
\left.\nabla_x^2 p(x)\right|_{x=\mu}\prec 0.
\]
For multimodal densities, the parameter of interest is therefore a finite, unordered collection of such critical points. We assume that \(p\) is strictly positive near each mode. Since the critical points of \(p\) coincide with those of \(\ell(x)=\log p(x)\), 
modes may equivalently be defined as local maximizers of \(\log p\). The gradient \(s(x):=\nabla_x\log p(x)\) is the score function, and density modes are characterized by zeros of \(s(x)\) together with the corresponding local curvature condition. In the private algorithm, we work with a stabilized version of $\hat s(x)$ defined below. All derivatives are taken with respect to the argument $x$; for notational simplicity, we omit the subscript $x$ when no ambiguity arises.

\subsection{Kernel Density Estimation}

Given independent and identically distributed (i.i.d.) samples $X_1,\dots,X_n\in\mathbb{R}^d$, the kernel density estimator (KDE, \cite{chen2017tutorial}) with bandwidth $h>0$ is
\[
\hat p(x)=\frac{1}{n h^d}\sum_{i=1}^n
K\!\left(\frac{x-X_i}{h}\right),
\]
where $K$ is a kernel function integrating to one. 
The detailed conditions are stated in Section~\ref{sec:theory}. 
Under standard regularity conditions, $\hat p$ is smooth and admits well-defined derivatives. On regions where $\hat p(x)>0$, it induces an estimator of the score function,
\[
\nabla \log \hat p(x)=\frac{\nabla \hat p(x)}{\hat p(x)}.
\]
To decrease the sensitivity of our score estimator from extreme observations, rather than working directly with \(\nabla_x\log \hat p(x)\), we use a stabilized score estimator. Given a clipping level $A>0$ and a density floor $p_{\mathrm{floor}}>0$, define
\begin{equation}\label{eq:stabilized-score}
\hat s_{A,p_{\mathrm{floor}}}(x)
=
\frac{\operatorname{clip}_A(\nabla \hat p(x))}
{\max\{\hat p(x),\,p_{\mathrm{floor}}\}},
\quad
\text{where}
\quad 
\operatorname{clip}_A(v):=
v\min\!\left\{1,\frac{A}{\|v\|_2}\right\}.
\end{equation}
Whenever $\hat p(x)\ge p_{\mathrm{floor}}$ and $\|\nabla \hat p(x)\|_2\le A$, this coincides with the ordinary score estimator $\nabla\log\hat p(x)$. We focus on the behavior of $\hat p$ and its derivatives near the population modes.

\subsection{Mode Estimation via Gradient Ascent}
To find critical points of the score function we employ a gradient ascent scheme on the score estimate:
\[
x_{t+1}=x_t+\eta\,\nabla\log\hat p(x_t).
\]
For Gaussian kernels, this gradient ascent on $\log\hat p$ matches exactly with the classical procedure of mean shift \citep{fukunaga1975estimation, cheng1995mean, comaniciu2002mean}. Unlike mean shift, the general gradient ascent formulation applies to arbitrary differentiable kernels. For our differentially private algorithm, we replace $\nabla\log\hat p(x_t)$ by the stabilized score estimator $\hat s_{A,p_{\mathrm{floor}}}(x_t)$ and add Gaussian perturbations calibrated to its sensitivity. Differentiating the KDE, one has
\[
\nabla \hat p(x)=\frac{1}{n h^{d+1}}\sum_{i=1}^n
\nabla K\!\left(\frac{x-X_i}{h}\right).
\]
Let us define
\[
g_i(x):=\frac{1}{h^{d+1}}\nabla K\!\left(\frac{x-X_i}{h}\right).
\]
Thus 
the stabilized score defined in \eqref{eq:stabilized-score} may be written as
\[
\hat s_{A,p_{\mathrm{floor}}}(x)
=
\frac{1}
{\max\!\left\{\hat p(x),\,p_{\mathrm{floor}}\right\}}
\operatorname{clip}_A\!\left(\frac{1}{n}\sum_{i=1}^n g_i(x)\right)
.
\]


\subsection{Differential Privacy}

Differential privacy \citep{dwork2006our, dwork2014algorithmic} is a popular choice for guaranteeing protection of individual data.

\begin{definition}[Differential privacy {\citep[Definition~2.4]{dwork2014algorithmic}}]
Let $\varepsilon>0$ and $\delta\in[0,1)$. A randomized mechanism $\mathcal{M}:\mathcal{X}^n\to\mathcal{Y}$ is
$(\varepsilon,\delta)$-differentially private if for all neighboring datasets
$\mathcal{X},\mathcal{X}'$ differing in one entry and all measurable
$A\subseteq\mathcal{Y}$,
\[
\Pr[\mathcal{M}(\mathcal{X})\in A \mid \mathcal{X}]
\le {\rm e}^\varepsilon \Pr[\mathcal{M}(\mathcal{X}')\in A\mid\mathcal{X}']+\delta.
\]
\end{definition}

\noindent We aim to find mode estimates that satisfy the above notion of privacy. To enforce such requirements, we use a stabilized score estimator together with Gaussian perturbations, placing our method within the framework of
differentially private stochastic optimization. Our privacy analysis relies on standard tools including the Gaussian mechanism, privacy amplification by subsampling, composition, and post-processing invariance \citep[see, e.g.,][]{dwork2014algorithmic,balle2018privacy,dwork2010boosting}.


\section{Differentially Private Mode Seeking Algorithm}\label{sec:dp-grams}

In this section, we introduce our differentially private mode estimation algorithm. At a high level, the algorithm, called \textsc{DP-GRAMS}, proceeds in
three stages:
(i) it generates multiple initial points concentrated in high-density
regions with privacy guarantees,
(ii) it refines each initialization via noisy ascent on the estimated score function,
and
(iii) it merges nearby outputs to produce a final set of private modes.
Privacy and statistical guarantees for the full pipeline are established in
Section~\ref{sec:theory}. The complete procedure is summarized in
Algorithms~\ref{alg:dpgrams} and~\ref{alg:dap-init}. Throughout, we write
\(
\varepsilon=\varepsilon_{\mathrm{init}}+\varepsilon_{\mathrm{modes}},
\)
where \(\varepsilon_{\mathrm{init}}\) is the privacy budget allocated to initialization and \(\varepsilon_{\mathrm{modes}}\) to the gradient ascent stage. We now describe
the different stages of this algorithm in more detail.

\begin{algorithm}[h]
\caption{DP-GRAMS: Differentially Private GRadient Ascent for Mode Seeking}
\label{alg:dpgrams}
\SetKwInOut{Input}{Input}
\SetKwInOut{Output}{Output}

\Input{
Private data \(S=\{X_i\}_{i=1}^n\subset\mathbb{R}^d\); fixed public finite candidate set
\(\mathcal Z=\{z_j\}_{j=1}^{N_{\mathrm{cand}}}\subset\mathbb{R}^d\); privacy parameters \((\varepsilon,\delta)\) with \(\varepsilon=\varepsilon_{\mathrm{init}}+\varepsilon_{\mathrm{modes}}\); minibatch size \(m\); number of iterations \(T\); number of DAP draws \(k\in\mathbb N\); DAP bandwidth \(h_{\mathrm{DAP}}>0\); suppression radius \(\rho_{\mathrm{init}}>0\); ascent bandwidth \(h_{\mathrm{mode}}>0\); stepsize \(\eta>0\); gradient clipping level \(A>0\); density floor \(p_{\mathrm{floor}}>0\).
}
\Output{Final private mode estimator \(\widehat{\mathcal M}\)}

Generate private initializations
\(\mathcal{I}=\{x_{0,1},\dots,x_{0,k}\}\)
using Algorithm~\ref{alg:dap-init} with candidate set \(\mathcal Z\), \(k\) draws, bandwidth \(h_{\mathrm{DAP}}\), and suppression radius \(\rho_{\mathrm{init}}\); write
\(a_\ell=x_{0,\ell}\) for the sampled anchors\;

Compute the noise scale \(\sigma\) according to \eqref{eq:sigma-corr}, using \(h=h_{\mathrm{mode}}\)\;

\For{\(t=0,\dots,T-1\)}{

Sample a minibatch \(\mathcal B_t\subset[n]\) uniformly without replacement, 
with \(|\mathcal B_t|=m\)\;

Compute the current correlation matrix
\[
(\mathbf K_t)_{\ell r}
=
\bar C_{h_{\mathrm{mode}}}(x_{t,\ell},x_{t,r})
=
\exp\!\left(-\frac{\|x_{t,\ell}-x_{t,r}\|}{h_{\mathrm{mode}}}\right),
\qquad
1\le \ell,r\le k
\]

Draw \(Z_t\in\mathbb{R}^{k\times d}\) with independent columns
\(
(Z_t)_{\cdot,j}\sim \mathcal N(0,\sigma^2\mathbf K_t),
\qquad j=1,\dots,d.
\)

\For{\(\ell=1,\dots,k\)}{

Compute the stabilized minibatch score
\[
\hat s_{A,p_{\mathrm{floor}};\mathcal B_t}(x_{t,\ell})
=
\frac{\operatorname{clip}_A(\nabla \hat p_{\mathcal B_t}(x_{t,\ell}))}
{\max\{\hat p_{\mathcal B_t}(x_{t,\ell}),\,p_{\mathrm{floor}}\}}
\]

Update
\[
x_{t+1,\ell}
=
x_{t,\ell}
+
\eta\Big(
\hat s_{A,p_{\mathrm{floor}};\mathcal B_t}(x_{t,\ell})
+
(Z_t)_{\ell,\cdot}
\Big)
\]

}

}

Set \(\widetilde{\mathcal M}=\{x_{T,1},\dots,x_{T,k}\}\)\;

Merge nearby points in \(\widetilde{\mathcal M}\) to obtain \(\widehat{\mathcal M}\)\;

\Return \(\widehat{\mathcal M}\)
\end{algorithm}

\begin{algorithm}[!htb]
\caption{Density-Aware Private (DAP) Initialization}
\label{alg:dap-init}
\SetKwInOut{Input}{Input}
\SetKwInOut{Output}{Output}

\Input{
Private data \(S=\{X_i\}_{i=1}^n\subset\mathbb{R}^d\); public candidate set
\(\mathcal Z=\{z_j\}_{j=1}^{N_{\mathrm{cand}}}\subset\mathbb{R}^d\) (see Section~\ref{dap-grid-construction}); privacy budget \(\varepsilon_{\mathrm{init}}\); number of DAP draws \(k\in\mathbb N\); DAP bandwidth \(h_{\mathrm{DAP}}>0\); suppression radius \(\rho_{\mathrm{init}}>0\).
}
\Output{Private initialization set \(\mathcal I=\{x_{0,1},\dots,x_{0,k}\}\)}

Set
\(
\varepsilon_{\mathrm{draw}}=\varepsilon_{\mathrm{init}}/k,
\qquad
A_1=[N_{\mathrm{cand}}].
\)

\For{\(j=1,\dots,N_{\mathrm{cand}}\)}{
Compute
\[
u_j=\frac1n\sum_{i=1}^n \mathbf 1\{\|X_i-z_j\|\le h_{\mathrm{DAP}}\}.
\]
}

\For{\(\ell=1,\dots,k\)}{
\If{\(A_\ell=\varnothing\)}{
set \(A_\ell=[N_{\mathrm{cand}}]\)\;
}
Sample \(J_\ell\in A_\ell\) using
\[
\Pr(J_\ell=j\mid A_\ell)
\propto
\exp\!\left(\frac{n\varepsilon_{\mathrm{draw}}}{2}u_j\right)\mathbf 1\{j\in A_\ell\},
\qquad
j\in[N_{\mathrm{cand}}].
\]

Set \(a_\ell=z_{J_\ell}\) and \(x_{0,\ell}=a_\ell\)\;

Update
\[
A_{\ell+1}
=
A_\ell\setminus
\bigl\{
j\in[N_{\mathrm{cand}}]:
\|z_j-a_\ell\|\le \rho_{\mathrm{init}}
\bigr\}.
\]
}

\Return \(\mathcal I=\{x_{0,1},\dots,x_{0,k}\}\)\;
\end{algorithm}

\medskip


\noindent\textbf{Initialization.}
A central challenge in differentially private estimation of \emph{multiple} modes is generating initial points that both respect privacy and lie in the modal basins of the true modes. To address this, we use a Density-Aware Private (DAP) initialization scheme in Algorithm~\ref{alg:dap-init}, which selects initializations from high-density regions while ensuring privacy via the exponential mechanism \citep{mcsherry2007mechanism}. To ensure that \emph{all} high density neighborhoods are visited, we combine this algorithm with a local suppression step.

More specifically, over a fixed public candidate set \(\mathcal Z=\{z_j\}_{j=1}^{N_{\mathrm{cand}}}\), we define the local empirical-mass utility
\begin{equation}
u_j=\frac1n\sum_{i=1}^n \mathbf 1\{\|X_i-z_j\|\le h_{\mathrm{DAP}}\},
\qquad j=1,\dots,N_{\mathrm{cand}},
\label{eq:uj-def}
\end{equation}
which approximates the local probability mass near \(z_j\) at scale \(h_{\mathrm{DAP}}\). Since the sensitivity of \(u_j\) is \(1/n\), the exponential mechanism uses weights proportional to
\[
\exp\!\left(\frac{n\varepsilon_{\mathrm{draw}}}{2}u_j\right),
\qquad
\varepsilon_{\mathrm{draw}}=\varepsilon_{\mathrm{init}}/k.
\]
The anchor draws are performed one at a time. At round \(\ell\), \(A_\ell\) denotes the current active candidate set after suppression. The utility \(u_j\) makes the scheme density-aware, since candidates with larger local empirical mass receive larger exponential-mechanism weight. After each selected anchor, we remove candidate points in a \(\rho_{\mathrm{init}}\)-neighborhood of that anchor before the next draw, encouraging the algorithm to visit new high density regions and thus find previously unexplored modes. If the active set becomes empty before all \(k\) rounds are completed, the full candidate set is reopened. The selected candidate locations are used directly as the initialization set \(\mathcal I=\{x_{0,1},\dots,x_{0,k}\}\). The privacy guarantee for this stage is given by Theorem~\ref{thm:dp_init_privacy}. In Section~\ref{sec:theory}, we show that with a public \(h_{\mathrm{DAP}}\)-grid and \(k\asymp M\log n\) draws, DAP places at least one initialization in each modal basin with high probability. Similar density-aware initialization ideas have been used earlier in \cite{rodriguez2014clustering,su2016differentially,li2016effective,fan2023k}. In the rest of this paper, we write \(a_1,\dots,a_k\) for the sampled anchors, so \(x_{0,\ell}=a_\ell\) for each \(\ell\).

\medskip

\noindent\textbf{Privatized Gradient Ascent.}
For the ascent stage, we use a possibly different bandwidth \(h_{\mathrm{mode}}\), chosen according to \eqref{h-opt}. 
Given the KDE \(\hat p\), 
\textsc{DP-GRAMS} iteratively updates candidate modes via ascent on an estimated score function. For any finite dataset \(\mathcal X=\{X_1,\dots,X_n\}\subset\mathbb{R}^d\), define
\begin{equation}
\hat p_{\mathcal X}(x)
=
\frac{1}{n h_{\mathrm{mode}}^d}\sum_{i=1}^n
K\!\left(\frac{x-X_i}{h_{\mathrm{mode}}}\right),
\qquad
\nabla \hat p_{\mathcal X}(x)
=
\frac{1}{n h_{\mathrm{mode}}^{d+1}}\sum_{i=1}^n
\nabla K\!\left(\frac{x-X_i}{h_{\mathrm{mode}}}\right).
\label{eq:kde-D}
\end{equation}
Rather than working directly with the ordinary score \(\nabla\log\hat p_{\mathcal X}(x)\), we use the stabilized score estimator in \eqref{eq:stabilized-score}. 
In \textsc{DP-GRAMS}, the ascent update at round \(t\) is built from the minibatch field \(\hat s_{A,p_{\mathrm{floor}};\mathcal B_t}(x)\) evaluated at the current iterates. Each candidate in the initialization pool \(\mathcal I\) undergoes \(T\) iterations of minibatch ascent on the stabilized score function built using the bandwidth \(h_{\mathrm{mode}}\). In each iteration, a minibatch \(\mathcal B_t\) of size \(m\) is sampled uniformly without replacement, the minibatch score \(\hat s_{A,p_{\mathrm{floor}};\mathcal B_t}(x_{t,\ell})\) is computed at each current iterate, and the resulting vector field is perturbed by Gaussian noise. Candidate modes are then updated using the fixed stepsize \(\eta>0\). After completing \(T\) iterations, the algorithm produces a set of private candidate modes \(\widetilde{\mathcal M}=\{x_{T,1},\dots,x_{T,k}\}\).

\medskip

\noindent\textbf{Correlated Noise.}
The use of multiple initializations means that the same data are reused across several ascent trajectories. To account for this, at each iteration we use noise vectors that are independent across iterations and coordinates, but correlated across initializations. More specifically, at round \(t\), after the current iterates \(x_{t,1},\dots,x_{t,k}\) are determined, we construct the correlation matrix
\[
(\mathbf K_t)_{\ell r}
=
\bar C_{h_{\mathrm{mode}}}(x_{t,\ell},x_{t,r})
=
\exp\!\left(-\frac{\|x_{t,\ell}-x_{t,r}\|}{h_{\mathrm{mode}}}\right),
\qquad
1\le \ell,r\le k,
\]
based on the exponential kernel. The noise matrix \(Z_t\in\mathbb R^{k\times d}\) then has independent columns with covariance \(\sigma^2\mathbf K_t\). Thus trajectories whose current iterates are nearby receive more strongly correlated perturbations than trajectories whose current iterates are far apart. This is the correlated-noise mechanism used throughout \textsc{DP-GRAMS}.

\medskip

\noindent\textbf{Merging Candidate Modes.}
Noise in gradient updates and randomness in initialization can produce multiple points corresponding to the same population mode. To reduce this redundancy, \textsc{DP-GRAMS} applies a final post-processing merge. When the number of modes is unknown, we use a radius-based merge with radius \(h_{\mathrm{mode}}\), grouping candidate modes within distance \(h_{\mathrm{mode}}\) and replacing each group by its mean. When the number of modes is known, we use Ward-linkage agglomerative clustering with target number of clusters equal to that mode count, again replacing each final cluster by its mean. The final merged output is denoted by \(\widehat{\mathcal M}\).

\subsection{Example: Modes of a Bivariate Gaussian}\label{sec:gaussian-example}

\begin{figure}[!htb]
    \centering
    \includegraphics[height=0.28\linewidth]{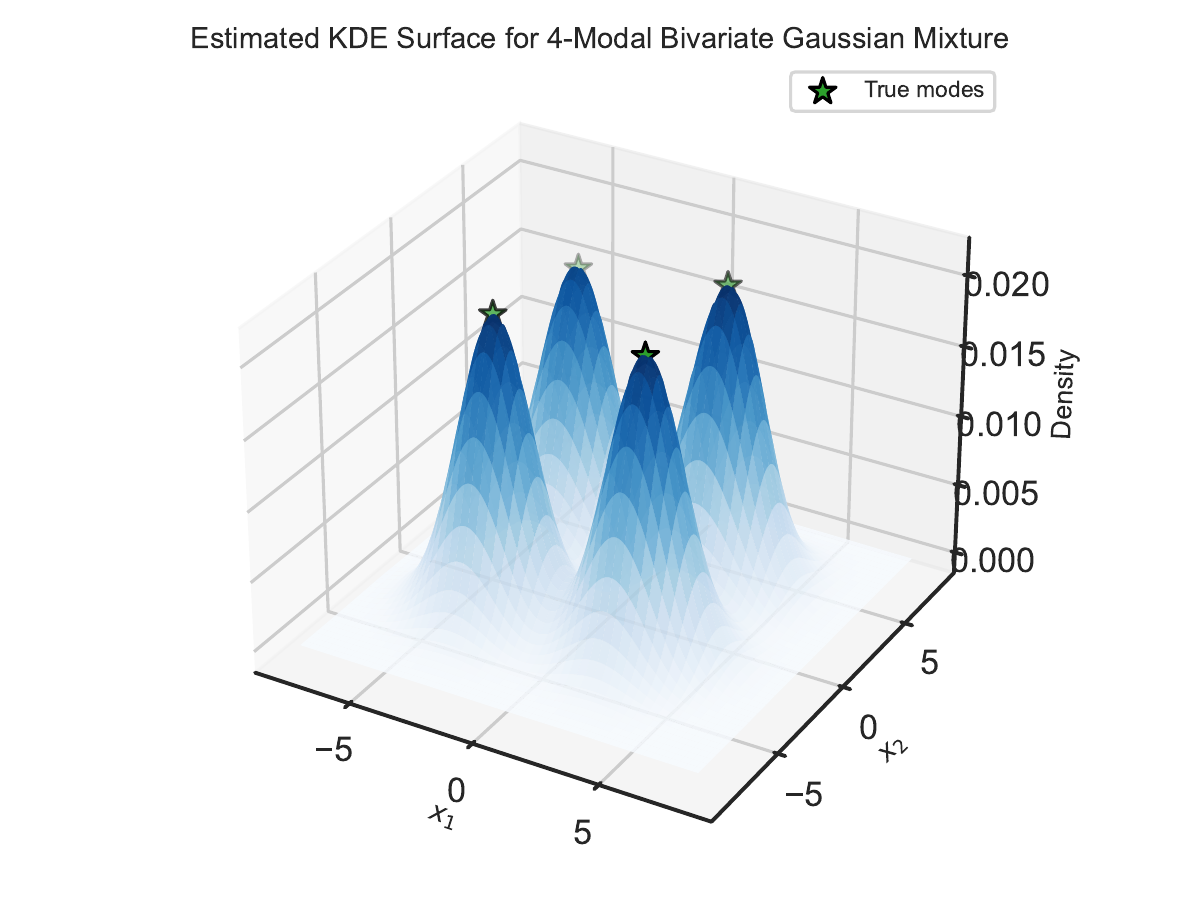}
    \qquad
    \includegraphics[height=0.25\textheight]{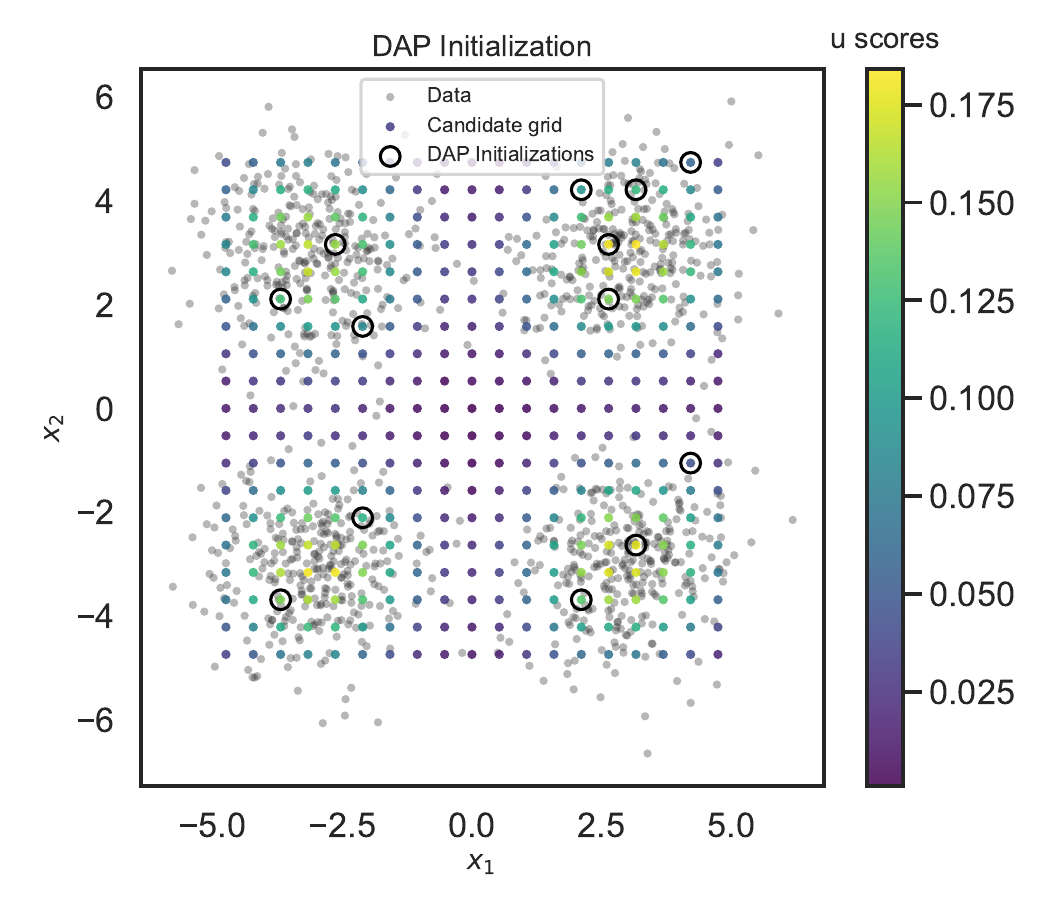}\\
    [0.4em]
    \includegraphics[height=0.28\linewidth]{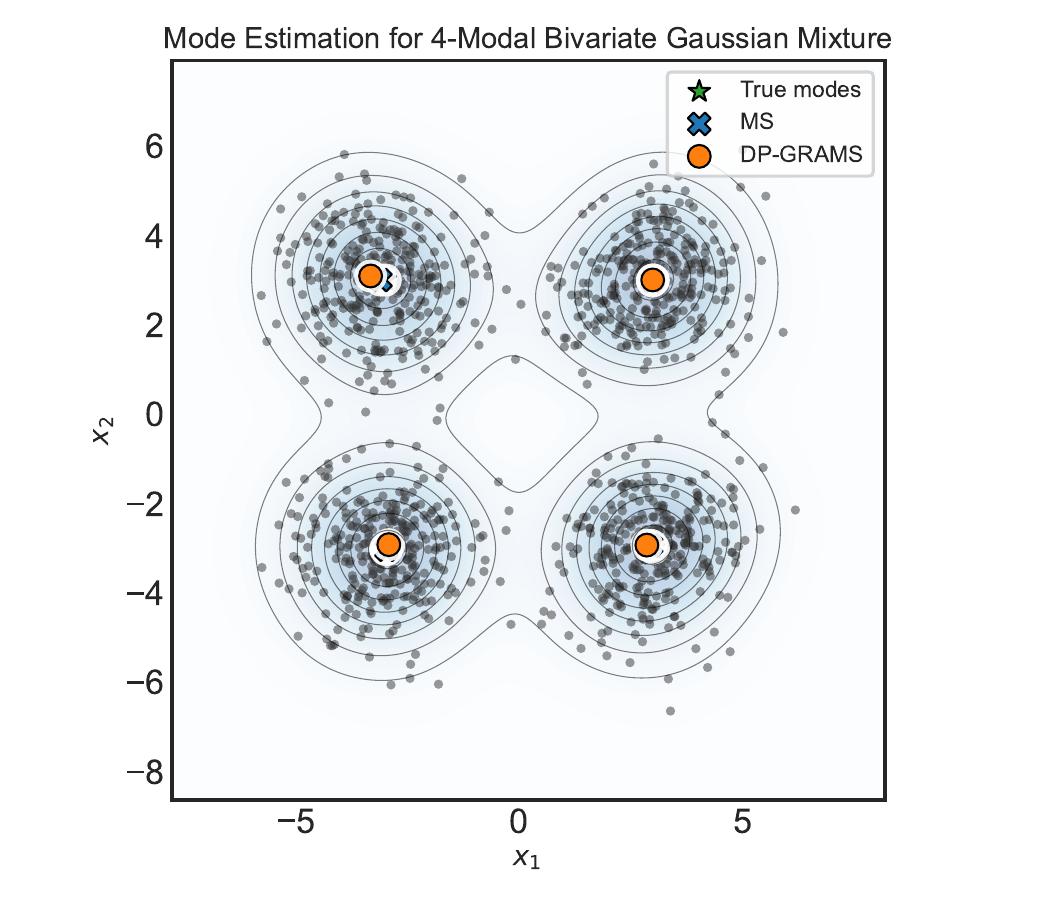}
    \qquad
    \includegraphics[height=0.28\linewidth]{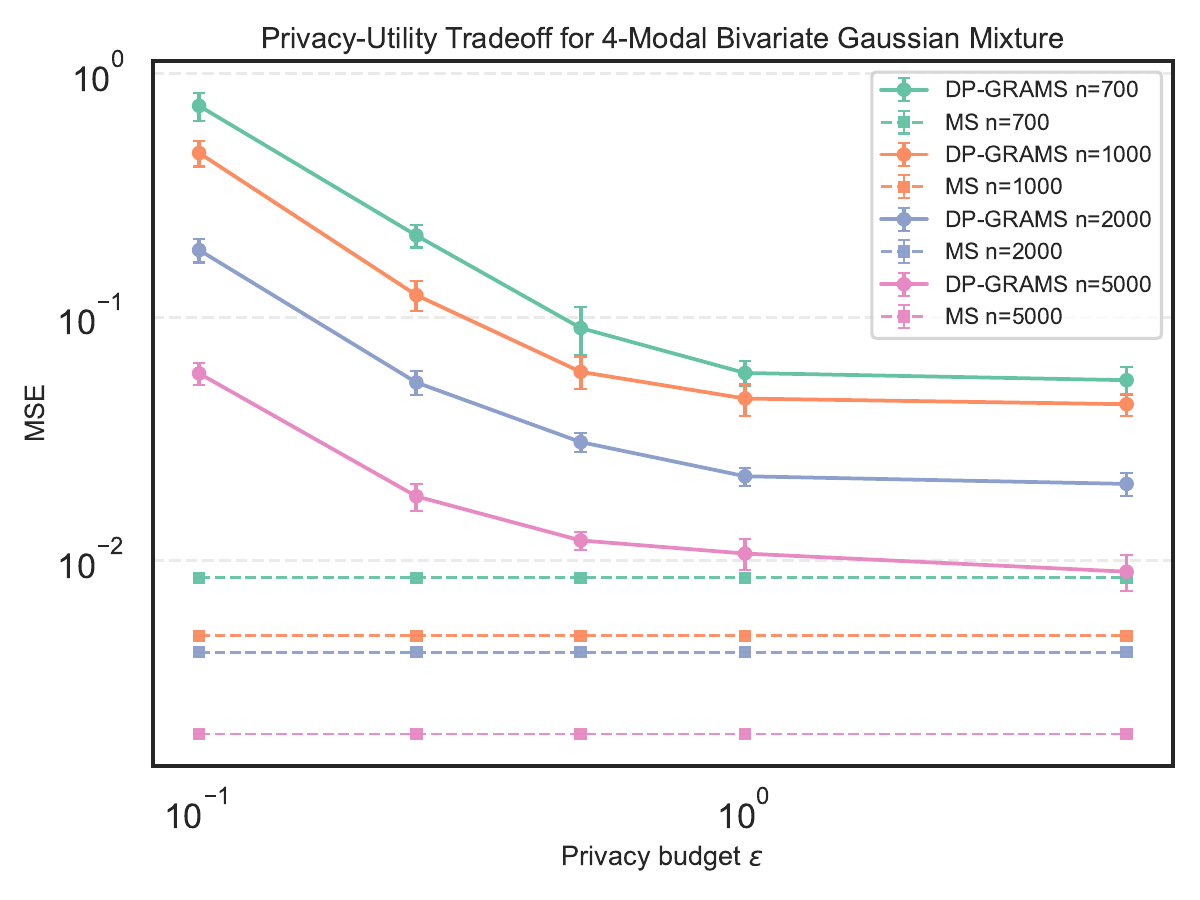}

        \caption{\small
    (a) Estimated KDE surface with true modes overlaid.
    (b) DAP initialization: candidate points colored by the local empirical-mass utility \(u_j\) in \eqref{eq:uj-def}, together with the privately selected anchors used as the starting points for \textsc{DP-GRAMS}.
    (c) Contour plot comparing true (green), mean shift (blue), and \textsc{DP-GRAMS} (orange) mode estimates on a single dataset.
    (d) Privacy--utility tradeoff: MSE vs.\ \(\varepsilon\) on a log scale for \(n\in\{700,1000,2000,5000\}\) and \(\varepsilon\in\{0.1,0.25,0.5,1,5\}\). Curves report averages over 20 runs with standard-error bars; dashed lines show non-private mean-shift baselines and solid curves show \textsc{DP-GRAMS}.}
    \label{fig:bivariate_gauss_fourpanel}
\end{figure}

We use this example to visualize the initialization and final estimates from \textsc{DP-GRAMS} in a clean, well-separated modal landscape. Data are generated from the four-component bivariate Gaussian mixture
\[
(X,Y) \sim \tfrac{1}{4}\sum_{k=1}^{4}\mathcal{N}(\mu_k,I_2),
\qquad
\mu_1=(3,3),\;\mu_2=(3,-3),\;\mu_3=(-3,3),\;\mu_4=(-3,-3),
\]
which has four population modes at the corners of a square. Panels (a)--(c) are produced from one representative dataset with \(n=1200\) and \((\varepsilon,\delta)=(1,10^{-6})\). The bandwidth is chosen by Silverman's rule.

Figure~\ref{fig:bivariate_gauss_fourpanel} shows that the DAP initialization concentrates anchors in the four high-density basins, while the suppression step discourages redundant selections from the same local neighborhood. After private ascent, the final \textsc{DP-GRAMS} estimates remain close to both the population modes and the non-private mean-shift outputs. The privacy--utility curves show the largest gains when moving from very strict privacy to moderate privacy budgets, with further improvement as \(n\) increases. At the largest sample sizes and privacy budgets, the private estimates are close to the non-private mean-shift baseline in this well-separated benchmark.

\section{Privacy Guarantees and Estimation Error Bounds}\label{sec:theory}

This section develops the privacy and utility guarantees for \textsc{DP-GRAMS}. We first state the regularity conditions under which the KDE and its derivatives are well behaved near the population modes. We then analyze privacy for the initialization and ascent stages, derive a basinwise error bound from a good initialization, and finally combine that local control with the DAP coverage argument to obtain global recovery of the final merged estimator.

We first formalize the appropriate notions of local smoothness of the true density, and introduce higher-order kernels. For any \(r>0\), we write the  Euclidean ball as \(\overline{B}(\mu,r):=\{x\in\R^d:\|x-\mu\|\le r\}\).

\begin{definition}[Local H\"older class, {\citep[Definition~1.2 in ][]{tsybakov2008nonparametric}}]\label{def:holder}
Let \(\mathcal U\subset\mathbb R^d\) be open, and let \(\beta,L>0\). A function \(f:\mathcal U\to\mathbb R\) belongs to \(\Sigma(\beta,L;\mathcal U)\) if \(f\) is \(\lfloor\beta\rfloor\)-times continuously differentiable on \(\mathcal U\) and
\[
\big\|D^{\lfloor\beta\rfloor} f(x)-D^{\lfloor\beta\rfloor} f(x')\big\|
\le
L\,\|x-x'\|^{\beta-\lfloor\beta\rfloor},
\qquad x,x'\in\mathcal U.
\]
We say \(f\) is locally \(\beta\)-H\"older around \(\mu\) on radius \(r\) if some open \(\mathcal U\) satisfies \(\overline B(\mu,r)\subset\mathcal U\) and \(f\in\Sigma(\beta,L;\mathcal U)\).
\end{definition}

\begin{definition}[Kernel of order \(\ell\), {\citep[Definition~1.3 in][]{tsybakov2008nonparametric}}]\label{def:kernel-order}
A function \(K:\mathbb R^d\to\mathbb R\) is a kernel of order \(\ell\ge1\) if
\[
\int K(u)\,du=1,
\qquad
\int u^\alpha K(u)\,du=0
\quad\text{for all }1\le|\alpha|\le\ell,
\]
and \(u^\alpha K(u)\) is integrable for all \(|\alpha|\le\ell\).
\end{definition}

Before presenting our theoretical results, we present the required assumptions. We begin with regularity conditions on the smoothing kernel, which will be used to control KDE bias, stochastic fluctuation, and the argument underlying the correlated-noise mechanism.

\begin{assumption}[Kernel regularity]\label{assump:kernel}
Let \(\beta>2\), \(\ell=\lfloor\beta\rfloor\), and let \(K:\mathbb R^d\to\mathbb R\) be a kernel of order \(\ell\) as defined in Definition~\ref{def:kernel-order}. Define \(K_\infty:=\|K\|_\infty\) and \(G_K:=\sup_{u\in\mathbb R^d}\|\nabla K(u)\|\). Assume:
\begin{enumerate}[label={(\roman*)}]
\item \(\sup_{u\in\mathbb R^d}|\partial^\alpha K(u)|<\infty\) for all multi-indices \(\alpha\) with \(|\alpha|\le3\).
\item For all multi-indices \(\alpha\) with \(|\alpha|\le2\), \((\partial^\alpha K)^2\in L^1(\mathbb R^d)\) and \(\|u\|^\beta|\partial^\alpha K(u)|\in L^1(\mathbb R^d)\); also, \(\|\nabla K\|^4\in L^1(\mathbb R^d)\).
\item For all multi-indices \(\alpha,\alpha',\gamma\) with \(|\alpha|\le2\), \(\alpha'\le\alpha\), and \(|\gamma|\le\ell\), \(u^\gamma\partial^\alpha K(u)\in L^1(\mathbb R^d)\) and \(\lim_{\|u\|\to\infty}u^\gamma\partial^{\alpha'}K(u)=0\).
\item \(\partial^\alpha K\in L^1(\mathbb R^d)\) for all multi-indices \(\alpha\) with \(|\alpha|\le d+2\).
\end{enumerate}
\end{assumption}

We next impose local positivity and smoothness conditions on the population density in neighborhoods of the true modes.

\begin{assumption}[Model assumptions]\label{assump:model-holder}
Assume $\beta>2$. Let $p:\mathbb{R}^d\to[0,\infty)$ be a density taking finite values, with exactly $M$
local modes at distinct points $\mu_1,\dots,\mu_M$ such that
\[
\min_{i\neq j}\|\mu_i-\mu_j\|>c_0
\]
for a constant $c_0>0$. Moreover, for each $j\in[M]$, there exist $r_j>0$, an open neighborhood $\mathcal U_j$ with
$\overline B(\mu_j,r_j)\subset\mathcal U_j$, and $L_j>0$ such that:
\begin{enumerate}[label={(\roman*)}]
\item $p_{\min,j}:=\inf_{x\in\overline B(\mu_j,r_j)}p(x)>c_1$, $p_{\max}:=\sup_{x\in\R^d} p(x)<\frac{1}{c_1}$ for a constant $c_1>0$.
\item $p\in\Sigma(\beta,L_j;\mathcal U_j)$, where $\Sigma(\beta,L_j;\mathcal U_j)$ is as defined in Definition~\ref{def:holder}.
\end{enumerate}

\end{assumption}
In particular, (ii) implies $p\in C^2(\mathcal U_j)$, and together with (i) the log-density
$\ell(x)=\log p(x)$ is well-defined on $\overline B(\mu_j,r_j)$.

\begin{remark}[Spurious modes]
Mode estimation is sensitive to outliers, which may create spurious local maxima in finite samples. Assumption~\ref{assump:model-holder}(i) rules out population modes supported on vanishing mass by requiring $p$ to be bounded away from zero near each $\mu_j$. Algorithmically, even if an outlier is selected as an anchor, its influence is controlled by gradient clipping at level \(A\) together with the density floor \(p_{\mathrm{floor}}\).
\end{remark}

The final assumption in this preliminary block specifies the bandwidth regime used throughout the asymptotic analysis.

\begin{assumption}[Bandwidth condition]\label{assump:bandwidth}
Let $(h_n)_{n\ge 1}$ be the bandwidth sequence. Assume
\[
h_n\downarrow 0
\qquad\text{and}\qquad
\frac{n\,h_n^{\,d+4}}{\log n}\to\infty
\qquad\text{as }n\to\infty.
\]
\end{assumption}

\begin{remark}[Bandwidth notation]\label{rem:bandwidth-notation}
In the remainder of Section~\ref{sec:theory}, unless stated otherwise, the bare bandwidth \(h\) refers to the ascent bandwidth \(h_{\mathrm{mode}}\) introduced in Section~\ref{sec:dp-grams}. The DAP initialization stage uses \(h_{\mathrm{DAP}}\), which is written explicitly in the DAP construction, proposition, and proofs below.
\end{remark}

Together, these conditions yield uniform control of $\hat p$ and its derivatives
near modes, which we use for both privacy calibration and convergence.

\subsection{Privacy and Sensitivity}

We organize the privacy analysis in three steps. We first quantify the pointwise sensitivity of the stabilized score estimator, then calibrate the corresponding single-start Gaussian mechanism, and finally analyze the correlated multi-start mechanism used by \textsc{DP-GRAMS}. 

\begin{lemma}[Deterministic sensitivity of the stabilized score]\label{lem:sensitivity}
Assume Assumption~\ref{assump:kernel}. Let \(\mathcal X=(X_1,\dots,X_n)\) and \(\mathcal X'=(X_1',\dots,X_n')\) be neighboring datasets, differing in one entry. Then, for every \(x\in\mathbb R^d\),
\[
\bigl\|
\hat s_{A,p_{\mathrm{floor}};\mathcal X}(x)
-
\hat s_{A,p_{\mathrm{floor}};\mathcal X'}(x)
\bigr\|
\le
\frac{S_h(A,p_{\mathrm{floor}})}{n},
\]
where
\begin{equation}\label{eq:def-Sh}
S_h(A,p_{\mathrm{floor}})
:=
\frac{2G_K}{p_{\mathrm{floor}}}\,h^{-(d+1)}
+
\frac{2A K_\infty}{p_{\mathrm{floor}}^2}\,h^{-d}.
\end{equation}
\end{lemma}

Lemma~\ref{lem:sensitivity} identifies the relevant deterministic pointwise sensitivity of the stabilized score field, which directly yields a privacy calibration for a single ascent trajectory. Let us set
\[
\delta_{\mathrm{iter}}=\frac{\delta}{2T},
\qquad
\varepsilon_{\mathrm{iter}}
=
\min\!\left\{
\log\!\left(1+({\rm e}-1)\frac{m}{n}\right),\;
\frac{\varepsilon_{\mathrm{modes}}}{2\sqrt{2T\ln(2/\delta)}},\;
\sqrt{\frac{\varepsilon_{\mathrm{modes}}}{4T}}
\right\}.
\]
For a single trajectory, we take
\begin{equation}\label{sigma-exact}
\sigma
=
\frac{S_h(A,p_{\mathrm{floor}})/m}{\log\!\bigl(1+n({\rm e}^{\varepsilon_{\mathrm{iter}}}-1)/m\bigr)}
\sqrt{2\log\!\left(\frac{2.5mT}{n\delta}\right)}.
\end{equation}

The next lemma records the resulting privacy guarantee for a single privatized ascent run.

\begin{lemma}[Privacy of a single \textsc{DP-GRAMS} run]
\label{thm:dp_grams_privacy}
Consider Algorithm~\ref{alg:dpgrams} with a fixed initialization \(x_0\), stabilization parameters \(A\) and \(p_{\mathrm{floor}}\), and Gaussian noise scale \(\sigma\) given by \eqref{sigma-exact}. Then the final iterate \(x_T\) is \((\varepsilon_{\mathrm{modes}},\delta)\)-DP.
\end{lemma}

Lemma~\ref{thm:dp_grams_privacy} treats a single ascent trajectory started from one initialization. In \textsc{DP-GRAMS}, however, we run several trajectories in parallel from different anchors, and at each round these trajectories are evaluated on the same minibatch. Treating the \(k\) current score vectors as \(k\) separate Gaussian releases would lead to a sub-optimal privacy accounting, worse by a factor of the number of starts. To avoid this, we instead analyze the round-\(t\) collection of score vectors as a single joint Gaussian release whose covariance reflects the spatial proximity of the current iterates: nearby trajectories receive more strongly correlated perturbations, which allows the full multi-start release to be handled through a kernelized Gaussian mechanism in the spirit of \citet{hall2013differential}. Concretely, to couple the releases across starts without incurring an additional \(k\)-dependent privacy penalty, we use the exponential kernel
\[
\bar C_h(x,y):=\exp\!\left(-\frac{\|x-y\|}{h}\right),
\]
and write \(\widehat{\bar C}\) and \(\widehat K\) for the Fourier transforms of
\(\bar C:=\bar C_1\) and \(K\), respectively. Define
\begin{equation}\label{eq:def-I0I1}
I_0
:=
\frac{1}{(2\pi)^d}
\int_{\mathbb R^d}
\frac{|\widehat K(t)|^2}{\widehat{\bar C}(t)}\,dt,
\qquad
I_1
:=
\frac{1}{(2\pi)^d}
\int_{\mathbb R^d}
\frac{\|t\|_2^2\,|\widehat K(t)|^2}{\widehat{\bar C}(t)}\,dt.
\end{equation}
Under Assumption~\ref{assump:kernel}(iv), both \(I_0\) and \(I_1\) are finite. Let
\begin{equation}\label{eq:def-Delta-corr}
\Delta_{h,\mathrm{corr}}(A,p_{\mathrm{floor}})
:=
2\sqrt{2}\left(
\frac{I_1^{1/2}}{p_{\mathrm{floor}}}\,h^{-(d+1)}
+
\frac{A I_0^{1/2}}{p_{\mathrm{floor}}^2}\,h^{-d}
\right).
\end{equation}

The next lemma shows that the joint sensitivity of the concurrent stabilized score vectors is of the same order as the single-start sensitivity \(S_h(A,p_{\mathrm{floor}})\), and therefore yields a joint privacy guarantee for the full multi-start ascent stage.

\begin{lemma}[Joint privacy across multiple initializations]
\label{lem:correlated_noise_privacy}
Consider Algorithm~\ref{alg:dpgrams} with a fixed initialization pool
\(
\mathcal I=\{x_{0,1},\dots,x_{0,k}\}.
\)
For each round \(t=0,\dots,T-1\), let us define
\[\mathbf K_t=[\bar C_h(x_{t,\ell},x_{t,r})]_{\ell,r=1}^k \qquad \text{and} \qquad
\mathfrak s_t(\mathcal B_t)
:=
\bigl(
\hat s_{A,p_{\mathrm{floor}};\mathcal B_t}(x_{t,\ell})^\top
\bigr)_{\ell=1}^k
\in\mathbb R^{kd}.
\]
Suppose that for each \(t\),
\(
Y_t
=
\mathfrak s_t(\mathcal B_t)+\Xi_t\) 
is the noisy joint score vector used to update the \(k\) trajectories, where
\(
\Xi_t\sim \mathcal N\!\bigl(0,\sigma^2(\mathbf K_t\otimes I_d)\bigr)
\)
are sampled independently across rounds, with
\begin{equation}\label{eq:sigma-corr}
\sigma
=
\frac{\Delta_{h,\mathrm{corr}}(A,p_{\mathrm{floor}})/m}
{\log\!\bigl(1+n({\rm e}^{\varepsilon_{\mathrm{iter}}}-1)/m\bigr)}
\sqrt{2\log\!\left(\frac{2.5mT}{n\delta}\right)}.
\end{equation}
Then the noisy score transcript \((Y_0,\dots,Y_{T-1})\) is
\((\varepsilon_{\mathrm{modes}},\delta)\)-DP. Consequently, the final merged estimator
\(\widehat{\mathcal M}\) are \((\varepsilon_{\mathrm{modes}},\delta)\)-DP.
\end{lemma}

We next quantify the privacy cost of the data-dependent initialization pool produced by DAP.

\begin{theorem}[Privacy of DAP initialization]
\label{thm:dp_init_privacy}
Let \(\mathcal Z=\{z_j\}_{j=1}^{N_{\mathrm{cand}}}\subset\mathbb R^d\) be a public grid fixed independently of the private sample \(\mathcal X=(X_1,\dots,X_n)\). Let the output of Algorithm~\ref{alg:dap-init} run for \(k\) rounds with total privacy budget \(\varepsilon_{\mathrm{init}}\), be 
\[
\mathcal I=\{x_{0,1},\dots,x_{0,k}\}.
\]
Then 
the initialization pool \(\mathcal I\) is \((\varepsilon_{\mathrm{init}},0)\)-DP.
\end{theorem}

Combining the private initialization stage with the private correlated ascent stage yields the following end-to-end guarantee for \textsc{DP-GRAMS}.

\begin{corollary}[End-to-end privacy of \textsc{DP-GRAMS}]
\label{cor:full_dp}
The complete \textsc{DP-GRAMS} algorithm, namely DAP Initialization with budget \((\varepsilon_{\mathrm{init}}, 0)\), followed by the correlated multi-start ascent stage of Lemma~\ref{lem:correlated_noise_privacy} with budget \((\varepsilon_{\mathrm{modes}},\delta)\), and the final merge step producing \(\widehat{\mathcal M}\), is \((\varepsilon_{\mathrm{init}}+\varepsilon_{\mathrm{modes}},\delta)\)-DP.
\end{corollary}

\subsection{Error Bounds for Mode Estimation}
We now turn to the utility analysis of \textsc{DP-GRAMS}. Since recovering multiple modes is a nonconvex problem \citep{carreira2007gaussian, ota2019quantile}, the argument depends crucially on initialization. We therefore proceed in two steps. First, we study a single \textsc{DP-GRAMS} trajectory started in a local neighborhood of a true mode and derive its mean-squared error bound. Second, we show that the DAP design places at least one start in the local neighborhood of every population mode with high probability.


\subsubsection{Local Convergence}

Fix a mode \(\mu_j\) for some \(j\in[M]\). We begin with the simpler problem in which one trajectory is already initialized inside the basin \(\overline B(\mu_j,r_j)\). Although the full algorithm uses correlated Gaussian perturbations across simultaneous starts, each individual row still has marginal law \(\mathcal N(0,\sigma^2 I_d)\), so the basinwise analysis applies to any such trajectory. The next assumptions isolate the local geometry needed for this argument.

\begin{assumption}[Curvature at modes]\label{assump:hessian-holder}
Assume $\beta>2$. For each $j\in[M]$ we assume that there exists a numerical constant $C>0$ such that
\[
\alpha_j:=-\lambda_{\max}\!\big(\nabla^2\log p(\mu_j)\big)>C.
\]
\end{assumption}
Here, for a symmetric matrix $\bA\in \RR^{d\times d}$, $\lambda_{\max}(\bA)$ denotes its largest eigenvalue. We refer to $\alpha_j$ as the local strong-concavity parameter of $\log p$ at $\mu_j$. Under Assumption~\ref{assump:model-holder} with $\beta>2$, \(\nabla^2\log p(x)\) is continuous, which together with Assumption~\ref{assump:hessian-holder} guarantees that, for each
$j\in[M]$, there exists $\widetilde r_j>0$ such that
\begin{equation}\label{eq:def-tilde-rj}
\nabla^2\log p(x)\preceq -\frac{\alpha_j}{2}I_d
\qquad\text{for all }x\in \overline B(\mu_j,\widetilde r_j).
\end{equation}
With this radius which guarantees sufficiently strong convexity of the Hessian, we define the local basins of attraction as follows.

\begin{definition}[Radius and separation]\label{assump:radius-holder}
	Assume $\beta>2$. For each $j\in[M]$, let us define
	\[
	r_j :=
	\min\Big\{
	\widetilde r_j,
	\;
	\frac{1}{2}\min_{i\ne j}\|\mu_i-\mu_j\|
	\Big\}
	\]
	where \(\widetilde r_j\) is as defined in \eqref{eq:def-tilde-rj}. For each \(j\in[M]\), we refer to \(\overline B(\mu_j,r_j)\) as the local basin neighborhood of \(\mu_j\).
\end{definition}
The above Definition~\ref{assump:radius-holder} records the local-radius and separation conditions used in the following analysis. Note that by Assumption~\ref{assump:model-holder}, the minimum separation between modes, i.e., \(\min_{i\ne j}\|\mu_i-\mu_j\|\) is bounded below by a constant. Thus the definition of \(r_j\) is not vacuous, and defines a strictly positive quantity.

Let us fix \(j\in[M]\) and let \(c_j:=\frac{p_{\min,j}}{2},\)
where \(p_{\min,j}\) is the local lower bound from Assumption~\ref{assump:model-holder}(i). Writing \(\ell(x)=\log p(x)\) and \(\hat\ell(x)=\log \hat p(x)\) for convenience, we define a local good event as follows. The components of the local good event are:
\begin{enumerate}
	\item The score function and its first two derivatives can be estimated suitably well by their sample counterparts.
	\item The density estimate is bounded away from zero, while its variance and the squared norm of the kernel gradients remain bounded.
	\item Once initialized in the basin of attraction of some mode, the privacy noise does not push the estimates out of that basin.
\end{enumerate}
This event relates to the conditions that guarantee local convergence of our algorithm.

\begin{definition}[Local good event]\label{def:good-event-local-holder}
Fix \(j\in[M]\) and \(\beta>2\). Define the static local analytic event \(\mathcal A_{n}^{\mathrm{local},j}\) to be the event that there exist deterministic constants \(C_j,D_j>0\), depending only on the local model and kernel constants and independent of \(n\), such that
\[
\sup_{x\in\overline B(\mu_j,r_j)}
\big\|\nabla^{s}\hat\ell(x)-\nabla^{s}\ell(x)\big\|
\le
C_j\!\left(h^{\beta-s}+\sqrt{\frac{\log n}{n h^{d+2s}}}\right),
\qquad s=0,1,2,
\]
and
\[
\inf_{x\in\overline B(\mu_j,r_j)}
\hat p(x)\ge c_j,
\qquad
\sup_{x\in\overline B(\mu_j,r_j)}
\frac{1}{n h^d}
\max\left\{
\sum_{i=1}^n
K\!\left(\frac{x-X_i}{h}\right)^2,
\sum_{i=1}^n
\left\|
\nabla K\!\left(\frac{x-X_i}{h}\right)
\right\|^2
\right\}
\le
D_j.
\]
Let us define the 
events
\[
\mathcal E_{n,T}^{\mathrm{stay},j}
:=
\{x_t\in \overline B(\mu_j,r_j)\ \text{for all }t=0,\dots,T\},
\quad 
\text{and}
\quad 
\mathcal G_{n,T}^{\mathrm{local},j}
:=
\mathcal A_{n}^{\mathrm{local},j}\cap \mathcal E_{n,T}^{\mathrm{stay},j}.
\]
Finally, let $\mathscr{X}_{n,T}^{\mathrm{local},j}$ be the set of all datasets $\mathcal{X}$ such that $\mathcal G_{n,T}^{\mathrm{local},j}$ holds.
\end{definition}

The next theorem gives the basin-wise conditional MSE bound that drives the global recovery argument. The optimizing bandwidth is recorded immediately after the theorem.

\begin{theorem}[Local convergence of \textsc{DP-GRAMS}]
\label{th:local-conv-holder}
Assume \(\beta>2\) and suppose Assumptions~\ref{assump:kernel},
\ref{assump:model-holder}, \ref{assump:bandwidth}, \ref{assump:hessian-holder} hold. Fix \(j\in[M]\) and \(x_0\in \overline B(\mu_j,r_j)\). Then there exist constants \(p_{\mathrm{floor}}\), \(A\),  \(\bar\eta_j>0\) such that for every constant \(0<\eta\le \bar\eta_j\), the following holds. Then, for the bandwidth \(h\) specified in \eqref{h-opt}, there exist numerical constants \(C_{T},C_{\mathrm{nonDP},j},C_{\mathrm{DP},j}>0\), such that for sufficiently large \(n\), and $\mathscr{X}_{n,T}^{\mathrm{local},j}$ defined in Definition~\ref{def:good-event-local-holder}, 
\[
\Pr\!\big(
\mathcal{X}\in \mathscr{X}_{n,T}^{\mathrm{local},j}
\big)\ge 1-5n^{-4},
\]
and
\[
\EE\!\left[
\|x_T-\mu_j\|^2
\,\middle|\,
\mathcal X
\right]
\le
C_{\mathrm{nonDP},j}
\Big(\frac{\log n}{n}\Big)^{\frac{2(\beta-1)}{d+2\beta}}
+
C_{\mathrm{DP},j}
\Big(\frac{T\,d\,\mathrm{polylog}(n,\delta)}{n^2\varepsilon_{\mathrm{modes}}^2}\Big)^{\frac{\beta-1}{d+\beta}}.
\]
provided \(T=C_T\log n\) , \(m= \ceil{n/\log n}\) and \(\mathcal{X}\in \mathscr{X}_{n,T}^{\mathrm{local},j}\).
\end{theorem}

\begin{remark}
    Some sufficient conditions for the choice of $p_{\rm floor}$ and $A$ are given in Proposition~\ref{prop:local-inactive}.
\end{remark}

The proof optimizes the underlying bandwidth-dependent error bound and yields the order-wise optimal choice
\begin{align}\label{h-opt}
h_{\rm opt}\asymp
\begin{cases}
\Big(\dfrac{\log n}{n}\Big)^{\frac{1}{d+2\beta}}
&\text{if }\ \varepsilon_{\mathrm{modes}} \gtrsim \varepsilon_{\rm thr},\\[1.2ex]
\Big(\dfrac{1}{n^2\varepsilon_{\mathrm{modes}}^2}\Big)^{\frac{1}{2d+2\beta}}
&\text{if }\ \varepsilon_{\mathrm{modes}} \lesssim \varepsilon_{\rm thr},
\end{cases}
\qquad
\varepsilon_{\rm thr}\asymp
\left(\dfrac{\log n}{n}\right)^{-\frac{d}{2(d+2\beta)}}\,
\left(\dfrac{1}{n\log n}\right)^{\frac{1}{2}} .
\end{align}
When $\varepsilon$ is large so the nonprivate term dominates, the resulting rate
matches the minimax-optimal mode estimation rate of
\cite{arias2016estimation,genovese2014nonparametric} up to logarithmic factors. We find that both the first and second terms of the MSE are affected by the traditional curse of dimensionality $d$, but is countered by the smoothness parameter $\beta$. If $\beta \to\infty$, i.e., the density is analytic, we recover near-parametric rates. 

\subsubsection{Density-Aware Private Initialization}\label{dap-grid-construction}

We now turn to DAP, whose role is to supply the good initializations required by the local convergence analysis. Recall that DAP is the initialization scheme in Algorithm~\ref{alg:dap-init}, based on the utility in \eqref{eq:uj-def}. We now specify the public candidate set used there. Let us consider a known public box 
\[
\mathcal Q=\prod_{q=1}^d [L_q,U_q]\subset\mathbb R^d,
\quad
\text{such that}
\quad
\bigcup_{j=1}^M \overline B(\mu_j,r_j)\subset \mathcal Q.
\]
For each \(n\), let us choose the number of DAP draws and the suppression radius so that
\[
k\asymp M\log n,
\qquad
\rho_{\mathrm{init}}\asymp(\log n)^{-1/d}.
\]
Let
\(
\mathcal Z_n:=\mathcal Q\cap h_{\mathrm{DAP}}\mathbb Z^d
=\{z_1,\dots,z_{N_{\mathrm{cand}}}\}
\)
be the candidate set used by Algorithm~\ref{alg:dap-init}. We take
\begin{equation}
h_{\mathrm{DAP}}
\asymp
\left(\frac{\log n}{n}\right)^{\frac{1}{d+2\beta}}.
\label{eq:h-dap}
\end{equation}
Then \(N_{\mathrm{cand}}\asymp h_{\mathrm{DAP}}^{-d}\asymp (n/\log n)^{d/(d+2\beta)}\), so that the public grid has polynomial size in \(n\). Since \(h_{\mathrm{DAP}}/\rho_{\mathrm{init}}\to0\), the grid is fine enough at the suppression scale used below.

The next proposition shows that, under the DAP design and a finite-sample lower bound on the initialization budget \(\epsilon_{\text{init}}\), \(k\asymp M\log n\) private draws localize in the modal basin neighborhoods and cover every such neighborhood with high probability.

\begin{proposition}[High-probability coverage of DAP Initialization]
\label{prop:init-dpball}
Suppose Assumptions~\ref{assump:model-holder} and
\ref{assump:hessian-holder} hold, and
 Algorithm~\ref{alg:dap-init} is run under the DAP design above. Assume that
\[
\varepsilon_{\mathrm{init}}
\gtrsim
M n^{-2\beta/(d+2\beta)}\mathrm{polylog}(n)
\qquad\text{and}\qquad
k\asymp M\log n.
\]
Then there exists a constant \(C_{\mathrm{init,cov}}>0\), independent of \(\varepsilon_{\mathrm{init}}\), such that, for all sufficiently large \(n\),
\[
\Pr\!\left(
\mathcal I\subseteq\bigcup_{j=1}^M \overline B(\mu_j,r_j),
\quad
\mathcal I\cap \overline B(\mu_j,r_j)\neq\varnothing\ \text{for every }j\in[M]
\right)
\ge 1-C_{\mathrm{init,cov}}\,n^{-2}.
\]
\end{proposition}

The above proposition pins together a number of crucial aspects for our initialization scheme. First of all, we ensure that the number of initializations scales as \(k\asymp M\log n\), which grows logarithmically in the sample size, and on average spends $C\log n$ initializations per mode. Secondly, on the high probability event defined within the proposition, the behavior of the chosen initializations is precisely characterized as follows. Each initialization lies within the neighborhood of \emph{some} density mode; and, for each density mode, there exists \emph{at least one} initialization in its convergence neighborhood. Both of these facts will be used in our global convergence analysis: since our iterative algorithm converges to local stationary points within the basin, which along with the first fact and the local concavity assumption, are  precisely the modes. The second fact ensures that every mode has a nearby initialization and is thus guaranteed to be recovered.

\subsubsection{Global Control}

The next theorem upgrades the basinwise local control to the final merged estimator by combining DAP coverage with the deterministic post-processing conditions used by the merge rule.

For each \(j\in[M]\), let us define the basin index set
\[
I_j:=\{\ell\in[k]:x_{0,\ell}\in \overline B(\mu_j,r_j)\}.
\]
which defines the set of particular initialization indices that are close to the \(j\)-th mode \(\mu_j\). Similar to our analysis for the local convergence, we now define a global good event, which in addition to the similar sub-events earlier defined for local convergence, also requires the initializations to be situated in the basins of attraction of each mode. 

\begin{definition}[Global good event]\label{def:good-event-global}
Let us define the event of good initialization as 
\(
\mathcal E_{\mathrm{init}}
:=
\Bigl\{
\forall j\in[M],\ \mathcal I\cap \overline B(\mu_j,r_j)\neq\varnothing
\Bigr\}
\)
and
\(
\mathcal A_n^{\mathrm{global}}
:=
\Bigl(\bigcap_{j=1}^M \mathcal A_n^{\mathrm{local},j}\Bigr)\cap \mathcal E_{\mathrm{init}}.
\)
For each \(j\in[M]\) and \(\ell\in I_j\), let us define
\(
\mathcal E_{n,T}^{\mathrm{stay},j}(\ell)
:=
\{x_{t,\ell}\in \overline B(\mu_j,r_j)\ \text{for all }t=0,\dots,T\}.
\)
Then we define the global good event as
\[
\mathcal G_{n,T}^{\mathrm{global}}
:=
\mathcal A_n^{\mathrm{global}}
\cap
\bigcap_{j=1}^M\bigcap_{\ell\in I_j}
\mathcal E_{n,T}^{\mathrm{stay},j}(\ell).
\]
Let $\mathscr{X}_{n,T}^{\mathrm{global}}$ be the set of all datasets $\mathcal{X}$ such that $\mathcal G_{n,T}^{\mathrm{global}}$ holds.
\end{definition}

\begin{theorem}[Global convergence of \textsc{DP-GRAMS}]\label{thm:global-holder}
Assume \(\beta>2\) and suppose Assumptions~\ref{assump:kernel},
\ref{assump:model-holder}, \ref{assump:bandwidth}, \ref{assump:hessian-holder} hold. Let $\widehat{\mathcal M}$ be the output of Algorithm~\ref{alg:dpgrams}. Then there exist constants \(C_{\mathrm{g}}, C_{\mathrm{nonDP},j}, C_{\mathrm{P},j}>0\) such that for $\mathscr{X}_{n,T}^{\mathrm{global}}$ defined in Definition~\ref{def:good-event-global},
\[
\Pr(\mathcal{X}\in \mathscr{X}_{n,T}^{\mathrm{global}})\ge 1-C_{\mathrm{global}}\,n^{-2}
\]
for all sufficiently large \(n\), and a permutation $\pi:[M]\to [M]$ such that \(\widehat{\mathcal{M}}:=\{\hat\mu_1,\dots,\hat\mu_M\}\) satisfies
\[
\EE\!\left[
\|\hat\mu_j-\mu_{\pi(j)}\|^2
\,\middle|\,
\mathcal X
\right]
\le
C_{\mathrm{nonDP},j}\Big(\frac{\log n}{n}\Big)^{\frac{2(\beta-1)}{d+2\beta}}
+
C_{\mathrm{DP},j}\Big(\frac{T d\,\mathrm{polylog}(n,\delta)}{n^2\varepsilon_{\mathrm{modes}}^2}\Big)^{\frac{\beta-1}{d+\beta}},
\]
for every \(j\in[M]\), whenever $\mathcal{X}\in \mathscr{X}_{n,T}^{\mathrm{global}}$.
\end{theorem}

Thus, on the global good event \(\mathcal G_{n,T}^{\mathrm{global}}\), each point in the final merged estimator \(\widehat{\mathcal M}\) corresponds to some population mode \(\mu_{\pi(j)}\) for an unknown permutation \(\pi:[M]\to [M]\). More importantly, the estimators satisfy the same conditional MSE rate as in Theorem~\ref{th:local-conv-holder}.


\subsection{Minimax Lower Bound}

We now investigate the optimality of the mode estimation errors achieved by our algorithm. While score estimation under smoothness assumptions have been studied in previous literature, the problem of quantifying the loss in estimation accuracy due to privacy requirements, remains unexplored. To this end, we derive minimax lower bounds for mode estimation under the smoothness assumptions, while constraining the estimators to satisfy $(\varepsilon,\delta)$ privacy.

\begin{theorem}
\label{thm:lower-bound}
Let $\beta>1$ and let $\mathcal P_\beta(L)$ denote the class of densities
$p:\mathbb R^d\to\mathbb R_+$ satisfying Assumptions~\ref{assump:model-holder}, \ref{assump:hessian-holder} and that $\nabla^2 \log p(x)$ has bounded singular values for all $x$.
Then there exists a constant
$C>0$ such that
\[
\inf_{\hat x\in \mathcal{T}(n,\varepsilon,\delta)}\;
\sup_{p\in\mathcal P_\beta(L)}
\EE_p\!\left[
\bigl\|
\hat{x} - x_0(p)
\bigr\|^2
\right]
\;\ge\;
C\, n^{-\frac{2(\beta-1)}{d+2\beta}}
+C (n\varepsilon)^{-\frac{2(\beta-1)}{d+\beta}}
\]
provided $\delta =o(n^{-1})$. Here $\mathcal{T}(n,\varepsilon,\delta)$ is the set of all possible estimators based on a sample of size $n$ and satisfying $(\varepsilon,\delta)$ differential privacy.
\end{theorem}

A comparison with Theorem~\ref{thm:global-holder} reveals that the error rates achieved by our \textsc{DP-GRAMS} algorithm are nearly minimax optimal, with the upper bound being worse than the lower bounds from Theorem~\ref{thm:lower-bound} only up to the logarithmic terms $\log n$ and $\log \delta$.

Our lower bound construction is based on first constructing a pointwise lower bound for differentially private score estimation, which to our knowledge is novel and might be of independent interest. More specifically, let the score function be $s_p(x)=\nabla \log p(x)$. Then there exists a constant
$C>0$ such that
\[
\inf_{\hat s\in T_{n,\varepsilon,\delta}(\cdot)}\;
\sup_{p\in\mathcal P_\beta(L)}
\mathbb E_p\!\left[
\bigl\|
\hat s - s_p
\bigr\|_{\rm \infty}^2
\right]
\;\ge\;
C\, n^{-\frac{2(\beta-1)}{d+2\beta}}
+C (n\varepsilon)^{-\frac{2(\beta-1)}{d+\beta}}
\]
provided $\delta =o(n^{-1})$. Given this result, the lower bound for mode estimation follows using the bounded singular values of the Hessian. The score estimation lower bound in turn depends on standard constructions for density perturbations, followed by the contraction of total variation under privacy constraints, from \cite{karwa2017finite}.


\section{Experiments}\label{sec:apps}

To complement the theoretical results, we evaluate \textsc{DP-GRAMS} and its downstream extensions on synthetic and real datasets. Unless stated otherwise, privacy is calibrated to \((\varepsilon,\delta)\), the ascent stage uses \(T=\lceil \log n\rceil\) iterations and minibatch size \(m=\lceil n/\log n\rceil\), and reported summaries are averages over \(20\) independent runs with standard-error bars.


\subsection{Implementation choices and tuning}
\label{subsec:exp-impl}

We first summarize the implementation choices shared across all experiments.


\paragraph{Kernel choice.}
The theory allows general differentiable kernels, but the experiments use the Gaussian kernel as the default. We validate this choice in Section~\ref{subsec:kernel_choice_validation} by comparing Gaussian \textsc{DP-GRAMS} with an order-\(4\) implementation. The Gaussian kernel is also convenient because classical mean shift is equivalent to a scaled ascent step on \(\log \hat p\).


\paragraph{Bandwidth selection.}
\label{subsubsec:bandwidth-exp}

We use either the theoretically motivated \(h=((\log n)/n)^{1/(d+6)}\), or Silverman's procedure \citep{silverman2018density} to fix a single bandwidth $h$ per dataset and keep it fixed across privacy budgets. 

\paragraph{Clipping.}
\label{subsubsec:clip-exp}

In the theoretical development, the ascent update is written using a clipped gradient and a density floor. In implementation, however, the key quantity is the size of the per-sample score contributions. Lemma~\ref{lem:local-qi-bound} shows that, on the local good event, these contributions are bounded with high probability at the scale \(h^{-(d+1)}\). We therefore parameterize clipping through the effective threshold
\begin{equation}\label{eq:clip-constant-exp}
C_*
=
(\texttt{clip\_multiplier})\,h^{-(d+1)}.
\end{equation}
The dimensionless \texttt{clip\_multiplier} is treated as a dataset-level hyperparameter and selected once per dataset using a small pilot grid search. We then fix the resulting \texttt{clip\_multiplier}, and hence \(C_*\), for all privacy budgets and all reported runs. The density floor and gradient clipping level in the theoretical pseudocode are not separately tuned or reported in the experiments.

\paragraph{Error metric.}\label{subsubsec:error_metric}
For all mode- and centroid-estimation tasks, accuracy is measured by the mean-squared matching error between a reference set \(\{\mu_1,\dots,\mu_k\}\) and an estimated set \(\{\hat\mu_1,\dots,\hat\mu_{\hat k}\}\):
\begin{equation}\label{MMMSE}
\mathrm{MSE}
:=
\frac{1}{\max\{k,\hat k\}}
\min_{\mathcal M:\,|\mathcal M|=\min\{k,\hat k\}}
\sum_{(j,\ell)\in\mathcal M}
\|\mu_j-\hat{\mu}_\ell\|^2,
\end{equation}
where the minimum is taken over all one-to-one matchings \(\mathcal M\subseteq [k]\times[\hat k]\) with no repeated indices and cardinality \(|\mathcal M|=\min\{k,\hat k\}\). The denominator \(\max\{k,\hat k\}\) keeps the error on the same scale even when the estimated number of modes differs from the reference number. The optimal matching is computed using the Hungarian algorithm \citep[see, e.g.,][]{kuhn1955hungarian,munkres1957algorithms}.

\subsection{Kernel Choice Validation}\label{subsec:kernel_choice_validation}

\begin{figure}[!htb]
    \centering
    \includegraphics[width=0.75\linewidth]{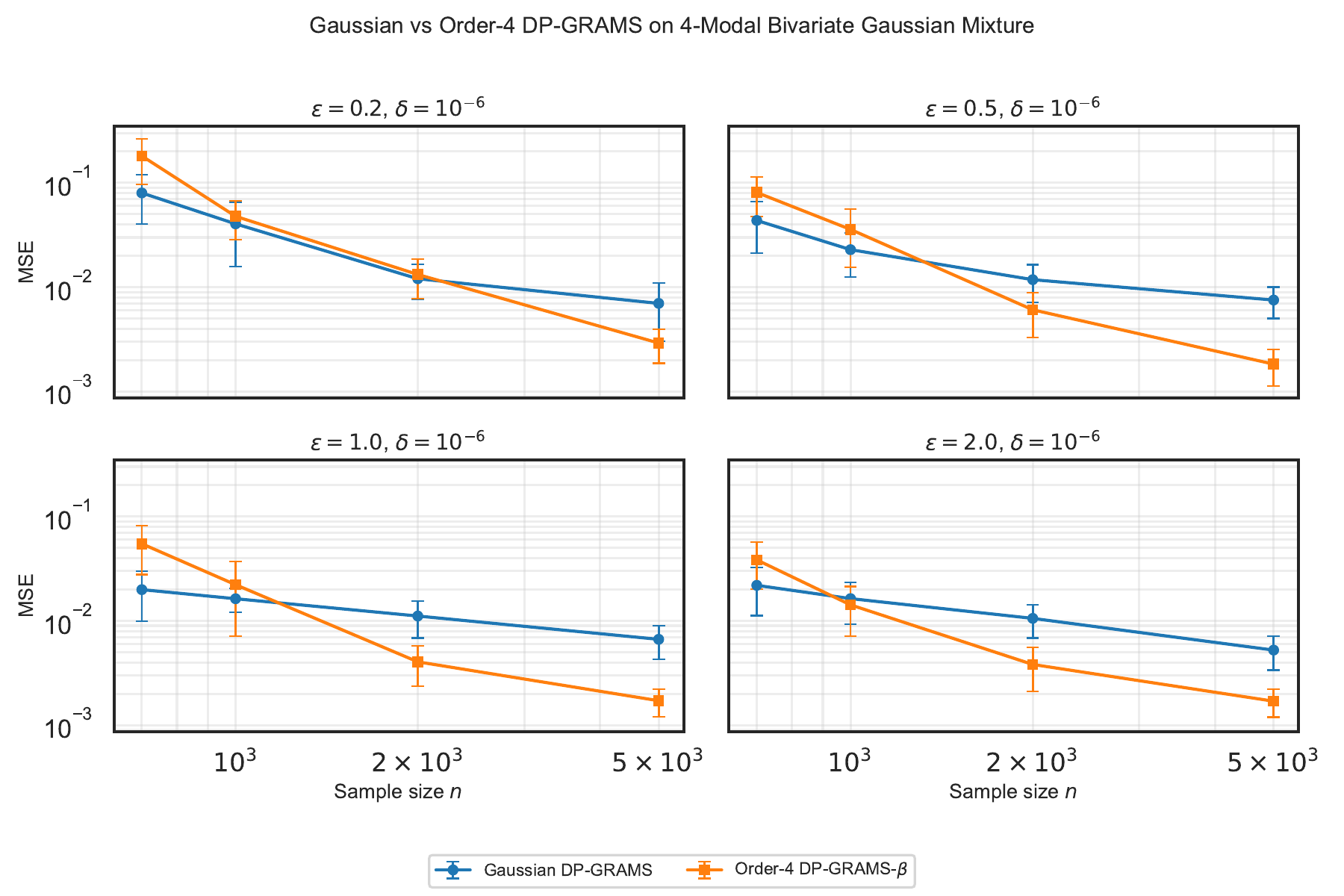}
    \caption{\small Kernel-choice validation on the four-corners Gaussian mixture. Gaussian \textsc{DP-GRAMS} is compared with order-\(4\) \textsc{DP-GRAMS}. Privacy is calibrated with \(\delta=10^{-6}\), and panels correspond to \(\varepsilon\in\{0.2,0.5,1,2\}\). Points show mean MSE \eqref{MMMSE} with standard-error bars.}
    \label{fig:kernel_order_validation}
\end{figure}


We compare Gaussian \textsc{DP-GRAMS} with \textsc{DP-GRAMS} using a compactly supported order-\(4\) kernel. The order-\(4\) kernel is constructed from the orthonormal Legendre expansion \citep[see Proposition 1.3 in][]{tsybakov2008nonparametric} and extended to \(\mathbb{R}^2\) by a product construction. We evaluate both methods on the four-corners Gaussian mixture and measure accuracy by the MSE formula in \eqref{MMMSE}. To isolate the effect of the kernel, both methods use the same bandwidth \(h\), computed from the order-\(\beta\) bandwidth rule in \eqref{h-opt}, for each \((n,\varepsilon)\) configuration.

Figure~\ref{fig:kernel_order_validation} shows that the Gaussian implementation is competitive with the order-\(4\) alternative across the displayed privacy budgets and sample sizes. The order-\(4\) kernel does not provide a systematic empirical advantage in this benchmark, while the Gaussian kernel avoids the finite-sample complications associated with sign-indefinite KDE estimates. We therefore use the Gaussian kernel as the default in the remaining experiments.

\subsection{Differentially Private Mode Estimation on Simulated Data}
\label{subsec:private-modes}

We next evaluate \textsc{DP-GRAMS} on simulated bivariate mixtures with known population modes. In each setting, we compare the population modes, non-private mean shift, and \textsc{DP-GRAMS} using the same bandwidth, iteration budget, and final merging rule. Privacy is calibrated with \(\delta=10^{-6}\) and \(\varepsilon\in\{0.1,0.25,0.5,1,5\}\), and results are averaged over \(20\) independent runs with standard-error bars. The four-corners Gaussian mixture in Section~\ref{sec:gaussian-example} provides a clean, well-separated benchmark, while the five-component \(t\)-mixture below tests recovery under heavier tails and unequal local scales. Additional diagnostics for both synthetic mode-estimation benchmarks are reported in Appendix~\ref{app:mode-estimation-diagnostics}, including repeated-run contour grids, sensitivity analyses for the clipping threshold \(C_*\), minibatch size \(m\), and step size \(\eta\), and full MSE and runtime tables across \((n,\varepsilon)\).

\subsubsection{Bivariate five-modal \(t\)-mixture}\label{sec:t-mixture-example}

\begin{figure}[!htb]
    \centering
    \includegraphics[height=0.26\linewidth]{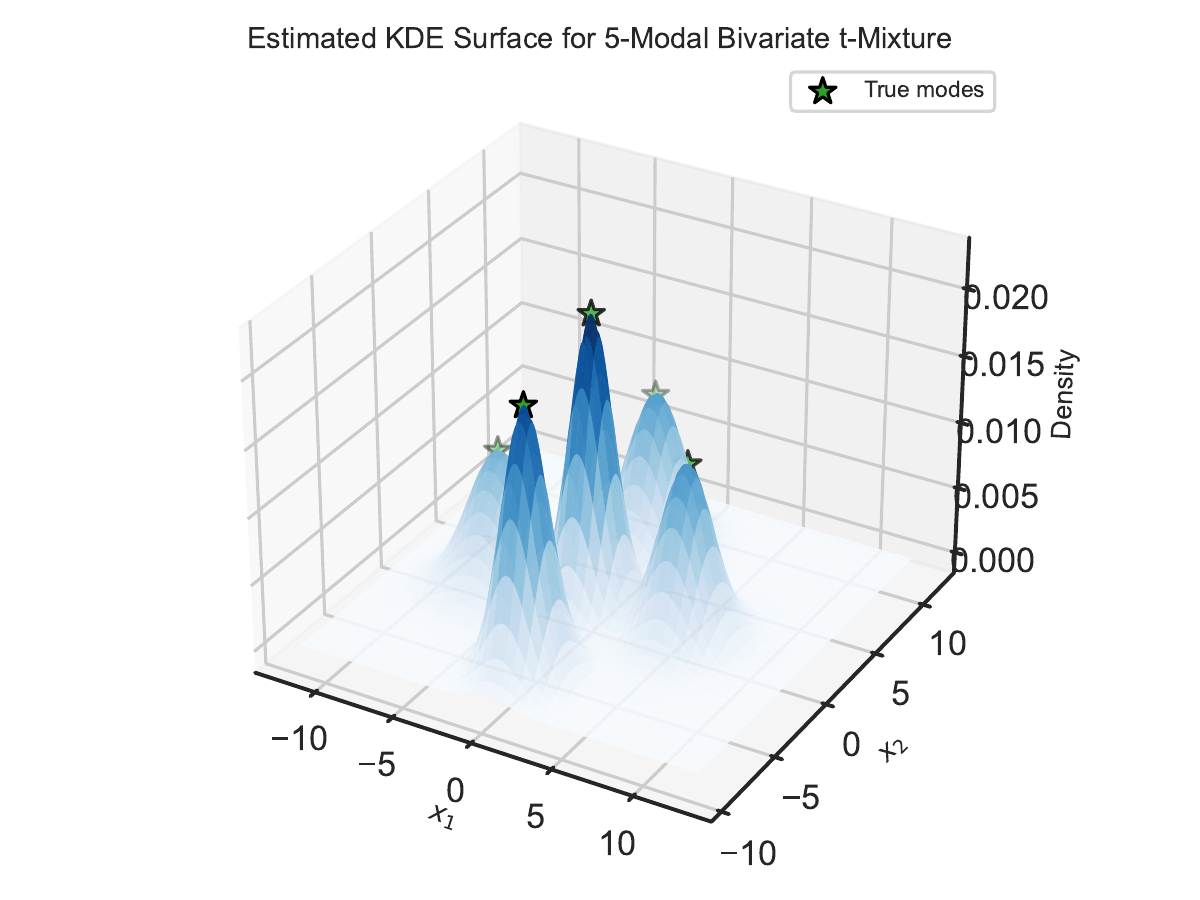}\hfill
    \includegraphics[height=0.26\linewidth]{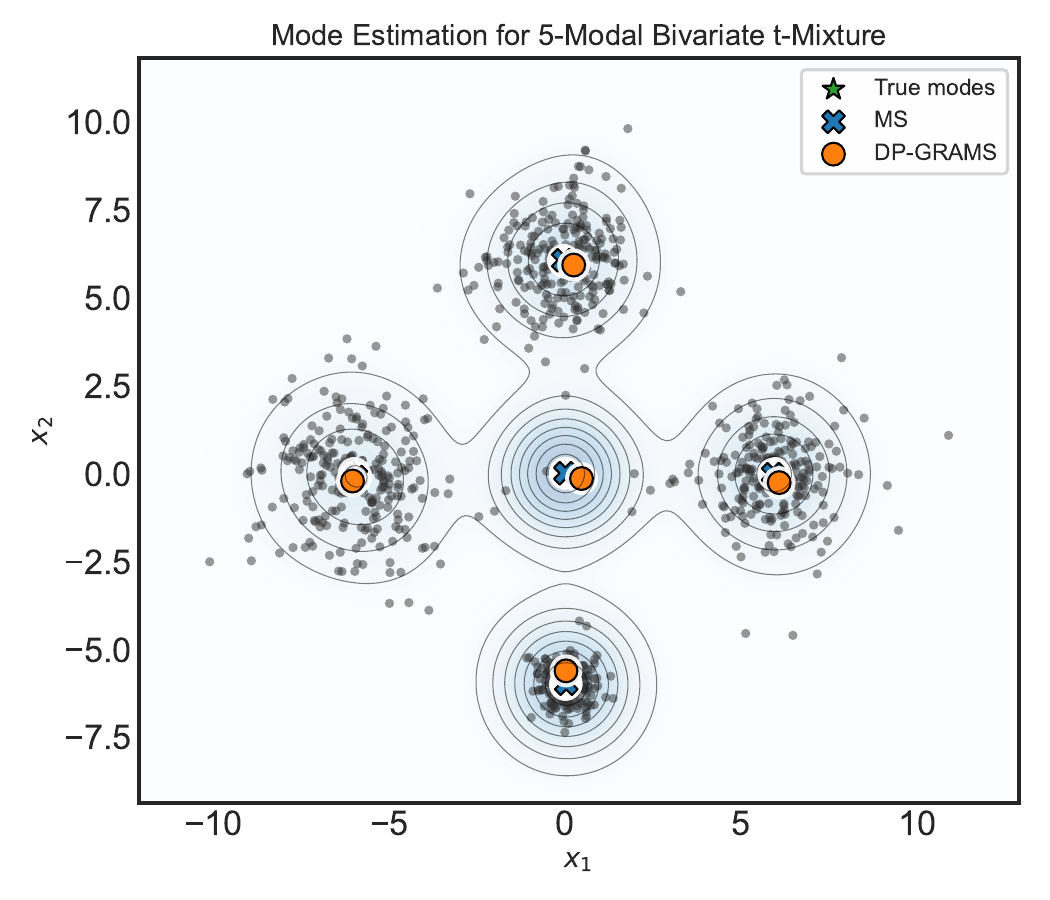}\hfill
    \includegraphics[height=0.26\linewidth]{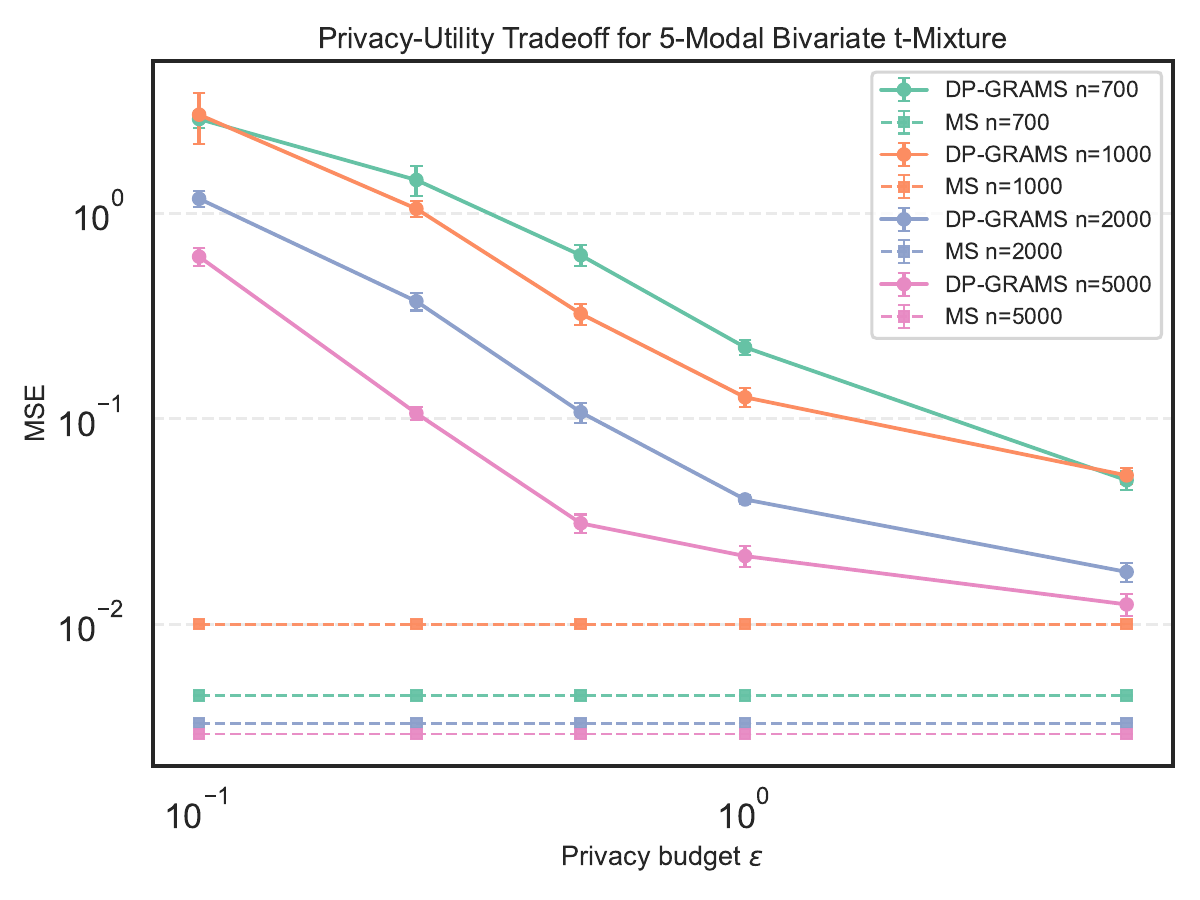}
        \caption{\small Private mode estimation on the five-component \(t\)-mixture.
    Panels (a) and (b) use one representative dataset with \(n=1200\) and \((\varepsilon,\delta)=(1,10^{-6})\).
    (a) Estimated KDE surface with true modes overlaid.
    (b) Contour plot comparing true (green), mean shift (blue), and \textsc{DP-GRAMS} (orange) mode estimates.
    (c) Privacy--utility tradeoff: \(\mathrm{MSE}\) in \eqref{MMMSE} versus \(\varepsilon\) on a log scale for \(n\in\{700,1000,2000,5000\}\) and \(\varepsilon\in\{0.1,0.25,0.5,1,5\}\); dashed lines denote mean-shift baselines and solid curves show \textsc{DP-GRAMS}.}
    \label{fig:bivariate_t_threepanel}
\end{figure}

To examine robustness beyond the Gaussian setting in Section~\ref{sec:gaussian-example}, we also consider the bivariate five-modal \(t\)-mixture
\[
(X,Y)\sim \sum_{k=1}^5 \pi_k\, t_{\nu_k}(\mu_k,\sigma_k^2 I_2),
\qquad
\pi_k=0.2,
\]
with centers
\(
\mu_1=(0,0),\quad
\mu_2=(6,0),\quad
\mu_3=(-6,0),\quad
\mu_4=(0,6),\quad
\mu_5=(0,-6),
\)
degrees of freedom \(\nu=(15,6,10,8,20)\), and scales \(\sigma=(0.1,0.9,1.3,1.0,0.4)\). This design is more challenging because its components have different tail behavior and local spread, so the resulting peaks are less homogeneous than in the Gaussian benchmark.

Figure~\ref{fig:bivariate_t_threepanel} shows that \textsc{DP-GRAMS} continues to recover the modal structure in this harder setting. Panels~(a) and~(b) show peaks of unequal height and spread, yet the private estimates remain near the five modal locations on a representative dataset. Panel~(c) shows that the privacy cost is largest for small samples and tight privacy budgets, while the gap to the mean-shift baseline narrows as \(n\) and \(\varepsilon\) increase. This behavior is consistent with the design: heavier tails and heterogeneous scales make the modal basins less uniform than in the four-corners Gaussian example.

\subsection{Differentially Private Modal Regression on Simulated Data}
\label{subsec:private-modal-regression}

Modal regression targets the conditional modes of \(Y\mid X=x\), rather than the conditional mean. This distinction is important when the conditional response distribution is multimodal: mean-based smoothers can average across distinct subpopulations and fail to represent a typical response value (see Figure~\ref{fig:modal_reg_sin_threepanel}a). Building on the private ascent framework of \textsc{DP-GRAMS}, we obtain a differentially private analogue of partial mean shift (PMS, see Algorithm 1 of \cite{chen2016nonparametric}), denoted \textsc{DP-PMS}; full pseudocode is given in Algorithm~\ref{alg:dp-pms} in Section~\ref{sec:apndx-algo}. In the experiments below, we use the fixed-design version of this problem: the predictor locations \(X\) are treated as public, and privacy is enforced for the responses \(Y\) conditional on those predictors. This is appropriate for the simulated designs considered here, where \(X\) lies in a known public domain and the privacy-sensitive quantity is the response distribution. If the predictors themselves were private, an additional privacy mechanism would be needed for the predictor locations, e.g., by releasing the modal curve on a public evaluation grid or by privatizing the predictor-side binning.

We evaluate PMS and \textsc{DP-PMS} on simulated regression problems with known oracle conditional modes. If \(\widehat y(x)\) denotes an estimated modal response at predictor value \(x\), and \(\mathcal M(x)\) denotes the population conditional mode set, we use the pointwise loss
\begin{equation}\label{eq:modal-reg-mse}
L\bigl(\widehat y(x),\mathcal M(x)\bigr)
:=
\min_{m\in \mathcal M(x)} \bigl(\widehat y(x)-m\bigr)^2.
\end{equation}
The reported regression error is the average of this loss over the predictor locations returned by the procedure. LOWESS is included only as a qualitative baseline targeting the conditional mean rather than the conditional modes. Privacy is calibrated to \((\varepsilon,\delta)\)-differential privacy with \(\delta=10^{-5}\) and \(\varepsilon\in\{0.1,0.2,0.5,1.0\}\).

\subsubsection{Sinusoidal two-component mixture}\label{sec:sinusoidal-modal-regression}
To study performance under smoothly varying nonlinear modal structure, we consider a sinusoidal two-component mixture. We draw predictors \(X\sim \mathrm{Uniform}(0,1)\) independently for each component and generate responses with Gaussian noise \(\sigma=0.15\) via
\[
Y_1 = 1.5 + 0.5 \sin(3\pi X) + \mathcal N(0,\sigma^2),
\qquad
Y_2 = 0.5 \sin(3\pi X) + \mathcal N(0,\sigma^2),
\]
so that the conditional density has two smooth modal curves. We consider sample sizes \(n\in\{200,600,1200,2400\}\) and compare PMS, \textsc{DP-PMS}, and LOWESS on the same simulated design.

\begin{figure}[!htb]
    \centering
    \includegraphics[height=0.24\linewidth]{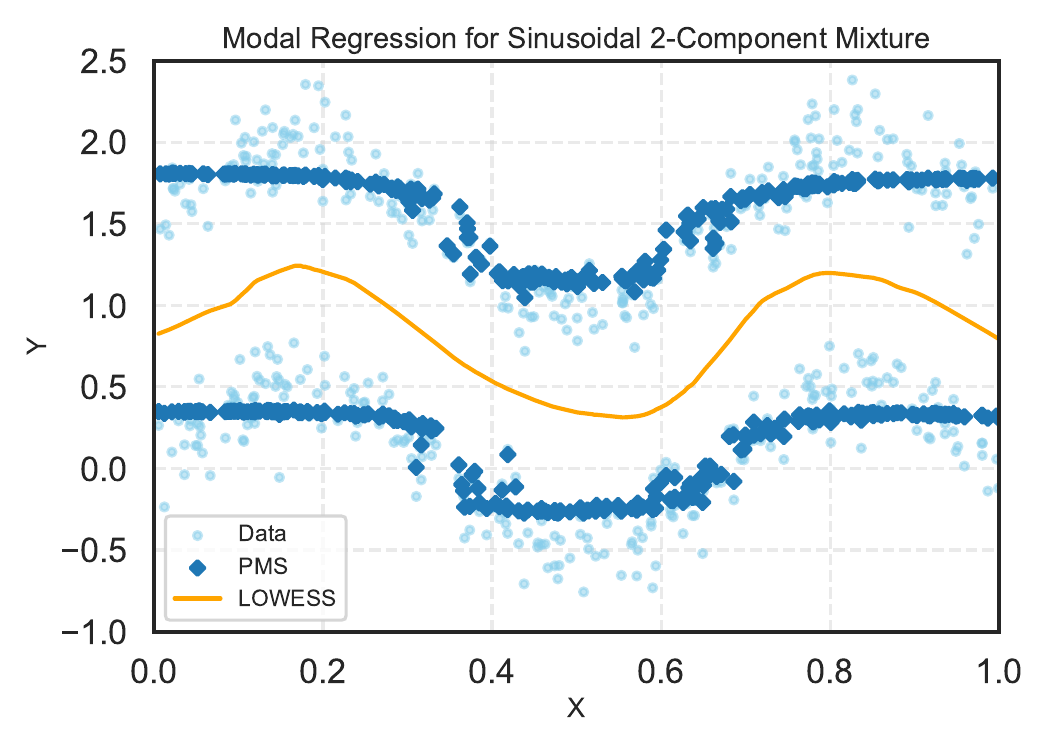}\hfill
    \includegraphics[height=0.24\linewidth]{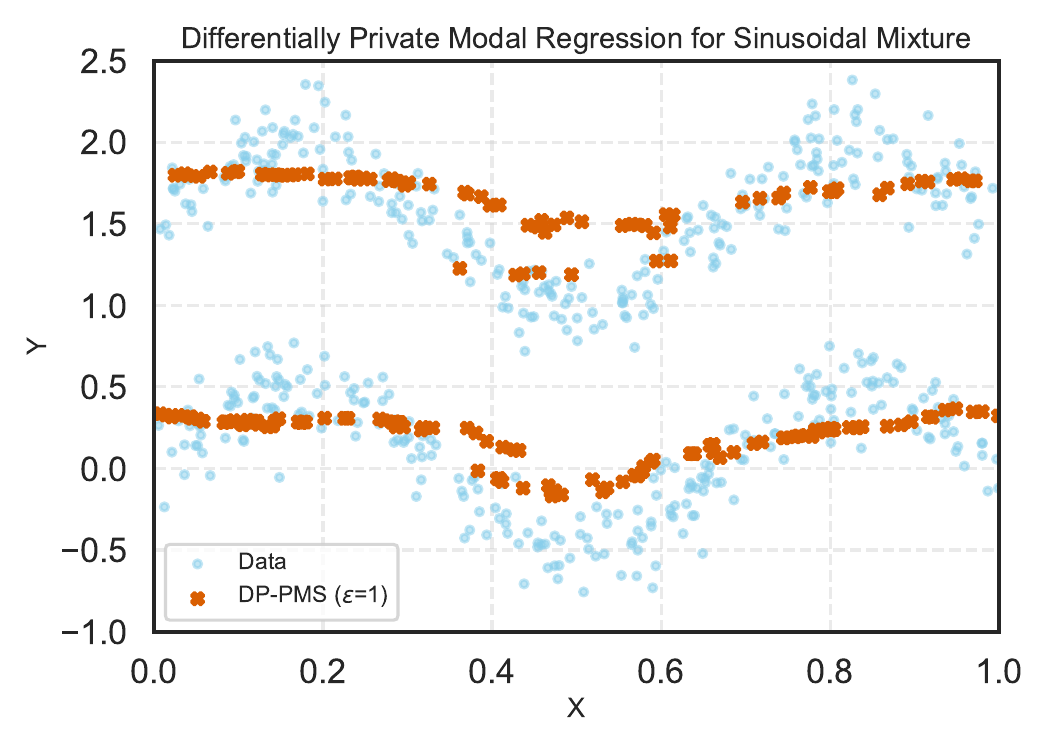}\hfill
    \includegraphics[height=0.24\linewidth]{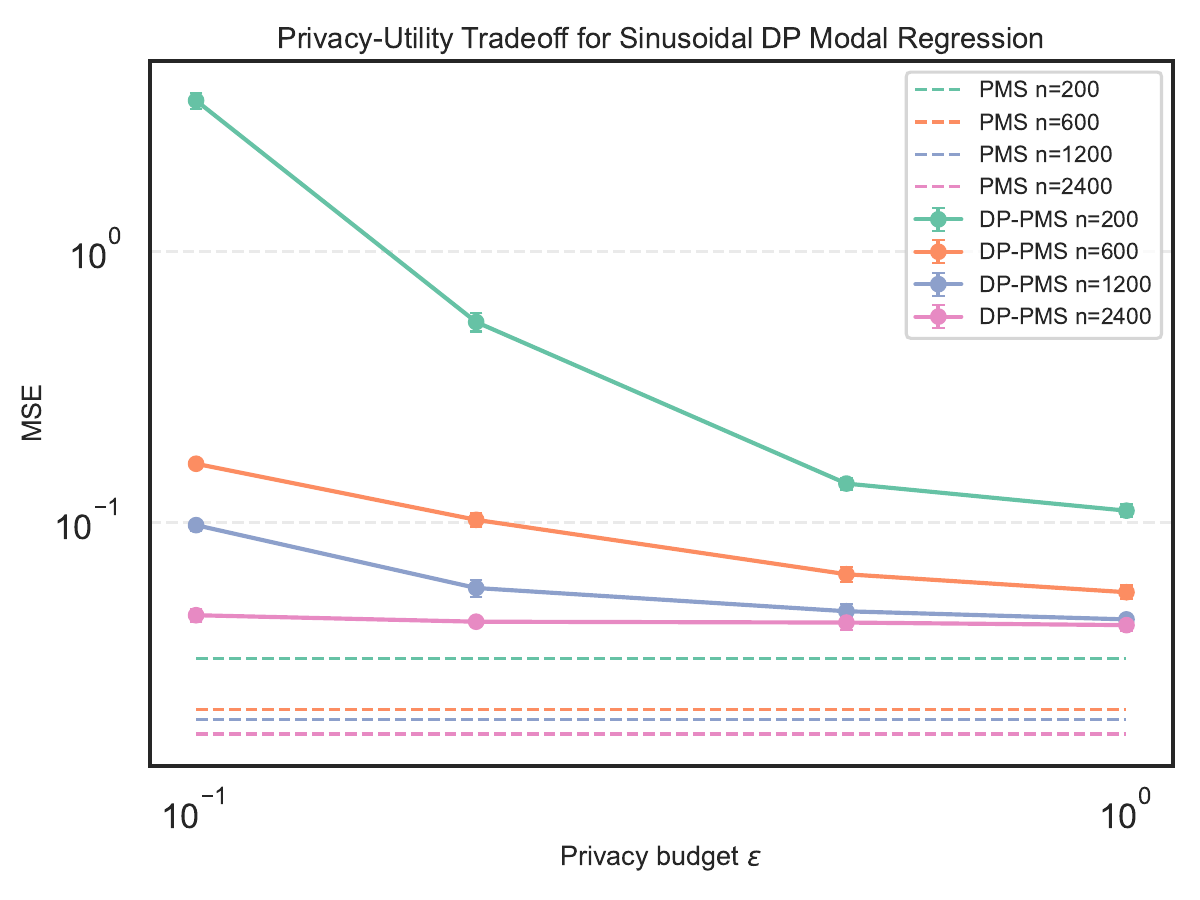}
       \caption{\small Private modal regression on sinusoidal two-component mixture data.
    Panels (a) and (b) use one representative dataset with \(n=500\); panel (b) uses \((\varepsilon,\delta)=(1,10^{-5})\).
    (a) PMS captures the two conditional modes, whereas LOWESS averages across the mixture components.
    (b) \textsc{DP-PMS} recovers the same two-branch modal structure under privacy.
    (c) Privacy--utility tradeoff: oracle MSE in \eqref{eq:modal-reg-mse} versus \(\varepsilon\) on a log scale for \(n\in\{200,600,1200,2400\}\); points show averages over \(20\) runs with standard-error bars, and dashed lines show the non-private PMS baselines.}
    \label{fig:modal_reg_sin_threepanel}
\end{figure}

Figure~\ref{fig:modal_reg_sin_threepanel} illustrates the qualitative distinction between mean and mode targets. PMS tracks the two modal branches, whereas LOWESS smooths across them. The private estimator preserves the two-branch structure in the representative run. The aggregate curves show the largest gains away from smallest \(n\) and tightest privacy budgets; for larger samples, further increases in \(\varepsilon\) produce smaller reductions in oracle MSE, so the remaining error is closer to the non-private PMS level. Additional modal-regression diagnostics are reported in Appendix~\ref{app:dp-pms-more}, including a complementary three-component piecewise-constant design, full privacy--utility and runtime tables for both regression benchmarks, and sensitivity analyses for  clipping threshold \(C_*\) and minibatch size \(m\).

\subsection{Differentially Private Modal Clustering}
\label{subsec:private-clustering-main}

Private mode estimation also leads naturally to a private clustering procedure. In mode-based clustering, the modes of the density act as cluster representatives, and observations are assigned to nearby modal centers. Building on \textsc{DP-GRAMS}, we obtain \textsc{DP-GRAMS-C} by first privately releasing a set of candidate modes, then merging them into a final collection of private centers, and finally assigning observations deterministically to the released centers. Full pseudocode is given in Algorithm~\ref{alg:dpgrams-c} in Section~\ref{sec:apndx-algo}.

The privacy guarantee applies to the released center set \(\widehat{\mathcal M}\), not to a separately released labeling of the original private sample. In the experiments below, nearest-center assignments are used only for evaluation, for example when computing ARI and NMI against known labels. Thus \textsc{DP-GRAMS-C} should be viewed as a private prototype-release procedure, analogous to differentially private \(k\)-means methods that release private cluster centers \citep{su2016differentially}: a curator releases private cluster representatives, and labels are then obtained by post-processing from those released representatives.

As a private clustering baseline, we use DP-\(k\)-Means via the implementation of \citet{su2016differentially,holohan2019diffprivlib}. Density-threshold private clustering methods, such as private DBSCAN \citep{bozdemir2021privacy,qiu2025approximate}, are also relevant but target a different notion of cluster structure. Since \textsc{DP-GRAMS-C} releases centers and is evaluated with centroid-based metrics, DP-\(k\)-Means provides the closest standard private baseline. We report results at the same nominal \(\varepsilon\), noting that DP-\(k\)-Means satisfies pure \(\varepsilon\)-DP whereas \textsc{DP-GRAMS-C} is calibrated under approximate \((\varepsilon,\delta)\)-DP with  \(\delta=o(n^{-1})\).

Across all clustering experiments, we evaluate methods using adjusted Rand index (ARI), normalized mutual information (NMI), and centroid mean-squared error. The centroid error is computed using the MSE formula \eqref{MMMSE} between released centers and reference centroids. For high-dimensional datasets, a full DAP lattice can become computationally prohibitive because the number of candidate grid points grows rapidly with dimension. We therefore run the image and gene-expression experiments in fixed PCA representations; MNIST uses public auxiliary candidates in PCA space, whereas Cancer RNA-Seq uses the default public DAP grid in reduced PCA space.

\subsubsection{Simulated Gaussian blobs}\label{subsec:private-clustering-sim}

We first evaluate \textsc{DP-GRAMS-C} on a simulated four-component Gaussian-blobs benchmark, where the true cluster identities and centroids are known. We generate two-dimensional datasets with \(k=4\) clusters using \texttt{make\_blobs} from \texttt{scikit-learn} \citep{scikit-learn}, with cluster standard deviation \(1.2\), and vary the sample size over \(n\in\{700,1000,2000,5000\}\). All clustering and centroid comparisons are performed in standardized feature space. For \textsc{DP-GRAMS-C}, privacy is calibrated with \(\delta=10^{-6}\) and \(\varepsilon\in\{0.1,0.2,0.5,1,5\}\). We compare non-private mean-shift clustering, \textsc{DP-GRAMS-C}, standard \(k\)-means, and DP-\(k\)-Means; the target number of clusters is treated as known when applying the final agglomerative merge.

\begin{figure}[!htb]
    \centering
    \includegraphics[width=\linewidth]{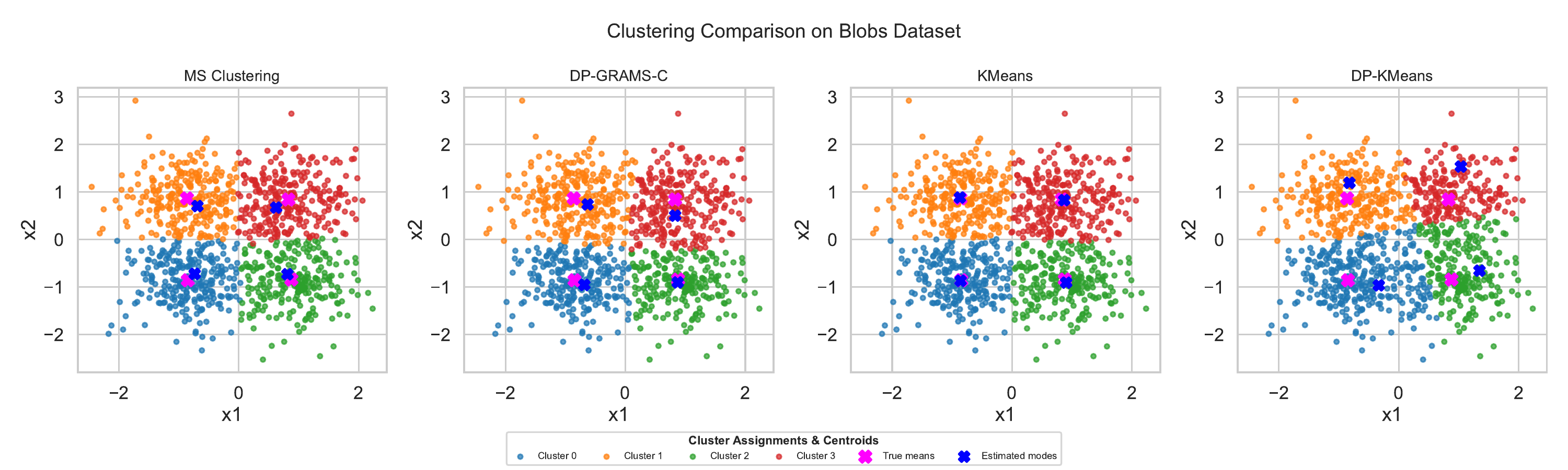}
    \caption{\small Private clustering on a four-component blobs dataset (\(n=1000\)).
    Each panel shows cluster assignments, true centroids, and estimated centroids for mean shift, \textsc{DP-GRAMS-C}, \(k\)-means, and DP-\(k\)-Means, with private methods run at \(\varepsilon=1\).}
    \label{fig:blobs_cluster_comparison}
\end{figure}

\begin{figure}[!htb]
    \centering
    \begin{subfigure}{0.32\linewidth}
        \centering
        \includegraphics[width=\linewidth]{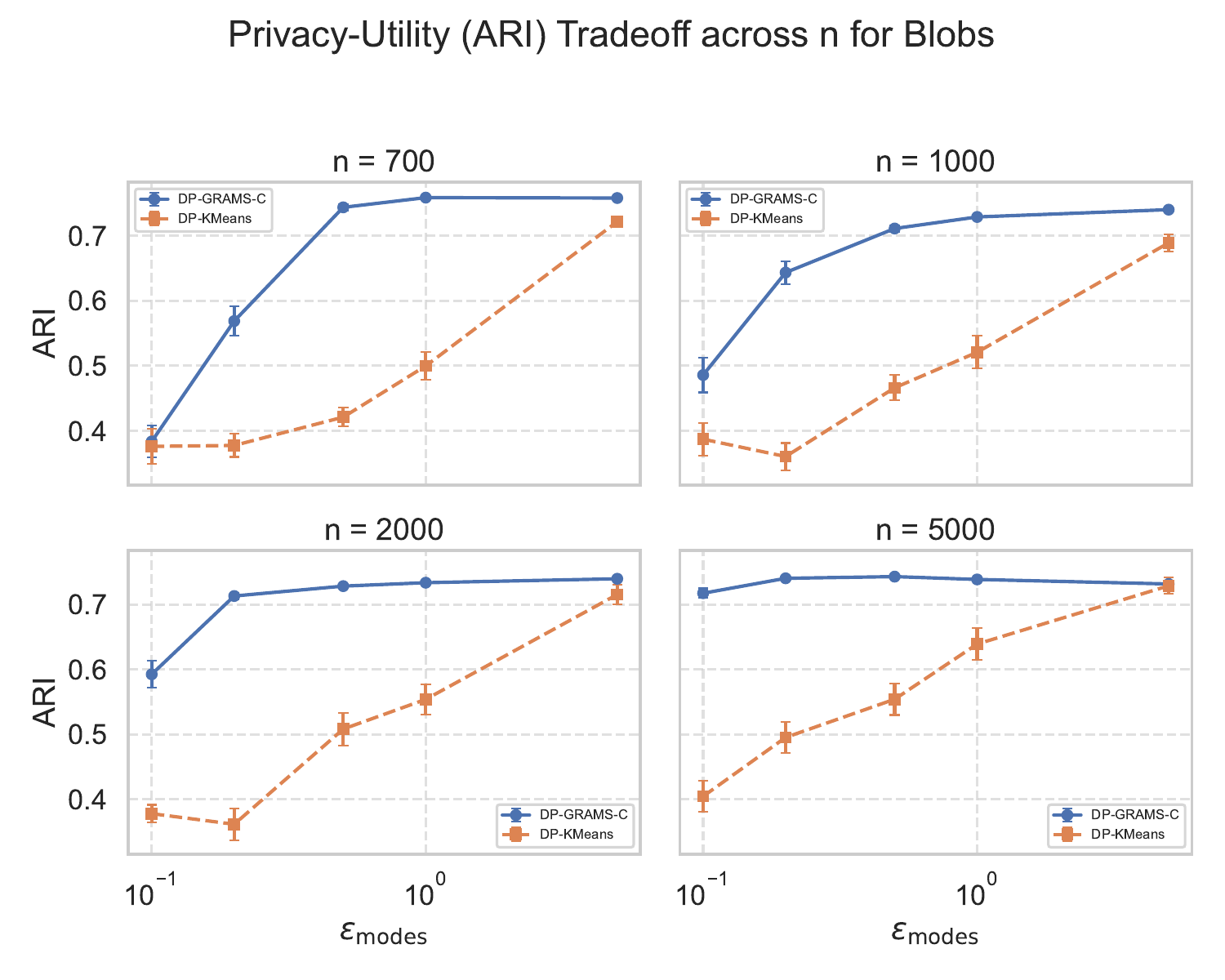}
        \caption{ARI.}
    \end{subfigure}\hfill
    \begin{subfigure}{0.32\linewidth}
        \centering
        \includegraphics[width=\linewidth]{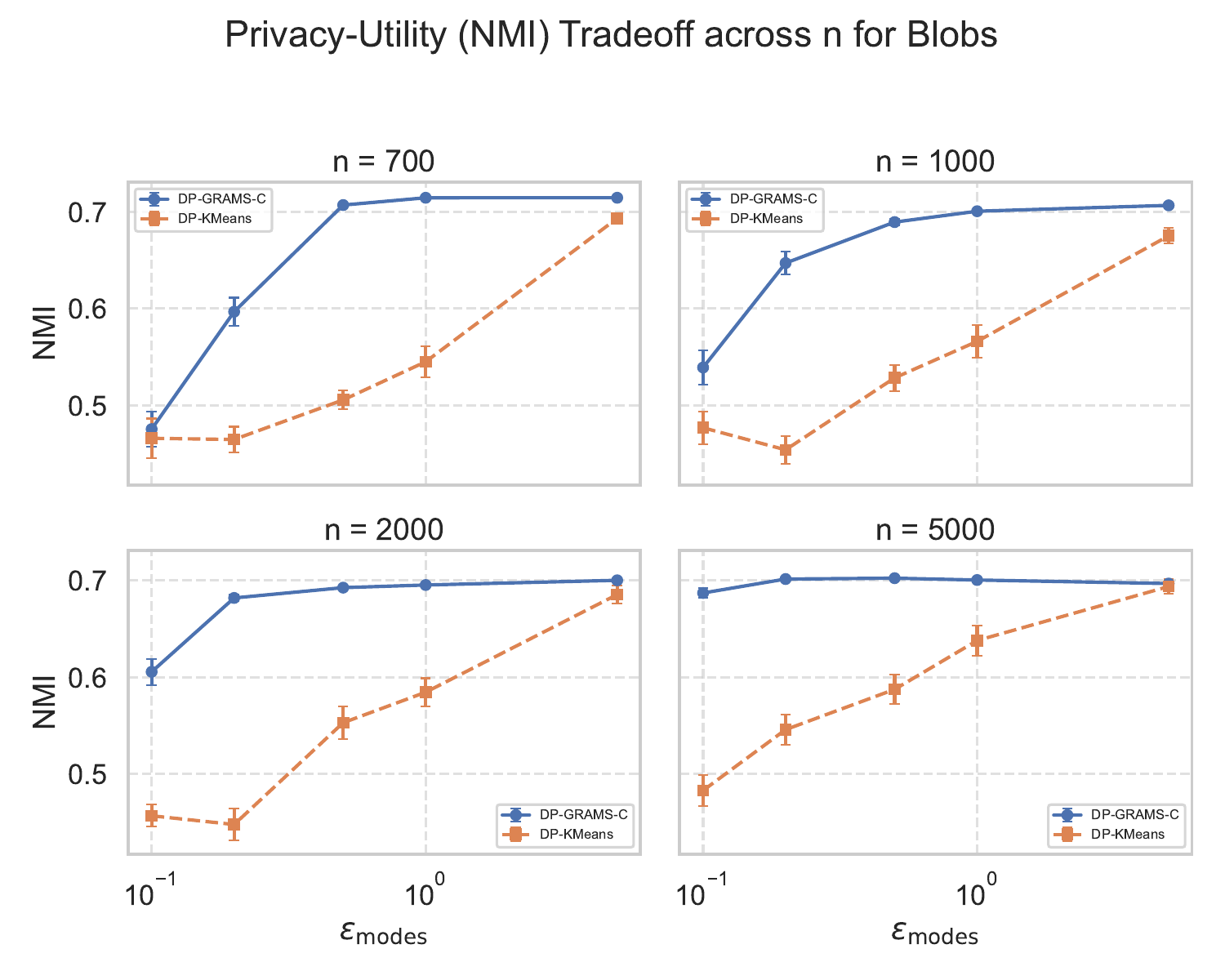}
        \caption{NMI.}
    \end{subfigure}\hfill
    \begin{subfigure}{0.32\linewidth}
        \centering
        \includegraphics[width=\linewidth]{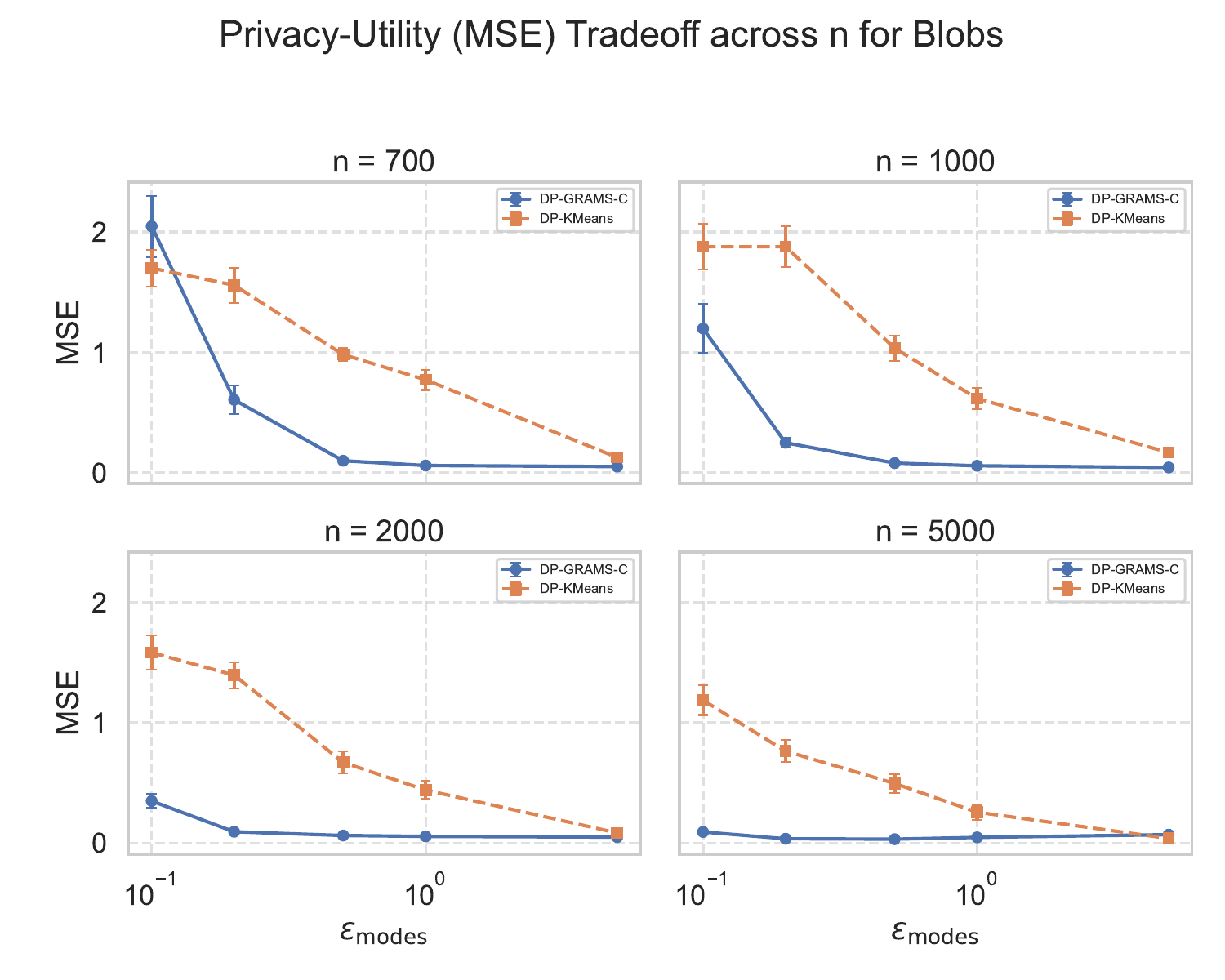}
        \caption{Centroid MSE.}
    \end{subfigure}
    \caption{\small Privacy--utility tradeoff for private clustering on blobs across \(n\in\{700,1000,2000,5000\}\) and \(\varepsilon\in\{0.1,0.2,0.5,1,5\}\). Each panel shows a \(2\times 2\) grid over sample size, plotting ARI, NMI, or centroid MSE versus \(\varepsilon\) on a log scale for \textsc{DP-GRAMS-C} and DP-\(k\)-Means. Points show averages over \(20\) runs with standard-error bars.}
    \label{fig:blobs_privacy_utility_all}
\end{figure}

Figures~\ref{fig:blobs_cluster_comparison} and~\ref{fig:blobs_privacy_utility_all} show that \textsc{DP-GRAMS-C} recovers the underlying cluster geometry well in this clean benchmark. The released centers remain close to the true centroids and to the non-private mean-shift solution. Across privacy budgets and sample sizes, ARI and NMI increase while centroid MSE decreases as either \(n\) or \(\varepsilon\) increases. In this benchmark, \textsc{DP-GRAMS-C} is competitive with DP-\(k\)-Means and is often stronger at moderate privacy budgets. Additional numerical summaries and sensitivity analyses for \(C_*\) and minibatch size \(m\) are reported in Appendix~\ref{app:private-clustering}.

\subsubsection{MNIST}\label{subsec:private-clustering-mnist}

We next study \textsc{DP-GRAMS-C} on MNIST. This experiment is intended to test the method at larger sample size in a moderate-dimensional representation. Because MNIST is public, we construct a stratified public auxiliary candidate set of \(1000\) images and project these candidates into the same PCA space used for clustering. These points are used as DAP candidates, while the remaining MNIST images form the private experimental sample; no full DAP lattice grid is built for MNIST.

We standardize the pixel features, fit a whitened five-dimensional PCA representation, and run clustering in this reduced space. For centroid-error evaluation, estimated centers are mapped back to standardized pixel space using the inverse PCA map and compared against standardized class means computed on the private experimental sample. We compare four methods: non-private mean-shift clustering, \textsc{DP-GRAMS-C}, standard \(k\)-means, and DP-\(k\)-Means. For \textsc{DP-GRAMS-C}, privacy is calibrated with \(\delta=10^{-5}\) and \(\varepsilon\in\{0.05,0.1,0.2,0.5,1\}\).

\begin{figure}[!htb]
    \centering
    \includegraphics[width=0.92\textwidth]{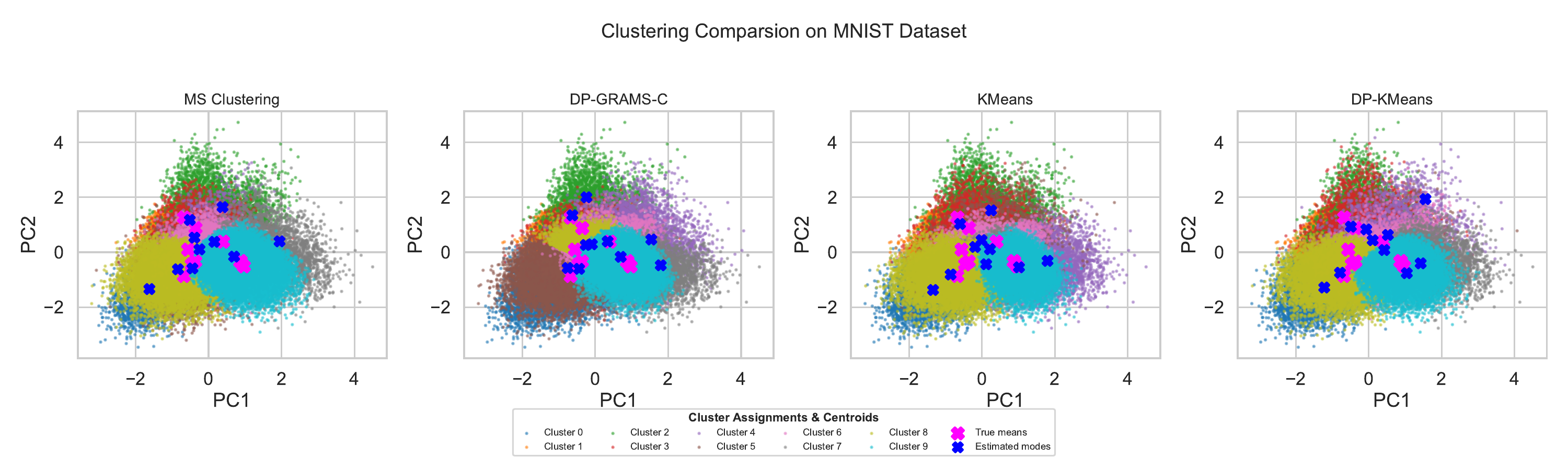}
   \caption{\small MNIST dataset with public auxiliary candidates. Two-dimensional visualization of the five-dimensional PCA clustering representation, comparing mean shift, \textsc{DP-GRAMS-C}, \(k\)-means, and DP-\(k\)-Means, with private methods run at \(\varepsilon=1\). True class centroids and estimated centroids are overlaid.}
    \label{fig:realworld_mnist}
\end{figure}

\begin{figure}[!htb]
    \centering
    \begin{minipage}{0.32\textwidth}
        \includegraphics[width=\linewidth]{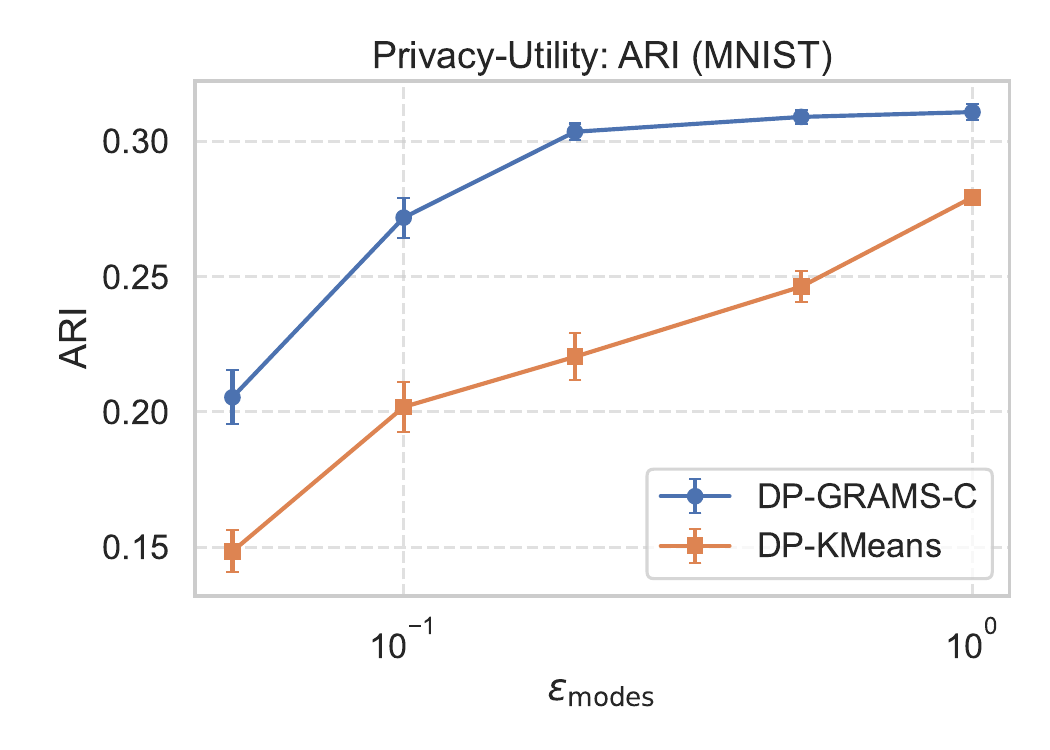}
    \end{minipage}\hfill
    \begin{minipage}{0.32\textwidth}
        \includegraphics[width=\linewidth]{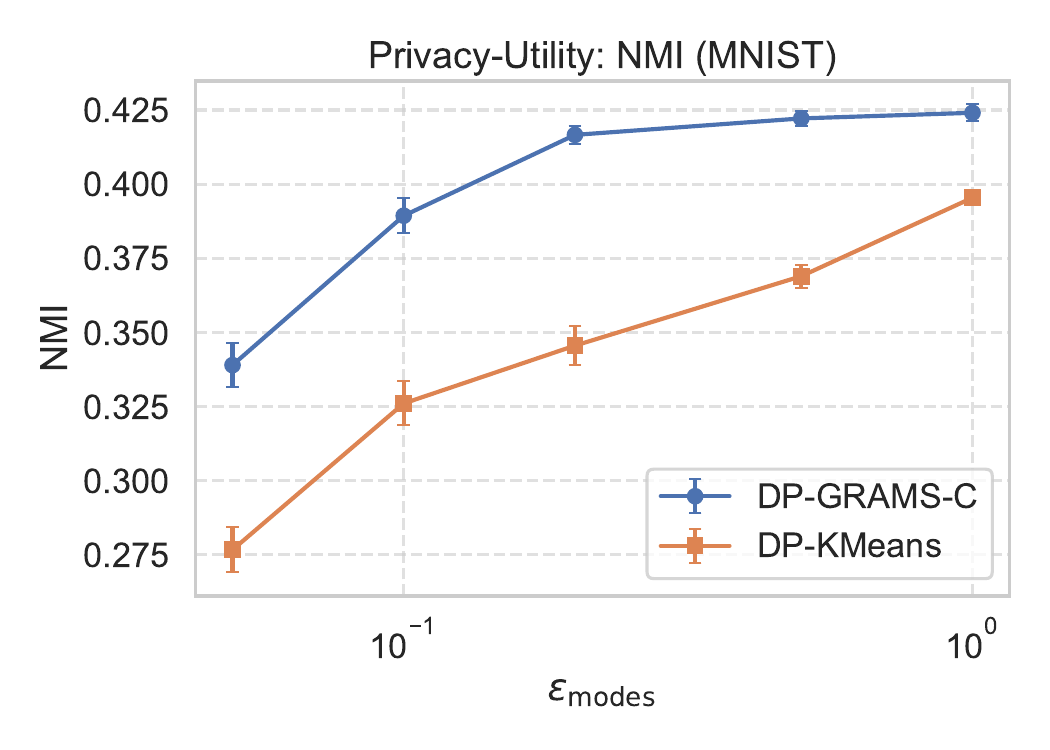}
    \end{minipage}\hfill
    \begin{minipage}{0.32\textwidth}
        \includegraphics[width=\linewidth]{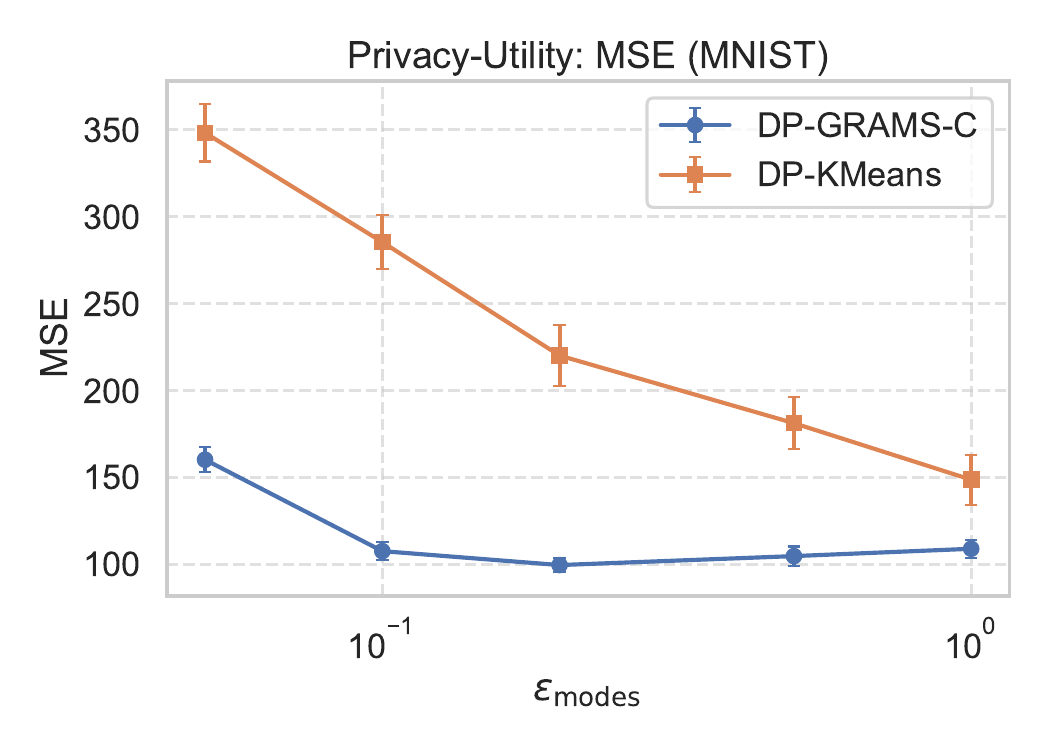}
    \end{minipage}
   \caption{\small Privacy--utility on MNIST with public auxiliary candidates in the whitened five-dimensional PCA representation: ARI, NMI, and centroid MSE versus \(\varepsilon\) on a log scale for \textsc{DP-GRAMS-C} and DP-\(k\)-Means, with \(\varepsilon\in\{0.05,0.1,0.2,0.5,1\}\). Points show averages over \(20\) runs with standard-error bars.}
    \label{fig:privacy_utility_mnist_all}
\end{figure}

Figures~\ref{fig:realworld_mnist} and~\ref{fig:privacy_utility_mnist_all} show that \textsc{DP-GRAMS-C} performs well in the whitened five-dimensional MNIST PCA representation. The public auxiliary candidates avoid the combinatorial growth of a full DAP grid, while the released private centers retain class structure visible in the PCA visualization. The largest improvement occurs from \(\varepsilon=0.05\) to moderate privacy budgets; at \(\varepsilon=1\), \textsc{DP-GRAMS-C} is close to non-private mean shift in ARI and NMI and remains stronger than the DP-\(k\)-Means baseline. Additional MNIST summaries and sensitivity analyses for \(C_*\) and minibatch size \(m\) are reported in Appendix~\ref{subsec:private-clustering-real-appendix}.

\subsubsection{Cancer RNA-Seq}\label{subsec:private-clustering-cancer}

We also evaluate \textsc{DP-GRAMS-C} on the UCI Gene Expression Cancer RNA-Seq dataset of \cite{gene_expression_cancer_rna-seq_401}, which consists of high-dimensional RNA-seq gene-expression profiles with five tumor types (BRCA, COAD, KIRC, LUAD, PRAD). We standardize each gene to zero mean and unit variance, project the data to a whitened six-dimensional PCA representation, and run all clustering methods in this reduced space. Unlike MNIST, this experiment uses the default public DAP grid in the reduced PCA space rather than public auxiliary candidate images. For centroid-error evaluation, estimated centers are mapped back to standardized gene space through the inverse PCA map and compared against standardized class means.

We compare mean-shift clustering, \textsc{DP-GRAMS-C}, standard \(k\)-means, and DP-\(k\)-Means. For \textsc{DP-GRAMS-C}, privacy is calibrated with \(\delta=10^{-5}\) and \(\varepsilon\in\{0.5,1,2,5,10\}\).

\begin{figure}[!htb]
    \centering
    \includegraphics[width=0.95\textwidth]{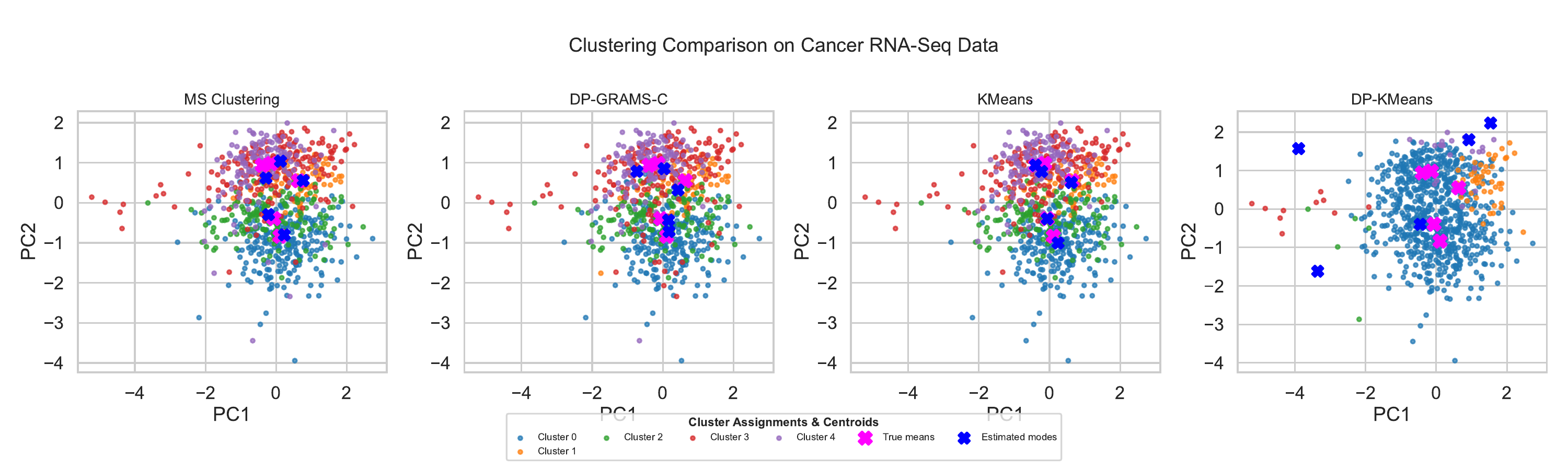}
    \caption{\small Cancer RNA-Seq dataset after gene-wise standardization and projection to six principal components. Two-dimensional PCA visualization of the six-dimensional clustering representation, comparing mean shift, \textsc{DP-GRAMS-C}, \(k\)-means, and DP-\(k\)-Means, with private methods run at \(\varepsilon=1\). True class centroids and estimated centroids are overlaid.}
    \label{fig:gene50_cluster_comparison}
\end{figure}

\begin{figure}[!htb]
    \centering
    \begin{minipage}{0.32\textwidth}
        \includegraphics[width=\linewidth]{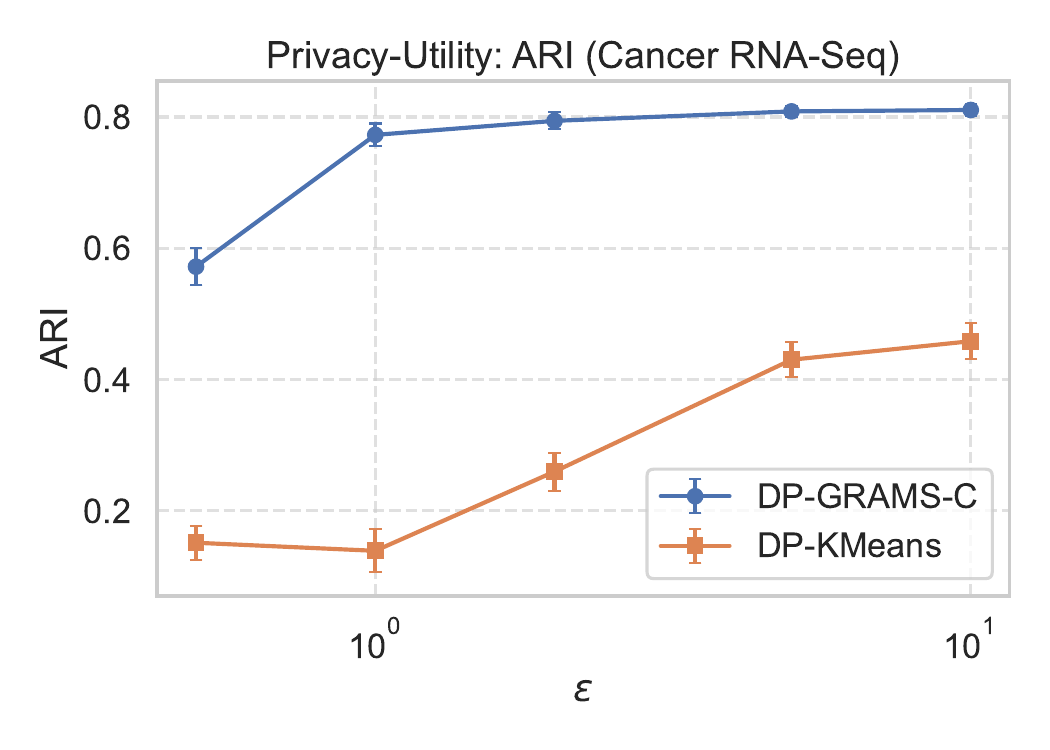}
    \end{minipage}\hfill
    \begin{minipage}{0.32\textwidth}
        \includegraphics[width=\linewidth]{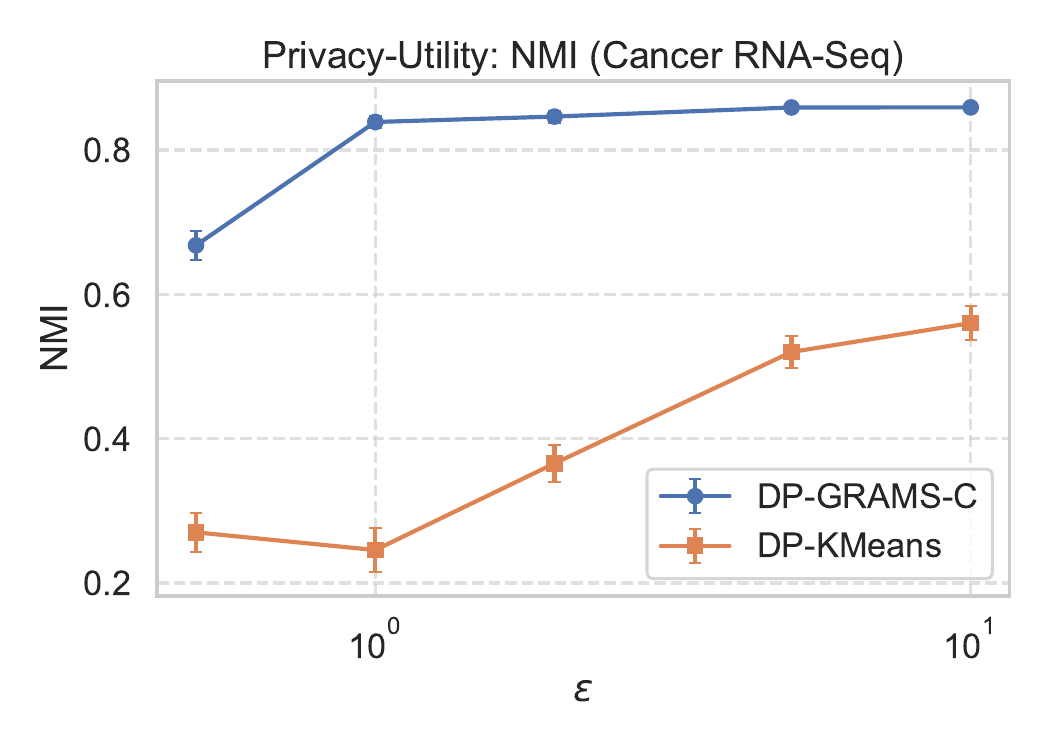}
    \end{minipage}\hfill
    \begin{minipage}{0.32\textwidth}
        \includegraphics[width=\linewidth]{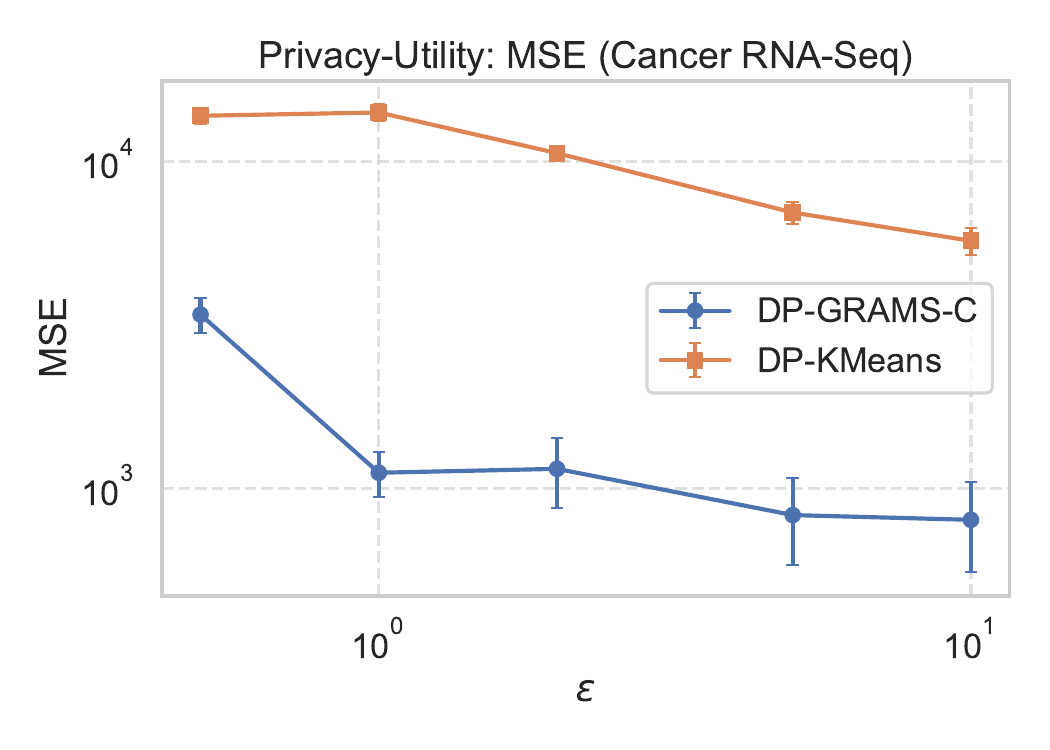}
    \end{minipage}
    \caption{\small Privacy--utility on Cancer RNA-Seq after gene-wise standardization and projection to six principal components: ARI, NMI, and centroid MSE versus \(\varepsilon\) on a log scale for \textsc{DP-GRAMS-C} and DP-\(k\)-Means, with \(\varepsilon\in\{0.5,1,2,5,10\}\). Points show averages over \(20\) runs with standard-error bars.}
    \label{fig:gene50_privacy_utility_all}
\end{figure}

\begin{table}[!htb]
\centering
\caption{\small One-run clustering performance at nominal privacy budget \(\varepsilon=1\) across datasets. MSE denotes centroid mean-squared error.}
\label{tab:real_world_clustering_updated}\small
\setlength{\tabcolsep}{4.5pt}
\renewcommand{\arraystretch}{1.15}
\begin{tabular}{lccccccccc}
\toprule
& \multicolumn{3}{c}{\textbf{Blobs}} & \multicolumn{3}{c}{\textbf{MNIST}} & \multicolumn{3}{c}{\textbf{Cancer RNA-Seq}} \\
\cmidrule(lr){2-4}\cmidrule(lr){5-7}\cmidrule(lr){8-10}
\textbf{Algorithm} & ARI & NMI & MSE & ARI & NMI & MSE & ARI & NMI & MSE \\
\midrule
Mean shift          & 0.760 & 0.720 & 0.0437 & 0.314 & 0.424 & 113.294 & 0.816 & 0.846 & 527.335 \\
\textsc{DP-GRAMS-C} & 0.727 & 0.704 & 0.0562 & 0.331 & 0.433 & 53.489  & 0.797 & 0.845 & 700.850 \\
\(k\)-Means         & 0.753 & 0.715 & 0.0018 & 0.295 & 0.409 & 68.104  & 0.817 & 0.863 & 340.236 \\
DP-\(k\)-Means      & 0.620 & 0.638 & 0.2904 & 0.302 & 0.411 & 176.529 & 0.054 & 0.120 & 14822.003 \\
\bottomrule
\end{tabular}
\end{table}

Figures~\ref{fig:gene50_cluster_comparison} and~\ref{fig:gene50_privacy_utility_all} show that the private mode-based procedure remains effective on the Cancer RNA-Seq task after reduction to a whitened six-dimensional PCA representation. The released centers preserve the class separation visible in the PCA display, and the aggregate curves show a sharp improvement from \(\varepsilon=0.5\) to \(\varepsilon=1\): ARI and NMI move near the non-private range, while centroid MSE decreases as the privacy budget increases. Table~\ref{tab:real_world_clustering_updated} provides a cross-dataset one-run comparison at \(\varepsilon=1\), where \textsc{DP-GRAMS-C} is close to non-private mean shift in ARI and NMI and much stronger than DP-\(k\)-Means on Cancer RNA-Seq. Additional Cancer RNA-Seq summaries, including full privacy--utility tables and sensitivity analyses for \(C_*\) and minibatch size \(m\), are reported in Appendix~\ref{subsec:private-clustering-cancer-appendix}.

\section{Discussion}\label{sec:discussion}

This paper introduced \textsc{DP-GRAMS}, a differentially private mode-seeking algorithm that leverages the equivalence between mean shift and gradient ascent on the log-density, bringing differentially private stochastic optimization tools to nonparametric mode estimation. The key insight is to decompose the log-KDE score field into per-sample contributions, clip those contributions to control sensitivity, and add Gaussian noise calibrated via standard $(\varepsilon,\delta)$-DP accounting. This perspective enables principled privatization of a classical nonparametric algorithm while preserving its underlying geometric interpretation.


Our analysis points to multiple directions that might be of interest. Firstly, initialization plays a crucial role in a mode seeking problem, especially in settings where local modal basins can be separated by low-density regions. To address this under privacy constraints, we incorporate a simple differentially private initialization scheme that combines a density-aware exponential-mechanism utility with a local suppression step. In our analysis, the public \(h_{\mathrm{DAP}}\)-grid together with the suppression rule yields a private initialization scheme that visits every modal basin with high probability using only \(k\asymp M\log n\) draws. 
Further improving the efficiency of private initialization might be a promising direction. Secondly, from a privacy accounting perspective, we adopt the standard $(\varepsilon,\delta)$-differential privacy framework rather than alternatives such as R\'enyi DP \cite{mironov2017renyi} or zero-Concentrated DP \cite{bun2016concentrated}. While RDP- and zCDP-based analyses can yield tighter composition bounds in some regimes, $(\varepsilon,\delta)$-DP remains a widely used notion with a direct and interpretable guarantee. Investigating RDP or zCDP variants for score and mode estimation
may be another avenue for future work.

Finally, there is potential in adapting our results to unknown smoothness and additional structures that the true density might enjoy. Adapting to the smoothness $\beta$, with and without privacy is interesting: see, e.g., \cite{lepskii1991problem,kroll2019pointwise,butucea2020local,schluttenhofer2022adaptive,auddy2025minimax}. More interestingly, deep learning based estimators have recently shown immense promise in adapting to underlying dimensionality and specific dependence patterns of the score function: see, e.g., \cite{nakada2020adaptive, song2019generative,oko2023diffusion}. Advancing these results to incorporate differential privacy is of both theoretical and practical interest. We intend to pursue this in the future.


\bibliographystyle{plainnat}  
\bibliography{references}

\appendix


\section{Proofs}\label{sec:main-proofs}

This appendix proves the theoretical results in the order in which they are used. We begin with the privacy statements, then assemble the local analytic and probabilistic ingredients behind Theorem~\ref{th:local-conv-holder}. We next prove DAP coverage and combine it with the local theorem and deterministic post-processing to prove Theorem~\ref{thm:global-holder}. We then prove the minimax lower bound. The final subsection collects the remaining auxiliary proofs in the same order of invocation.

\subsection{Privacy results}

We first verify privacy one stage at a time. The initialization pool is handled through repeated exponential-mechanism draws with suppression, after which the end-to-end guarantee follows by composing that initialization step with the correlated Gaussian ascent mechanism.

\begin{proof}[Proof of Theorem~\ref{thm:dp_init_privacy}]
Fix neighboring datasets \(\mathcal X=(X_1,\dots,X_n)\) and \(\mathcal X'=(X_1',\dots,X_n')\) differing in one entry. Since the candidate set \(\mathcal Z\) is fixed independently of the data, each score
\[
u_j(\mathcal X)
=
\frac1n\sum_{i=1}^n \mathbf 1\{\|X_i-z_j\|\le h_{\mathrm{DAP}}\}
\]
has global sensitivity at most \(1/n\).

Set \(\varepsilon_{\mathrm{draw}}=\varepsilon_{\mathrm{init}}/k\). Fix a round \(\ell\in[k]\), and condition on the previously selected indices \(J_1,\dots,J_{\ell-1}\). Under this conditioning, the set \(A_\ell\) is fixed. Hence the range from which \(J_\ell\) is sampled is also fixed. More precisely, define
\[
R_\ell
:=
\begin{cases}
A_\ell, & \text{if } A_\ell\neq\varnothing,\\[1mm]
[N_{\mathrm{cand}}], & \text{if } A_\ell=\varnothing.
\end{cases}
\]
Then Algorithm~\ref{alg:dap-init} samples \(J_\ell\) from this fixed range \(R_\ell\). Thus, for \(j\in[N_{\mathrm{cand}}]\),
\[
\Pr(J_\ell=j\mid R_\ell,\mathcal X)
=
\frac{
\exp\!\left(\frac{n\varepsilon_{\mathrm{draw}}}{2}u_j(\mathcal X)\right)\mathbf 1\{j\in R_\ell\}
}{
\sum_{r\in R_\ell}
\exp\!\left(\frac{n\varepsilon_{\mathrm{draw}}}{2}u_r(\mathcal X)\right)
}.
\]
This is exactly the exponential mechanism on the fixed restricted range \(R_\ell\) with score sensitivity \(1/n\) \citep{mcsherry2007mechanism}. Hence round \(\ell\) is \((\varepsilon_{\mathrm{draw}},0)\)-DP.

By sequential composition, the full index sequence \((J_1,\dots,J_k)\) is
\[
\Big(\sum_{\ell=1}^k \varepsilon_{\mathrm{draw}},\,0\Big)
=
(\varepsilon_{\mathrm{init}},0)\text{-DP}
\]
\citep[see, e.g.,][]{dwork2014algorithmic}. Finally, since \(a_\ell=z_{J_\ell}\) and \(x_{0,\ell}=a_\ell\), both the anchor sequence \((a_1,\dots,a_k)\) and the initialization pool \(\mathcal I=\{x_{0,1},\dots,x_{0,k}\}\) are deterministic post-processing of \((J_1,\dots,J_k)\), and are therefore also \((\varepsilon_{\mathrm{init}},0)\)-DP \citep[Proposition~2.1]{dwork2014algorithmic}.
\end{proof}

\begin{proof}[Proof of Corollary~\ref{cor:full_dp}]
By Theorem~\ref{thm:dp_init_privacy}, the DAP initialization pool
\[
\mathcal I=\{x_{0,1},\dots,x_{0,k}\}
\]
is \((\varepsilon_{\mathrm{init}},0)\)-DP. Conditional on any fixed initialization pool \(\mathcal I\), Lemma~\ref{lem:correlated_noise_privacy} shows that the remainder of Algorithm~\ref{alg:dpgrams}, including the final merged estimator \(\widehat{\mathcal M}\), is \((\varepsilon_{\mathrm{modes}},\delta)\)-DP. Therefore, by sequential composition,
\[
(\varepsilon_{\mathrm{init}},0)+(\varepsilon_{\mathrm{modes}},\delta)
=
(\varepsilon_{\mathrm{init}}+\varepsilon_{\mathrm{modes}},\delta),
\]
so the complete \textsc{DP-GRAMS} algorithm is \((\varepsilon_{\mathrm{init}}+\varepsilon_{\mathrm{modes}},\delta)\)-DP \citep[see, e.g.,][]{dwork2014algorithmic}.
\end{proof}

\subsection{Local convergence: ingredient statements and proof roadmap}

Fix \(j\in[M]\) and consider a single \textsc{DP-GRAMS} trajectory \(x_0,\dots,x_T\) initialized in \(\overline B(\mu_j,r_j)\). In this local block, \(x_t\) denotes this single trajectory and \(z_t\sim \mathcal N(0,\sigma^2 I_d)\) denotes its Gaussian perturbation at round \(t\). The argument starts from the local geometry of \(\log p\) near \(\mu_j\), upgrades this to uniform control of the KDE and log-KDE on the modal basin, packages those bounds into the event \(\mathcal A_n^{\mathrm{local},j}\), and then combines the stopped recursion with basin retention before optimizing in \(h\).

\begin{lemma}[Continuity of $\nabla^2 \log p$ at each mode]\label{lem:logp-holder}
Under Assumption~\ref{assump:model-holder} with \(\beta>2\), for each \(j\in[M]\),
the map \(x\mapsto \nabla^2\log p(x)\) is continuous at \(\mu_j\). Consequently,
for every \(\xi>0\) there exists \(\widetilde r_j(\xi)>0\) such that
\[
\big\|\nabla^2\log p(x)-\nabla^2\log p(\mu_j)\big\|
\le \xi,
\qquad \forall x\in \overline B(\mu_j,\widetilde r_j(\xi)).
\]
In particular, under Assumption~\ref{assump:hessian-holder}, taking
\(\xi=\alpha_j/2\) yields
\[
\nabla^2\log p(x)\preceq -\frac{\alpha_j}{2}I_d,
\qquad \forall x\in \overline B(\mu_j,\widetilde r_j(\alpha_j/2)).
\]
\end{lemma}


\begin{lemma}[Derived inward drift on the local neighborhood]\label{lem:derived-inward-drift}
Assume $\beta>2$ and suppose Assumptions~\ref{assump:model-holder}, \ref{assump:hessian-holder}, and \ref{assump:radius-holder} hold. Then, for each $j\in[M]$ and every $x\in \overline B(\mu_j,r_j)$,
\[
\langle x-\mu_j,\nabla\log p(x)\rangle
\le
-\frac{\alpha_j}{2}\|x-\mu_j\|^2.
\]
\end{lemma}

The next elementary geometric observation records the role of the radius restriction in Assumption~\ref{assump:radius-holder}: the chosen local basin neighborhood around \(\mu_j\) contains no other population mode.

%


\begin{lemma}[Uniform KDE derivative rates on local balls]\label{lem:uniform_kde_ball}
Assume $\beta>2$ and let $\ell=\lfloor\beta\rfloor\ge 2$. Let $\hat p$ be the KDE defined in
Section~\ref{sec:dp-grams}, based on i.i.d.\ data from a density $p$ satisfying
Assumptions~\ref{assump:model-holder} and \ref{assump:bandwidth}, with kernel $K$
satisfying Assumption~\ref{assump:kernel}. Fix $j\in[M]$. Then with probability
at least $1-n^{-4}$, simultaneously for all $s\in\{0,1,2\}$,
\[
\sup_{x\in\overline B(\mu_j,r_j)}
\|\nabla^s\hat p(x)-\nabla^s p(x)\|
\le
C\!\left(h^{\beta-s}+\sqrt{\frac{\log n}{n h^{d+2s}}}\right),
\]
where $C>0$ depends only on $d,\beta$, the local Hölder constants of $p$ on
$\overline B(\mu_j,r_j)$, and kernel moments.
\end{lemma}

\begin{lemma}[Uniform lower bound for $\hat p$ on mode basins]\label{lem:kde-lower-bound}
Let $\hat p$ be the KDE from Section~\ref{sec:dp-grams}, built from i.i.d.\ samples drawn
from $p$ satisfying Assumptions~\ref{assump:model-holder}, \ref{assump:bandwidth}, and
\ref{assump:kernel}. Fix $j\in[M]$. Then there exists a constant $c_j>0$ such that for all
sufficiently large $n$,
\[
\Pr\!\left(\inf_{x\in\overline B(\mu_j,r_j)} \hat p(x)\ge c_j\right)\ge 1-n^{-4}.
\]
\end{lemma}

\begin{lemma}[Uniform-in-$x$ log-KDE derivative rates on local balls]\label{lem:uniform_log_ball}
Fix $j\in[M]$ and assume $\beta>2$. Let $\ell=\log p$ and $\hat\ell=\log\hat p$.
There exists a constant $C_j>0$ (depending only on $p_{\min,j}$ and
$\sup_{x\in\overline B(\mu_j,r_j)}\|\nabla^r p(x)\|$ for $r\le 2$) such that, for all sufficiently large $n$,
with probability at least $1-2n^{-4}$, simultaneously for all $s\in\{0,1,2\}$,
\[
\sup_{x\in\overline B(\mu_j,r_j)}
\big\|\nabla^{s}\hat\ell(x)-\nabla^{s}\ell(x)\big\|
\;\le\;
C_j\!\left(h^{\beta-s}+\sqrt{\frac{\log n}{n h^{d+2s}}}\right).
\]
\end{lemma}

\begin{lemma}[High-probability analytic local event]\label{lem:local-good-prob}
Under Assumptions~\ref{assump:model-holder}, \ref{assump:bandwidth}, and \ref{assump:kernel}, for each \(j\in[M]\),
\[
\Pr\big(\mathcal A_{n}^{\mathrm{local},j}\big) \ge 1 - 4n^{-4}
\]
for all sufficiently large \(n\).
\end{lemma}

The next steps use \(\mathcal A_n^{\mathrm{local},j}\) to control the stopped local dynamics and then intersect it with \(\mathcal E_{n,T}^{\mathrm{stay},j}\) to recover the full event \(\mathcal G_{n,T}^{\mathrm{local},j}\).

\begin{lemma}[Local score-contribution bound]\label{lem:local-qi-bound}
Let us write
\(
p_{\min}:=\min_{j\in[M]}p_{\min,j},
\) and 
\(C_*:=\frac{2G_K}{p_{\min}}\,h^{-(d+1)}
\). For
\[
q_i(x):=\frac{g_i(x)}{\hat p(x)},
\qquad
g_i(x):=\frac{1}{h^{d+1}}\nabla K\!\left(\frac{x-X_i}{h}\right),
\]
one has, for every \(j\in[M]\), on \(\mathcal A_n^{\mathrm{local},j}\),
\[
\sup_{x\in\overline B(\mu_j,r_j)}\max_{1\le i\le n}\|q_i(x)\|\le C_*
\]
for all sufficiently large \(n\). In particular, \(C_*\asymp h^{-(d+1)}\), which is the local scale used to control the privacy perturbation in the proof of Theorem~\ref{th:local-conv-holder}.
\end{lemma}

\begin{proposition}[Local inactivity of stabilization]\label{prop:local-inactive}
Suppose $\beta>2$ and let Assumptions~\ref{assump:kernel},
\ref{assump:model-holder}, and \ref{assump:bandwidth} hold. Let us fix
\(
p_{\min}:=\min_{j\in[M]} p_{\min,j},
\),  
\(0<c_{\mathrm{floor}}\le \tfrac12,
\)
and set
\(
p_{\mathrm{floor}}=c_{\mathrm{floor}}\,p_{\min}
\). 
Let us choose $A>0$ so that
\(
A>2\max_{j\in[M]}\sup_{x\in \overline B(\mu_j,r_j)}\|\nabla p(x)\|
\). 
Then, for each $j\in[M]$, on $\mathcal A_n^{\mathrm{local},j}$ and for all sufficiently large $n$,
\[
\hat p(x)\ge p_{\mathrm{floor}},
\qquad
\|\nabla\hat p(x)\|\le A,
\qquad x\in \overline B(\mu_j,r_j).
\]
Consequently,
\(
\hat s_{A,p_{\mathrm{floor}};\mathcal X}(x)=\nabla\log\hat p(x)
\qquad\text{for all }x\in \overline B(\mu_j,r_j).
\)
\end{proposition}

\medskip

\begin{proof}[Proof of Proposition~\ref{prop:local-inactive}]
Fix \(j\in[M]\) and work on \(\mathcal A_n^{\mathrm{local},j}\). By
Definition~\ref{def:good-event-local-holder},
\[
\hat p(x)\ge c_j=\frac{p_{\min,j}}{2}\ge \frac{p_{\min}}{2}\ge p_{\mathrm{floor}}
\qquad\text{for all }x\in \overline B(\mu_j,r_j),
\]
so the denominator floor is inactive on the whole basin. It remains to show that gradient clipping is also inactive. Let
\[
a_n:=
\sup_{x\in\overline B(\mu_j,r_j)}|\hat\ell(x)-\ell(x)|,
\qquad
b_n:=
\sup_{x\in\overline B(\mu_j,r_j)}\|\nabla\hat\ell(x)-\nabla\ell(x)\|.
\]
On \(\mathcal A_n^{\mathrm{local},j}\), we have \(a_n\to 0\) and \(b_n\to 0\). Since
\(\hat p(x)=p(x)e^{\hat\ell(x)-\ell(x)}\),
\[
\sup_{x\in\overline B(\mu_j,r_j)}|\hat p(x)-p(x)|
\le
p_{\max,j}\bigl(e^{a_n}-1\bigr)
\to 0,
\]
where
\(
p_{\max,j}:=\sup_{x\in\overline B(\mu_j,r_j)} p(x)<\infty.
\)
Also, using
\[
\nabla \hat p(x)=\hat p(x)\nabla\hat\ell(x),
\qquad
\nabla p(x)=p(x)\nabla\ell(x),
\]
we obtain
\begin{align*}
\|\nabla\hat p(x)-\nabla p(x)\|
&\le
|\hat p(x)-p(x)|\,\|\nabla\hat\ell(x)\|
+
p(x)\,\|\nabla\hat\ell(x)-\nabla\ell(x)\| \\
&\le
|\hat p(x)-p(x)|\bigl(\|\nabla\ell(x)\|+b_n\bigr)+p_{\max,j} b_n .
\end{align*}
Taking suprema over \(x\in\overline B(\mu_j,r_j)\), and using compactness of
\(\overline B(\mu_j,r_j)\) together with continuity of \(p\) and \(\ell\), gives
\[
\sup_{x\in\overline B(\mu_j,r_j)}\|\nabla\hat p(x)-\nabla p(x)\|\to 0
\qquad\text{on }\mathcal A_n^{\mathrm{local},j}.
\]
This convergence is the \(s=1\) derivative control in Lemma~\ref{lem:uniform_kde_ball}, used under Assumption~\ref{assump:bandwidth}. Therefore
\[
\sup_{x\in\overline B(\mu_j,r_j)}\|\nabla\hat p(x)\|
\le
\sup_{x\in\overline B(\mu_j,r_j)}\|\nabla p(x)\|+o(1)
< A
\]
for all sufficiently large \(n\), by the choice of \(A\).

Hence, on \(\mathcal A_n^{\mathrm{local},j}\) and for all sufficiently large \(n\),
\(
\operatorname{clip}_A(\nabla\hat p(x))=\nabla\hat p(x)
\text{ and }
\max\{\hat p(x),p_{\mathrm{floor}}\}=\hat p(x)
\)
for all \(x\in \overline B(\mu_j,r_j)\). Therefore
\[
\hat s_{A,p_{\mathrm{floor}};\mathcal X}(x)=\nabla\log\hat p(x)
\qquad\text{for all }x\in \overline B(\mu_j,r_j),
\]
as claimed.
\end{proof}

We now collect the local bounds used to control the stabilized minibatch field. Lemma~\ref{lem:local-qi-bound} gives a uniform bound on the samplewise score contributions, Lemma~\ref{lem:minibatch-field-second} gives a conditional second-moment bound for the stabilized minibatch field, and Lemma~\ref{lem:minibatch-field-sup} upgrades this to a high-probability fluctuation bound along adapted trajectories.

\begin{lemma}[Second-moment control of the minibatch stabilized field]\label{lem:minibatch-field-second}
Assume $\beta>2$ and suppose Assumptions~\ref{assump:kernel},
\ref{assump:model-holder}, and \ref{assump:bandwidth} hold. Fix $j\in[M]$ and work on
$\mathcal A_n^{\mathrm{local},j}$. Then there exists a constant $C_{j,\zeta}>0$ such that, for every deterministic $x\in \overline B(\mu_j,r_j)$,
\[
\EE\!\left[\left\|\hat s_{A,p_{\mathrm{floor}};\mathcal B_t}(x)-\nabla\log\hat p(x)\right\|^2\,\middle|\,x,\mathcal X,\mathcal A_n^{\mathrm{local},j}\right]
\le
\frac{C_{j,\zeta}}{m h^{d+2}}.
\]
\end{lemma}

\begin{lemma}[High-probability minibatch fluctuation of the stabilized field]\label{lem:minibatch-field-sup}
Assume \(\beta>2\) and suppose Assumptions~\ref{assump:kernel},
\ref{assump:model-holder}, and \ref{assump:bandwidth} hold. Fix \(j\in[M]\), assume
\[
\frac{m h^{d+2}}{\log(eTn)}\ge C,
\]
for a sufficiently large constant $C>0$. Let \((x_t)_{t=0}^{T-1}\) be any process adapted to the algorithmic history such that \(x_t\in \overline B(\mu_j,r_j)\) for all \(t\). Then there exists a constant \(C_{j,\mathrm{mb}}>0\) such that
\[
\Pr\!\left(
\max_{0\le t\le T-1}
\left\|
\hat s_{A,p_{\mathrm{floor}};\mathcal B_t}(x_t)-\nabla\log\hat p(x_t)
\right\|
\le
C_{j,\mathrm{mb}}\sqrt{\frac{\log(eTn)}{m h^{d+2}}}
\,\middle|\,
\mathcal X,\mathcal A_n^{\mathrm{local},j}
\right)
\ge 1-T(eTn)^{-6}
\]
for all sufficiently large \(n\).
\end{lemma}

\begin{lemma}[Local score bounds for \(k_n\)]\label{lem:kn-bound-holder}
Assume \(\beta>2\) and suppose Assumptions~\ref{assump:kernel},
\ref{assump:model-holder}, \ref{assump:bandwidth}, \ref{assump:hessian-holder},
and \ref{assump:radius-holder} hold. Fix \(j\in[M]\). For \(x_t\in \overline B(\mu_j,r_j)\), write
\[
\delta_t:=x_t-\mu_j,
\qquad
k_n(x):=\nabla\log\hat p(x)=\nabla\hat\ell(x).
\]
Then, on \(\mathcal A_n^{\mathrm{local},j}\), there exists a constant \(C_j>0\) such that for all sufficiently large \(n\),
\[
\langle \delta_t,k_n(x_t)\rangle
\le
-\frac{3\alpha_j}{8}\|\delta_t\|^2
+
C_j\!\Big(h^{2(\beta-1)}+\frac{\log n}{n h^{d+2}}\Big)
+
C_j\!\Big(h^{\beta-2}+\sqrt{\frac{\log n}{n h^{d+4}}}\Big)\|\delta_t\|^2,
\]
and
\[
\|k_n(x_t)\|^2
\le
2H_j^2\|\delta_t\|^2
+
C_j\!\Big(h^{2(\beta-2)}+\frac{\log n}{n h^{d+4}}\Big)\|\delta_t\|^2
+
C_j\!\Big(h^{2(\beta-1)}+\frac{\log n}{n h^{d+2}}\Big).
\]
\end{lemma}

\begin{lemma}[Finite-population second moment for the local score average]\label{lem:T1-bound-holder}
Assume the hypotheses of Lemma~\ref{lem:kn-bound-holder}. For
\[
\bar q_t:=\frac{1}{m}\sum_{i\in\mathcal B_t} q_i(x_t),
\qquad
q_i(x_t)=\frac{g_i(x_t)}{\hat p(x_t)},
\]
there exist constants \(C_{3,j},C_{4,j}>0\) such that, on \(\mathcal A_n^{\mathrm{local},j}\), for every \(x_t\in \overline B(\mu_j,r_j)\) and all sufficiently large \(n\),
\[
\EE[\|\bar q_t\|^2\mid x_t,\mathcal X,\mathcal A_n^{\mathrm{local},j}]
\le
C_{3,j}\|\delta_t\|^2
+
C_{4,j}\left(
h^{2(\beta-1)}
+
\frac{\log n}{n h^{d+2}}
+
\frac{1}{m h^{d+2}}
\right).
\]
\end{lemma}

\begin{lemma}[Minibatch remainder bound]\label{lem:T3-bound-holder}
Let us define
\[
r_t:=\hat s_{A,p_{\mathrm{floor}};\mathcal B_t}(x_t)-\bar q_t,
\qquad
\bar q_t:=\frac{1}{m}\sum_{i\in\mathcal B_t} q_i(x_t).
\]
Then, on $\mathcal A_n^{\mathrm{local},j}$ and for $x_t\in \overline B(\mu_j,r_j)$,
\[
\EE[\|r_t\|^2\mid x_t,\mathcal X,\mathcal A_n^{\mathrm{local},j}]
\lesssim \frac{1}{m h^{d+2}}.
\]
\end{lemma}

For the local proofs, fix \(j\in[M]\) and write \(\ell=\log p\). Define
\[
H_j:=\sup_{x\in \overline B(\mu_j,r_j)}\|\nabla^2 \ell(x)\|<\infty,
\qquad
L_j:=\sup_{x\in \overline B(\mu_j,r_j)}\|\nabla\log p(x)\| .
\]
The finiteness of \(H_j\) follows from Lemma~\ref{lem:logp-holder} and compactness of \(\overline B(\mu_j,r_j)\). Fix any deterministic \(B_j>L_j\), and set
\[
\eta_{0,j}:=\frac{\alpha_j r_j^2}{8B_j^2},
\qquad
\eta_{1,j}:=\frac{\alpha_j}{48(H_j^2+1)},
\qquad
\bar\eta_j:=\min\{\eta_{0,j},\eta_{1,j}\}.
\]

The next proposition gives the fixed-bandwidth local error bound on the stay-in-basin event, conditional on the static analytic event \(\mathcal A_n^{\mathrm{local},j}\). Basin retention is handled afterward and will then allow us to pass from the stopped estimate to the full local theorem.

\begin{proposition}[General stopped local error bound]\label{prop:local-truncated-holder}
Assume \(\beta>2\) and suppose Assumptions~\ref{assump:kernel},
\ref{assump:model-holder}, \ref{assump:bandwidth}, \ref{assump:hessian-holder},
and \ref{assump:radius-holder} hold. Let \(x_0\in\overline B(\mu_j,r_j)\) for some \(j\in[M]\), choose \(p_{\mathrm{floor}}\) and \(A\) as in Proposition~\ref{prop:local-inactive}, and let \(0<\eta\le \eta_{1,j}\). Then there exist constants \(\kappa_j>0\) and \(C_j>0\), independent of \(n\), such that for all sufficiently large \(n\),
\[
\EE\!\left[
\|x_T-\mu_j\|^2\,\mathbf 1_{\mathcal E_{n,T}^{\mathrm{stay},j}}
\,\middle|\,
\mathcal X,\mathcal A_n^{\mathrm{local},j}
\right]
\le
(1-\kappa_j\eta)^T\|x_0-\mu_j\|^2
+
C_j\left(
h^{2(\beta-1)}
+
\frac{\log n}{n h^{d+2}}
+
\frac{1}{m h^{d+2}}
+
\eta d\sigma^2
\right).
\]
\end{proposition}

\begin{proof}[Proof of Proposition~\ref{prop:local-truncated-holder}]
Condition on the observed sample \(\mathcal X\) and on the event
\(\mathcal A_n^{\mathrm{local},j}\) throughout. For \(t=0,\dots,T\), let us define
\[
E_t:=\mathcal E_{n,t}^{\mathrm{stay},j}
=
\{x_s\in \overline B(\mu_j,r_j)\ \text{for all }s=0,\dots,t\}.
\]
and the filtration
\[
\mathcal F_t
:=
\sigma(\mathcal B_0,Z_0,\dots,\mathcal B_{t-1},Z_{t-1},x_0),
\qquad t\ge 0.
\]
Then \(E_t\in \mathcal F_t\).

The proof tracks the squared error only up to the first exit from the local basin. The indicators \(\mathbf 1_{E_t}\) therefore localize the recursion to the regime in which the basinwise analytic controls from the preceding lemmas are valid. We write
\[
\delta_t:=x_t-\mu_j,
\qquad
\bar q_t:=\frac{1}{m}\sum_{i\in\mathcal B_t} q_i(x_t),
\qquad
r_t:=\hat s_{A,p_{\mathrm{floor}};\mathcal B_t}(x_t)-\bar q_t,
\]
so that
\[
x_{t+1}=x_t+\eta(\bar q_t+r_t+z_t),
\qquad
z_t\sim \mathcal N(0,\sigma^2 I_d).
\]
Since \(E_{t+1}\subseteq E_t\),
\[
\|\delta_{t+1}\|^2\mathbf 1_{E_{t+1}}
\le
\mathbf 1_{E_t}\,
\bigl\|
\delta_t+\eta(\bar q_t+r_t+z_t)
\bigr\|^2.
\]
Taking conditional expectation given \(\mathcal F_t\), and writing
\[
T_2
:=
\left\langle
\delta_t,
\EE[\bar q_t\mid \mathcal F_t,\mathcal X,\mathcal A_n^{\mathrm{local},j}]
\right\rangle,
\]
therefore yields
\begin{align*}
\EE\!\left[
\|\delta_{t+1}\|^2\mathbf 1_{E_{t+1}}
\,\middle|\,
\mathcal F_t,\mathcal X,\mathcal A_n^{\mathrm{local},j}
\right]
&\le
\mathbf 1_{E_t}
\Big(
\|\delta_t\|^2
+
2\eta T_2
+
2\eta\langle \delta_t,\EE[r_t\mid \mathcal F_t,\mathcal X,\mathcal A_n^{\mathrm{local},j}]\rangle\\
&\hspace{2.3cm}
+
\eta^2\EE[\|\bar q_t+r_t+z_t\|^2\mid \mathcal F_t,\mathcal X,\mathcal A_n^{\mathrm{local},j}]
\Big).
\end{align*}

We next separate the one-step recursion into its principal drift term, the stabilization remainder, and the quadratic second-moment term.

On \(E_t\), one has \(x_t\in \overline B(\mu_j,r_j)\), so the local lemmas apply. First,
\[
\EE[\bar q_t\mid \mathcal F_t,\mathcal X,\mathcal A_n^{\mathrm{local},j}]
=
\nabla\log\hat p(x_t)
=:k_n(x_t),
\]
and Lemma~\ref{lem:kn-bound-holder} gives
\begin{align*}
T_2
&=
\langle \delta_t,k_n(x_t)\rangle
\le
-\frac{3\alpha_j}{8}\|\delta_t\|^2
+
C_j\!\left(
h^{2(\beta-1)}
+
\frac{\log n}{n h^{d+2}}
\right)
+
C_j\!\left(
h^{\beta-2}
+
\sqrt{\frac{\log n}{n h^{d+4}}}
\right)\|\delta_t\|^2.
\end{align*}
Next,
\[
\EE[r_t\mid \mathcal F_t,\mathcal X,\mathcal A_n^{\mathrm{local},j}]
=
\EE\!\left[
\hat s_{A,p_{\mathrm{floor}};\mathcal B_t}(x_t)-\nabla\log\hat p(x_t)
\,\middle|\,
\mathcal F_t,\mathcal X,\mathcal A_n^{\mathrm{local},j}
\right],
\]
so Lemma~\ref{lem:minibatch-field-second} and Jensen's inequality yield
\[
\left\|
\EE[r_t\mid \mathcal F_t,\mathcal X,\mathcal A_n^{\mathrm{local},j}]
\right\|
\le 
\EE[\left\|r_t\right\|\mid \mathcal F_t,\mathcal X,\mathcal A_n^{\mathrm{local},j}]
\le 
\sqrt{\EE[\left\|r_t\right\|^2\mid \mathcal F_t,\mathcal X,\mathcal A_n^{\mathrm{local},j}]}
\le
C_j\sqrt{\frac{1}{m h^{d+2}}}.
\]
Hence, by Young's inequality,
\[
2\eta
\left\langle
\delta_t,
\EE[r_t\mid \mathcal F_t,\mathcal X,\mathcal A_n^{\mathrm{local},j}]
\right\rangle
\le
\frac{\alpha_j\eta}{8}\|\delta_t\|^2
+
C_j\eta\,\frac{1}{m h^{d+2}}.
\]

It remains to control the quadratic contribution. Here the average score term and the stabilization remainder are handled by the second-moment bounds from Lemmas~\ref{lem:T1-bound-holder} and~\ref{lem:T3-bound-holder}, while the Gaussian perturbation contributes the explicit \(d\sigma^2\) term.

For the quadratic term, by \((a+b+c)^2\le 3a^2+3b^2+3c^2\),
\begin{align*}
\EE[\|\bar q_t+r_t+z_t\|^2\mid \mathcal F_t,\mathcal X,\mathcal A_n^{\mathrm{local},j}]
&\le
3\,\EE[\|\bar q_t\|^2\mid \mathcal F_t,\mathcal X,\mathcal A_n^{\mathrm{local},j}]
+
3\,\EE[\|r_t\|^2\mid \mathcal F_t,\mathcal X,\mathcal A_n^{\mathrm{local},j}]
+
3d\sigma^2.
\end{align*}
By Lemmas~\ref{lem:T1-bound-holder} and~\ref{lem:T3-bound-holder}, there exist constants \(C_{3,j},C_{4,j}>0\) such that, on \(E_t\),
\begin{align*}
&~\EE[\|\bar q_t\|^2\mid \mathcal F_t,\mathcal X,\mathcal A_n^{\mathrm{local},j}]
+
\EE[\|r_t\|^2\mid \mathcal F_t,\mathcal X,\mathcal A_n^{\mathrm{local},j}]\\
\le&~
\ C_{3,j}\|\delta_t\|^2
+
C_{4,j}\left(
h^{2(\beta-1)}
+
\frac{\log n}{n h^{d+2}}
+
\frac{1}{m h^{d+2}}
\right)
\end{align*}
for all sufficiently large \(n\).

Combining the previous displays and using Assumption~\ref{assump:bandwidth}, which gives
\[
h^{\beta-2}+\sqrt{\frac{\log n}{n h^{d+4}}}=o(1),
\]
together with \(0<\eta\le \eta_{1,j}\), we obtain constants \(\kappa_j>0\) and \(C_j'>0\) such that
\begin{align*}
\EE\!\left[
\|\delta_{t+1}\|^2\mathbf 1_{E_{t+1}}
\,\middle|\,
\mathcal F_t,\mathcal X,\mathcal A_n^{\mathrm{local},j}
\right]
\le&
\ (1-\kappa_j\eta)\,\|\delta_t\|^2\mathbf 1_{E_t}\\
&+
C_j'\eta\left(
h^{2(\beta-1)}
+
\frac{\log n}{n h^{d+2}}
+
\frac{1}{m h^{d+2}}
\right)
+
C_j'\eta^2 d\sigma^2
\end{align*}
for all sufficiently large \(n\). Let us now define
\[
\Delta_t^{\mathrm{tr}}
:=
\EE\!\left[
\|\delta_t\|^2\mathbf 1_{E_t}
\,\middle|\,
\mathcal X,\mathcal A_n^{\mathrm{local},j}
\right].
\]
Taking expectations in the last display gives
\[
\Delta_{t+1}^{\mathrm{tr}}
\le
(1-\kappa_j\eta)\Delta_t^{\mathrm{tr}}
+
C_j'\eta\left(
h^{2(\beta-1)}
+
\frac{\log n}{n h^{d+2}}
+
\frac{1}{m h^{d+2}}
\right)
+
C_j'\eta^2 d\sigma^2.
\]
Unrolling the recursion yields
\[
\Delta_T^{\mathrm{tr}}
\le
(1-\kappa_j\eta)^T\|x_0-\mu_j\|^2
+
C_j''\left(
h^{2(\beta-1)}
+
\frac{\log n}{n h^{d+2}}
+
\frac{1}{m h^{d+2}}
+
\eta d\sigma^2
\right),
\]
where \(C_j''>0\) is independent of \(n,m,h,T\). Since \(\Delta_T^{\mathrm{tr}}
=
\EE\!\left[
\|x_T-\mu_j\|^2 \mathbf 1_{\mathcal E_{n,T}^{\mathrm{stay},j}}
\,\middle|\,
\mathcal X,\mathcal A_n^{\mathrm{local},j}
\right]\), this proves the proposition.
\end{proof}

The next two results control the perturbation term uniformly up to exit and then show that, under the stated tuning conditions, the trajectory remains in \(\overline B(\mu_j,r_j)\) with high probability.

\begin{lemma}[Uniform perturbation control up to exit]\label{lem:local-perturb-holder}
Assume \(\beta>2\). Suppose Assumptions~\ref{assump:kernel},
\ref{assump:model-holder}, \ref{assump:bandwidth}, \ref{assump:hessian-holder},
and \ref{assump:radius-holder} hold, and assume
\[
\frac{m h^{d+2}}{\log(eTn)}\to\infty.
\]
Fix \(j\in[M]\), and let \((x_t)_{t=0}^{T-1}\) be any process adapted to the algorithmic filtration such that
\(
x_0\in \overline B(\mu_j,r_j)
\).
For each \(t=0,\dots,T-1\), let us define
\[
E_t^{(j)}
:=
\{x_s\in \overline B(\mu_j,r_j)\ \text{for all }s=0,\dots,t\},
\qquad
\widetilde x_t^{(j)}
:=
x_t\,\mathbf 1_{E_t^{(j)}}
+
x_0\,\mathbf 1_{(E_t^{(j)})^c}.
\]
Also recall
\(
C_*:=(2G_K/p_{\min})\,h^{-(d+1)}
\)
from Lemma~\ref{lem:local-qi-bound}. Define
\[
\Xi_{n,T,m,h}^{(j)}
:=
C_{1,j}\!\left(
h^{\beta-1}
+
\sqrt{\frac{\log n}{n h^{d+2}}}
+
\sqrt{\frac{\log(eTn)}{m h^{d+2}}}
\right)
+
C_{2,j}\,
\frac{C_*\sqrt{T d\,\mathrm{polylog}(T,n,\delta)}}{n\varepsilon_{\mathrm{modes}}},
\]
where \(C_{1,j},C_{2,j}>0\) are deterministic constants independent of \(n,m,h,T\). Then, on \(\mathcal A_n^{\mathrm{local},j}\),
\[
\Pr\!\left(
\max_{0\le t\le T-1}
\left\|
\hat s_{A,p_{\mathrm{floor}};\mathcal B_t}(\widetilde x_t^{(j)})
-\nabla\log p(\widetilde x_t^{(j)})
+z_t
\right\|
\le
\Xi_{n,T,m,h}^{(j)}
\,\middle|\,
\mathcal X,\mathcal A_n^{\mathrm{local},j}
\right)
\ge
1-2T(eTn)^{-6}.
\]
\end{lemma}

\medskip

\begin{proposition}[High-probability basin retention]\label{prop:local-stay-holder}
Assume \(\beta>2\). Suppose Assumptions~\ref{assump:kernel},
\ref{assump:model-holder}, \ref{assump:bandwidth}, \ref{assump:hessian-holder},
and \ref{assump:radius-holder} hold, and assume
\(
\frac{m h^{d+2}}{\log(eTn)}\to\infty.
\)
Let
\(x_0\in\overline B(\mu_j,r_j)\) for some \(j\in[M]\), and choose
\(p_{\mathrm{floor}}\) and \(A\) as in Proposition~\ref{prop:local-inactive}. Define
\[
\mathcal E_{n,T}^{\mathrm{stay},j}
:=
\{x_t\in \overline B(\mu_j,r_j)\ \text{for all }t=0,\dots,T\}.
\]
With \(B_j\) and \(\eta_{0,j}\) as defined in the local proof notation above, if
\(0<\eta\le \eta_{0,j}\) and
\(
\Xi_{n,T,m,h}^{(j)}
\le
\min\!\left\{
B_j,\frac{\alpha_j r_j}{4}
\right\},
\)
then on \(\mathcal A_n^{\mathrm{local},j}\), for all sufficiently large \(n\)
\[
\Pr\!\left(
\mathcal E_{n,T}^{\mathrm{stay},j}
\,\middle|\,
\mathcal X,\mathcal A_n^{\mathrm{local},j}
\right)\ge 1-2T(eTn)^{-6}.
\]
\end{proposition}

\medskip

\begin{proof}[Proof of Proposition~\ref{prop:local-stay-holder}]
We condition on \(\mathcal X\) and on \(\mathcal A_n^{\mathrm{local},j}\) throughout. For \(t=0,\dots,T\), define
\[
E_t^{(j)}
:=
\{x_s\in \overline B(\mu_j,r_j)\ \text{for all }s=0,\dots,t\}.
\]
For \(t=0,\dots,T-1\), define the stopped surrogate
\[
\widetilde x_t^{(j)}
:=
x_t\,\mathbf 1_{E_t^{(j)}}
+
x_0\,\mathbf 1_{(E_t^{(j)})^c}.
\]
Let
\[
F_{n,T}^{(j)}
:=
\left\{
\max_{0\le t\le T-1}
\left\|
\hat s_{A,p_{\mathrm{floor}};\mathcal B_t}(\widetilde x_t^{(j)})
-\nabla\log p(\widetilde x_t^{(j)})
+z_t
\right\|
\le
\Xi_{n,T,m,h}^{(j)}
\right\}.
\]
By Lemma~\ref{lem:local-perturb-holder},
\[
\Pr\!\left(
F_{n,T}^{(j)}
\,\middle|\,
\mathcal X,\mathcal A_n^{\mathrm{local},j}
\right)
\ge
1-2T(eTn)^{-6}.
\]

We show that on \(F_{n,T}^{(j)}\), the actual trajectory never leaves \(\overline B(\mu_j,r_j)\). We proceed by induction on \(t\). The base case \(E_0^{(j)}\) holds because \(x_0\in \overline B(\mu_j,r_j)\) by assumption. Now suppose \(E_t^{(j)}\) holds for some \(t\in\{0,\dots,T-1\}\). Then
\[
\widetilde x_t^{(j)}=x_t.
\]
On \(F_{n,T}^{(j)}\), there therefore exists a vector \(u_t\) with
\(
\|u_t\|\le \Xi_{n,T,m,h}^{(j)}
\)
such that
\[
\hat s_{A,p_{\mathrm{floor}};\mathcal B_t}(x_t)+z_t
=
\nabla\log p(x_t)+u_t.
\]
By the algorithm update,
\[
x_{t+1}
=
x_t+\eta\Big(
\hat s_{A,p_{\mathrm{floor}};\mathcal B_t}(x_t)+z_t
\Big)
=
x_t+\eta\bigl(\nabla\log p(x_t)+u_t\bigr).
\]
We write
\[
\delta_t:=x_t-\mu_j.
\]
Then
\begin{align*}
\|x_{t+1}-\mu_j\|^2
&=
\|\delta_t\|^2
+
2\eta\langle \delta_t,\nabla\log p(x_t)\rangle
+
2\eta\langle \delta_t,u_t\rangle
+
\eta^2\|\nabla\log p(x_t)+u_t\|^2.
\end{align*}
By Lemma~\ref{lem:derived-inward-drift},
\[
\langle \delta_t,\nabla\log p(x_t)\rangle
\le
-\frac{\alpha_j}{2}\|\delta_t\|^2.
\]
Also,
\[
\langle \delta_t,u_t\rangle
\le
\|\delta_t\|\,\|u_t\|
\le
r_j\,\Xi_{n,T,m,h}^{(j)},
\]
and, since \(\|\nabla\log p(x_t)\|\le L_j<B_j\) and
\(\Xi_{n,T,m,h}^{(j)}\le B_j\),
\[
\|\nabla\log p(x_t)+u_t\|^2
\le
(\|\nabla\log p(x_t)\|+\|u_t\|)^2
\le
4B_j^2.
\]
Therefore,
\[
\|x_{t+1}-\mu_j\|^2
\le
(1-\alpha_j\eta)\|\delta_t\|^2
+
2\eta r_j\,\Xi_{n,T,m,h}^{(j)}
+
4\eta^2 B_j^2.
\]
Because \(E_t^{(j)}\) holds, \(\|\delta_t\|\le r_j\). Using
\[
\Xi_{n,T,m,h}^{(j)}\le \frac{\alpha_j r_j}{4}
\qquad\text{and}\qquad
\eta\le \eta_{0,j}=\frac{\alpha_j r_j^2}{8B_j^2},
\]
we obtain
\[
2\eta r_j\,\Xi_{n,T,m,h}^{(j)}
\le
\frac{\alpha_j\eta r_j^2}{2},
\qquad
4\eta^2 B_j^2
\le
\frac{\alpha_j\eta r_j^2}{2}.
\]
Hence
\[
\|x_{t+1}-\mu_j\|^2
\le
(1-\alpha_j\eta)r_j^2
+
\frac{\alpha_j\eta r_j^2}{2}
+
\frac{\alpha_j\eta r_j^2}{2}
=
r_j^2.
\]
Thus \(x_{t+1}\in \overline B(\mu_j,r_j)\), so \(E_{t+1}^{(j)}\) holds.

By induction, \(E_t^{(j)}\) holds for all \(t=0,\dots,T\) on \(F_{n,T}^{(j)}\). Therefore
\(
F_{n,T}^{(j)}\subseteq \mathcal E_{n,T}^{\mathrm{stay},j},
\)
and so
\[
\Pr\!\left(
\mathcal E_{n,T}^{\mathrm{stay},j}
\,\middle|\,
\mathcal X,\mathcal A_n^{\mathrm{local},j}
\right)
\ge
1-2T(eTn)^{-6}.
\]
This proves the proposition.
\end{proof}

\medskip

We now combine the stopped local recursion with basin retention to derive the stated rate and identify the optimizing bandwidth.
\begin{proof}[Proof of Theorem~\ref{th:local-conv-holder}]
Recall from Definition~\ref{def:good-event-local-holder} that
\(
\mathcal G_{n,T}^{\mathrm{local},j}
=
\mathcal A_n^{\mathrm{local},j}\cap \mathcal E_{n,T}^{\mathrm{stay},j}.
\)
By the notation used in the proof of the local convergence, \(\bar\eta_j=\min\{\eta_{0,j},\eta_{1,j}\}\). Hence \(0<\eta\le\bar\eta_j\) implies \(0<\eta\le\eta_{0,j}\) and \(0<\eta\le\eta_{1,j}\).
By Lemma~\ref{lem:local-good-prob},
\[
\Pr\!\big(\mathcal A_n^{\mathrm{local},j}\big)\ge 1-4n^{-4}
\]
for all sufficiently large \(n\).

We first derive the local error bound for a generic bandwidth \(h\), and only then optimize in \(h\). Conditional on \((\mathcal X,\mathcal A_n^{\mathrm{local},j})\), Proposition~\ref{prop:local-inactive} shows that the stabilized score agrees with the ordinary log-KDE score throughout \(\overline B(\mu_j,r_j)\) for all sufficiently large \(n\). On the same event, Lemma~\ref{lem:local-qi-bound} gives the uniform samplewise bound
\[
\sup_{x\in\overline B(\mu_j,r_j)}\max_{1\le i\le n}\|q_i(x)\|\le C_*,
\qquad
C_*=\frac{2G_K}{p_{\min}}\,h^{-(d+1)}.
\]
This is the local quantity that will control the privacy contribution.

By Proposition~\ref{prop:local-truncated-holder}, there exists a constant \(C_j>0\), independent of \(n\), such that on \(\mathcal A_n^{\mathrm{local},j}\),
\[
\EE\!\left[
\|x_T-\mu_j\|^2\,\mathbf 1_{\mathcal E_{n,T}^{\mathrm{stay},j}}
\,\middle|\,
\mathcal X,\mathcal A_n^{\mathrm{local},j}
\right]
\le
(1-\kappa_j\eta)^T\|x_0-\mu_j\|^2
+
C_j\left(
h^{2(\beta-1)}
+
\frac{\log n}{n h^{d+2}}
+
\frac{1}{m h^{d+2}}
+
\eta d\sigma^2
\right).
\]

We next bound the privacy term. Since \(p_{\mathrm{floor}}\) and \(A\) are fixed as in Proposition~\ref{prop:local-inactive}, the definition \eqref{eq:def-Delta-corr} gives
\[
\Delta_{h,\mathrm{corr}}(A,p_{\mathrm{floor}})
=
2\sqrt{2}\left(
\frac{I_1^{1/2}}{p_{\mathrm{floor}}}\,h^{-(d+1)}
+
\frac{A I_0^{1/2}}{p_{\mathrm{floor}}^2}\,h^{-d}
\right).
\]
Because \(h\to 0\) under Assumption~\ref{assump:bandwidth}, one has \(h\le 1\) for all sufficiently large \(n\), hence \(h^{-d}\le h^{-(d+1)}\). Therefore
\[
\Delta_{h,\mathrm{corr}}(A,p_{\mathrm{floor}})
\lesssim h^{-(d+1)}
\asymp C_*.
\]
Using \eqref{eq:sigma-corr}, the definition of \(\varepsilon_{\mathrm{iter}}\), and the fact that \(\eta\) is fixed independently of \(n\), we obtain
\[
\eta d\sigma^2
\lesssim
\frac{T d\,\mathrm{polylog}(n,\delta)}{n^2\varepsilon_{\mathrm{modes}}^2}\,h^{-2(d+1)}.
\]
Also, since \(m\asymp n/\log n\),
\[
\frac{1}{m h^{d+2}}
\lesssim
\frac{\log n}{n h^{d+2}}.
\]
Thus, after absorbing constants,
\begin{equation}\label{eq:local-generic-h-bound}
\EE\!\left[
\|x_T-\mu_j\|^2\,\mathbf 1_{\mathcal E_{n,T}^{\mathrm{stay},j}}
\,\middle|\,
\mathcal X,\mathcal A_n^{\mathrm{local},j}
\right]
\le
(1-\kappa_j\eta)^T\|x_0-\mu_j\|^2
+
C_j R_n(h),
\end{equation}
where
\[
R_n(h)
:=
h^{2(\beta-1)}
+
\frac{\log n}{n h^{d+2}}
+
A_n h^{-2(d+1)},
\qquad
A_n:=\frac{T d\,\mathrm{polylog}(n,\delta)}{n^2\varepsilon_{\mathrm{modes}}^2}.
\]

We now optimize \(R_n(h)\). We first balance the approximation term and the nonprivate stochastic term:
\[
h^{2(\beta-1)}
\asymp
\frac{\log n}{n h^{d+2}}.
\]
This gives
\[
h_{\mathrm{np}}
\asymp
\Big(\frac{\log n}{n}\Big)^{\frac{1}{d+2\beta}},
\quad 
\text{and}
\quad
h_{\mathrm{np}}^{2(\beta-1)}
\asymp
\frac{\log n}{n h_{\mathrm{np}}^{d+2}}
\asymp
\Big(\frac{\log n}{n}\Big)^{\frac{2(\beta-1)}{d+2\beta}}.
\]
Thus
\[
R_n(h_{\mathrm{np}})
\lesssim
\Big(\frac{\log n}{n}\Big)^{\frac{2(\beta-1)}{d+2\beta}}
+
A_n h_{\mathrm{np}}^{-2(d+1)}.
\]
The nonprivate choice remains optimal whenever
\[
A_n h_{\mathrm{np}}^{-2(d+1)}
\lesssim
\Big(\frac{\log n}{n}\Big)^{\frac{2(\beta-1)}{d+2\beta}},
\]
which is equivalent, up to logarithmic factors and using \(T=C_T\log n\), to
\(
\varepsilon_{\mathrm{modes}}\gtrsim \varepsilon_{\rm thr}.
\)
In that regime,
\[
R_n(h_{\mathrm{np}})
\lesssim
\Big(\frac{\log n}{n}\Big)^{\frac{2(\beta-1)}{d+2\beta}}.
\]
Let us next balance the approximation term and the privacy term:
\(
h^{2(\beta-1)}
\asymp
A_n h^{-2(d+1)}.
\)
This gives
\[
h_{\mathrm{dp}}
\asymp
A_n^{\frac{1}{2d+2\beta}}
=
\Big(\frac{T d\,\mathrm{polylog}(n,\delta)}{n^2\varepsilon_{\mathrm{modes}}^2}\Big)^{\frac{1}{2d+2\beta}},
\]
and hence
\[
h_{\mathrm{dp}}^{2(\beta-1)}
\asymp
A_n h_{\mathrm{dp}}^{-2(d+1)}
\asymp
A_n^{\frac{\beta-1}{d+\beta}}
=
\Big(\frac{T d\,\mathrm{polylog}(n,\delta)}{n^2\varepsilon_{\mathrm{modes}}^2}\Big)^{\frac{\beta-1}{d+\beta}}.
\]
If
\(
\varepsilon_{\mathrm{modes}}\lesssim \varepsilon_{\rm thr},
\)
then the remaining middle term is dominated by the common order of the first and third terms, so
\[
R_n(h_{\mathrm{dp}})
\lesssim
\Big(\frac{T d\,\mathrm{polylog}(n,\delta)}{n^2\varepsilon_{\mathrm{modes}}^2}\Big)^{\frac{\beta-1}{d+\beta}}.
\]
Combining the two regimes, for the bandwidth choice \(h=h_{\rm opt}\) in \eqref{h-opt},
\begin{equation}\label{eq:local-optimized-Rn}
R_n(h_{\rm opt})
\lesssim
\Big(\frac{\log n}{n}\Big)^{\frac{2(\beta-1)}{d+2\beta}}
+
\Big(\frac{T d\,\mathrm{polylog}(n,\delta)}{n^2\varepsilon_{\mathrm{modes}}^2}\Big)^{\frac{\beta-1}{d+\beta}}.
\end{equation}

We next verify the stay-in-basin hypotheses of Proposition~\ref{prop:local-stay-holder}. Since \(m\asymp n/\log n\) and \(T=C_T\log n\), the first requirement
\[
\frac{m h^{d+2}}{\log(eTn)}\to\infty
\quad 
\text{is equivalent to}
\quad 
\frac{n h^{d+2}}{(\log n)^2}\to\infty.
\]
For
\(
h=h_{\mathrm{np}}
\asymp
\Big(\log n/n\Big)^{\frac{1}{d+2\beta}},
\)
this gives
\[
\frac{n h_{\mathrm{np}}^{d+2}}{(\log n)^2}
\asymp
n^{\frac{2(\beta-1)}{d+2\beta}}
(\log n)^{-2+\frac{d+2}{d+2\beta}}
\to\infty.
\]
For
\[
h=h_{\mathrm{dp}}
\asymp
A_n^{\frac{1}{2d+2\beta}},
\qquad
A_n:=\frac{T d\,\mathrm{polylog}(n,\delta)}{n^2\varepsilon_{\mathrm{modes}}^2},
\]
the privacy-dominated condition
\(
\varepsilon_{\mathrm{modes}}\lesssim \varepsilon_{\rm thr}
\)
is equivalent, up to logarithmic factors, to
\[
A_n
\gtrsim
\Big(\frac{\log n}{n}\Big)^{\frac{2(d+\beta)}{d+2\beta}}.
\]
Therefore
\begin{align*}
\frac{n h_{\mathrm{dp}}^{d+2}}{(\log n)^2}
=
\frac{n A_n^{\frac{d+2}{2d+2\beta}}}{(\log n)^2}
&\gtrsim
\frac{
n\Big(\frac{\log n}{n}\Big)^{\frac{d+2}{d+2\beta}}
}{(\log n)^2}
=
n^{\frac{2(\beta-1)}{d+2\beta}}
(\log n)^{-2+\frac{d+2}{d+2\beta}}
\to\infty.
\end{align*}
Thus the minibatch-growth condition from Proposition~\ref{prop:local-stay-holder} holds in both regimes. It remains to verify that
\[
\Xi_{n,T,m,h}^{(j)}
\le
\min\!\left\{B_j,\frac{\alpha_j r_j}{4}\right\}
\]
for all sufficiently large \(n\). Since
\[
\Xi_{n,T,m,h}^{(j)}
=
C_{1,j}\!\left(
h^{\beta-1}
+
\sqrt{\frac{\log n}{n h^{d+2}}}
+
\sqrt{\frac{\log(eTn)}{m h^{d+2}}}
\right)
+
C_{2,j}\,
\frac{C_*\sqrt{T d\,\mathrm{polylog}(T,n,\delta)}}{n\varepsilon_{\mathrm{modes}}},
\]
it is enough to show that each term vanishes in both regimes.

For \(h=h_{\mathrm{np}}\),
\[
h_{\mathrm{np}}^{\beta-1}
=
\Big(\frac{\log n}{n}\Big)^{\frac{\beta-1}{d+2\beta}}
\to 0,
\quad
\sqrt{\frac{\log n}{n h_{\mathrm{np}}^{d+2}}}
=
\Big(\frac{\log n}{n}\Big)^{\frac{\beta-1}{d+2\beta}}
\to 0,
\]
and, since \(m\asymp n/\log n\) and \(\log(eTn)\asymp \log n\),
\[
\sqrt{\frac{\log(eTn)}{m h_{\mathrm{np}}^{d+2}}}
\lesssim
\sqrt{\frac{\log^2 n}{n h_{\mathrm{np}}^{d+2}}}
=
n^{-\frac{\beta-1}{d+2\beta}}
(\log n)^{1-\frac{d+2}{2(d+2\beta)}}
\to 0.
\]
Also, \(C_*\asymp h_{\mathrm{np}}^{-(d+1)}\), so
\[
\frac{C_*\sqrt{T d\,\mathrm{polylog}(T,n,\delta)}}{n\varepsilon_{\mathrm{modes}}}
\asymp
\sqrt{A_n}\,h_{\mathrm{np}}^{-(d+1)}
=
\bigl(A_n h_{\mathrm{np}}^{-2(d+1)}\bigr)^{1/2}.
\]
In the nonprivate regime,
\[
A_n h_{\mathrm{np}}^{-2(d+1)}
\lesssim
\Big(\frac{\log n}{n}\Big)^{\frac{2(\beta-1)}{d+2\beta}},
\quad
\text{hence}
\quad
\frac{C_*\sqrt{T d\,\mathrm{polylog}(T,n,\delta)}}{n\varepsilon_{\mathrm{modes}}}
\lesssim
\Big(\frac{\log n}{n}\Big)^{\frac{\beta-1}{d+2\beta}}
\to 0.
\]
For \(h=h_{\mathrm{dp}}\),
\(
h_{\mathrm{dp}}^{\beta-1}
=
A_n^{\frac{\beta-1}{2d+2\beta}}
\)
and, since \(h_{\mathrm{dp}}\) is defined by balancing approximation and privacy,
\[
\sqrt{A_n}\,h_{\mathrm{dp}}^{-(d+1)}
=
\bigl(A_n h_{\mathrm{dp}}^{-2(d+1)}\bigr)^{1/2}
\asymp
h_{\mathrm{dp}}^{\beta-1}.
\]
Moreover, the minibatch-growth condition already proved implies
\[
\sqrt{\frac{\log(eTn)}{m h_{\mathrm{dp}}^{d+2}}}\to 0,
\quad 
\text{and therefore also}
\quad 
\sqrt{\frac{\log n}{n h_{\mathrm{dp}}^{d+2}}}\to 0.
\]
Thus every term in \(\Xi_{n,T,m,h}^{(j)}\) vanishes in both regimes. Therefore \(\Xi_{n,T,m,h}^{(j)}\to 0\), and so for all sufficiently large \(n\),
\[
\Xi_{n,T,m,h}^{(j)}
\le
\min\!\left\{B_j,\frac{\alpha_j r_j}{4}\right\}.
\]
Applying Proposition~\ref{prop:local-stay-holder} with the corresponding bandwidth choice now yields
\[
\Pr\!\left(
\mathcal E_{n,T}^{\mathrm{stay},j}
\,\middle|\,
\mathcal X,\mathcal A_n^{\mathrm{local},j}
\right)\ge 1-n^{-5}
\]
for all sufficiently large \(n\). Therefore
\[
\Pr\!\big(\mathcal G_{n,T}^{\mathrm{local},j}\big)
=
\Pr\!\big(\mathcal A_n^{\mathrm{local},j}\cap \mathcal E_{n,T}^{\mathrm{stay},j}\big)
\ge
1-4n^{-4}-n^{-5}
\ge
1-5n^{-4}.
\]
Since
\(
\mathcal G_{n,T}^{\mathrm{local},j}
=
\mathcal A_n^{\mathrm{local},j}\cap \mathcal E_{n,T}^{\mathrm{stay},j},
\)
we have
\begin{align*}
\EE\!\left[
\|x_T-\mu_j\|^2
\,\middle|\,
\mathcal X,\mathcal G_{n,T}^{\mathrm{local},j}
\right]
&\le
\frac{
\EE\!\left[
\|x_T-\mu_j\|^2\,\mathbf 1_{\mathcal E_{n,T}^{\mathrm{stay},j}}
\,\middle|\,
\mathcal X,\mathcal A_n^{\mathrm{local},j}
\right]
}{
\Pr\!\left(
\mathcal E_{n,T}^{\mathrm{stay},j}
\,\middle|\,
\mathcal X,\mathcal A_n^{\mathrm{local},j}
\right)
}.
\end{align*}
Using the lower bound \(1-n^{-5}\) in the denominator and the bound \eqref{eq:local-generic-h-bound} together with \eqref{eq:local-optimized-Rn}, we obtain
\[
\EE\!\left[
\|x_T-\mu_j\|^2
\,\middle|\,
\mathcal X,\mathcal G_{n,T}^{\mathrm{local},j}
\right]
\le
\frac{
(1-\kappa_j\eta)^T\|x_0-\mu_j\|^2
+
C_j R_n(h_{\rm opt})
}{
1-n^{-5}
}.
\]
Finally, since \(T=C_T\log n\),
\[
(1-\kappa_j\eta)^T\le e^{-\kappa_j\eta T}=n^{-\kappa_j\eta C_T}.
\]
We choose \(C_{T,j}^\star>0\) so that
\[
\kappa_j\eta C_{T,j}^\star>\frac{2(\beta-1)}{d+2\beta}
\quad 
\text{and hence}
\quad 
(1-\kappa_j\eta)^T
=
o\!\left(
\Big(\frac{\log n}{n}\Big)^{\frac{2(\beta-1)}{d+2\beta}}
\right)
\]
for every \(C_T\ge C_{T,j}^\star\). Absorbing this transient term and the factor \((1-n^{-5})^{-1}\) into the constants completes the proof:
\[
\EE\!\left[
\|x_T-\mu_j\|^2
\,\middle|\,
\mathcal X,\mathcal G_{n,T}^{\mathrm{local},j}
\right]
\le
C_{\mathrm{nonDP},j}
\Big(\frac{\log n}{n}\Big)^{\frac{2(\beta-1)}{d+2\beta}}
+
C_{\mathrm{DP},j}
\Big(\frac{T d\,\mathrm{polylog}(n,\delta)}{n^2\varepsilon_{\mathrm{modes}}^2}\Big)^{\frac{\beta-1}{d+\beta}}.
\]
\end{proof}

\subsection{Global convergence: DAP coverage, basinwise trajectories, and post-processing}
We next turn from basinwise control to global recovery. The first group of lemmas analyzes the DAP design and shows that the initialization pool covers all modal basins with high probability. The final step combines this coverage event with the local theorem and the deterministic post-processing conditions used by the merge rule to prove Theorem~\ref{thm:global-holder}.

Throughout the DAP coverage argument, for \(j\in[M]\) and \(z_r\in\mathcal Z_n\), write
\[
\mathcal G_{j,n}:=\mathcal Z_n\cap \overline B(\mu_j,\rho_{\mathrm{init}}/4),
\qquad
u_r:=\frac1n\sum_{i=1}^n\mathbf 1\{\|X_i-z_r\|\le h_{\mathrm{DAP}}\},
\qquad
m_r:=\Pr(\|X-z_r\|\le h_{\mathrm{DAP}}).
\]

\begin{lemma}[Geometric consequences of the DAP design]
\label{lem:dap-geometry}
Assume Assumption~\ref{assump:model-holder}, and work under the DAP design specified above. Then, for all sufficiently large \(n\), the following hold:
\begin{enumerate}[label={(\roman*)}]
\item \(\mathcal G_{j,n}\neq\varnothing\) for every \(j\in[M]\).
\item If \(a\in \overline B(\mu_j,\rho_{\mathrm{init}}/4)\), then
\[
\overline B(\mu_j,\rho_{\mathrm{init}}/4)\subseteq \overline B(a,\rho_{\mathrm{init}}).
\]
\item If \(a\in \overline B(\mu_j,\rho_{\mathrm{init}}/4)\), then
\[
\overline B(a,\rho_{\mathrm{init}})
\cap
\overline B(\mu_i,\rho_{\mathrm{init}}/4)
=\varnothing
\quad 
\text{for every }
i\neq j.
\]
\end{enumerate}
\end{lemma}

\begin{lemma}[Within-basin local-mass ordering]
\label{lem:dap-localmass}
Suppose Assumptions~\ref{assump:model-holder},
\ref{assump:hessian-holder} hold. Work under the DAP design specified above. Then, for each \(j\in[M]\), there exist constants \(\Delta_j>0\) and \(n_j\in\mathbb N\) such that, for all \(n\ge n_j\),
\[
\inf_{r:\,z_r\in \mathcal Z_n\cap \overline B(\mu_j,\rho_{\mathrm{init}}/4)} m_r
\ge
\sup_{s:\,z_s\in \mathcal Z_n\cap(\overline B(\mu_j,r_j)\setminus B(\mu_j,\rho_{\mathrm{init}}/4))} m_s
+
\Delta_j h_{\mathrm{DAP}}^d\rho_{\mathrm{init}}^2.
\]
\end{lemma}

\begin{lemma}[Uniform concentration of the DAP utilities]
\label{lem:dap-concentration}
Suppose Assumption~\ref{assump:model-holder} holds, and work under the DAP design specified above. Fix \(c_{\mathrm{conc}}>0\). Then there exists \(C_{\mathrm{DAP,tail}}>0\), independent of \(\varepsilon_{\mathrm{init}}\), such that, for all sufficiently large \(n\),
\[
\Pr\!\left(
\max_{1\le r\le N_{\mathrm{cand}}}|u_r-m_r|
\le c_{\mathrm{conc}}h_{\mathrm{DAP}}^d\rho_{\mathrm{init}}^2
\right)
\ge 1-C_{\mathrm{DAP,tail}}n^{-4}.
\]
\end{lemma}

\begin{lemma}[DAP competitive-region complexity]
\label{lem:dap-cap}
Suppose Assumptions~\ref{assump:model-holder},
\ref{assump:hessian-holder} hold. Work under the DAP design specified above. Fix \(m\in[M]\). For a sufficiently small fixed \(\gamma_0>0\), define
\[
\mathcal C_{m,n}:=
\left\{z_r\in\mathcal Z_n\setminus \overline B(\mu_m,r_m):
m_r\ge \inf_{s:\,z_s\in\mathcal G_{m,n}}m_s-
\gamma_0h_{\mathrm{DAP}}^d\rho_{\mathrm{init}}^2\right\}.
\]
Then, for all sufficiently large \(n\),
\[
\mathcal C_{m,n}
\subseteq
\bigcup_{\substack{j\in[M]\setminus\{m\}:\\ p(\mu_j)\ge p(\mu_m)}}
\overline B(\mu_j,r_j),
\]
and \(\mathcal C_{m,n}\) can be covered by at most \(L_{\mathrm{cap}}M\log n\) Euclidean balls of radius \(\rho_{\mathrm{init}}/2\), where \(L_{\mathrm{cap}}>0\) is uniform over \(m\in[M]\) and independent of \(n\) and \(\varepsilon_{\mathrm{init}}\).
\end{lemma}

\medskip

\begin{proof}[Proof of Proposition~\ref{prop:init-dpball}]
Let \(c_{\mathrm{conc}}\le \gamma_0/8\), where \(\gamma_0\) is the constant in Lemma~\ref{lem:dap-cap}. Define
\[
\mathcal E_{\mathrm{conc}}
:=
\left\{
\max_{1\le r\le N_{\mathrm{cand}}}|u_r-m_r|
\le c_{\mathrm{conc}}h_{\mathrm{DAP}}^d\rho_{\mathrm{init}}^2
\right\}.
\]
Lemma~\ref{lem:dap-concentration} gives
\[
\Pr(\mathcal E_{\mathrm{conc}}^c)\le C_{\mathrm{DAP,tail}}n^{-4}.
\]
Work on \(\mathcal E_{\mathrm{conc}}\), and put
\[
\lambda:=\frac{n\varepsilon_{\mathrm{init}}}{2k}.
\]
For any active set \(A_\ell\), any nonempty  \(T\subseteq A_\ell\), and any \(S\subseteq A_\ell\), the exponential mechanism gives
\[
\begin{aligned}
\Pr(J_\ell\in S\mid A_\ell)
&=
\frac{\sum_{r:\,z_r\in S}\exp(\lambda u_r)}
{\sum_{q:\,z_q\in A_\ell}\exp(\lambda u_q)}                                                        
\le
\frac{|S|\exp\{\lambda\sup_{r:\,z_r\in S}u_r\}}
{\sum_{q:\,z_q\in T}\exp(\lambda u_q)}                                                                 \\
&\le
|S|\exp\left[-\lambda\left\{\inf_{q:\,z_q\in T}u_q-
\sup_{r:\,z_r\in S}u_r\right\}\right].
\end{aligned}
\]
For each \(j\in[M]\), the separation argument in Lemma~\ref{lem:dap-cap} gives a constant \(\eta_j>0\) such that
\[
\sup_{x\in\mathcal Q\setminus\cup_{i=1}^M\overline B(\mu_i,r_i)}
p(x)
\le
p(\mu_j)-\eta_j.
\]
Choose \(s_j\in(0,r_j/3)\) so that
\[
\inf_{\|x-\mu_j\|\le s_j}p(x)\ge p(\mu_j)-\eta_j/4.
\]
The local-mass expansion gives a constant \(C_{\mathrm{lm}}>0\) such that, for all \(z_r\in\mathcal Z_n\cap\mathcal Q\),
\[
\left|
m_r-h_{\mathrm{DAP}}^d\operatorname{Vol}(B(0,1))p(z_r)
\right|
\le C_{\mathrm{lm}}h_{\mathrm{DAP}}^{d+1}.
\]
By the uniform DAP separation constants, \(\eta_\star:=\inf_{1\le j\le M}\eta_j>0\). Set
\[
c_{\mathrm{out}}:=\frac12\operatorname{Vol}(B(0,1))\eta_\star .
\]
Whenever
\[
2C_{\mathrm{lm}}h_{\mathrm{DAP}}
\le
\frac14\operatorname{Vol}(B(0,1))\eta_\star,
\]
we have
\[
\inf_{r:\,z_r\in\mathcal Z_n\cap\cup_{j=1}^M\overline B(\mu_j,s_j)}m_r
\ge
\sup_{s:\,z_s\in\mathcal Z_n\setminus\cup_{j=1}^M\overline B(\mu_j,r_j)}m_s
+
c_{\mathrm{out}}h_{\mathrm{DAP}}^d.
\]

For each \(j\), take a \(4\rho_{\mathrm{init}}\)-separated subset of \(\overline B(\mu_j,s_j/2)\) of cardinality at least \(c\rho_{\mathrm{init}}^{-d}\), and project each point to its nearest grid point in \(\mathcal Z_n\). Choose \(n_0\) so that, for all \(n\ge n_0\),
\[
\frac{\sqrt d}{2}h_{\mathrm{DAP}}\le\frac{\rho_{\mathrm{init}}}{8},
\qquad
\frac{\rho_{\mathrm{init}}}{8}<\frac12\min_{1\le j\le M}s_j,
\qquad
\rho_{\mathrm{init}}<\frac{c_0}{3}.
\]
The projected points lie in \(\overline B(\mu_j,s_j)\). Points projected from the same mode remain \(2\rho_{\mathrm{init}}\)-separated because
\[
\|z-z'\|\ge 4\rho_{\mathrm{init}}-2\cdot\frac{\rho_{\mathrm{init}}}{8}
>2\rho_{\mathrm{init}}.
\]
If \(z\) and \(z'\) are projected from different modes \(i\neq j\), then
\[
\|z-z'\|
\ge
\|\mu_i-\mu_j\|-s_i-s_j
\ge
\frac23\|\mu_i-\mu_j\|
\ge
\frac23c_0
>
2\rho_{\mathrm{init}}.
\]
Let \(\mathcal P_n\) be the union of the projected grid points over \(j\in[M]\). Since \(\rho_{\mathrm{init}}\asymp(\log n)^{-1/d}\), there is \(c_{\mathrm{pack}}>0\) such that
\[
|\mathcal P_n|\ge c_{\mathrm{pack}}M\rho_{\mathrm{init}}^{-d}
\ge c_{\mathrm{pack}}M\log n.
\]
No suppression ball of radius \(\rho_{\mathrm{init}}\) can remove two points of \(\mathcal P_n\). The upper implicit constant in \(k\asymp M\log n\) is chosen so that
\[
k\le \frac{c_{\mathrm{pack}}}{2}M\log n.
\]
Thus, before every draw \(\ell\le k\), at least one point of \(\mathcal P_n\) remains active.

Set
\[
\Gamma:=\min\{c_{\mathrm{out}}/2,\,3\gamma_0/4\}.
\]
We next use the lower bound on \(\varepsilon_{\mathrm{init}}\) explicitly. The notation
\[
\varepsilon_{\mathrm{init}}
\gtrsim
M n^{-2\beta/(d+2\beta)}\mathrm{polylog}(n)
\]
is used here in the following sufficient form: there exists a constant \(C_{\varepsilon,\mathrm{init}}>0\) such that
\[
\varepsilon_{\mathrm{init}}
\ge
C_{\varepsilon,\mathrm{init}}
M n^{-2\beta/(d+2\beta)}
(\log n)^{2-d/(d+2\beta)+2/d}.
\]
Since \(k\asymp M\log n\), there is \(C_k>0\) such that \(k\le C_kM\log n\). Since \(h_{\mathrm{DAP}}\asymp(\log n/n)^{1/(d+2\beta)}\) and \(\rho_{\mathrm{init}}\asymp(\log n)^{-1/d}\), there are \(c_h,c_\rho>0\) such that
\[
h_{\mathrm{DAP}}^d
\ge
c_h\left(\frac{\log n}{n}\right)^{d/(d+2\beta)},
\qquad
\rho_{\mathrm{init}}^2
\ge
c_\rho(\log n)^{-2/d}.
\]
Therefore
\[
\begin{aligned}
\lambda h_{\mathrm{DAP}}^d\rho_{\mathrm{init}}^2
&=
\frac{n\varepsilon_{\mathrm{init}}}{2k}
h_{\mathrm{DAP}}^d\rho_{\mathrm{init}}^2                                                   \\
&\ge
\frac{n}{2C_kM\log n}
\cdot
C_{\varepsilon,\mathrm{init}}
M n^{-2\beta/(d+2\beta)}
(\log n)^{2-d/(d+2\beta)+2/d}                                      \\
&\qquad\qquad{}\times
c_h\left(\frac{\log n}{n}\right)^{d/(d+2\beta)}
c_\rho(\log n)^{-2/d}                                                \\
&=
\frac{C_{\varepsilon,\mathrm{init}}c_hc_\rho}{2C_k}\log n .
\end{aligned}
\]
Choose
\[
C_{\varepsilon,\mathrm{init}}
\ge
\frac{18C_k}{\Gamma c_hc_\rho}.
\]
Then
\[
\Gamma\lambda h_{\mathrm{DAP}}^d\rho_{\mathrm{init}}^2\ge 9\log n.
\]

Fix any draw \(\ell\le k\). Let \(T\) be any active singleton contained in \(\mathcal P_n\), and set
\[
S=A_\ell\setminus\bigcup_{j=1}^M\overline B(\mu_j,r_j).
\]
On \(\mathcal E_{\mathrm{conc}}\), the modal-comparator gap gives
\[
\begin{aligned}
\inf_{q:\,z_q\in T}u_q-\sup_{r:\,z_r\in S}u_r
&\ge
\inf_{q:\,z_q\in T}m_q-\sup_{r:\,z_r\in S}m_r
-2c_{\mathrm{conc}}h_{\mathrm{DAP}}^d\rho_{\mathrm{init}}^2  \\
&\ge
c_{\mathrm{out}}h_{\mathrm{DAP}}^d
-2c_{\mathrm{conc}}h_{\mathrm{DAP}}^d\rho_{\mathrm{init}}^2.
\end{aligned}
\]
Increase \(n_0\), if necessary, so that \(2c_{\mathrm{conc}}\rho_{\mathrm{init}}^2\le c_{\mathrm{out}}/2\) and \(\rho_{\mathrm{init}}^2\le1\) for all \(n\ge n_0\). Then
\[
\inf_{q:\,z_q\in T}u_q-\sup_{r:\,z_r\in S}u_r
\ge
\Gamma h_{\mathrm{DAP}}^d\rho_{\mathrm{init}}^2.
\]
Since \(|S|\le N_{\mathrm{cand}}\le Ch_{\mathrm{DAP}}^{-d}\le Cn\), the exponential-mechanism bound gives
\[
\Pr\!\left(
J_\ell\notin\bigcup_{j=1}^M\overline B(\mu_j,r_j)\mid A_\ell
\right)
\le
Cn\exp\left(-\Gamma\lambda h_{\mathrm{DAP}}^d\rho_{\mathrm{init}}^2\right)
\le
Cn^{-8}.
\]
A union bound over \(k\asymp M\log n\) draws gives
\[
\Pr\!\left(
\exists \ell\le k:
a_\ell\notin\bigcup_{j=1}^M\overline B(\mu_j,r_j)
\,\middle|\,\mathcal E_{\mathrm{conc}}
\right)
\le
CM(\log n)n^{-8}
\le Cn^{-7},
\]
where the last inequality uses the implicit growth regime for \(M\) in the DAP design.
Let \(\mathcal E_{\mathrm{loc}}\) denote the event that all selected anchors lie in
\(\bigcup_{j=1}^M\overline B(\mu_j,r_j)\).

It remains to show that every modal basin is reached. Let
\(
\lambda^{(1)}>\lambda^{(2)}>\cdots>\lambda^{(R)}
\)
be the distinct values among \(\{p(\mu_j):j\in[M]\}\), and define
\[
\mathcal H_r:=\{j\in[M]:p(\mu_j)=\lambda^{(r)}\},
\qquad r\in[R].
\]
We work on \(\mathcal E_{\mathrm{conc}}\cap\mathcal E_{\mathrm{loc}}\). By Lemma~\ref{lem:dap-cap},
\[
\bigcup_{m=1}^M\mathcal C_{m,n}
\subseteq
\bigcup_{j=1}^M\overline B(\mu_j,r_j)
\subseteq \mathcal Q .
\]
Since the public box \(\mathcal Q\) is fixed and \(\rho_{\mathrm{init}}\asymp(\log n)^{-1/d}\), this union of competitive sets can be covered by at most \(L_{\mathrm{peel}}M\log n\) Euclidean balls of radius \(\rho_{\mathrm{init}}/2\), with \(L_{\mathrm{peel}}\) independent of \(M\), \(n\), and \(\varepsilon_{\mathrm{init}}\). The constants implicit in the DAP design choice \(k\asymp M\log n\) are fixed so that, for all \(n\ge n_0\),
\[
L_{\mathrm{peel}}M\log n+M\le k\le \frac{c_{\mathrm{pack}}}{2}M\log n.
\]
The upper bound is the packing requirement used above to keep an active modal comparator throughout the \(k\) draws; the lower bound is the peeling requirement used below.

We prove coverage level by level. Suppose that all modes in \(\mathcal H_1,\dots,\mathcal H_{r-1}\) have been hit. Let \(V_r\subseteq\mathcal H_r\) be the modes in the current level that have already been hit, and set \(U_r:=\mathcal H_r\setminus V_r\). If \(U_r=\varnothing\), move to level \(r+1\). Otherwise define
\[
\mathcal T_{U_r,n}:=\bigcup_{j\in U_r}\mathcal G_{j,n},
\qquad
\mathcal B_{U_r}:=\bigcup_{j\in U_r}\overline B(\mu_j,r_j).
\]
For every \(j\in U_r\), \(\mathcal G_{j,n}\) remains active until mode \(j\) is hit. Indeed, on \(\mathcal E_{\mathrm{loc}}\), every earlier anchor lies in some \(\overline B(\mu_i,r_i)\). If \(i\neq j\), \(a\in\overline B(\mu_i,r_i)\), and \(z\in\mathcal G_{j,n}\), then
\[
\|z-a\|
\ge
\|\mu_i-\mu_j\|-\|a-\mu_i\|-\|z-\mu_j\|
\ge
\frac{c_0}{2}-\frac{\rho_{\mathrm{init}}}{4}.
\]
Increasing \(n_0\), if necessary, so that \(\rho_{\mathrm{init}}\le c_0/5\), the last display is larger than \(\rho_{\mathrm{init}}\). Thus suppressions from other basins do not remove \(\mathcal G_{j,n}\), and no anchor has yet fallen in \(\overline B(\mu_j,r_j)\) because \(j\in U_r\).

Choose \(m_r^\star\in U_r\) such that
\[
\inf_{q:\,z_q\in\mathcal G_{m_r^\star,n}}m_q
=
\inf_{q:\,z_q\in\mathcal T_{U_r,n}}m_q.
\]
For a draw made before level \(r\) is completed, set
\[
S_0:=A_\ell\setminus(\mathcal B_{U_r}\cup\mathcal C_{m_r^\star,n}).
\]
Then the definition of \(\mathcal C_{m_r^\star,n}\) gives
\[
\sup_{s:\,z_s\in S_0}m_s
<
\inf_{q:\,z_q\in\mathcal T_{U_r,n}}m_q
-
\gamma_0h_{\mathrm{DAP}}^d\rho_{\mathrm{init}}^2.
\]
On \(\mathcal E_{\mathrm{conc}}\),
\[
\inf_{q:\,z_q\in\mathcal T_{U_r,n}}u_q-
\sup_{s:\,z_s\in S_0}u_s
\ge
\frac{3\gamma_0}{4}h_{\mathrm{DAP}}^d\rho_{\mathrm{init}}^2
\ge
\Gamma h_{\mathrm{DAP}}^d\rho_{\mathrm{init}}^2.
\]
The exponential-mechanism bound therefore gives
\[
\Pr(J_\ell\in S_0\mid A_\ell)
\le
Cn\exp\left(-\Gamma\lambda h_{\mathrm{DAP}}^d\rho_{\mathrm{init}}^2\right)
\le
Cn^{-8}.
\]
Thus, except on an event of conditional probability at most \(Cn^{-8}\), each draw before the current level is completed either lands in \(\mathcal B_{U_r}\), which hits at least one mode in \(U_r\), or lands in the active part of \(\mathcal C_{m_r^\star,n}\). By Lemma~\ref{lem:dap-cap}, the latter set lies in basin neighborhoods whose heights are at least \(\lambda^{(r)}\). Since points in \(A_\ell\setminus\mathcal B_{U_r}\) are outside the unvisited level-\(r\) basins, any such active competitive point lies in a higher-level basin or in an already-hit same-level basin.

If a selected anchor \(a_\ell\) lies in one of the fixed competitive-cover balls with center \(y\), then every grid point \(z\) in that same ball is removed at the next suppression step, since
\[
\|z-a_\ell\|
\le
\|z-y\|+\|a_\ell-y\|
\le
\frac{\rho_{\mathrm{init}}}{2}+\frac{\rho_{\mathrm{init}}}{2}
=
\rho_{\mathrm{init}}.
\]
Because the packing argument above keeps the active set nonempty throughout the \(k\) draws, the reset clause in Algorithm~\ref{alg:dap-init} is not invoked on this event. Therefore a competitive-cover ball, once suppressed, is not charged again. Over the full run, at most \(L_{\mathrm{peel}}M\log n\) non-hit draws can be charged to active competitive-cover balls. Since at most \(M\) successful hits are needed to hit all basins and \(k\ge L_{\mathrm{peel}}M\log n+M\), failure to hit every basin by time \(k\) implies that at least one exceptional event \(J_\ell\in S_0\) occurred. A union bound over \(k\asymp M\log n\) draws yields
\[
\Pr\!\left(
\exists j\in[M]:\mathcal I\cap\overline B(\mu_j,r_j)=\varnothing
\,\middle|\,\mathcal E_{\mathrm{conc}}\cap\mathcal E_{\mathrm{loc}}
\right)
\le
CM(\log n)n^{-8}
\le Cn^{-7},
\]
again using the implicit growth regime for \(M\) in the DAP design.
Combining the concentration, localization, and coverage bounds gives
\[
\begin{aligned}
&\Pr\!\left(
\mathcal I\subseteq\bigcup_{j=1}^M\overline B(\mu_j,r_j),
\quad
\mathcal I\cap\overline B(\mu_j,r_j)\neq\varnothing\ \text{for every }j\in[M]
\right) \\
&\qquad\ge
1-C_{\mathrm{DAP,tail}}n^{-4}-Cn^{-7}-Cn^{-7}
\ge
1-C_{\mathrm{init,cov}}n^{-2}.
\end{aligned}
\]
\end{proof}

\begin{lemma}[High-probability static global event]\label{lem:global-good-prob}
Under Assumptions~\ref{assump:kernel}, \ref{assump:model-holder},
and \ref{assump:bandwidth},
\[
\Pr\big(\mathcal A_n^{\mathrm{global}}\big)\ge 1-C_{\mathrm{global,stat}}\,n^{-2}
\]
for all sufficiently large \(n\), for some constant \(C_{\mathrm{global,stat}}>0\).
\end{lemma}

\begin{lemma}[Post-processing preserves basinwise rates]
\label{lem:postprocess-preserves}
Recall \(I_j:=\{\ell\in[k]:x_{0,\ell}\in \overline B(\mu_j,r_j)\}\) for \(j\in[M]\). Fix an event \(\mathcal E\) and deterministic bounds \(R_{n,j}\) such that
\[
\EE\!\left[
\|x_{T,\ell}-\mu_j\|^2
\,\middle|\,
\mathcal X,\mathcal E
\right]
\le R_{n,j}
\qquad
\text{for every }j\in[M]\text{ and every }\ell\in I_j.
\]
Suppose \(\widehat{\mathcal M}\) contains points \(\hat\mu_1,\dots,\hat\mu_M\) such that, after relabeling if necessary,
\[
\hat\mu_j
=
\frac{1}{|C_j|}\sum_{\ell\in C_j} x_{T,\ell}
\qquad\text{for some nonempty }C_j\subseteq I_j,
\qquad j\in[M].
\]
Then
\[
\EE\!\left[
\|\hat\mu_j-\mu_j\|^2
\,\middle|\,
\mathcal X,\mathcal E
\right]
\le R_{n,j},
\qquad j\in[M].
\]
\end{lemma}

The next two propositions verify the within-basin averaging hypothesis in the two merge regimes used in practice: radius merge when the number of modes is unknown, and Ward agglomerative merge when it is known.

\begin{proposition}[Radius merge when the number of modes is unknown]
\label{prop:merge-radius}
Assume \(I_j\neq\varnothing\) for every \(j\in[M]\). Suppose the final merge is the radius-based rule with merge radius \(h_{\mathrm{mode}}\). If
\[
\max_{j\in[M]}\max_{\ell\in I_j}\|x_{T,\ell}-\mu_j\|
\le \frac{h_{\mathrm{mode}}}{4},
\]
and, when the routine is applied to all \(k\) endpoints,
\[
\min_{m\notin \cup_{j=1}^M I_j}\min_{j\in[M]}
\|x_{T,m}-\mu_j\|
>
\frac{5}{4}h_{\mathrm{mode}},
\]
then, for all sufficiently large \(n\), the merged output contains points \(\hat\mu_1,\dots,\hat\mu_M\) such that
\[
\hat\mu_j
=
\frac{1}{|C_j|}\sum_{\ell\in C_j} x_{T,\ell}
\qquad\text{for some nonempty }C_j\subseteq I_j,
\qquad j\in[M].
\]
\end{proposition}

\begin{proposition}[Ward agglomerative merge when the number of modes is known]
\label{prop:merge-ward}
Assume \(I_j\neq\varnothing\) for every \(j\in[M]\). Suppose the final merge is Ward agglomerative clustering with target number of clusters equal to \(M\), run on an input set contained in
\(
\{x_{T,\ell}:\ell\in \cup_{j=1}^M I_j\}
\).
If
\(
\max_{j\in[M]}\max_{\ell\in I_j}\|x_{T,\ell}-\mu_j\|
\le \frac{h_{\mathrm{mode}}}{4}
\),
then for all sufficiently large \(n\), the merged output contains points \(\hat\mu_1,\dots,\hat\mu_M\) such that
\[
\hat\mu_j
=
\frac{1}{|C_j|}\sum_{\ell\in C_j} x_{T,\ell}
\qquad\text{for some nonempty }C_j\subseteq I_j,
\qquad j\in[M].
\]
\end{proposition}

It remains to pass from basinwise endpoints to the merged estimator. The proof below first upgrades the local stay-in-basin event to all basin-started trajectories and then invokes Lemma~\ref{lem:postprocess-preserves} together with the appropriate merge proposition.

\begin{proof}[Proof of Theorem~\ref{thm:global-holder}]
Recall from Definition~\ref{def:good-event-global} that
\[
\mathcal G_{n,T}^{\mathrm{global}}
=
\mathcal A_n^{\mathrm{global}}
\cap
\bigcap_{j=1}^M\bigcap_{\ell\in I_j}
\mathcal E_{n,T}^{\mathrm{stay},j}(\ell).
\]
By Lemma~\ref{lem:global-good-prob},
\(
\Pr(\mathcal A_n^{\mathrm{global}})\ge 1-C_{\mathrm{global,stat}}\,n^{-2}
\)
for all sufficiently large \(n\). We focus on \(\mathcal A_n^{\mathrm{global}}\).

Since \(\mathcal A_n^{\mathrm{global}}\subseteq \mathcal A_n^{\mathrm{local},j}\)
and \(x_{0,\ell}\in \overline B(\mu_j,r_j)\) for every \(\ell\in I_j\), the verification carried out in the proof of Theorem~\ref{th:local-conv-holder} shows that, under the same tuning
\[
T=C_T\log n,
\qquad
m\asymp \frac{n}{\log n},
\qquad
h=h_{\mathrm{opt}},
\]
the hypotheses of Proposition~\ref{prop:local-stay-holder} hold for the trajectory \(\{x_{t,\ell}\}_{t=0}^T\) for each pair \((j,\ell)\). Therefore
\[
\Pr\!\left(
\mathcal E_{n,T}^{\mathrm{stay},j}(\ell)
\,\middle|\,
\mathcal X,\mathcal A_n^{\mathrm{global}}
\right)\ge 1-n^{-5}
\]
for all sufficiently large \(n\). Since
\[
\sum_{j=1}^M |I_j|\le k\asymp \log n,
\]
a union bound gives
\[
\Pr\!\left(
\bigcap_{j=1}^M\bigcap_{\ell\in I_j}
\mathcal E_{n,T}^{\mathrm{stay},j}(\ell)
\,\middle|\,
\mathcal X,\mathcal A_n^{\mathrm{global}}
\right)
\ge 1-k n^{-5}.
\]
Hence
\[
\Pr(\mathcal G_{n,T}^{\mathrm{global}})
\ge 1-C_{\mathrm{global}}\,n^{-2}
\]
for some constant \(C_{\mathrm{global}}>0\) and all sufficiently large \(n\).

Now fix \(j\in[M]\) and \(\ell\in I_j\). Since
\(\mathcal A_n^{\mathrm{global}}\subseteq \mathcal A_n^{\mathrm{local},j}\),
the initialization satisfies \(x_{0,\ell}\in \overline B(\mu_j,r_j)\), and
\(\mathcal A_n^{\mathrm{global}}\) does not involve the ascent-stage randomness,
Proposition~\ref{prop:local-truncated-holder} applies conditionally on
\((\mathcal X,\mathcal A_n^{\mathrm{global}})\). Thus, for some constant \(C_j>0\),
\begin{align*}
&\EE\!\left[
\|x_{T,\ell}-\mu_j\|^2
\mathbf 1_{\mathcal E_{n,T}^{\mathrm{stay},j}(\ell)}
\,\middle|\,
\mathcal X,\mathcal A_n^{\mathrm{global}}
\right]\\
&\qquad\le
(1-\kappa_j\eta)^T\|x_{0,\ell}-\mu_j\|^2
+
C_j\left(
h^{2(\beta-1)}
+
\frac{\log n}{n h^{d+2}}
+
\frac{1}{m h^{d+2}}
+
\eta d\sigma^2
\right).
\end{align*}
Using the generic bound \eqref{eq:local-generic-h-bound}, the optimized bound \eqref{eq:local-optimized-Rn}, and the same transient-term choice as in Theorem~\ref{th:local-conv-holder}, with \(T=C_T\log n\), \(m\asymp n/\log n\), \(h=h_{\mathrm{opt}}\), and \(C_T\ge C_T^\star:=\max_{j\in[M]} C_{T,j}^\star\), this becomes
\[
\EE\!\left[
\|x_{T,\ell}-\mu_j\|^2
\mathbf 1_{\mathcal E_{n,T}^{\mathrm{stay},j}(\ell)}
\,\middle|\,
\mathcal X,\mathcal A_n^{\mathrm{global}}
\right]
\le
C_{\mathrm{nonDP},j}\Big(\frac{\log n}{n}\Big)^{\frac{2(\beta-1)}{d+2\beta}}
+
C_{\mathrm{DP},j}\Big(\frac{T d\,\mathrm{polylog}(n,\delta)}{n^2\varepsilon_{\mathrm{modes}}^2}\Big)^{\frac{\beta-1}{d+\beta}}.
\]
Since
\(
\mathcal G_{n,T}^{\mathrm{global}}
\subseteq
\mathcal E_{n,T}^{\mathrm{stay},j}(\ell)
\),
we have
\begin{align*}
\EE\!\left[
\|x_{T,\ell}-\mu_j\|^2
\,\middle|\,
\mathcal X,\mathcal G_{n,T}^{\mathrm{global}}
\right]
&=
\frac{
\EE\!\left[
\|x_{T,\ell}-\mu_j\|^2
\mathbf 1_{\mathcal G_{n,T}^{\mathrm{global}}}
\,\middle|\,
\mathcal X,\mathcal A_n^{\mathrm{global}}
\right]
}{
\Pr\!\left(
\mathcal G_{n,T}^{\mathrm{global}}
\,\middle|\,
\mathcal X,\mathcal A_n^{\mathrm{global}}
\right)
}\\
&\le
\frac{
\EE\!\left[
\|x_{T,\ell}-\mu_j\|^2
\mathbf 1_{\mathcal E_{n,T}^{\mathrm{stay},j}(\ell)}
\,\middle|\,
\mathcal X,\mathcal A_n^{\mathrm{global}}
\right]
}{
\Pr\!\left(
\mathcal G_{n,T}^{\mathrm{global}}
\,\middle|\,
\mathcal X,\mathcal A_n^{\mathrm{global}}
\right)
}.
\end{align*}
Using
\(
\Pr\!\left(
\mathcal G_{n,T}^{\mathrm{global}}
\,\middle|\,
\mathcal X,\mathcal A_n^{\mathrm{global}}
\right)\ge 1-k n^{-5},
\)
and absorbing \((1-k n^{-5})^{-1}\) into the constants for all sufficiently large
\(n\), we obtain
\[
\EE\!\left[
\|x_{T,\ell}-\mu_j\|^2
\,\middle|\,
\mathcal X,\mathcal G_{n,T}^{\mathrm{global}}
\right]
\le
C_{\mathrm{nonDP},j}\Big(\frac{\log n}{n}\Big)^{\frac{2(\beta-1)}{d+2\beta}}
+
C_{\mathrm{DP},j}\Big(\frac{T d\,\mathrm{polylog}(n,\delta)}{n^2\varepsilon_{\mathrm{modes}}^2}\Big)^{\frac{\beta-1}{d+\beta}},
\]
for every \(j\in[M]\) and every \(\ell\in I_j\).

Therefore Lemma~\ref{lem:postprocess-preserves} applies with
\(\mathcal E=\mathcal G_{n,T}^{\mathrm{global}}\). If the number of modes is unknown and radius merge is used, Proposition~\ref{prop:merge-radius} supplies the required within-basin representation. If the number of modes is known and Ward agglomerative merge is used, Proposition~\ref{prop:merge-ward} supplies it. In either case, \(\widehat{\mathcal M}\) contains points \(\hat\mu_1,\dots,\hat\mu_M\) satisfying the claimed conditional MSE bound. This proves the theorem.
\end{proof}

\subsection{Minimax lower bound}

\begin{proof}[Proof of Theorem~\ref{thm:lower-bound}]
Let $p_0$ be a $C^\infty$ density on $\mathbb R^d$ with a unique mode at $\mathbf{0}$,
such that $p_0(\mathbf{0})>1/C$ and $\nabla^2\log p_0(\mathbf{0})\prec 0$. Clearly $p_0\in\mathcal H^\beta(L)$ for all
$\beta$. We next create the localized shift alternatives as follows.

For a function \(\psi\in \mathcal{H}^{\beta}(L)\) with $\psi(0)=1$, $\int \psi(u) du =0$, $\int |\psi(u)| du \le C$ ,$\|\nabla \psi(0)\|\ge 1/C$ , let us define
\[
p_\theta(x) := p_0(x)(1 + \theta h^{\beta}\psi(x/h)).
\]
for $\theta\in \{+1,-1\}$. Since translations preserve H{\" o}lder regularity, for $h$ small enough we have
$p_\theta\in\mathcal H_\beta(L)$. Recall that $s_p(x)=\nabla\log p(x)$. For any $x$, using a Taylor expansion of
$\nabla\log p_0$ around $x$, we have
\[
s_{p_\theta}(x)
= 
\frac{1}{p_{\theta}(x)}
\nabla p_0(x)
+\frac{\theta h^{\beta -1}}{\psi(x/h)}\nabla \psi(x/h)
\]
Since $\|\nabla^2\log p_0(x)\|\le C$ for sufficiently small $h$, it can be checked that
\begin{equation}\label{eq:score-wrt-h}
\inf_{x:\|x\|\le h}
\|s_{p_1}(x) - s_{p_{-1}}(x)\|
\ge 
C h^{\beta-1}.
\end{equation}
Note that
\begin{align*}
    {\rm TV}(p_1,p_{-1})
    = h^{\beta}\int |\psi(x/h)| dx
    =h^{\beta+d}\int |\psi(u)| du
    \le Ch^{\beta +d}
\end{align*}
and
\begin{align*}
    {\rm KL}(p_1,p_{-1})
    =&~
    \int p_1(x)
    \log \frac{p_1(x)}{p_{-1}(x)} dx
    \le~
    2\int p_1(x)\cdot (h^{\beta}\psi(x/h))^2dx\\
    \le&~ Ch^{2\beta+d}\int p_1(uh)(\psi(u))^2du
    \le Ch^{2\beta+d}.
\end{align*}
We next use Lemma 6.1 of \cite{karwa2017finite} to write
\begin{align*}
\sup_{S; k,k'\in \{-1,1\}; \PP_{\theta_k}(T_{\varepsilon,\delta}\in S)>\delta' }
\log 
\left[\frac{\PP_{\theta_k}(T_{\varepsilon,\delta }\in S)-\delta' }{
\PP_{\theta_{k'}}(T_{\varepsilon,\delta }\in S)
}\right]
\le \varepsilon'
\end{align*}
where $\delta'=\exp(\varepsilon')n\delta {\rm TV}(p_1 ,p_{-1})$, $\varepsilon'=6\varepsilon n {\rm TV}(p_1 ,p_{-1})$, and for a fixed $x$, let \(T_{\varepsilon,\delta}(x)\) be any \(\varepsilon,\delta\) transcript from $n$ iid observations. Next, following lemmas C4 to C7 of \cite{cai2024optimal} we have:
\begin{align*}
({\rm TV}(p_1^{T_{\varepsilon,\delta}},p_{-1}^{T_{\varepsilon,\delta}}))^2
\le&~ 
4\min\{\varepsilon'(\exp(\varepsilon')-1),n{\rm KL}(p_{1},p_{-1})\}
+32\exp(2\varepsilon')n^2\delta^2 ({\rm TV}(p_{1},p_{-1}))^2\\
\le&~
C\min\{n^2\varepsilon^2h^{2\beta +2d},nh^{2\beta+d}\}
+
n^2\delta^2 h^{2\beta +2d}
\end{align*}
Now choosing $h= \max\{n^{-\frac{1}{2\beta +d}}, (n\varepsilon)^{-\frac{1}{\beta +d}}\}$ and since $\delta =o(n^{-1})$ by assumption, we have 
\[
({\rm TV}(p_1^{T_{\varepsilon,\delta}},p_{-1}^{T_{\varepsilon,\delta}}))^2\le c
\]
for a sufficiently small constant $c>0$. Now using Le Cam two-point lemma, in particular Lemma 1 of \cite{yu1997assouad}, in conjunction with \eqref{eq:score-wrt-h} we have for any $x$ with $\|x\|\le h$ that
\[
\inf_{\hat{s}}
\sup_{s}
\EE \|\hat{s}(x)-s(x)\|^2
\ge n^{-\frac{2(\beta-1)}{2\beta+d}}
+ (n\varepsilon)^{-\frac{2(\beta-1)}{\beta +d}}.
\]
Now to obtain rates for mode estimation, note that, if $\hat{s}$ is any score estimator and estimated mode is $\hat{x}$, we have 
\begin{align*}
    0=\hat{s}(\hat{x})
    =&~\hat{s}(\hat{x})-s(\hat{x})+s(\hat{x})
    =~\hat{s}(\hat{x})-s(\hat{x})+s(0)+(H(\xi))(\hat{x}-0)
\end{align*}
for $\xi=t\hat{x}$ for some $t\in [0,1]$, which implies
\begin{align*}
    \EE\|\hat{x}-0\|^2
    \ge [\EE \lambda_{\max}(H(\xi))]^{-2}\EE\|\hat{s}(\hat{x})-s(\hat{x})\|^2.
\end{align*}
where $H(\xi):=\nabla s(x)\vert_{x=\xi}$ is the Hessian of log density evaluated at $\xi$. By assumptions, the Hessian at all $x$ has bounded singular values, and hence the minimax lower bound for score estimation implies
\begin{align*}
    \inf_{\hat{x}}\sup_{P}
    \EE\|\hat{x}-x_0(P)\|^2
    \ge&~ C\inf_{\hat{s},\|x\|\le h}\sup_{P}\EE\|\hat{s}(x)-s_P(x)\|^2
    \gtrsim~ n^{-\frac{2(\beta-1)}{2\beta+d}}
+ (n\varepsilon)^{-\frac{2(\beta-1)}{\beta +d}}.
\end{align*}
\end{proof}

\subsection{Proofs of auxiliary lemmas}\label{sec:proof-lemmas}

This subsection verifies the supporting lemmas used earlier. The proofs are grouped to mirror the main argument: privacy first, then local geometric and analytic control, then stabilization and minibatch bounds, and finally the DAP and global lemmas.

\paragraph{Differential Privacy lemmas.}

\begin{proof}[Proof of Lemma~\ref{lem:sensitivity}]
Fix \(x\in\mathbb R^d\). Write
\[
p:=\hat p_{\mathcal X}(x),\qquad p':=\hat p_{\mathcal X'}(x),\qquad
g:=\nabla \hat p_{\mathcal X}(x),\qquad g':=\nabla \hat p_{\mathcal X'}(x).
\]
Since \(\mathcal X\) and \(\mathcal X'\) differ in one entry, the KDE and KDE-gradient
differences satisfy
\[
|p-p'|
\le
\frac{2K_\infty}{n h^d},
\qquad
\|g-g'\|
\le
\frac{2G_K}{n h^{d+1}}.
\]
Now set
\[
d:=\max\{p,p_{\mathrm{floor}}\},
\qquad
d':=\max\{p',p_{\mathrm{floor}}\}.
\]
Then \(d,d'\ge p_{\mathrm{floor}}\), and since \(u\mapsto \max\{u,p_{\mathrm{floor}}\}\) is \(1\)-Lipschitz,
\(
|d-d'|\le |p-p'|
\). 
Using
\[
\hat s_{A,p_{\mathrm{floor}};\mathcal X}(x)=\frac{\operatorname{clip}_A(g)}{d},
\qquad
\hat s_{A,p_{\mathrm{floor}};\mathcal X'}(x)=\frac{\operatorname{clip}_A(g')}{d'},
\]
we obtain
\begin{align*}
\left\|
\hat s_{A,p_{\mathrm{floor}};\mathcal X}(x)
-
\hat s_{A,p_{\mathrm{floor}};\mathcal X'}(x)
\right\|
&\le
\left\|
\frac{\operatorname{clip}_A(g)-\operatorname{clip}_A(g')}{d}
\right\|
+
\left\|
\operatorname{clip}_A(g')
\left(\frac{1}{d}-\frac{1}{d'}\right)
\right\| \\
&\le
\frac{1}{p_{\mathrm{floor}}}
\|\operatorname{clip}_A(g)-\operatorname{clip}_A(g')\|
+
\frac{\|\operatorname{clip}_A(g')\|}{d\,d'}\,|d-d'|.
\end{align*}
The clipping map \(\operatorname{clip}_A\) is the Euclidean projection onto the closed
ball of radius \(A\), hence it is \(1\)-Lipschitz, and also
\(\|\operatorname{clip}_A(g')\|\le A\). Therefore
\begin{align*}
\left\|
\hat s_{A,p_{\mathrm{floor}};\mathcal X}(x)
-
\hat s_{A,p_{\mathrm{floor}};\mathcal X'}(x)
\right\|
&\le
\frac{1}{p_{\mathrm{floor}}}\|g-g'\|
+
\frac{A}{p_{\mathrm{floor}}^2}|p-p'| 
\le
\frac{2G_K}{n\,p_{\mathrm{floor}}}\,h^{-(d+1)}
+
\frac{2A K_\infty}{n\,p_{\mathrm{floor}}^2}\,h^{-d} \\
&=
\frac{S_h(A,p_{\mathrm{floor}})}{n}.
\end{align*}
This proves the claim.
\end{proof}

\medskip

\begin{proof}[Proof of Lemma~\ref{thm:dp_grams_privacy}]
We prove privacy of the full \(T\)-round ascent stage.

Let us fix a round \(t\in\{0,\dots,T-1\}\), and condition on the full transcript up to the start
of round \(t\). Under this conditioning, the current iterate \(x_t\) is fixed. The round-\(t\)
mechanism releases
\[
Y_t
=
\hat s_{A,p_{\mathrm{floor}};\mathcal B_t}(x_t)+Z_t,
\qquad
Z_t\sim \mathcal N(0,\sigma^2 I_d).
\]

Let \(\mathcal B_t\) and \(\mathcal B_t'\) be neighboring minibatches of size \(m\), differing
in one entry. Applying Lemma~\ref{lem:sensitivity} with \(n\) replaced by \(m\) gives
\[
\left\|
\hat s_{A,p_{\mathrm{floor}};\mathcal B_t}(x_t)
-
\hat s_{A,p_{\mathrm{floor}};\mathcal B_t'}(x_t)
\right\|
\le
\frac{S_h(A,p_{\mathrm{floor}})}{m}.
\]
Thus the full-data \(\ell_2\)-sensitivity for one round is
\[
\Delta_t=\frac{S_h(A,p_{\mathrm{floor}})}{m}.
\]
We write
\[
\rho:=\frac{m}{n},
\qquad
\varepsilon^*
:=
\log\!\Bigl(1+\frac{e^{\varepsilon_{\mathrm{iter}}}-1}{\rho}\Bigr)
=
\log\!\Bigl(1+\frac{n(e^{\varepsilon_{\mathrm{iter}}}-1)}{m}\Bigr),
\qquad
\delta^*
:=
\frac{\delta_{\mathrm{iter}}}{\rho}
=
\frac{n\delta}{2mT}.
\]
By the Gaussian mechanism \citep[Theorem~3.22]{dwork2014algorithmic}, adding Gaussian noise with standard deviation
\[
\sigma
=
\frac{\Delta_t}{\varepsilon^*}
\sqrt{2\log\!\Bigl(\frac{1.25}{\delta^*}\Bigr)}
\]
makes the corresponding full-data round mechanism \((\varepsilon^*,\delta^*)\)-DP. Since
\[
\frac{1.25}{\delta^*}
=
\frac{1.25\cdot 2mT}{n\delta}
=
\frac{2.5mT}{n\delta},
\]
this is exactly
\[
\sigma
=
\frac{S_h(A,p_{\mathrm{floor}})/m}
{\log\!\bigl(1+n(e^{\varepsilon_{\mathrm{iter}}}-1)/m\bigr)}
\sqrt{2\log\!\left(\frac{2.5mT}{n\delta}\right)},
\]
which is precisely \eqref{sigma-exact}.

Now sample the minibatch uniformly without replacement. By privacy amplification by
subsampling without replacement \citep{balle2018privacy}, the actual round-\(t\) mechanism is
\((\varepsilon_{\mathrm{iter}},\delta_{\mathrm{iter}})\)-DP. We now compose the \(T\) rounds. By advanced composition \citep{dwork2010boosting},
\[
\varepsilon_{\mathrm{comp}}
\le
\sqrt{2T\log(2/\delta)}\,\varepsilon_{\mathrm{iter}}
+
T\varepsilon_{\mathrm{iter}}(e^{\varepsilon_{\mathrm{iter}}}-1),
\qquad
\delta_{\mathrm{comp}}
\le
\frac{\delta}{2}+T\delta_{\mathrm{iter}}.
\]
Since \(\delta_{\mathrm{iter}}=\delta/(2T)\), we have
\[
\delta_{\mathrm{comp}}\le \frac{\delta}{2}+T\cdot\frac{\delta}{2T}=\delta.
\]
Also, by the definition of \(\varepsilon_{\mathrm{iter}}\),
\[
\sqrt{2T\log(2/\delta)}\,\varepsilon_{\mathrm{iter}}
\le
\frac{\varepsilon_{\mathrm{modes}}}{2}.
\]
Further, \(\varepsilon_{\mathrm{iter}}\le \log(1+(e-1)m/n)\le 1\), so
\(e^{\varepsilon_{\mathrm{iter}}}-1\le 2\varepsilon_{\mathrm{iter}}\), and therefore
\[
T\varepsilon_{\mathrm{iter}}(e^{\varepsilon_{\mathrm{iter}}}-1)
\le
2T\varepsilon_{\mathrm{iter}}^2
\le
2T\cdot \frac{\varepsilon_{\mathrm{modes}}}{4T}
=
\frac{\varepsilon_{\mathrm{modes}}}{2},
\]
where we also used
\(\varepsilon_{\mathrm{iter}}\le \sqrt{\varepsilon_{\mathrm{modes}}/(4T)}\).
Hence
\[
\varepsilon_{\mathrm{comp}}\le \varepsilon_{\mathrm{modes}}.
\]
Thus the full \(T\)-round ascent stage is \((\varepsilon_{\mathrm{modes}},\delta)\)-DP.
\end{proof}

\medskip

\begin{proof}[Proof of Lemma~\ref{lem:correlated_noise_privacy}]
Let us fix the initialization pool
\(
\mathcal I=\{x_{0,1},\dots,x_{0,k}\}.
\),
a round \(t\in\{0,\dots,T-1\}\), and condition on the full transcript up to the start of round \(t\). Under this conditioning, the current iterates \(x_{t,1},\dots,x_{t,k}\) are fixed. Let us define
\[
\mathbf K_t
:=
\bigl[\bar C_h(x_{t,\ell},x_{t,r})\bigr]_{\ell,r=1}^k .
\]
Let \(\mathcal B_t\) and \(\mathcal B_t'\) be neighboring minibatches of size \(m\), differing
in one entry. Define the concurrent stabilized score difference
\[
\Delta_t
:=
\mathfrak s_t(\mathcal B_t)-\mathfrak s_t(\mathcal B_t')
\in\mathbb R^{kd}.
\]
Also define the corresponding density and gradient difference vectors
\[
\Delta p_t
:=
\Big(
\hat p_{\mathcal B_t}(x_{t,\ell})-\hat p_{\mathcal B_t'}(x_{t,\ell})
\Big)_{\ell=1}^k,
\quad 
\text{and}
\quad 
\Delta g_t
:=
\Big(
\nabla \hat p_{\mathcal B_t}(x_{t,\ell})-\nabla \hat p_{\mathcal B_t'}(x_{t,\ell})
\Big)_{\ell=1}^k.
\]

By Proposition~8 of \cite{hall2013differential}, applied with the exponential kernel
\(\bar C_h\), the corresponding RKHS sensitivities satisfy
\[
\|\Delta p_t\|_{\mathcal H_{\bar C_h}}
\le
\frac{2 I_0^{1/2}}{m}\,h^{-d},
\qquad
\|\Delta g_t\|_{\mathcal H_{\bar C_h}^d}
\le
\frac{2 I_1^{1/2}}{m}\,h^{-(d+1)}.
\]
Let us now define
\[
\Psi(u,v):=\frac{\operatorname{clip}_A(v)}{\max\{u,p_{\mathrm{floor}}\}},
\qquad
u\in\mathbb R,\ v\in\mathbb R^d.
\]
For any \(u,u'\in\mathbb R\) and \(v,v'\in\mathbb R^d\),
\[
\|\Psi(u,v)-\Psi(u',v')\|
\le
\frac{1}{p_{\mathrm{floor}}}\|v-v'\|
+
\frac{A}{p_{\mathrm{floor}}^2}|u-u'|.
\]
Applying this componentwise across the \(k\) starts and using
\((a+b)^2\le 2a^2+2b^2\), we obtain
\begin{align*}
\|\Delta_t\|_{\mathcal H_{\bar C_h}^d}
&\le
\sqrt{2}\left(
\frac{1}{p_{\mathrm{floor}}}\|\Delta g_t\|_{\mathcal H_{\bar C_h}^d}
+
\frac{A}{p_{\mathrm{floor}}^2}\|\Delta p_t\|_{\mathcal H_{\bar C_h}}
\right) \\
&\le
\frac{2\sqrt{2}}{m}\left(
\frac{I_1^{1/2}}{p_{\mathrm{floor}}}\,h^{-(d+1)}
+
\frac{A I_0^{1/2}}{p_{\mathrm{floor}}^2}\,h^{-d}
\right) \\
&=
\frac{\Delta_{h,\mathrm{corr}}(A,p_{\mathrm{floor}})}{m}.
\end{align*}

For the finite set of current evaluation points \(\{x_{t,1},\dots,x_{t,k}\}\), the RKHS norm induced by \(\bar C_h\) corresponds to the Mahalanobis \emph{seminorm} associated with the kernel matrix
\[
\mathbf K_t=\bigl[\bar C_h(x_{t,\ell},x_{t,r})\bigr]_{\ell,r=1}^k,
\quad 
\text{and}
\quad 
\|v\|_{\mathcal H_{\bar C_h}^d}^2
=
v^\top(\mathbf K_t^{\dagger}\otimes I_d)v,
\qquad v\in\mathbb R^{kd},
\]
where \(\mathbf K_t^{\dagger}\) denotes the Moore--Penrose pseudoinverse. Hence the full-data round mechanism
\[
Y_t
=
\mathfrak s_t(\mathcal B_t)+\Xi_t,
\qquad
\Xi_t\sim \mathcal N\!\bigl(0,\sigma^2(\mathbf K_t\otimes I_d)\bigr),
\]
is a Gaussian mechanism calibrated to the sensitivity
\(\Delta_{h,\mathrm{corr}}(A,p_{\mathrm{floor}})/m\) measured in this seminorm. We write
\[
\rho:=\frac{m}{n},
\qquad
\varepsilon^*
:=
\log\!\Bigl(1+\frac{e^{\varepsilon_{\mathrm{iter}}}-1}{\rho}\Bigr)
=
\log\!\Bigl(1+\frac{n(e^{\varepsilon_{\mathrm{iter}}}-1)}{m}\Bigr),
\qquad
\delta^*
:=
\frac{\delta_{\mathrm{iter}}}{\rho}
=
\frac{n\delta}{2mT}.
\]
By the Gaussian mechanism \citep[Theorem~3.22]{dwork2014algorithmic}, choosing
\[
\sigma
=
\frac{\Delta_{h,\mathrm{corr}}(A,p_{\mathrm{floor}})/m}{\varepsilon^*}
\sqrt{2\log\!\Bigl(\frac{1.25}{\delta^*}\Bigr)}
\]
makes the full-data round mechanism \((\varepsilon^*,\delta^*)\)-DP. Since
\(
1.25/\delta^*
=
2.5mT/(n\delta),
\)
this is exactly \eqref{eq:sigma-corr}.

By privacy amplification by subsampling without replacement \citep{balle2018privacy}, the actual round-\(t\) mechanism is \((\varepsilon_{\mathrm{iter}},\delta_{\mathrm{iter}})\)-DP. By advanced composition \citep{dwork2010boosting}, the joint mechanism \((Y_0,\dots,Y_{T-1})\) is \((\varepsilon_{\mathrm{modes}},\delta)\)-DP.

Now the pre-merge endpoint set
\(
\widetilde{\mathcal M}:=\{x_{T,1},\dots,x_{T,k}\}
\)
is a deterministic function of the fixed initialization pool \(\mathcal I\) and the joint mechanism output \((Y_0,\dots,Y_{T-1})\). Hence \(\widetilde{\mathcal M}\) is also \((\varepsilon_{\mathrm{modes}},\delta)\)-DP. Since the final estimator \(\widehat{\mathcal M}\) is obtained from \(\widetilde{\mathcal M}\) by deterministic merging, \(\widehat{\mathcal M}\) is \((\varepsilon_{\mathrm{modes}},\delta)\)-DP by post-processing \citep[Proposition~2.1 of][]{dwork2014algorithmic}.
\end{proof}

\subsubsection{Local geometric and analytic lemmas.}

\begin{proof}[Proof of Lemma~\ref{lem:logp-holder}]
Fix \(j\in[M]\). Since \(p\in \Sigma(\beta,L_j;\mathcal U_j)\) with \(\beta>2\),
the function \(p\) is \(C^2\) on \(\mathcal U_j\). In particular,
\(p\), \(\nabla p\), and \(\nabla^2 p\) are continuous in a neighborhood of \(\mu_j\).

By Assumption~\ref{assump:model-holder}(i), \(p(\mu_j)>0\). Hence, by continuity of \(p\),
there exists an open neighborhood \(V_j\subset\mathcal U_j\) of \(\mu_j\) such that
\[
p(x)\ge \frac{p(\mu_j)}{2}>0
\qquad\text{for all }x\in V_j.
\]
Therefore \(\log p\) is well defined on \(V_j\), and on \(V_j\) we have
\[
\nabla^2\log p(x)
=
\frac{\nabla^2 p(x)}{p(x)}
-
\frac{\nabla p(x)\nabla p(x)^\top}{p(x)^2}.
\]
Since \(p\), \(\nabla p\), and \(\nabla^2 p\) are continuous on \(V_j\), and
\(p(x)\) is bounded away from zero there, it follows that
\(x\mapsto \nabla^2\log p(x)\) is continuous at \(\mu_j\).

Hence, for every \(\xi>0\), there exists \(\widetilde r_j(\xi)>0\) such that
\[
\big\|\nabla^2\log p(x)-\nabla^2\log p(\mu_j)\big\|
\le \xi,
\qquad \forall x\in \overline B(\mu_j,\widetilde r_j(\xi)).
\]
Now suppose Assumption~\ref{assump:hessian-holder} also holds, so that
\(
\nabla^2\log p(\mu_j)\preceq -\alpha_j I_d
\).
Taking \(\xi=\alpha_j/2\), for any \(x\in \overline B(\mu_j,\widetilde r_j(\alpha_j/2))\),
\[
\nabla^2\log p(x)
\preceq
\nabla^2\log p(\mu_j)
+
\big\|\nabla^2\log p(x)-\nabla^2\log p(\mu_j)\big\|\,I
\preceq
-\alpha_j I+\frac{\alpha_j}{2}I
=
-\frac{\alpha_j}{2}I.
\]
This proves the claim.
\end{proof}


%

\begin{proof}[Proof of Lemma~\ref{lem:derived-inward-drift}]
Fix $j\in[M]$ and $x\in \overline B(\mu_j,r_j)$. Since $\nabla\log p(\mu_j)=0$, the fundamental theorem of calculus gives
\[
\nabla\log p(x)
=
\int_0^1 \nabla^2\log p\bigl(\mu_j+s(x-\mu_j)\bigr)(x-\mu_j)\,ds.
\]
Therefore
\[
\langle x-\mu_j,\nabla\log p(x)\rangle
=
\int_0^1
(x-\mu_j)^\top
\nabla^2\log p\bigl(\mu_j+s(x-\mu_j)\bigr)
(x-\mu_j)\,ds.
\]
By Assumption~\ref{assump:radius-holder}, since
\(
r_j\le \widetilde r_j
\)
and hence
\[
\nabla^2\log p(y)\preceq -\frac{\alpha_j}{2}I
\qquad\text{for all }y\in\overline B(\mu_j,r_j).
\]
Hence
\[
\langle x-\mu_j,\nabla\log p(x)\rangle
\le
-\frac{\alpha_j}{2}\|x-\mu_j\|^2.
\]
If $x\in \partial \overline B(\mu_j,r_j)$, then $\|x-\mu_j\|=r_j$, so
\[
\langle x-\mu_j,\nabla\log p(x)\rangle
\le
-\frac{\alpha_j r_j^2}{2}.
\]
\end{proof}


%

\begin{proof}[Proof of Lemma~\ref{lem:uniform_kde_ball}]
Fix $j\in[M]$ and write $B_j:=\overline B(\mu_j,r_j)$. Fix $s\in\{0,1,2\}$ and a multi-index
$\alpha$ with $|\alpha|=s$. Recall
\[
\partial^\alpha \hat p(x)
=\frac{1}{n h^{d+s}}\sum_{i=1}^n \partial^\alpha K\!\Big(\frac{x-X_i}{h}\Big).
\]
For each $x\in B_j$, decompose
\[
\partial^\alpha \hat p(x)-\partial^\alpha p(x)
=
\Big(\partial^\alpha \hat p(x)-\EE\,\partial^\alpha \hat p(x)\Big)
+
\Big(\EE\,\partial^\alpha \hat p(x)-\partial^\alpha p(x)\Big).
\]
We first bound the bias term. A change of variables gives
\[
\EE\,\partial^\alpha \hat p(x)
=\frac{1}{h^{s}}\int_{\R^d} \partial^\alpha K(u)\,p(x-hu)\,du .
\]
Since $B_j\subset \mathcal U_j$ and $B_j$ is compact, there exists $\eta_j>0$ such that the closed
$\eta_j$-neighborhood of $B_j$ is contained in $\mathcal U_j$. Split the integral as
\[
\EE\,\partial^\alpha \hat p(x)
=
\frac{1}{h^s}\int_{\|u\|\le \eta_j/h}\partial^\alpha K(u)\,p(x-hu)\,du
+
\frac{1}{h^s}\int_{\|u\|> \eta_j/h}\partial^\alpha K(u)\,p(x-hu)\,du
=: I_{1,\alpha}(x)+I_{2,\alpha}(x).
\]

For the local term $I_{1,\alpha}(x)$, if $\|u\|\le \eta_j/h$ and $x\in B_j$, then $x-hu\in\mathcal U_j$.
Let $\ell=\lfloor\beta\rfloor\ge 2$. Since $p\in\Sigma(\beta,L_j;\mathcal U_j)$, Taylor's theorem with
integral remainder yields, uniformly for $x\in B_j$ and $\|u\|\le \eta_j/h$,
\[
p(x-hu)
=\sum_{|\gamma|\le \ell}\frac{(-h)^{|\gamma|}}{\gamma!}\,(\partial^\gamma p)(x)\,u^\gamma
+R_\ell(x,u),
\qquad
|R_\ell(x,u)|\le C_j\,h^\beta\|u\|^\beta .
\]
Inserting this into $I_{1,\alpha}(x)$ gives
\[
I_{1,\alpha}(x)
=
\sum_{|\gamma|\le \ell}\frac{(-h)^{|\gamma|-s}}{\gamma!}\,(\partial^\gamma p)(x)
\int_{\|u\|\le \eta_j/h}\partial^\alpha K(u)\,u^\gamma\,du
+
\frac{1}{h^s}\int_{\|u\|\le \eta_j/h}\partial^\alpha K(u)\,R_\ell(x,u)\,du .
\]

For each $|\gamma|\le \ell$, write
\[
\int_{\|u\|\le \eta_j/h}\partial^\alpha K(u)\,u^\gamma\,du
=
\int_{\R^d}\partial^\alpha K(u)\,u^\gamma\,du
-
\int_{\|u\|>\eta_j/h}\partial^\alpha K(u)\,u^\gamma\,du .
\]
Since $K$ is a kernel of order $\ell$, the full-space moments satisfy
\[
\int_{\R^d}\partial^\alpha K(u)\,u^\gamma\,du
=
\begin{cases}
(-1)^{|\alpha|}\alpha!, & \gamma=\alpha,\\[1mm]
0, & \gamma\neq \alpha,\ |\gamma|\le \ell,
\end{cases}
\]
by integration by parts and the order-$\ell$ moment conditions. Hence
\[
I_{1,\alpha}(x)
=
\partial^\alpha p(x)+R_{1,\alpha}(x)+R_{2,\alpha}(x),
\]
where
\[
R_{1,\alpha}(x)
=
-\sum_{|\gamma|\le \ell}\frac{(-h)^{|\gamma|-s}}{\gamma!}\,(\partial^\gamma p)(x)
\int_{\|u\|>\eta_j/h}\partial^\alpha K(u)\,u^\gamma\,du
\]
and
\[
R_{2,\alpha}(x)
=
\frac{1}{h^s}\int_{\|u\|\le \eta_j/h}\partial^\alpha K(u)\,R_\ell(x,u)\,du .
\]

Since $B_j$ is compact and $p\in\Sigma(\beta,L_j;\mathcal U_j)$, all derivatives $\partial^\gamma p$
with $|\gamma|\le \ell$ are bounded on $B_j$. Using Assumption~\ref{assump:kernel}(iv), for each
$|\gamma|\le \ell$ and $|\alpha|\le 2$,
\[
\int_{\|u\|>\eta_j/h} |\partial^\alpha K(u)|\,\|u\|^{|\gamma|}\,du
\le
\Big(\frac{h}{\eta_j}\Big)^{\beta-|\gamma|}
\int_{\R^d}\|u\|^\beta |\partial^\alpha K(u)|\,du,
\]
since $|\gamma|\le \ell\le \beta$. Therefore
\[
|R_{1,\alpha}(x)|\le C h^{\beta-s},
\qquad x\in B_j.
\]
Also, using the remainder bound and Assumption~\ref{assump:kernel}(iv),
\[
|R_{2,\alpha}(x)|
\le
C_j h^{\beta-s}\int_{\R^d}\|u\|^\beta |\partial^\alpha K(u)|\,du
\le C h^{\beta-s},
\qquad x\in B_j.
\]

For the tail term $I_{2,\alpha}(x)$, using Assumption~\ref{assump:model-holder}(iii) and
Assumption~\ref{assump:kernel}(iv),
\[
|I_{2,\alpha}(x)|
\le
\frac{p_{\max}}{h^s}
\int_{\|u\|>\eta_j/h} |\partial^\alpha K(u)|\,du
\le
\frac{p_{\max}}{h^s}\Big(\frac{h}{\eta_j}\Big)^\beta
\int_{\R^d}\|u\|^\beta |\partial^\alpha K(u)|\,du
\le C h^{\beta-s},
\]
uniformly in $x\in B_j$.

Combining the bounds for $I_{1,\alpha}(x)$ and $I_{2,\alpha}(x)$ yields
\[
\sup_{x\in B_j}\big|\EE\,\partial^\alpha \hat p(x)-\partial^\alpha p(x)\big|
\le C\,h^{\beta-s}.
\]

We next bound the stochastic term. Define the centered process
\[
Z_\alpha(x)
=\frac{1}{n}\sum_{i=1}^n\Big(f_x(X_i)-\EE f_x(X)\Big),
\qquad
f_x(u):=h^{-(d+s)}\partial^\alpha K\!\Big(\frac{x-u}{h}\Big).
\]
By Assumption~\ref{assump:kernel}(iii),
\[
\|f_x\|_\infty\le h^{-(d+s)}\sup_{u\in\R^d}|\partial^\alpha K(u)|<\infty.
\]
Also, by Assumption~\ref{assump:model-holder}(iii) and Assumption~\ref{assump:kernel}(iv),
\[
\EE f_x(X)^2
=
\frac{1}{h^{2(d+s)}}\int_{\R^d}\Big(\partial^\alpha K\!\Big(\frac{x-y}{h}\Big)\Big)^2 p(y)\,dy
\le
h^{-(d+2s)}\,p_{\max}\int_{\R^d}(\partial^\alpha K(u))^2\,du
\lesssim h^{-(d+2s)},
\]
uniformly over $x\in B_j$, after the change of variables $u=(x-y)/h$.
Bernstein's inequality therefore gives, for any fixed $x\in B_j$ and any $t>0$,
\[
\Pr\big(|Z_\alpha(x)|>t\big)
\le 2\exp\!\left(
-\frac{c\,n t^2}{h^{-(d+2s)}+h^{-(d+s)}t}
\right).
\]
If $\|x-x'\|\le\rho$, then by the mean value theorem and Assumption~\ref{assump:kernel}(iii),
\[
|f_x(u)-f_{x'}(u)|
\le C\,h^{-(d+s+1)}\|x-x'\|,
\]
uniformly in $u$, and hence
\[
|Z_\alpha(x)-Z_\alpha(x')|
\le C\,h^{-(d+s+1)}\rho .
\]
Let $\{x_m\}_{m=1}^N$ be a $\rho$-net of $B_j$ with $N\lesssim (r_j/\rho)^d$. Then
\[
\sup_{x\in B_j}|Z_\alpha(x)|
\le \max_{1\le m\le N}|Z_\alpha(x_m)| + C\,h^{-(d+s+1)}\rho .
\]
Choose
\[
t:=C\sqrt{\frac{\log n}{n h^{d+2s}}},
\qquad
\rho:=c\,h^{d+s+1}t,
\]
so that the second term is at most $t/2$. A union bound over the net points and the above Bernstein
bound yield, for all sufficiently large $n$,
\[
\Pr\!\left(
\sup_{x\in B_j}|Z_\alpha(x)|>t
\right)
\le n^{-5}.
\]
Consequently, with probability at least $1-n^{-5}$,
\[
\sup_{x\in B_j}|Z_\alpha(x)|
\lesssim \sqrt{\frac{\log n}{n h^{d+2s}}}.
\]
Combining the bias and stochastic bounds yields, with probability at least $1-n^{-5}$,
\[
\sup_{x\in B_j}\big|\partial^\alpha \hat p(x)-\partial^\alpha p(x)\big|
\lesssim
h^{\beta-s}+\sqrt{\frac{\log n}{n h^{d+2s}}}.
\]
For fixed $s$, there are finitely many multi-indices $\alpha$ with $|\alpha|=s$; a union bound over them
absorbs into constants. A final union bound over $s\in\{0,1,2\}$ gives, for all sufficiently large $n$,
with probability at least $1-n^{-4}$,
\[
\sup_{x\in B_j}\|\nabla^s\hat p(x)-\nabla^s p(x)\|
\le
C\!\left(h^{\beta-s}+\sqrt{\frac{\log n}{n h^{d+2s}}}\right),
\qquad s=0,1,2,
\]
as claimed.
\end{proof}

\medskip

\begin{proof}[Proof of Lemma~\ref{lem:kde-lower-bound}]
Fix $j\in[M]$ and set $B_j:=\overline B(\mu_j,r_j)$. By Assumption~\ref{assump:model-holder}(i),
$p(x)\ge p_{\min,j}>0$ for all $x\in B_j$.
By Lemma~\ref{lem:uniform_kde_ball} with $s=0$, with probability at least $1-n^{-4}$,
\[
\sup_{x\in B_j}|\hat p(x)-p(x)|
\le
C\Big(h^{\beta}+\sqrt{\frac{\log n}{n h^{d}}}\Big).
\]
Under Assumption~\ref{assump:bandwidth}, \(h\to0\) and \(n h^{d+4}/\log n\to\infty\). Since \(h\le1\) for all sufficiently large \(n\), these imply
\(h^\beta+\sqrt{\log n/(n h^d)}=o(1)\). Hence, for all sufficiently large \(n\), the right-hand side is at most \(p_{\min,j}/2\). On this event,
\[
\inf_{x\in B_j}\hat p(x)
\ge
\inf_{x\in B_j}p(x)-\sup_{x\in B_j}|\hat p(x)-p(x)|
\ge
p_{\min,j}-\frac12 p_{\min,j}
=\frac12 p_{\min,j}.
\]
Thus the claim holds with $c_j:=p_{\min,j}/2$ and probability at least $1-n^{-4}$ for sufficiently large $n$.
\end{proof}

\medskip

\begin{proof}[Proof of Lemma~\ref{lem:uniform_log_ball}]
Fix $j\in[M]$ and write $B_j:=\overline B(\mu_j,r_j)$. Let $\ell=\log p$ and $\hat\ell=\log\hat p$.

Define the events
\[
A:=\left\{\inf_{x\in B_j}\hat p(x)\ge c_j\right\},\qquad
B:=\bigcap_{s=0}^2\left\{
\sup_{x\in B_j}\|\nabla^s\hat p(x)-\nabla^s p(x)\|
\le
C\left(h^{\beta-s}+\sqrt{\frac{\log n}{n h^{d+2s}}}\right)
\right\}.
\]
Lemma~\ref{lem:kde-lower-bound} gives $\Pr(A)\ge 1-n^{-4}$ for all large $n$.
Lemma~\ref{lem:uniform_kde_ball} gives $\Pr(B)\ge 1-n^{-4}$ for all large $n$.
Therefore, by the union bound,
\[
\Pr(A^c\cup B^c)\le \Pr(A^c)+\Pr(B^c)\le 2n^{-4},
\]
so
\begin{equation}\label{eq:prob-intersection-log}
\Pr(A\cap B)\ge 1-2n^{-4}.
\end{equation}
Work on the event $A\cap B$ below; in particular $\inf_{x\in B_j}\hat p(x)\ge c_j$ and also
$p(x)\ge p_{\min,j}$ for all $x\in B_j$.

For any $x\in B_j$, the mean value theorem gives
\[
|\hat\ell(x)-\ell(x)|
=|\log\hat p(x)-\log p(x)|
=\frac{|\hat p(x)-p(x)|}{\xi(x)},
\]
where $\xi(x)$ lies between $\hat p(x)$ and $p(x)$. Hence $\xi(x)\ge \min\{c_j,p_{\min,j}\}$ and
\[
\sup_{x\in B_j}|\hat\ell(x)-\ell(x)|
\le
\frac{1}{\min\{c_j,p_{\min,j}\}}\sup_{x\in B_j}|\hat p(x)-p(x)|.
\]
On $B$, the right-hand side is bounded by $C_j\!\left(h^\beta+\sqrt{\frac{\log n}{nh^d}}\right)$.

For any $x\in B_j$,
\[
\nabla\hat\ell(x)-\nabla\ell(x)
=\frac{\nabla\hat p(x)}{\hat p(x)}-\frac{\nabla p(x)}{p(x)}
=
\frac{\nabla\hat p(x)-\nabla p(x)}{\hat p(x)}
+\nabla p(x)\left(\frac{1}{\hat p(x)}-\frac{1}{p(x)}\right).
\]
Since $|\hat p^{-1}-p^{-1}|=\frac{|\hat p-p|}{\hat p\,p}$, on $A$ and $p\ge p_{\min,j}$,
\[
\|\nabla\hat\ell(x)-\nabla\ell(x)\|
\le
\frac{1}{c_j}\|\nabla\hat p(x)-\nabla p(x)\|
+\frac{\|\nabla p(x)\|}{c_j\,p_{\min,j}}|\hat p(x)-p(x)|.
\]
Taking suprema over $x\in B_j$ and using boundedness of $\|\nabla p(x)\|$ on $B_j$ (from Assumption~\ref{assump:model-holder})
gives a constant $C_j$ such that
\[
\sup_{x\in B_j}\|\nabla\hat\ell(x)-\nabla\ell(x)\|
\le
C_j\left(
\sup_{x\in B_j}\|\nabla\hat p(x)-\nabla p(x)\|
+
\sup_{x\in B_j}|\hat p(x)-p(x)|
\right).
\]
On $B$, each term is bounded by the stated rate with $s=1$ and $s=0$, hence the RHS is bounded by
$C_j\!\left(h^{\beta-1}+\sqrt{\frac{\log n}{n h^{d+2}}}\right)$. For any $x\in B_j$,
\[
\nabla^2\hat\ell(x)-\nabla^2\ell(x)
=
\left(\frac{\nabla^2\hat p(x)}{\hat p(x)}-\frac{\nabla^2 p(x)}{p(x)}\right)
-
\left(\frac{\nabla\hat p(x)\nabla\hat p(x)^\top}{\hat p(x)^2}
-\frac{\nabla p(x)\nabla p(x)^\top}{p(x)^2}\right).
\]
For the Hessian-fraction term,
\[
\frac{\nabla^2\hat p}{\hat p}-\frac{\nabla^2 p}{p}
=
\frac{\nabla^2\hat p-\nabla^2 p}{\hat p}
+\nabla^2 p\left(\frac{1}{\hat p}-\frac{1}{p}\right),
\]
so on $A$ and $p\ge p_{\min,j}$,
\[
\left\|\frac{\nabla^2\hat p(x)}{\hat p(x)}-\frac{\nabla^2 p(x)}{p(x)}\right\|
\le
\frac{1}{c_j}\|\nabla^2\hat p(x)-\nabla^2 p(x)\|
+\frac{\|\nabla^2 p(x)\|}{c_j\,p_{\min,j}}|\hat p(x)-p(x)|.
\]
For the quadratic-gradient term, add and subtract
$\nabla\hat p(x)\nabla p(x)^\top$:
\begin{align*}
\frac{\nabla\hat p\nabla\hat p^\top}{\hat p^2}-\frac{\nabla p\nabla p^\top}{p^2}
&=
\frac{\nabla\hat p(\nabla\hat p-\nabla p)^\top}{\hat p^2}
+\frac{(\nabla\hat p-\nabla p)\nabla p^\top}{\hat p^2}
+\nabla p\nabla p^\top\left(\frac{1}{\hat p^2}-\frac{1}{p^2}\right),
\end{align*}
and use $\|uv^\top\| \le \|u\|\,\|v\|$ and
\[
\left|\frac{1}{\hat p^2}-\frac{1}{p^2}\right|
=\frac{|p-\hat p|\,|p+\hat p|}{\hat p^2 p^2}
\le
\frac{(p_{\max,j}+ \|\hat p\|_\infty)}{c_j^2 p_{\min,j}^2}\,|\hat p-p|
\]
on $A$, where $p_{\max,j}:=\sup_{x\in B_j}p(x)<\infty$ and $\|\hat p\|_\infty<\infty$ (since on $B$ we have $\sup_{x\in B_j}|\hat p(x)|\le \sup_{x\in B_j}|p(x)|+\sup_{x\in B_j}|\hat p(x)-p(x)|<\infty$).
Also, on $A\cap B$, $\sup_{x\in B_j}\|\nabla\hat p(x)\|\le \sup\|\nabla p\|+\sup\|\nabla\hat p-\nabla p\|<\infty$.
Therefore there is a constant $C_j$ such that
\[
\sup_{x\in B_j}\left\|\frac{\nabla\hat p(x)\nabla\hat p(x)^\top}{\hat p(x)^2}
-\frac{\nabla p(x)\nabla p(x)^\top}{p(x)^2}\right\|
\le
C_j\left(
\sup_{x\in B_j}\|\nabla\hat p(x)-\nabla p(x)\|
+
\sup_{x\in B_j}|\hat p(x)-p(x)|
\right).
\]
Combine the two displays and take suprema over $x\in B_j$ to obtain
\[
\sup_{x\in B_j}\|\nabla^2\hat\ell(x)-\nabla^2\ell(x)\|
\le
C_j\left(
\sup_{x\in B_j}\|\nabla^2\hat p(x)-\nabla^2 p(x)\|
+
\sup_{x\in B_j}\|\nabla\hat p(x)-\nabla p(x)\|
+
\sup_{x\in B_j}|\hat p(x)-p(x)|
\right).
\]
On $B$, each supremum is bounded by the corresponding rate with $s=2,1,0$, hence the RHS is bounded by
$C_j\!\left(h^{\beta-2}+\sqrt{\frac{\log n}{n h^{d+4}}}\right)$.

On $A\cap B$ we have, simultaneously for $s=0,1,2$,
\[
\sup_{x\in B_j}\|\nabla^s\hat\ell(x)-\nabla^s\ell(x)\|
\le
C_j\!\left(h^{\beta-s}+\sqrt{\frac{\log n}{n h^{d+2s}}}\right).
\]
Together with \eqref{eq:prob-intersection-log} this proves the lemma.
\end{proof}

\medskip

\begin{proof}[Proof of Lemma~\ref{lem:local-good-prob}]
Let \(A\) be the event in Lemma~\ref{lem:uniform_log_ball}, let \(B\) be the event in
Lemma~\ref{lem:kde-lower-bound}, and let \(C\) be the event that there exists a
deterministic constant \(D_j>0\) such that
\[
\sup_{x\in\overline B(\mu_j,r_j)}
\frac{1}{n h^d}\sum_{i=1}^n
K\!\left(\frac{x-X_i}{h}\right)^2
\le
D_j,
\qquad
\sup_{x\in\overline B(\mu_j,r_j)}
\frac{1}{n h^d}\sum_{i=1}^n
\left\|
\nabla K\!\left(\frac{x-X_i}{h}\right)
\right\|^2
\le
D_j.
\]
The same net-and-Bernstein argument used in the proof of
Lemma~\ref{lem:uniform_kde_ball}, applied to the bounded integrable classes
\[
u\mapsto K\!\left(\frac{x-u}{h}\right)^2,
\qquad
u\mapsto
\left\|
\nabla K\!\left(\frac{x-u}{h}\right)
\right\|^2,
\]
yields \(\Pr(C)\ge 1-n^{-4}\) for all sufficiently large \(n\).

On \(A\cap B\cap C\), every requirement in the definition of
\(\mathcal A_n^{\mathrm{local},j}\) is satisfied. Therefore
\[
A\cap B\cap C\subseteq \mathcal A_n^{\mathrm{local},j}.
\]
Hence, for all sufficiently large \(n\),
\begin{align*}
\Pr\big(\mathcal A_n^{\mathrm{local},j}\big)
&\ge
\Pr(A\cap B\cap C)
\ge
1-\Pr(A^c)-\Pr(B^c)-\Pr(C^c)\\
&\ge
1-2n^{-4}-n^{-4}-n^{-4}\\
&=
1-4n^{-4}.
\end{align*}
This proves the claim.
\end{proof}

\subsubsection{Stabilization and minibatch-control lemmas.}

\begin{proof}[Proof of Lemma~\ref{lem:local-qi-bound}]
Fix $j\in[M]$ and work on $\mathcal A_n^{\mathrm{local},j}$. For every $x\in \overline B(\mu_j,r_j)$ and every $i\in[n]$,
\[
\|g_i(x)\|
=
\frac{1}{h^{d+1}}
\left\|
\nabla K\!\left(\frac{x-X_i}{h}\right)
\right\|
\le G_K h^{-(d+1)}.
\]
Also, by Definition~\ref{def:good-event-local-holder},
\[
\hat p(x)\ge c_j=\frac{p_{\min,j}}{2}\ge \frac{p_{\min}}{2}
\qquad\text{for all }x\in \overline B(\mu_j,r_j).
\]
Hence
\[
\|q_i(x)\|=\frac{\|g_i(x)\|}{\hat p(x)}
\le \frac{2G_K}{p_{\min}}h^{-(d+1)}
= C_*,
\]
uniformly over $x\in \overline B(\mu_j,r_j)$ and $i\in[n]$.
\end{proof}

\medskip

\begin{proof}[Proof of Lemma~\ref{lem:minibatch-field-second}]
For fixed $x\in \overline B(\mu_j,r_j)$, write
\[
\Delta_{p,t}(x):=\hat p_{\mathcal B_t}(x)-\hat p(x),
\qquad
\Delta_{g,t}(x):=\nabla\hat p_{\mathcal B_t}(x)-\nabla\hat p(x).
\]
These are centered finite-population sample means. By the standard variance formula for sampling without replacement,
\[
\EE\!\left[|\Delta_{p,t}(x)|^2\,\middle|\,x,\mathcal X,\mathcal A_n^{\mathrm{local},j}\right]
\le
\frac{n-m}{m(n-1)}\cdot \frac1n\sum_{i=1}^n K_i(x)^2,
\]
\[
\EE\!\left[\|\Delta_{g,t}(x)\|^2\,\middle|\,x,\mathcal X,\mathcal A_n^{\mathrm{local},j}\right]
\le
\frac{n-m}{m(n-1)}\cdot \frac1n\sum_{i=1}^n \|G_i(x)\|^2,
\]
where
\[
K_i(x):=\frac{1}{h^d}K\!\left(\frac{x-X_i}{h}\right),
\qquad
G_i(x):=\frac{1}{h^{d+1}}\nabla K\!\left(\frac{x-X_i}{h}\right).
\]
Now,
\[
\frac1n\sum_{i=1}^n K_i(x)^2
=
\frac{1}{h^d}
\left[
\frac{1}{n h^d}\sum_{i=1}^n
K\!\left(\frac{x-X_i}{h}\right)^2
\right],
\]
and
\[
\frac1n\sum_{i=1}^n \|G_i(x)\|^2
=
\frac{1}{h^{d+2}}
\left[
\frac{1}{n h^d}\sum_{i=1}^n
\left\|
\nabla K\!\left(\frac{x-X_i}{h}\right)
\right\|^2
\right].
\]
By Definition~\ref{def:good-event-local-holder}, both bracketed quantities are bounded by \(D_j\) uniformly over \(x\in \overline B(\mu_j,r_j)\) on \(\mathcal A_n^{\mathrm{local},j}\). Therefore,
\[
\EE\!\left[|\Delta_{p,t}(x)|^2\,\middle|\,x,\mathcal X,\mathcal A_n^{\mathrm{local},j}\right]
\le \frac{D_j}{m h^d},
\qquad
\EE\!\left[\|\Delta_{g,t}(x)\|^2\,\middle|\,x,\mathcal X,\mathcal A_n^{\mathrm{local},j}\right]
\le \frac{D_j}{m h^{d+2}}
\]
after enlarging constants if necessary.

Next define
\[
f(u,v):=\frac{\operatorname{clip}_A(v)}{\max\{u,p_{\mathrm{floor}}\}}.
\]
By Proposition~\ref{prop:local-inactive}, on $\mathcal A_n^{\mathrm{local},j}$ and for all sufficiently large $n$,
\[
f\bigl(\hat p(x),\nabla\hat p(x)\bigr)=\nabla\log\hat p(x)
\qquad\text{for all }x\in \overline B(\mu_j,r_j).
\]
Moreover, exactly as in the proof of Lemma~\ref{lem:sensitivity},
\[
\|f(u,v)-f(u',v')\|
\le
\frac{1}{p_{\mathrm{floor}}}\|v-v'\|
+
\frac{A}{p_{\mathrm{floor}}^2}|u-u'|.
\]
Hence
\begin{align*}
&\left\|\hat s_{A,p_{\mathrm{floor}};\mathcal B_t}(x)-\nabla\log\hat p(x)\right\|^2 \\
&\qquad=
\left\|f\bigl(\hat p_{\mathcal B_t}(x),\nabla\hat p_{\mathcal B_t}(x)\bigr)-f\bigl(\hat p(x),\nabla\hat p(x)\bigr)\right\|^2 \\
&\qquad\le
\frac{2}{p_{\mathrm{floor}}^2}\|\Delta_{g,t}(x)\|^2
+
\frac{2A^2}{p_{\mathrm{floor}}^4}|\Delta_{p,t}(x)|^2.
\end{align*}
Taking conditional expectation and using the previous bounds gives
\[
\EE\!\left[\left\|\hat s_{A,p_{\mathrm{floor}};\mathcal B_t}(x)-\nabla\log\hat p(x)\right\|^2\,\middle|\,x,\mathcal X,\mathcal A_n^{\mathrm{local},j}\right]
\le
\frac{C_{j,\zeta}}{m h^{d+2}}
\]
after enlarging constants, since $h^{-d}\le h^{-(d+2)}$ for all sufficiently large $n$.
Finally, Jensen's inequality gives the mean bound.
\end{proof}

\medskip

\begin{proof}[Proof of Lemma~\ref{lem:minibatch-field-sup}]
For fixed $x$, use the notation from the proof of Lemma~\ref{lem:minibatch-field-second}. As in that proof, for fixed $x$ these are centered finite-population sample means and, on $\mathcal A_n^{\mathrm{local},j}$,
\[
\frac1n\sum_{i=1}^n K_i(x)^2 \lesssim h^{-d},
\qquad
\frac1n\sum_{i=1}^n \|G_i(x)\|^2 \lesssim h^{-d-2},
\]
while the deterministic envelopes are
\[
|K_i(x)|\le K_\infty h^{-d},
\qquad
\|G_i(x)\|\le G_K h^{-(d+1)}.
\]
Bernstein-type concentration inequalities for sampling without replacement therefore yield constants $c,C>0$ such that
\[
\Pr\!\left(
|\Delta_{p,t}(x)|>u
\,\middle|\,
x,\mathcal X,\mathcal A_n^{\mathrm{local},j}
\right)
\le
2\exp\!\left(-\frac{c m u^2}{h^{-d}+h^{-d}u}\right),
\]
and, for each coordinate $r\in[d]$,
\[
\Pr\!\left(
|(\Delta_{g,t}(x))_r|>u
\,\middle|\,
x,\mathcal X,\mathcal A_n^{\mathrm{local},j}
\right)
\le
2\exp\!\left(-\frac{c m u^2}{h^{-d-2}+h^{-d-1}u}\right).
\]
Set
\[
u_n:=C\sqrt{\frac{\log(eTn)}{m h^{d+2}}},
\qquad
v_n:=\frac{u_n}{\sqrt d}.
\]
Because \(m h^{d+2}/\log(eTn)\to\infty\), we have
\(h u_n=C\{h^2\log(eTn)/(m h^{d+2})\}^{1/2}=o(1)\), and similarly \(h v_n=o(1)\). Thus the linear Bernstein terms are lower order. After increasing \(C\) if necessary,
\[
\Pr\!\left(
|\Delta_{p,t}(x)|>u_n
\,\middle|\,
x,\mathcal X,\mathcal A_n^{\mathrm{local},j}
\right)
\le (eTn)^{-7},
\]
and, for each coordinate \(r\in[d]\),
\[
\Pr\!\left(
|(\Delta_{g,t}(x))_r|>v_n
\,\middle|\,
x,\mathcal X,\mathcal A_n^{\mathrm{local},j}
\right)
\le d^{-1}(eTn)^{-7}
\]
for all sufficiently large \(n\). A union bound gives
\[
\Pr\!\left(
\|\Delta_{g,t}(x)\|>u_n
\,\middle|\,
x,\mathcal X,\mathcal A_n^{\mathrm{local},j}
\right)
\le (eTn)^{-7}.
\]
Now define
\[
f(u,v):=\frac{\operatorname{clip}_A(v)}{\max\{u,p_{\mathrm{floor}}\}}.
\]
By Proposition~\ref{prop:local-inactive},
\(
f\bigl(\hat p(x),\nabla\hat p(x)\bigr)=\nabla\log\hat p(x)
\)
on \(\mathcal A_n^{\mathrm{local},j}\) for all sufficiently large \(n\), and
\[
\|f(u,v)-f(u',v')\|
\le
\frac{1}{p_{\mathrm{floor}}}\|v-v'\|
+
\frac{A}{p_{\mathrm{floor}}^2}|u-u'|.
\]
Hence, on the event
\(
\{
|\Delta_{p,t}(x)|\le u_n,
\|\Delta_{g,t}(x)\|\le u_n\},
\)
we have
\[
\left\|
\hat s_{A,p_{\mathrm{floor}};\mathcal B_t}(x)-\nabla\log\hat p(x)
\right\|
\le
\left(\frac{1}{p_{\mathrm{floor}}}+\frac{A}{p_{\mathrm{floor}}^2}\right)u_n.
\]
Absorbing constants into \(C_{j,\mathrm{mb}}\) proves the first display. The second follows by conditioning at each time \(t\) and applying a union bound over \(t=0,\dots,T-1\).
\end{proof}

\medskip

\begin{proof}[Proof of Lemma~\ref{lem:kn-bound-holder}]
Set
\[
r_{n,1}:=
h^{\beta-1}+\sqrt{\frac{\log n}{n h^{d+2}}},
\qquad
r_{n,2}:=
h^{\beta-2}+\sqrt{\frac{\log n}{n h^{d+4}}}.
\]
By Lemma~\ref{lem:uniform_log_ball}, on \(\mathcal A_n^{\mathrm{local},j}\),
\[
\sup_{x\in\overline B(\mu_j,r_j)}
\|\nabla\hat\ell(x)-\nabla\ell(x)\|
\le C_{j,1}r_{n,1},
\qquad
\sup_{x\in\overline B(\mu_j,r_j)}
\|\nabla^2\hat\ell(x)-\nabla^2\ell(x)\|
\le C_{j,2}r_{n,2}
\]
for deterministic constants \(C_{j,1},C_{j,2}>0\).

Fix \(x_t\in \overline B(\mu_j,r_j)\), write \(\delta_t:=x_t-\mu_j\), and recall that \(\nabla\ell(\mu_j)=0\). By the fundamental theorem of calculus,
\[
k_n(x_t)
=
\nabla\hat\ell(\mu_j)
+
\left(\int_0^1 \nabla^2\hat\ell(\mu_j+s\delta_t)\,ds\right)\delta_t.
\]
Adding and subtracting \(\nabla^2\ell\), and using Lemma~\ref{lem:derived-inward-drift}, gives
\[
\langle \delta_t,k_n(x_t)\rangle
\le
-\frac{\alpha_j}{2}\|\delta_t\|^2
+
\|\delta_t\|\,\|\nabla\hat\ell(\mu_j)\|
+
C_{j,2}r_{n,2}\|\delta_t\|^2.
\]
Since \(\nabla\ell(\mu_j)=0\), Lemma~\ref{lem:uniform_log_ball} also gives
\(
\|\nabla\hat\ell(\mu_j)\|
\le C_{j,1}r_{n,1}
\).
Therefore
\[
\langle \delta_t,k_n(x_t)\rangle
\le
-\frac{\alpha_j}{2}\|\delta_t\|^2
+
C_{j,1}r_{n,1}\|\delta_t\|
+
C_{j,2}r_{n,2}\|\delta_t\|^2.
\]
Applying Young's inequality to the middle term and using
\[
r_{n,1}^2\lesssim h^{2(\beta-1)}+\frac{\log n}{n h^{d+2}}
\]
yields the first display.

For the second display, let
\(
H_j:=\sup_{x\in \overline B(\mu_j,r_j)}\|\nabla^2\ell(x)\|
\).
The same expansion gives
\[
\|k_n(x_t)\|
\le
\|\nabla\hat\ell(\mu_j)\|
+
\left\|
\left(\int_0^1 \nabla^2\ell(\mu_j+s\delta_t)\,ds\right)\delta_t
\right\|
+
C_{j,2}r_{n,2}\|\delta_t\|
\le
C_{j,1}r_{n,1}+H_j\|\delta_t\|+C_{j,2}r_{n,2}\|\delta_t\|.
\]
Squaring both sides and using
\[
r_{n,1}^2\lesssim h^{2(\beta-1)}+\frac{\log n}{n h^{d+2}},
\qquad
r_{n,2}^2\lesssim h^{2(\beta-2)}+\frac{\log n}{n h^{d+4}},
\]
gives the second display.
\end{proof}

\medskip

\begin{proof}[Proof of Lemma~\ref{lem:T1-bound-holder}]
Because $\mathcal B_t$ is a uniformly random $m$-subset of $\{1,\dots,n\}$,
\[
\Pr(i\in\mathcal B_t)=\frac{m}{n},
\qquad
\Pr(i,j\in\mathcal B_t)=\frac{m(m-1)}{n(n-1)}\quad(i\ne j).
\]
Writing the minibatch mean with indicators,
\[
\bar q_t=\frac{1}{m}\sum_{i=1}^n q_i(x_t)\mathbf{1}\{i\in\mathcal B_t\},
\]
we obtain
\[
\EE[\bar q_t\mid x_t,\mathcal X,\mathcal A_n^{\mathrm{local},j}]
= \frac{1}{n}\sum_{i=1}^n q_i(x_t)
= \nabla \log \hat p(x_t)=:k_n(x_t).
\]
For the second moment,
\[
\EE[\|\bar q_t\|^2\mid x_t,\mathcal X,\mathcal A_n^{\mathrm{local},j}]
= \frac{1}{m^2}\Bigg(\frac{m}{n}\sum_{i=1}^n\|q_i(x_t)\|^2
+ \frac{m(m-1)}{n(n-1)}\sum_{i\ne j} q_i(x_t)^\top q_j(x_t)\Bigg).
\]
Using
\[
\sum_{i\ne j} q_i^\top q_j = \big\|\sum_{i=1}^n q_i\big\|^2 - \sum_{i=1}^n\|q_i\|^2
= n^2\|k_n(x_t)\|^2-\sum_{i=1}^n\|q_i(x_t)\|^2,
\]
we get 
\begin{equation}\label{eq:finite-population-identity}
    \EE[\|\bar q_t\|^2\mid x_t,\mathcal X,\mathcal A_n^{\mathrm{local},j}]
= \frac{n-m}{m n(n-1)}\sum_{i=1}^n\|q_i(x_t)\|^2
+ \Big(1-\frac{n-m}{m(n-1)}\Big)\|k_n(x_t)\|^2.
\end{equation}
On \(\mathcal A_n^{\mathrm{local},j}\), Definition~\ref{def:good-event-local-holder}
gives
\(
\hat p(x_t)\ge c_j>0,
\)
and there exists a constant \(D_j>0\) such that
\[
\sup_{x\in\overline B(\mu_j,r_j)}
\frac{1}{n h^d}\sum_{i=1}^n
\left\|
\nabla K\!\left(\frac{x-X_i}{h}\right)
\right\|^2
\le
D_j.
\]
Since
\[
q_i(x_t)
=
\frac{g_i(x_t)}{\hat p(x_t)},
\qquad
g_i(x_t):=\frac{1}{h^{d+1}}\nabla K\!\Big(\frac{x_t-X_i}{h}\Big),
\]
it follows that
\begin{align*}
\frac{1}{n}\sum_{i=1}^n \|q_i(x_t)\|^2
&=
\frac{1}{\hat p(x_t)^2}\cdot
\frac{1}{n}\sum_{i=1}^n
\left\|
\frac{1}{h^{d+1}}\nabla K\!\Big(\frac{x_t-X_i}{h}\Big)
\right\|^2 
=
\frac{1}{\hat p(x_t)^2 h^{d+2}}
\left[
\frac{1}{n h^d}\sum_{i=1}^n
\left\|
\nabla K\!\Big(\frac{x_t-X_i}{h}\Big)
\right\|^2
\right] \\
&\le
\frac{D_j}{c_j^2}\,h^{-d-2}.
\end{align*}
Thus
\[
\EE[\|\bar q_t\|^2\mid x_t,\mathcal X,\mathcal A_n^{\mathrm{local},j}]
\lesssim
\frac{n-m}{m(n-1)}\,h^{-d-2}
+
\|k_n(x_t)\|^2.
\]
Combining this with \eqref{eq:finite-population-identity} and Lemma~\ref{lem:kn-bound-holder}
yields the claim.
\end{proof}

\medskip

\begin{proof}[Proof of Lemma~\ref{lem:T3-bound-holder}]
Set $k_n(x_t)=\nabla\log\hat p(x_t)$. Then
\[
r_t
=\bigl(\hat s_{A,p_{\mathrm{floor}};\mathcal B_t}(x_t)-k_n(x_t)\bigr)
-\bigl(\bar q_t-k_n(x_t)\bigr).
\]
Hence, by $(a+b)^2\le 2a^2+2b^2$,
\begin{align*}
&~\EE[\|r_t\|^2\mid x_t,\mathcal X,\mathcal A_n^{\mathrm{local},j}]\\
\le&~
2\EE\!\left[\left\|\hat s_{A,p_{\mathrm{floor}};\mathcal B_t}(x_t)-k_n(x_t)\right\|^2\middle|x_t,\mathcal X,\mathcal A_n^{\mathrm{local},j}\right]
+2\EE[\|\bar q_t-k_n(x_t)\|^2\mid x_t,\mathcal X,\mathcal A_n^{\mathrm{local},j}].
\end{align*}
The first term is $O((m h^{d+2})^{-1})$ by Lemma~\ref{lem:minibatch-field-second}. For the second term, \eqref{eq:finite-population-identity} gives
\[
\EE[\|\bar q_t-k_n(x_t)\|^2\mid x_t,\mathcal X,\mathcal A_n^{\mathrm{local},j}]
=
\frac{n-m}{m n(n-1)}\sum_{i=1}^n\|q_i(x_t)\|^2
\lesssim \frac{1}{m h^{d+2}}.
\]
This proves the claim.
\end{proof}

\medskip

\begin{proof}[Proof of Lemma~\ref{lem:local-perturb-holder}]
We condition on \(\mathcal X\) and on \(\mathcal A_n^{\mathrm{local},j}\) throughout. By construction, the process \((\widetilde x_t^{(j)})_{t=0}^{T-1}\) is adapted to the algorithmic filtration and satisfies
\[
\widetilde x_t^{(j)}\in \overline B(\mu_j,r_j)
\qquad\text{for all }t=0,\dots,T-1.
\]
Define
\[
b_t:=\nabla\log\hat p(\widetilde x_t^{(j)})-\nabla\log p(\widetilde x_t^{(j)}),
\qquad
\zeta_t:=\hat s_{A,p_{\mathrm{floor}};\mathcal B_t}(\widetilde x_t^{(j)})-\nabla\log\hat p(\widetilde x_t^{(j)}).
\]
Then
\[
\hat s_{A,p_{\mathrm{floor}};\mathcal B_t}(\widetilde x_t^{(j)})
-\nabla\log p(\widetilde x_t^{(j)})
+z_t
=
b_t+\zeta_t+z_t.
\]

Since \(\widetilde x_t^{(j)}\in \overline B(\mu_j,r_j)\) for all \(t\), Definition~\ref{def:good-event-local-holder} gives
\[
\max_{0\le t\le T-1}\|b_t\|
\le
C_j\!\left(
h^{\beta-1}
+
\sqrt{\frac{\log n}{n h^{d+2}}}
\right).
\]

Next, Lemma~\ref{lem:minibatch-field-sup} applies to the adapted process \((\widetilde x_t^{(j)})_{t=0}^{T-1}\), yielding
\[
\Pr\!\left(
\max_{0\le t\le T-1}\|\zeta_t\|
\le
C_{j,\mathrm{mb}}\sqrt{\frac{\log(eTn)}{m h^{d+2}}}
\,\middle|\,
\mathcal X,\mathcal A_n^{\mathrm{local},j}
\right)
\ge
1-T(eTn)^{-6}.
\]

For the Gaussian noise, recall that each \(z_t\) is a \(d\)-dimensional Gaussian row with covariance \(\sigma^2 I_d\). Standard Gaussian norm concentration and a union bound over \(t=0,\dots,T-1\) give a deterministic constant \(C_{2,j}>0\) such that
\[
\Pr\!\left(
\max_{0\le t\le T-1}\|z_t\|
\le
C_{2,j}\sigma\sqrt{d\log(eTn)}
\,\middle|\,
\mathcal X,\mathcal A_n^{\mathrm{local},j}
\right)
\ge
1-T(eTn)^{-6}.
\]

By \eqref{eq:def-Delta-corr},
\[
\Delta_{h,\mathrm{corr}}(A,p_{\mathrm{floor}})
=
2\sqrt{2}\left(
\frac{I_1^{1/2}}{p_{\mathrm{floor}}}\,h^{-(d+1)}
+
\frac{A I_0^{1/2}}{p_{\mathrm{floor}}^2}\,h^{-d}
\right).
\]
Since \(A\) and \(p_{\mathrm{floor}}\) are fixed in \(n\) and \(h\to 0\) under Assumption~\ref{assump:bandwidth}, one has \(h\le 1\) for all sufficiently large \(n\), hence \(h^{-d}\le h^{-(d+1)}\). Therefore
\[
\Delta_{h,\mathrm{corr}}(A,p_{\mathrm{floor}})
\lesssim h^{-(d+1)}
\asymp C_*.
\]
Using \eqref{eq:sigma-corr} and the fact that \(\varepsilon_{\mathrm{iter}}\lesssim \varepsilon_{\mathrm{modes}}/\sqrt{T}\), there exists a constant \(C_{3,j}>0\) such that
\[
\sigma
\le
C_{3,j}\,
\frac{C_*\sqrt{T\,\mathrm{polylog}(T,n,\delta)}}{n\varepsilon_{\mathrm{modes}}}
\]
for all sufficiently large \(n\). Therefore
\[
C_{2,j}\sigma\sqrt{d\log(eTn)}
\le
\widetilde C_{2,j}\,
\frac{C_*\sqrt{T d\,\mathrm{polylog}(T,n,\delta)}}{n\varepsilon_{\mathrm{modes}}}
\]
after enlarging constants.

Intersecting the minibatch and Gaussian-noise good events and combining the three bounds above yields
\[
\max_{0\le t\le T-1}\|b_t+\zeta_t+z_t\|
\le
\Xi_{n,T,m,h}^{(j)}
\]
with conditional probability at least \(1-2T(eTn)^{-6}\). This proves the lemma.
\end{proof}
\subsubsection{Initialization lemmas.}

\begin{proof}[Proof of Lemma~\ref{lem:dap-geometry}]
The design gives
\[
\frac{h_{\mathrm{DAP}}}{\rho_{\mathrm{init}}}
\asymp
n^{-1/(d+2\beta)}(\log n)^{1/(d+2\beta)+1/d}\to0,
\qquad
\rho_{\mathrm{init}}\asymp(\log n)^{-1/d}\to0.
\]
Since \(\min_j r_j>0\) and \(\min_{i\ne j}\|\mu_i-\mu_j\|\ge c_0>0\), choose \(n_0\) such that, for all \(n\ge n_0\),
\[
\frac{\sqrt d}{2}h_{\mathrm{DAP}}\le \frac{\rho_{\mathrm{init}}}{8},
\qquad
\rho_{\mathrm{init}}\le \frac12\min_j r_j,
\qquad
\rho_{\mathrm{init}}\le \frac23 c_0.
\]
For each \(j\), the lattice-grid property gives \(z_r\in\mathcal Z_n\) with
\[
\|z_r-\mu_j\|\le \frac{\sqrt d}{2}h_{\mathrm{DAP}}\le \frac{\rho_{\mathrm{init}}}{8}<\frac{\rho_{\mathrm{init}}}{4}.
\]
Thus \(z_r\in\mathcal G_{j,n}\), proving (i). If \(x,a\in\overline B(\mu_j,\rho_{\mathrm{init}}/4)\), then
\[
\|x-a\|
\le
\|x-\mu_j\|+\|a-\mu_j\|
\le
\frac{\rho_{\mathrm{init}}}{4}+\frac{\rho_{\mathrm{init}}}{4}
<\rho_{\mathrm{init}},
\]
which proves (ii). If \(x\in\overline B(\mu_i,\rho_{\mathrm{init}}/4)\), \(a\in\overline B(\mu_j,\rho_{\mathrm{init}}/4)\), and \(i\ne j\), then
\[
\|x-a\|
\ge
\|\mu_i-\mu_j\|-\|x-\mu_i\|-\|a-\mu_j\|
\ge
c_0-\frac{\rho_{\mathrm{init}}}{2}
\ge
\rho_{\mathrm{init}},
\]
which proves (iii).
\end{proof}

\medskip

\begin{proof}[Proof of Lemma~\ref{lem:dap-localmass}]
Fix \(j\in[M]\), and write \(\ell=\log p\). By the local Hessian assumption, there are constants \(0<a_j<A_j<\infty\) such that, for every \(0<t\le r_j\),
\[
\inf_{\|x-\mu_j\|\le t}\ell(x)
\ge
\ell(\mu_j)-a_jt^2,
\qquad
\sup_{t\le \|y-\mu_j\|\le r_j}\ell(y)
\le
\ell(\mu_j)-A_jt^2.
\]
Choose \(t=\rho_{\mathrm{init}}/4\). Since \(p\) is bounded below on \(\overline B(\mu_j,r_j)\), the preceding display implies, after decreasing \(A_j-a_j\) if necessary, that for all \(x\in \overline B(\mu_j,\rho_{\mathrm{init}}/4)\) and all \(y\in \overline B(\mu_j,r_j)\setminus B(\mu_j,\rho_{\mathrm{init}}/4)\),
\[
p(x)-p(y)
\ge
c_j\rho_{\mathrm{init}}^2
\]
for a constant \(c_j>0\).

For \(z_r\in\overline B(\mu_j,\rho_{\mathrm{init}}/4)\), the change of variables \(x=z_r+h_{\mathrm{DAP}}u\) gives
\[
m_r
=
h_{\mathrm{DAP}}^d
\int_{\|u\|\le1}p(z_r+h_{\mathrm{DAP}}u)\,du.
\]
The same identity holds for \(z_s\). Because \(p\in C^2(\mathcal U_j)\) and \(h_{\mathrm{DAP}}/\rho_{\mathrm{init}}\to0\),
\[
\sup_{\substack{z\in\overline B(\mu_j,r_j)\\ \|u\|\le1}}
|p(z+h_{\mathrm{DAP}}u)-p(z)|
\le
C h_{\mathrm{DAP}}
=
o(\rho_{\mathrm{init}}^2).
\]
Hence, for \(z_r\in\mathcal Z_n\cap\overline B(\mu_j,\rho_{\mathrm{init}}/4)\) and
\(z_s\in\mathcal Z_n\cap(\overline B(\mu_j,r_j)\setminus B(\mu_j,\rho_{\mathrm{init}}/4))\),
\[
\begin{aligned}
m_r-m_s
&=
h_{\mathrm{DAP}}^d
\int_{\|u\|\le1}
\{p(z_r+h_{\mathrm{DAP}}u)-p(z_s+h_{\mathrm{DAP}}u)\}\,du       \\
&\ge
h_{\mathrm{DAP}}^d
\int_{\|u\|\le1}
\{c_j\rho_{\mathrm{init}}^2-o(\rho_{\mathrm{init}}^2)\}\,du       \\
&\ge
\Delta_j h_{\mathrm{DAP}}^d\rho_{\mathrm{init}}^2
\end{aligned}
\]
for some \(\Delta_j>0\) and all \(n\ge n_j\). Taking the infimum over \(r\) and the supremum over \(s\) proves the claim.
\end{proof}

\medskip

\begin{proof}[Proof of Lemma~\ref{lem:dap-concentration}]
For each \(r\), \(u_r\) is the average of Bernoulli variables with mean \(m_r\). Since \(p_{\max}<\infty\),
\[
m_r
\le
p_{\max}\operatorname{Vol}\{x:\|x-z_r\|\le h_{\mathrm{DAP}}\}
\le
C h_{\mathrm{DAP}}^d.
\]
Bernstein's inequality gives, for every \(t>0\),
\[
\Pr(|u_r-m_r|>t)
\le
2\exp\left(
-\frac{nt^2}{2m_r+2t/3}
\right)
\le
2\exp\left(
-\frac{nt^2}{C h_{\mathrm{DAP}}^d+t}
\right).
\]
Taking
\[
t=C_1\left\{
\sqrt{\frac{h_{\mathrm{DAP}}^d\log N_{\mathrm{cand}}}{n}}
+
\frac{\log N_{\mathrm{cand}}}{n}
\right\}
\]
with \(C_1\) large enough and union bounding over \(r=1,\dots,N_{\mathrm{cand}}\) yields
\[
\Pr\left(
\max_{1\le r\le N_{\mathrm{cand}}}|u_r-m_r|>t
\right)
\le C N_{\mathrm{cand}}^{-4}.
\]
Since \(N_{\mathrm{cand}}\le Ch_{\mathrm{DAP}}^{-d}\le Cn\), the right side is at most \(Cn^{-4}\). It remains to compare \(t\) with \(h_{\mathrm{DAP}}^d\rho_{\mathrm{init}}^2\). Under the DAP design,
\[
h_{\mathrm{DAP}}^d
\asymp
\left(\frac{\log n}{n}\right)^{d/(d+2\beta)},
\qquad
\rho_{\mathrm{init}}^2
\asymp
(\log n)^{-2/d}.
\]
Therefore
\[
\frac{\sqrt{h_{\mathrm{DAP}}^d\log N_{\mathrm{cand}}/n}}
{h_{\mathrm{DAP}}^d\rho_{\mathrm{init}}^2}
\asymp
n^{-\beta/(d+2\beta)}
(\log n)^{1/2-d/(2(d+2\beta))+2/d}\to0
\]
and
\[
\frac{\log N_{\mathrm{cand}}/n}{h_{\mathrm{DAP}}^d\rho_{\mathrm{init}}^2}
\asymp
n^{-2\beta/(d+2\beta)}
(\log n)^{1-d/(d+2\beta)+2/d}\to0.
\]
Thus, for any fixed \(c_{\mathrm{conc}}>0\), \(t\le c_{\mathrm{conc}}h_{\mathrm{DAP}}^d\rho_{\mathrm{init}}^2\) for all large enough \(n\). This proves the stated bound.
\end{proof}

\medskip

\begin{proof}[Proof of Lemma~\ref{lem:dap-cap}]
Fix \(m\in[M]\). Let \(z_r\in\mathcal C_{m,n}\), and suppose that
\[
z_r\notin
\bigcup_{\substack{j\in[M]\setminus\{m\}:\\ p(\mu_j)\ge p(\mu_m)}}
\overline B(\mu_j,r_j).
\]
Since \(z_r\notin\overline B(\mu_m,r_m)\) by definition of \(\mathcal C_{m,n}\), the point \(z_r\) lies outside the basin neighborhood of every mode whose height is at least \(p(\mu_m)\).

By the compactness separation argument used for the DAP competitive regions, there is a constant \(\eta_m>0\) such that, outside these higher-or-equal modal neighborhoods,
\[
p(z_r)
\le
p(\mu_m)-\eta_m.
\]
On the other hand, for every \(z_s\in\mathcal G_{m,n}\), the definition of \(\mathcal G_{m,n}\) gives
\[
\|z_s-\mu_m\|\le \rho_{\mathrm{init}}/4.
\]
Since \(\rho_{\mathrm{init}}\to0\) and \(p\) is continuous at \(\mu_m\), there is \(n_m\) such that, for all \(n\ge n_m\),
\[
p(z_s)\ge p(\mu_m)-\eta_m/4
\qquad\text{for every }z_s\in\mathcal G_{m,n}.
\]
Using the local-mass expansion from Lemma~\ref{lem:dap-localmass}, there is a constant \(C_m>0\) such that, uniformly over the grid points under consideration,
\[
\left|
m_s-h_{\mathrm{DAP}}^d\operatorname{Vol}(B(0,1))p(z_s)
\right|
\le
C_mh_{\mathrm{DAP}}^{d+1}.
\]
Increase \(n_m\) if needed so that
\[
2C_mh_{\mathrm{DAP}}
\le
\frac{\eta_m}{4}\operatorname{Vol}(B(0,1)).
\]
Then, for every \(z_s\in\mathcal G_{m,n}\),
\[
\begin{aligned}
m_s-m_r
&\ge
h_{\mathrm{DAP}}^d\operatorname{Vol}(B(0,1))
\{p(z_s)-p(z_r)\}
-
2C_mh_{\mathrm{DAP}}^{d+1}                                      \\
&\ge
h_{\mathrm{DAP}}^d\operatorname{Vol}(B(0,1))
\left\{p(\mu_m)-\frac{\eta_m}{4}-(p(\mu_m)-\eta_m)\right\}
-
2C_mh_{\mathrm{DAP}}^{d+1}                                      \\
&\ge
\frac{\eta_m}{2}\operatorname{Vol}(B(0,1))h_{\mathrm{DAP}}^d.
\end{aligned}
\]
Increase \(n_m\) if needed so that
\[
\gamma_0\rho_{\mathrm{init}}^2
<
\frac{\eta_m}{2}\operatorname{Vol}(B(0,1)).
\]
Then
\[
m_r
<
\inf_{s:\,z_s\in\mathcal G_{m,n}}m_s
-
\gamma_0h_{\mathrm{DAP}}^d\rho_{\mathrm{init}}^2,
\]
contradicting \(z_r\in\mathcal C_{m,n}\). Therefore
\[
\mathcal C_{m,n}
\subseteq
\bigcup_{\substack{j\in[M]\setminus\{m\}:\\ p(\mu_j)\ge p(\mu_m)}}
\overline B(\mu_j,r_j).
\]

It remains to cover this set. Since \(\mathcal C_{m,n}\subseteq\mathcal Z_n\subseteq\mathcal Q\), it is enough to cover the fixed public box \(\mathcal Q\). A bounded box in \(\mathbb R^d\) can be covered by at most \(C_Q\rho_{\mathrm{init}}^{-d}\) cubes of side length \(\rho_{\mathrm{init}}/(2\sqrt d)\), and each such cube is contained in a Euclidean ball of radius \(\rho_{\mathrm{init}}/2\). Hence
\[
N\!\left(\mathcal C_{m,n},\|\cdot\|,\rho_{\mathrm{init}}/2\right)
\le
C_Q\rho_{\mathrm{init}}^{-d}.
\]
Since \(\rho_{\mathrm{init}}\asymp(\log n)^{-1/d}\) and \(M\ge1\), there is a constant \(C'>0\) such that
\[
C_Q\rho_{\mathrm{init}}^{-d}
\le
C'M\log n.
\]
Absorbing \(C'\) into \(L_{\mathrm{cap}}\) proves the covering bound.
\end{proof}

\subsubsection{Global-control lemmas.}

\begin{proof}[Proof of Lemma~\ref{lem:global-good-prob}]
For each \(j\in[M]\), let \(A_j:=\mathcal A_n^{\mathrm{local},j}\), and let \(B:=\mathcal E_{\mathrm{init}}\).
By Lemma~\ref{lem:local-good-prob}, for each \(j\in[M]\),
\[
\Pr(A_j^c)\le 4n^{-4}
\]
for all sufficiently large \(n\). By Proposition~\ref{prop:init-dpball}, there exists \(C_{\mathrm{init,cov}}>0\) such that
\[
\Pr(B^c)\le C_{\mathrm{init,cov}}\,n^{-2}
\]
for all sufficiently large \(n\). Therefore, for all sufficiently large \(n\),
\[
\Pr\Big(\bigcap_{j=1}^M A_j \cap B\Big)
\ge
1-\sum_{j=1}^M\Pr(A_j^c)-\Pr(B^c)
\ge
1-4Mn^{-4}-C_{\mathrm{init,cov}}\,n^{-2}.
\]
Hence
\(
\Pr\big(\mathcal A_n^{\mathrm{global}}\big)\ge 1-C_{\mathrm{global,stat}}\,n^{-2}
\)
for a constant \(C_{\mathrm{global,stat}}>0\) and sufficiently large \(n\).
\end{proof}

The next argument isolates the only point needed to pass from \(\widetilde{\mathcal M}\) to \(\widehat{\mathcal M}\): once a merged point is an average of endpoints attached to a single population mode, the pre-merge rate is inherited unchanged.

\begin{proof}[Proof of Lemma~\ref{lem:postprocess-preserves}]
By convexity of \(\|\cdot\|^2\),
\[
\|\hat\mu_j-\mu_j\|^2
=
\left\|
\frac{1}{|C_j|}\sum_{\ell\in C_j}(x_{T,\ell}-\mu_j)
\right\|^2
\le
\frac{1}{|C_j|}\sum_{\ell\in C_j}\|x_{T,\ell}-\mu_j\|^2.
\]
Taking conditional expectation given \((\mathcal X,\mathcal E)\) and using the assumed endpoint bounds for all \(\ell\in C_j\subseteq I_j\) gives
\[
\EE\!\left[
\|\hat\mu_j-\mu_j\|^2
\,\middle|\,
\mathcal X,\mathcal E
\right]
\le
\frac{1}{|C_j|}\sum_{\ell\in C_j}
\EE\!\left[
\|x_{T,\ell}-\mu_j\|^2
\,\middle|\,
\mathcal X,\mathcal E
\right]
\le
R_{n,j}.
\]
\end{proof}

\medskip

\begin{proof}[Proof of Proposition~\ref{prop:merge-radius}]
Assume first that \(\ell,m\in I_j\). Then
\[
\|x_{T,\ell}-x_{T,m}\|
\le
\|x_{T,\ell}-\mu_j\|+\|x_{T,m}-\mu_j\|
\le
\frac{h_{\mathrm{mode}}}{2}
<
h_{\mathrm{mode}}.
\]
So all endpoints from the same basin lie within the merge radius of one another.

If \(\ell\in I_i\), \(m\in I_j\), and \(i\neq j\), then
\[
\|x_{T,\ell}-x_{T,m}\|
\ge
\|\mu_i-\mu_j\|
-
\|x_{T,\ell}-\mu_i\|
-
\|x_{T,m}-\mu_j\|
\ge
\|\mu_i-\mu_j\|-\frac{h_{\mathrm{mode}}}{2}.
\]
Since
\(
\Delta_{\min}:=\min_{i\neq j}\|\mu_i-\mu_j\|>0
\)
is fixed and \(h_{\mathrm{mode}}\to 0\), we have
\[
\Delta_{\min}-\frac{h_{\mathrm{mode}}}{2}>h_{\mathrm{mode}}
\]
for all sufficiently large \(n\). Thus endpoints from different basins are not merged.

If the routine is applied to all \(k\) endpoints, the additional assumption
\[
\min_{m\notin \cup_{j=1}^M I_j}\min_{j\in[M]}
\|x_{T,m}-\mu_j\|
>
\frac{5}{4}h_{\mathrm{mode}}
\]
implies that no endpoint outside \(\cup_{j=1}^M I_j\) can lie within distance \(h_{\mathrm{mode}}\) of any endpoint in \(I_j\).

Therefore, for each \(j\in[M]\), radius merge produces a cluster consisting only of indices from \(I_j\); call that index set \(C_j\). Its merged point is exactly
\[
\hat\mu_j=\frac{1}{|C_j|}\sum_{\ell\in C_j}x_{T,\ell},
\]
with nonempty \(C_j\subseteq I_j\). This proves the proposition.
\end{proof}

\medskip

\begin{proof}[Proof of Proposition~\ref{prop:merge-ward}]
Let
\(
N_{\max}:=\max_{j\in[M]}|I_j|.
\)
If \(A,B\subseteq I_j\) are nonempty subclusters from the same mode, then their centroids satisfy
\[
\|\bar x_A-\bar x_B\|
\le \frac{h_{\mathrm{mode}}}{2},
\]
so their Ward linkage obeys
\[
\Delta_{\mathrm{Ward}}(A,B)
=
\frac{|A||B|}{|A|+|B|}\,\|\bar x_A-\bar x_B\|^2
\le
\frac{N_{\max}h_{\mathrm{mode}}^2}{4}.
\]
If \(A\subseteq I_i\) and \(B\subseteq I_j\) with \(i\neq j\), then
\[
\|\bar x_A-\bar x_B\|
\ge
\|\mu_i-\mu_j\|-\frac{h_{\mathrm{mode}}}{2}
\ge
\Delta_{\min}-\frac{h_{\mathrm{mode}}}{2},
\]
where \(\Delta_{\min}:=\min_{i\neq j}\|\mu_i-\mu_j\|>0\). Since \(\frac{|A||B|}{|A|+|B|}\ge \frac12\) for nonempty clusters,
\[
\Delta_{\mathrm{Ward}}(A,B)
\ge
\frac12\Bigl(\Delta_{\min}-\frac{h_{\mathrm{mode}}}{2}\Bigr)^2.
\]

Because \(N_{\max}\le k\asymp \log n\) and \(h_{\mathrm{mode}}\to 0\), one has \(N_{\max}h_{\mathrm{mode}}^2\to 0\). Hence, for all sufficiently large \(n\),
\[
\frac{N_{\max}h_{\mathrm{mode}}^2}{4}
<
\frac12\Bigl(\Delta_{\min}-\frac{h_{\mathrm{mode}}}{2}\Bigr)^2.
\]
So every within-mode Ward merge has smaller cost than every cross-mode Ward merge. Since the algorithm is stopped at \(M\) clusters, the final partition is mode-pure. If \(C_j\) denotes the final cluster associated with \(\mu_j\), then \(C_j\subseteq I_j\), and the corresponding Ward centroid is
\[
\hat\mu_j=\frac{1}{|C_j|}\sum_{\ell\in C_j}x_{T,\ell}.
\]
This proves the proposition.
\end{proof}

\section{Additional Experimental Results}
\label{app:additional-sim}

The code to reproduce all experimental results can be found at
\url{https://github.com/ArkaB-DS/DP-GRAMS}. All experiments are implemented in Python and run on a machine with an ARM CPU (8 cores, 8 logical processors), 8.6\,GB RAM, and macOS~15.6.1. Unless otherwise noted, reported summaries are averages over 20 runs with standard errors, and runtimes are reported in seconds.

\subsection{Private Mode Estimation}
\label{app:mode-estimation-diagnostics}

This section provides additional diagnostics for the two mode-estimation benchmarks used in the paper: the bivariate 4-modal Gaussian mixture introduced in Section~\ref{sec:gaussian-example} and the bivariate 5-modal \(t\)-mixture studied in Section~\ref{subsec:private-modes}. These results supplement the main-text privacy--utility curves in Figures~\ref{fig:bivariate_gauss_fourpanel} and~\ref{fig:bivariate_t_threepanel} by showing run-to-run variability, sensitivity to the clipping threshold \(C_*\), minibatch size \(m\), and step size \(\eta\), and MSE and runtime summaries across \((n,\varepsilon)\).

\paragraph{Bivariate 4-modal Gaussian mixture.}
These diagnostics assess whether the four-corners Gaussian results in Section~\ref{sec:gaussian-example} are stable across independent private runs and moderate tuning changes.

\begin{figure}[htbp]
\centering
\includegraphics[width=0.95\textwidth,height=0.8\textheight,
  keepaspectratio]{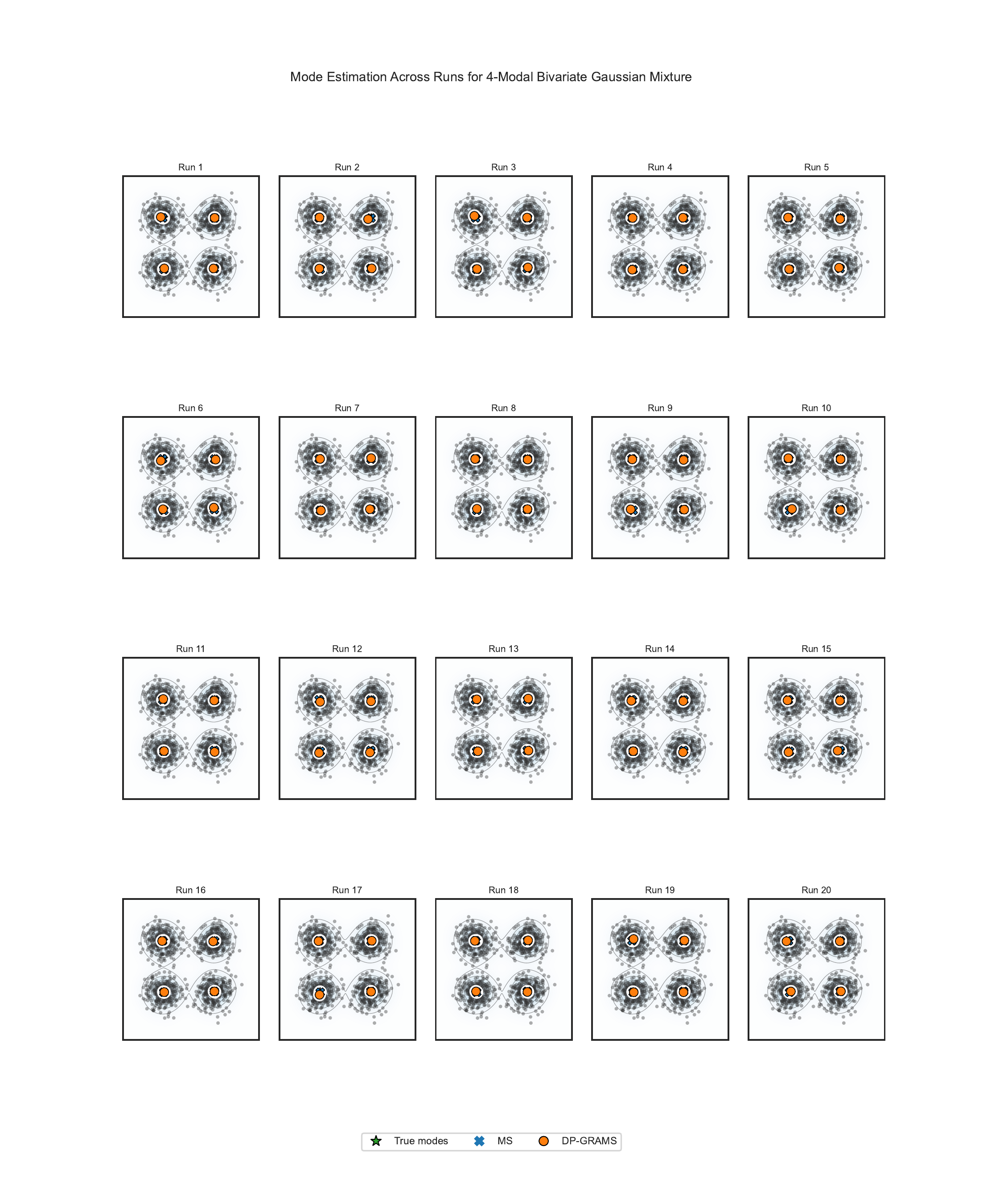}
\caption{Bivariate 4-modal Gaussian mixture: grid of 20 \textsc{DP-GRAMS} runs on one fixed dataset. Each subplot shows KDE contours with true modes (green stars), non-private mean-shift estimates (blue crosses), and \textsc{DP-GRAMS} estimates (orange circles). The grid visualizes run-to-run variability from private DAP initialization and injected ascent noise.}
\label{fig:bivariate_4mix_grid}
\end{figure}

\begin{figure}[!htb]
    \centering
    \begin{subfigure}{0.48\linewidth}
        \centering
        \includegraphics[width=\linewidth]{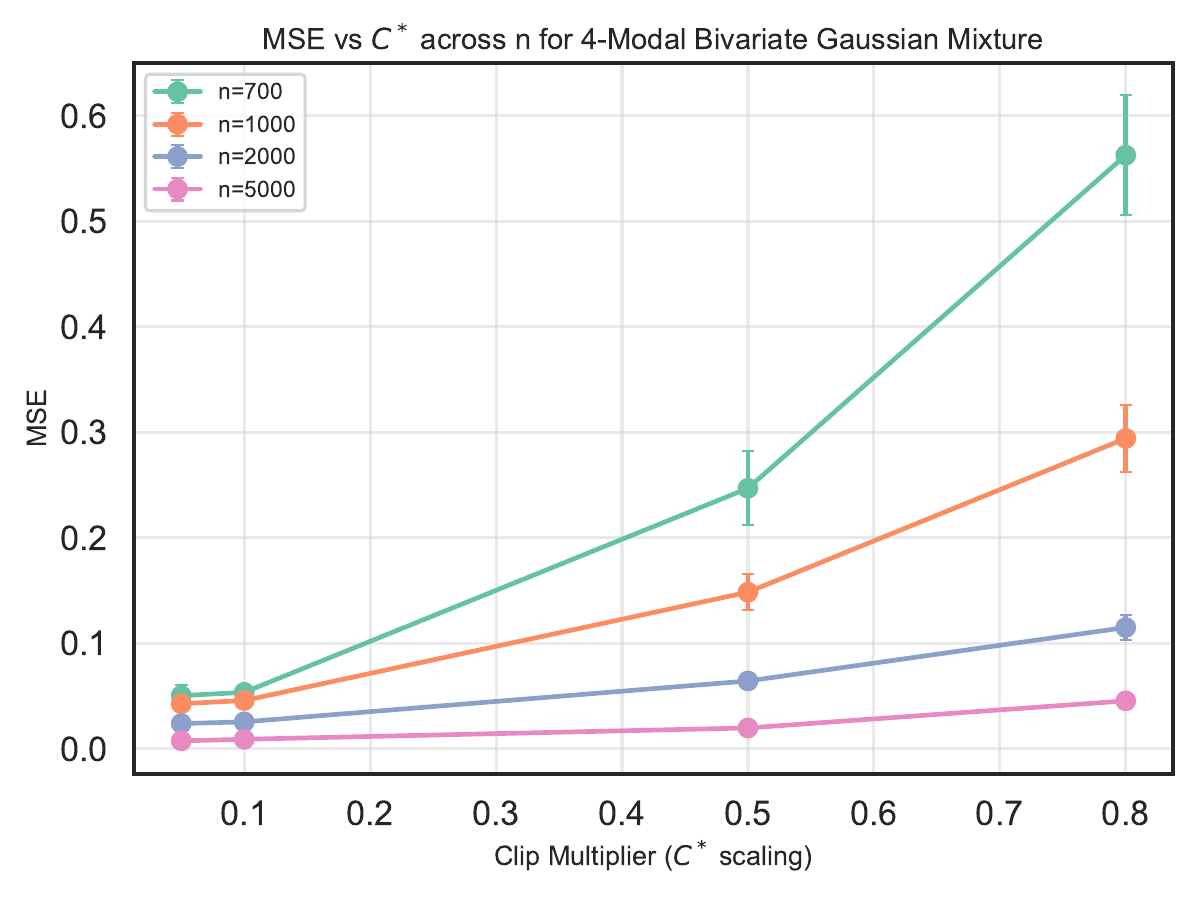}
        \caption{MSE vs.\ clipping multiplier \texttt{clip\_multiplier}.}
    \end{subfigure}\hfill
    \begin{subfigure}{0.48\linewidth}
        \centering
        \includegraphics[width=\linewidth]{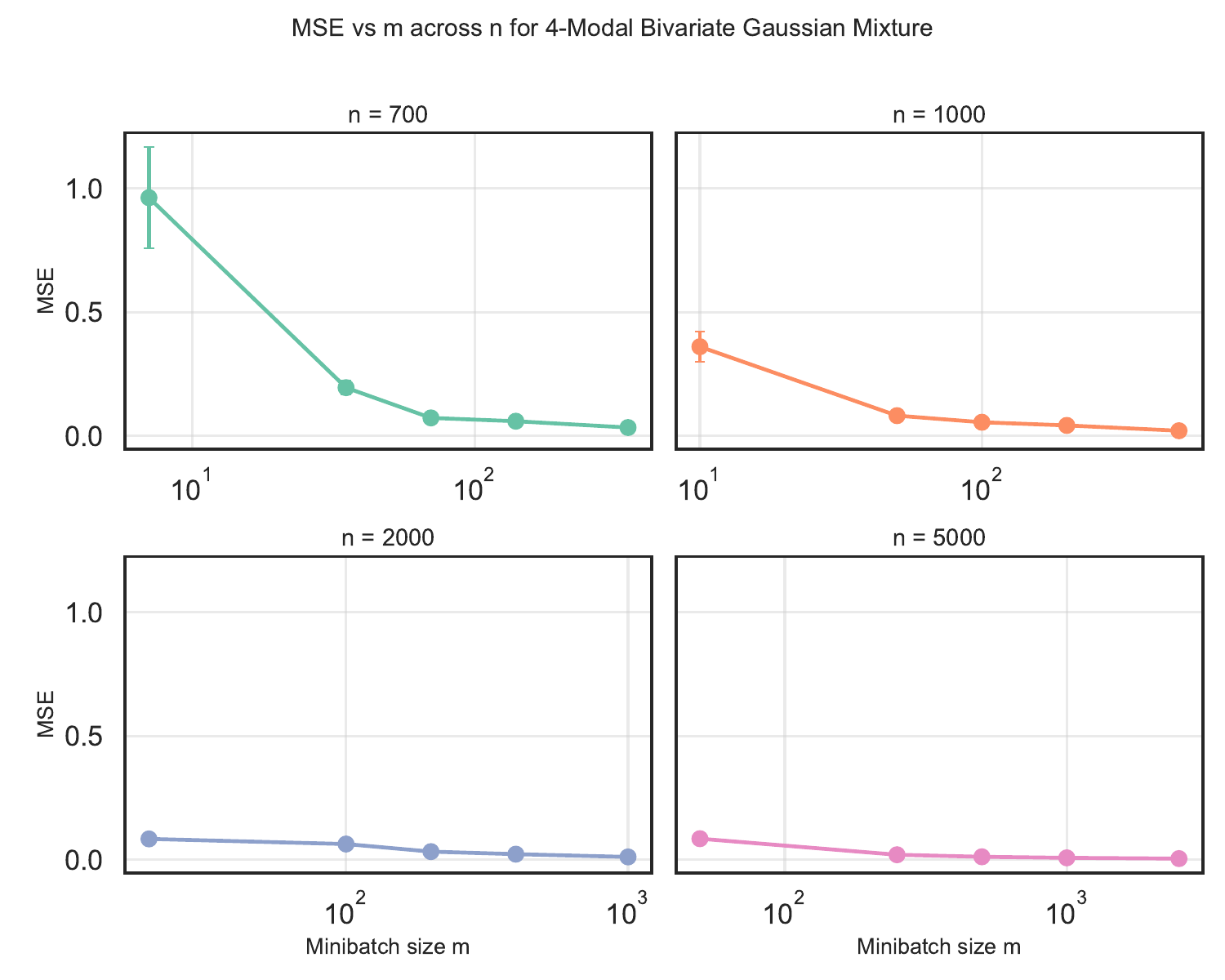}
        \caption{MSE vs.\ minibatch size $m$.}
    \end{subfigure}
    \caption{\textbf{Hyperparameter sensitivity for \textsc{DP-GRAMS} on the 4-modal Gaussian mixture.}
    (a) Effect of clipping multiplier \texttt{clip\_multiplier} on MSE for \(n\in\{700,1000,2000,5000\}\) at fixed \(\varepsilon=1\).
    (b) Effect of minibatch size \(m\) on MSE across the same sample sizes.
    The MSE is relatively stable around the default choices for both tuning parameters.}
    \label{fig:bivariate_4mix_hparam}
\end{figure}

\begin{figure}[!htb]
    \centering
    \includegraphics[width=0.75\linewidth]{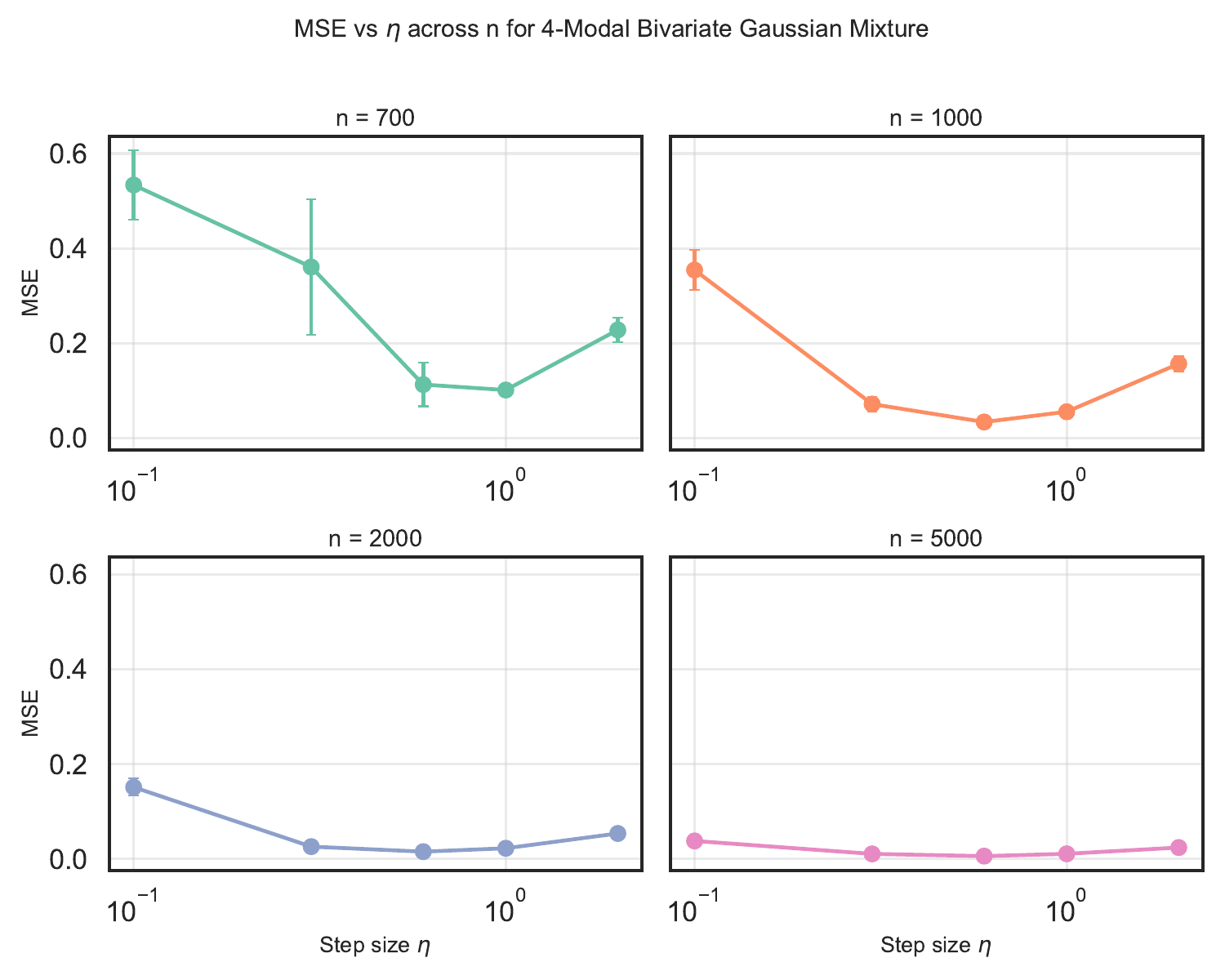}
    \caption{\textbf{Step-size sensitivity for \textsc{DP-GRAMS} on the 4-modal Gaussian mixture.}
    The figure reports MSE versus step size \(\eta\) across \(n\in\{700,1000,2000,5000\}\) at fixed \(\varepsilon=1\). The sweep does not show sharp degradation near the selected default.}
    \label{fig:bivariate_4mix_eta}
\end{figure}

\begin{table}[!htb]
\centering
\caption{Bivariate 4-modal Gaussian mixture: MSE and runtime (mean \(\pm\) SE) for \textsc{DP-GRAMS} and non-private mean shift across \((n,\varepsilon)\).}
\label{tab:bivariate4mix-new}
{\small \resizebox{\textwidth}{!}{%
\begin{tabular}{c c c c c c}
\toprule
\multirow{2}{*}{$n$} & \multirow{2}{*}{$\varepsilon$} &
\multicolumn{2}{c}{MSE} & \multicolumn{2}{c}{Runtime (s)} \\
\cmidrule(lr){3-4} \cmidrule(lr){5-6}
& & DP-GRAMS & MS & DP-GRAMS & MS \\
\midrule
\multirow{5}{*}{700}
 & 0.1  & $0.7375 \pm 0.0978$ & $0.00851 \pm 0.00000$ & $0.00250 \pm 0.00002$ & $0.09364 \pm 0.00014$ \\
 & 0.25 & $0.2166 \pm 0.0233$ & $0.00851 \pm 0.00000$ & $0.00253 \pm 0.00001$ & $0.09364 \pm 0.00014$ \\
 & 0.5  & $0.0903 \pm 0.0201$ & $0.00851 \pm 0.00000$ & $0.00251 \pm 0.00002$ & $0.09364 \pm 0.00014$ \\
 & 1.0  & $0.0590 \pm 0.0070$ & $0.00851 \pm 0.00000$ & $0.00248 \pm 0.00001$ & $0.09364 \pm 0.00014$ \\
 & 5.0  & $0.0551 \pm 0.0072$ & $0.00851 \pm 0.00000$ & $0.00256 \pm 0.00003$ & $0.09364 \pm 0.00014$ \\
\midrule
\multirow{5}{*}{1000}
 & 0.1  & $0.4724 \pm 0.0567$ & $0.00492 \pm 0.00000$ & $0.00331 \pm 0.00005$ & $0.17251 \pm 0.00050$ \\
 & 0.25 & $0.1232 \pm 0.0171$ & $0.00492 \pm 0.00000$ & $0.00306 \pm 0.00000$ & $0.17251 \pm 0.00050$ \\
 & 0.5  & $0.0596 \pm 0.0087$ & $0.00492 \pm 0.00000$ & $0.00308 \pm 0.00003$ & $0.17251 \pm 0.00050$ \\
 & 1.0  & $0.0463 \pm 0.0070$ & $0.00492 \pm 0.00000$ & $0.00292 \pm 0.00001$ & $0.17251 \pm 0.00050$ \\
 & 5.0  & $0.0439 \pm 0.0046$ & $0.00492 \pm 0.00000$ & $0.00296 \pm 0.00002$ & $0.17251 \pm 0.00050$ \\
\midrule
\multirow{5}{*}{2000}
 & 0.1  & $0.1887 \pm 0.0210$ & $0.00421 \pm 0.00000$ & $0.00482 \pm 0.00006$ & $0.69787 \pm 0.00188$ \\
 & 0.25 & $0.0539 \pm 0.0062$ & $0.00421 \pm 0.00000$ & $0.00473 \pm 0.00001$ & $0.69787 \pm 0.00188$ \\
 & 0.5  & $0.0307 \pm 0.0028$ & $0.00421 \pm 0.00000$ & $0.00472 \pm 0.00001$ & $0.69787 \pm 0.00188$ \\
 & 1.0  & $0.0222 \pm 0.0019$ & $0.00421 \pm 0.00000$ & $0.00470 \pm 0.00000$ & $0.69787 \pm 0.00188$ \\
 & 5.0  & $0.0207 \pm 0.0023$ & $0.00421 \pm 0.00000$ & $0.00473 \pm 0.00001$ & $0.69787 \pm 0.00188$ \\
\midrule
\multirow{5}{*}{5000}
 & 0.1  & $0.0588 \pm 0.0062$ & $0.00194 \pm 0.00000$ & $0.01137 \pm 0.00022$ & $6.09457 \pm 0.12074$ \\
 & 0.25 & $0.0184 \pm 0.0023$ & $0.00194 \pm 0.00000$ & $0.01107 \pm 0.00001$ & $6.09457 \pm 0.12074$ \\
 & 0.5  & $0.0121 \pm 0.0010$ & $0.00194 \pm 0.00000$ & $0.01144 \pm 0.00027$ & $6.09457 \pm 0.12074$ \\
 & 1.0  & $0.0107 \pm 0.0015$ & $0.00194 \pm 0.00000$ & $0.01110 \pm 0.00001$ & $6.09457 \pm 0.12074$ \\
 & 5.0  & $0.0090 \pm 0.0015$ & $0.00194 \pm 0.00000$ & $0.01105 \pm 0.00001$ & $6.09457 \pm 0.12074$ \\
\bottomrule
\end{tabular}}}
\end{table}

Figures~\ref{fig:bivariate_4mix_grid}--\ref{fig:bivariate_4mix_eta} show that, in the Gaussian benchmark, the private estimates remain concentrated near the four modal basins and do not exhibit sharp degradation near the selected clipping, minibatch, or step-size defaults. Table~\ref{tab:bivariate4mix-new} shows that most of the MSE reduction occurs when moving from the tightest privacy budget to moderate \(\varepsilon\), with additional improvement as \(n\) increases. The reported runtimes are small in this implementation because \textsc{DP-GRAMS} follows a fixed number of private starts rather than running mean shift from every data point. These diagnostics support the main privacy--utility trends shown in Figure~\ref{fig:bivariate_gauss_fourpanel}.

\paragraph{Bivariate 5-modal \(t\)-mixture.}

We repeat the same diagnostics for the heavier-tailed, heterogeneous mixture studied in Section~\ref{subsec:private-modes}.

\begin{figure}[!htb] 
\centering 
\includegraphics[width=0.95\textwidth,height=0.85\textheight,
  keepaspectratio]{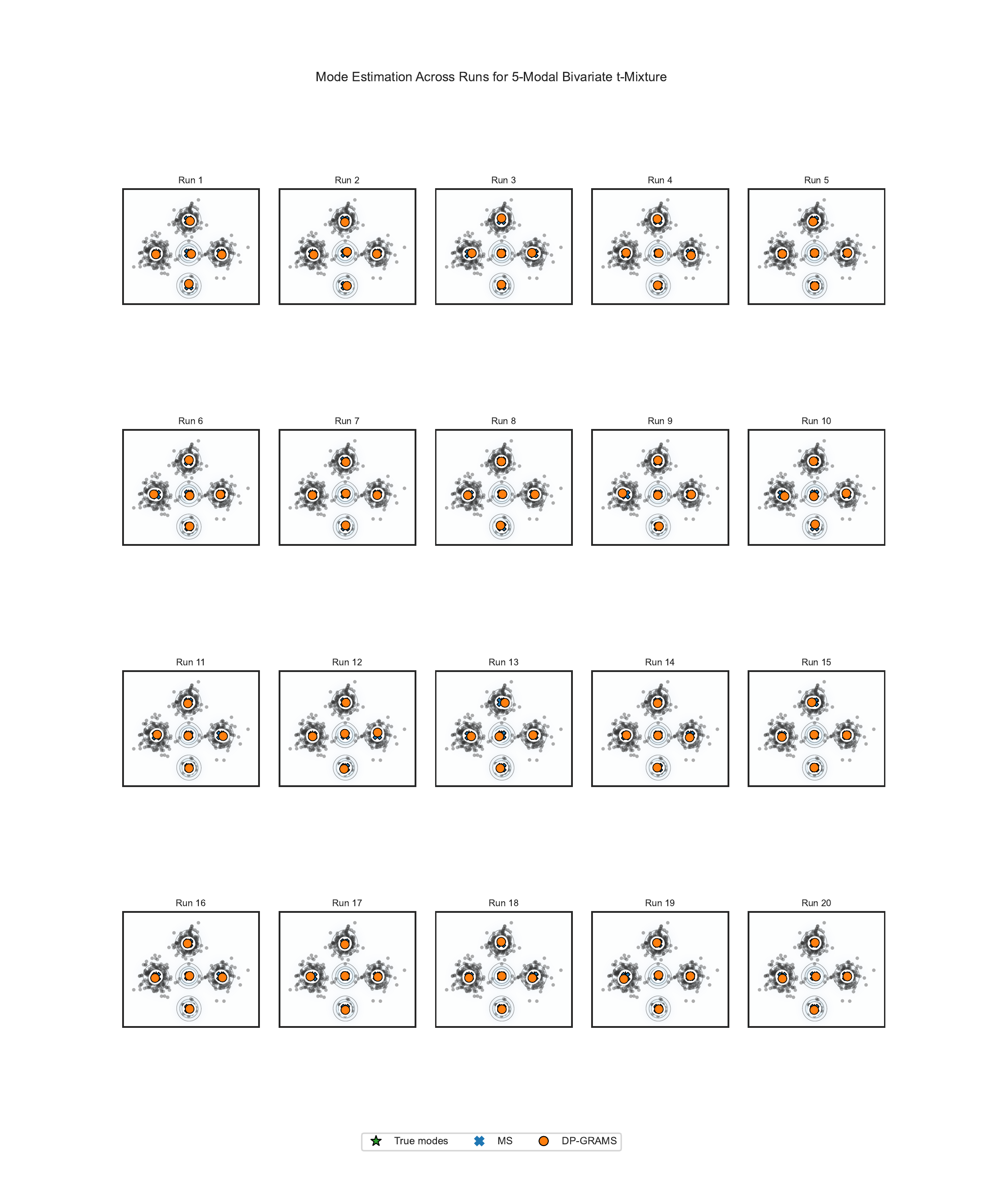}
\caption{Bivariate 5-modal \(t\)-mixture: grid of 20 \textsc{DP-GRAMS} runs on one fixed dataset. Each subplot shows KDE contours with true modes (green stars), non-private mean-shift estimates (blue crosses), and \textsc{DP-GRAMS} estimates (orange circles). The grid visualizes run-to-run variability from private DAP initialization and injected ascent noise in the heavier-tailed setting.}
\label{fig:bivariate_5mix_grid} 
\end{figure}

\begin{figure}[!htb]
    \centering
    \begin{subfigure}{0.48\linewidth}
        \centering
        \includegraphics[width=\linewidth]{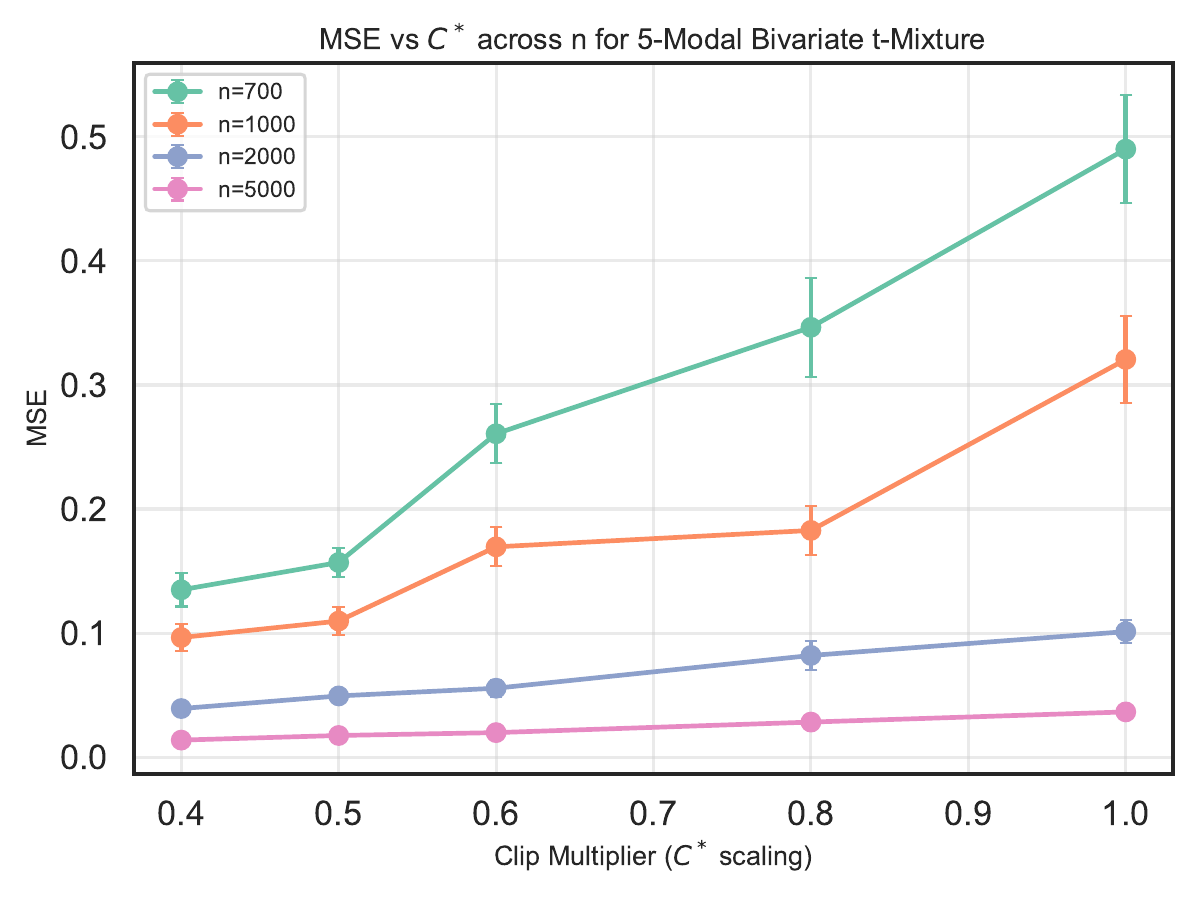}
        \caption{MSE vs.\ clipping multiplier \texttt{clip\_multiplier}.}
    \end{subfigure}\hfill
    \begin{subfigure}{0.48\linewidth}
        \centering
        \includegraphics[width=\linewidth]{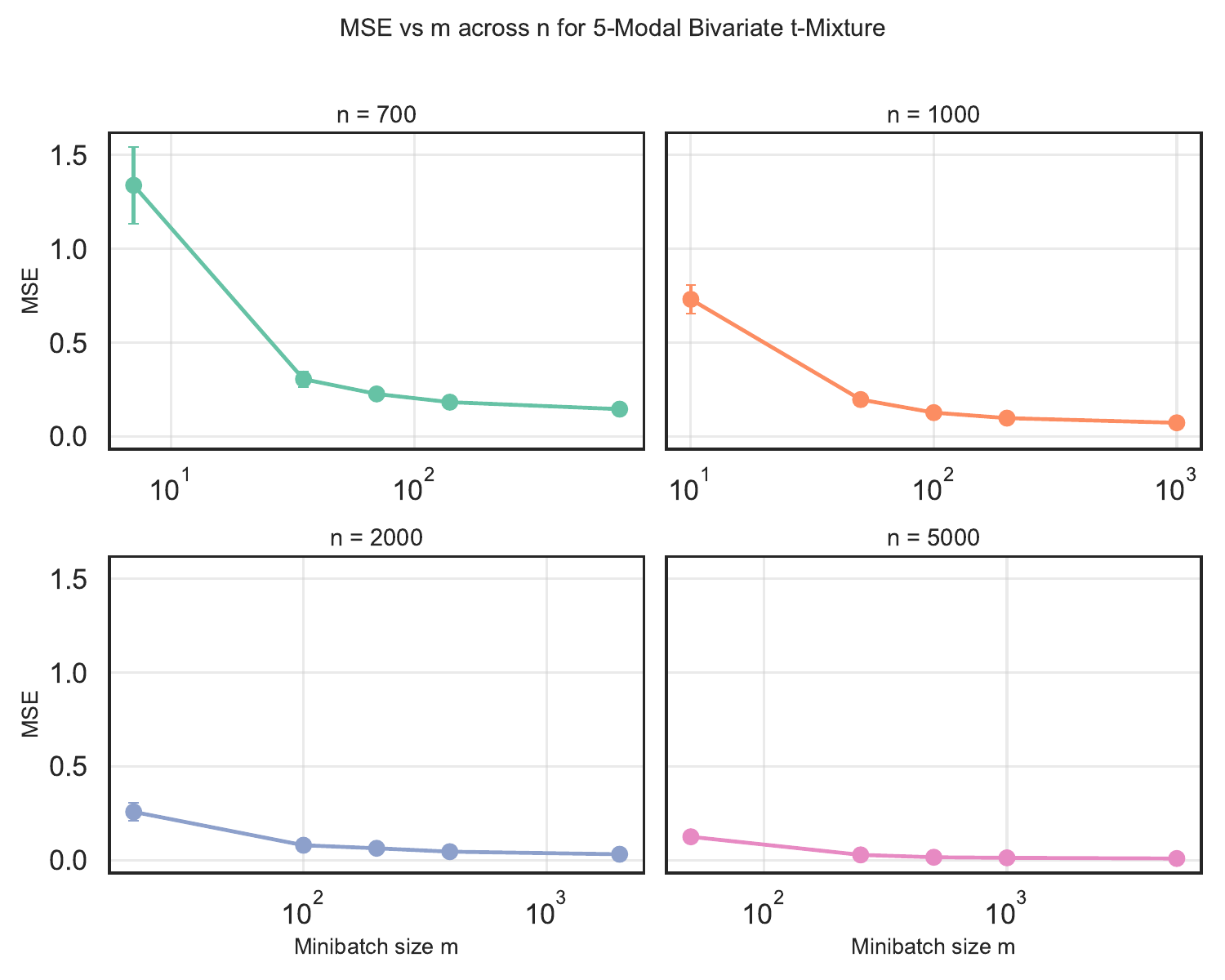}
        \caption{MSE vs.\ minibatch size $m$.}
    \end{subfigure}
        \caption{\textbf{Hyperparameter sensitivity for \textsc{DP-GRAMS} on the 5-modal \(t\)-mixture.}
    (a) Effect of clipping multiplier \texttt{clip\_multiplier} on MSE for \(n\in\{700,1000,2000,5000\}\) at fixed \(\varepsilon=1\).
    (b) Effect of minibatch size \(m\) on MSE across the same sample sizes.
    The sweeps do not show sharp degradation near the selected defaults, despite the heavier tails and heterogeneous component scales.}
    \label{fig:bivariate_5mix_hparam}
\end{figure}

\begin{figure}[!htb]
    \centering
    \includegraphics[width=0.75\linewidth]{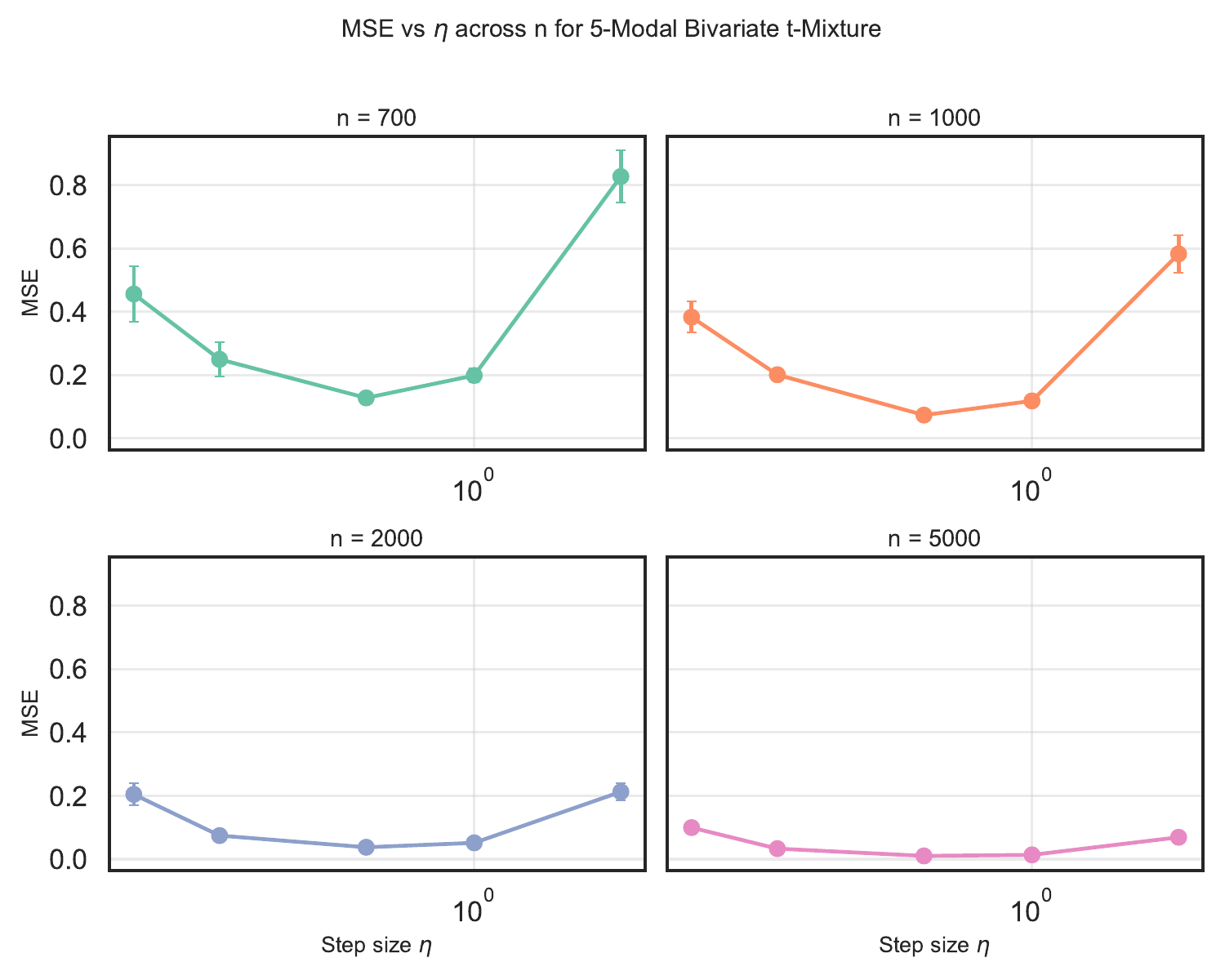}
    \caption{\textbf{Step-size sensitivity for \textsc{DP-GRAMS} on the 5-modal \(t\)-mixture.}
    The figure reports MSE versus step size \(\eta\) across \(n\in\{700,1000,2000,5000\}\) at fixed \(\varepsilon=1\). The sweep does not show sharp degradation near the selected default, although the heavier-tailed setting exhibits some variability across sample sizes.}
    \label{fig:bivariate_5mix_eta}
\end{figure}

\begin{table}[!htb]
\centering
{\small \caption{Bivariate 5-modal \(t\)-mixture: MSE and runtime (mean \(\pm\) SE) for \textsc{DP-GRAMS} and non-private mean shift across \((n,\varepsilon)\).}
\label{tab:bivariate5mix-t}
\resizebox{\textwidth}{!}{%
\begin{tabular}{c c c c c c}
\toprule
\multirow{2}{*}{$n$} & \multirow{2}{*}{$\varepsilon$} &
\multicolumn{2}{c}{MSE} & \multicolumn{2}{c}{Runtime (s)} \\
\cmidrule(lr){3-4} \cmidrule(lr){5-6}
& & DP-GRAMS & MS & DP-GRAMS & MS \\
\midrule
\multirow{5}{*}{700}
 & 0.1  & $2.8695 \pm 0.2815$ & $0.00450 \pm 0.00000$ & $0.00539 \pm 0.00002$ & $0.09284 \pm 0.00010$ \\
 & 0.25 & $1.4499 \pm 0.2419$ & $0.00450 \pm 0.00000$ & $0.00533 \pm 0.00000$ & $0.09284 \pm 0.00010$ \\
 & 0.5  & $0.6244 \pm 0.0723$ & $0.00450 \pm 0.00000$ & $0.00533 \pm 0.00000$ & $0.09284 \pm 0.00010$ \\
 & 1.0  & $0.2226 \pm 0.0191$ & $0.00450 \pm 0.00000$ & $0.00532 \pm 0.00000$ & $0.09284 \pm 0.00010$ \\
 & 5.0  & $0.0502 \pm 0.0053$ & $0.00450 \pm 0.00000$ & $0.00535 \pm 0.00001$ & $0.09284 \pm 0.00010$ \\
\midrule
\multirow{5}{*}{1000}
 & 0.1  & $3.0136 \pm 0.8413$ & $0.01001 \pm 0.00000$ & $0.00654 \pm 0.00002$ & $0.16933 \pm 0.00033$ \\
 & 0.25 & $1.0509 \pm 0.0954$ & $0.01001 \pm 0.00000$ & $0.00649 \pm 0.00001$ & $0.16933 \pm 0.00033$ \\
 & 0.5  & $0.3243 \pm 0.0381$ & $0.01001 \pm 0.00000$ & $0.00647 \pm 0.00001$ & $0.16933 \pm 0.00033$ \\
 & 1.0  & $0.1272 \pm 0.0133$ & $0.01001 \pm 0.00000$ & $0.00648 \pm 0.00001$ & $0.16933 \pm 0.00033$ \\
 & 5.0  & $0.0530 \pm 0.0046$ & $0.01001 \pm 0.00000$ & $0.00659 \pm 0.00004$ & $0.16933 \pm 0.00033$ \\
\midrule
\multirow{5}{*}{2000}
 & 0.1  & $1.1756 \pm 0.1011$ & $0.00331 \pm 0.00000$ & $0.01094 \pm 0.00008$ & $0.68341 \pm 0.00030$ \\
 & 0.25 & $0.3725 \pm 0.0363$ & $0.00331 \pm 0.00000$ & $0.01083 \pm 0.00001$ & $0.68341 \pm 0.00030$ \\
 & 0.5  & $0.1076 \pm 0.0119$ & $0.00331 \pm 0.00000$ & $0.01079 \pm 0.00001$ & $0.68341 \pm 0.00030$ \\
 & 1.0  & $0.0405 \pm 0.0020$ & $0.00331 \pm 0.00000$ & $0.01079 \pm 0.00001$ & $0.68341 \pm 0.00030$ \\
 & 5.0  & $0.0180 \pm 0.0019$ & $0.00331 \pm 0.00000$ & $0.01080 \pm 0.00001$ & $0.68341 \pm 0.00030$ \\
\midrule
\multirow{5}{*}{5000}
 & 0.1  & $0.6145 \pm 0.0600$ & $0.00293 \pm 0.00000$ & $0.02399 \pm 0.00017$ & $4.49957 \pm 0.06095$ \\
 & 0.25 & $0.1063 \pm 0.0082$ & $0.00293 \pm 0.00000$ & $0.02373 \pm 0.00002$ & $4.49957 \pm 0.06095$ \\
 & 0.5  & $0.0310 \pm 0.0032$ & $0.00293 \pm 0.00000$ & $0.02372 \pm 0.00001$ & $4.49957 \pm 0.06095$ \\
 & 1.0  & $0.0215 \pm 0.0024$ & $0.00293 \pm 0.00000$ & $0.02371 \pm 0.00001$ & $4.49957 \pm 0.06095$ \\
 & 5.0  & $0.0125 \pm 0.0015$ & $0.00293 \pm 0.00000$ & $0.02378 \pm 0.00004$ & $4.49957 \pm 0.06095$ \\
\bottomrule
\end{tabular}}}
\end{table}

Figures~\ref{fig:bivariate_5mix_grid}--\ref{fig:bivariate_5mix_eta} and Table~\ref{tab:bivariate5mix-t} show the same qualitative pattern as the main \(t\)-mixture experiment: errors are large at the tightest privacy budgets, decrease sharply as \(\varepsilon\) increases, and continue to improve with \(n\). The gap to the mean-shift baseline remains larger than in the Gaussian benchmark, reflecting the heavier tails and heterogeneous component scales. Across both synthetic benchmarks, the appendix diagnostics support the main-text privacy--utility trends and show that the selected clipping, minibatch, and step-size defaults are not isolated tuning choices.

\subsection{Private Modal Regression}
\label{app:dp-pms-more}

This subsection provides additional diagnostics for the private modal-regression experiments in Section~\ref{subsec:private-modal-regression}. We first report a complementary three-component piecewise-constant design, then give expanded privacy--utility tables and hyperparameter sweeps for both this design and the sinusoidal two-component mixture from the main text. Unless otherwise stated, privacy is calibrated as in Section~\ref{subsec:private-modal-regression}, and regression error is the oracle modal MSE in \eqref{eq:modal-reg-mse}. Runtime summaries for PMS and \textsc{DP-PMS} are aggregated over the same 20 runs.

\paragraph{3-component mixture.}

This example is a three-component mixture with piecewise-constant conditional modes. For sample size \(n\), let
\[
n_1=\lfloor n/3\rfloor,\qquad
n_2=\lfloor n/3\rfloor,\qquad
n_3=n-n_1-n_2.
\]
We generate independent samples
\[
X_{1,i}\sim \mathrm{Uniform}(0,0.5),\quad
Y_{1,i}\sim \mathcal N(3,\sigma^2),\qquad i=1,\dots,n_1,
\]
\[
X_{2,i}\sim \mathrm{Uniform}(0.4,0.7),\quad
Y_{2,i}\sim \mathcal N(2,\sigma^2),\qquad i=1,\dots,n_2,
\]
\[
X_{3,i}\sim \mathrm{Uniform}(0.6,1),\quad
Y_{3,i}\sim \mathcal N(1,\sigma^2),\qquad i=1,\dots,n_3,
\]
with \(\sigma=0.2\). Because the predictor intervals overlap, the population conditional mode set is piecewise:
\[
\mathcal M(x)=
\begin{cases}
\{3\}, & x<0.4,\\
\{3,2\}, & 0.4\le x\le 0.5,\\
\{2\}, & 0.5<x<0.6,\\
\{2,1\}, & 0.6\le x\le 0.7,\\
\{1\}, & x>0.7.
\end{cases}
\]

For this support-limited design, we use a binned conditional-DAP initializer. This is a variant of Algorithm~\ref{alg:dp-pms}: instead of selecting sparse predictor locations directly from the full public design, we partition the fixed predictor domain into bins and, within each nonempty bin, sample a private \(Y\)-anchor from a public \(Y\)-grid using an exponential-mechanism score based on nearby responses. The resulting \((x,y)\) anchors are passed to \textsc{DP-PMS} as initialization points. After private ascent, a fixed-radius cleanup is applied using only the already private \textsc{DP-PMS} outputs, so this cleanup is post-processing.

\begin{figure}[!htb]
    \centering
    \includegraphics[height=0.24\linewidth]{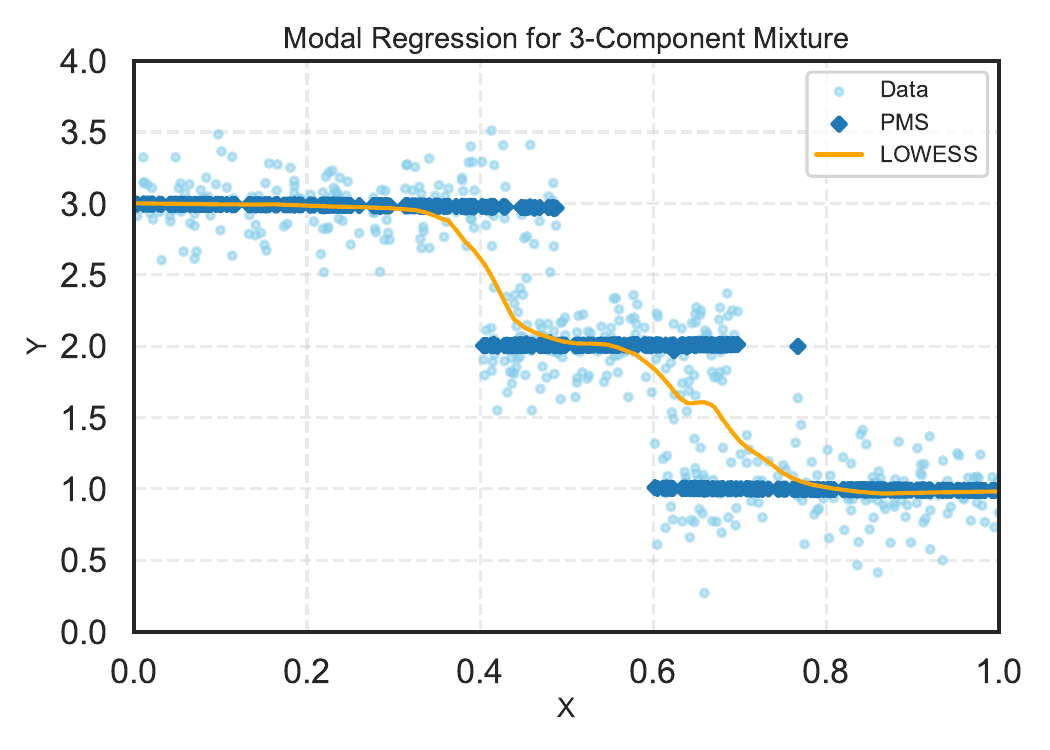}
    \includegraphics[height=0.24\linewidth]{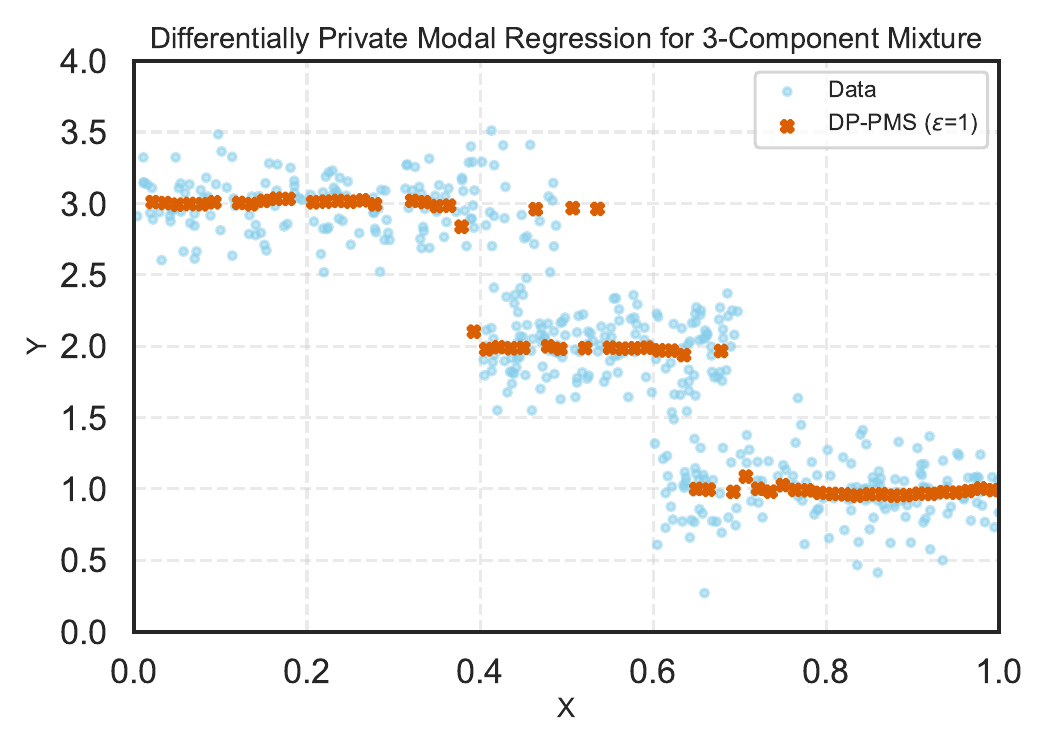}
    \includegraphics[height=0.24\linewidth]{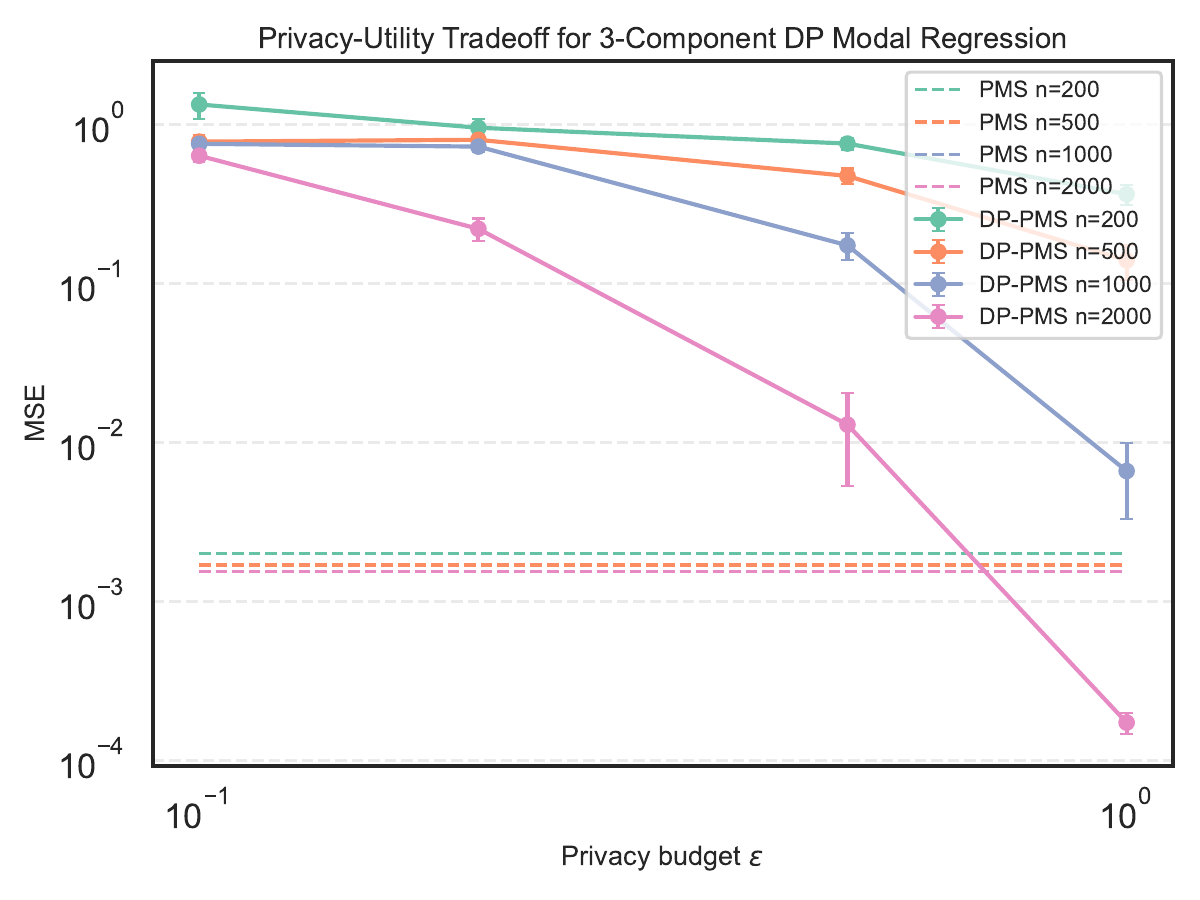}
        \caption{\small Private modal regression on three-component piecewise-constant mixture data.
    Panels (a) and (b) use one representative dataset with \(n=500\); panel (b) uses \((\varepsilon,\delta)=(1,10^{-5})\).
    (a) PMS tracks the modal branches, whereas LOWESS smooths across them.
    (b) \textsc{DP-PMS} recovers the piecewise branch structure under privacy.
    (c) Privacy--utility tradeoff: oracle MSE in \eqref{eq:modal-reg-mse} versus \(\varepsilon\) on a log scale for \(n\in\{200,500,1000,2000\}\).}
    \label{fig:modal_reg_threepanel_appendix}
\end{figure}

\begin{figure}[!htb]
    \centering
    \begin{subfigure}{0.48\linewidth}
        \centering
        \includegraphics[width=\linewidth]{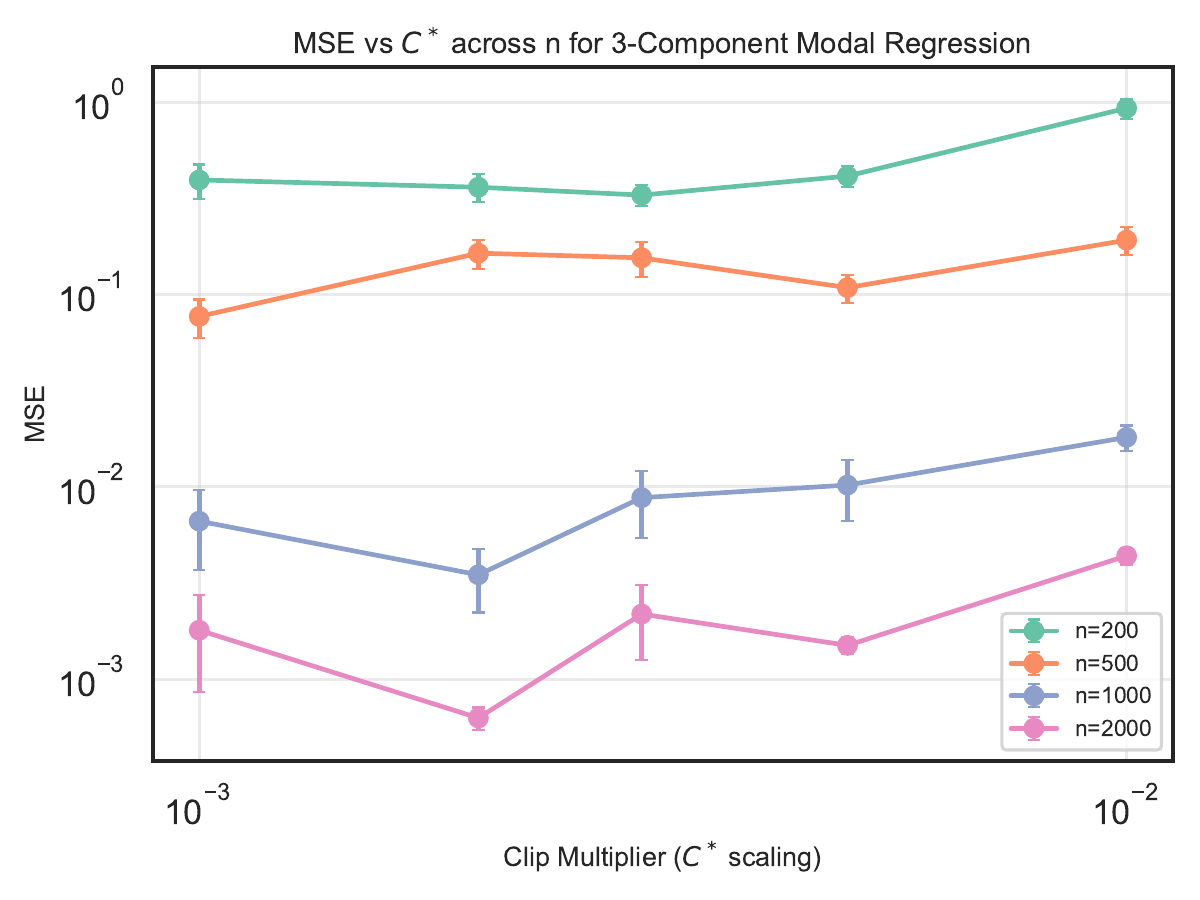}
        \caption{MSE vs.\ clipping multiplier \texttt{clip\_multiplier} across $n$.}
    \end{subfigure}\hfill
    \begin{subfigure}{0.48\linewidth}
        \centering
        \includegraphics[width=\linewidth]{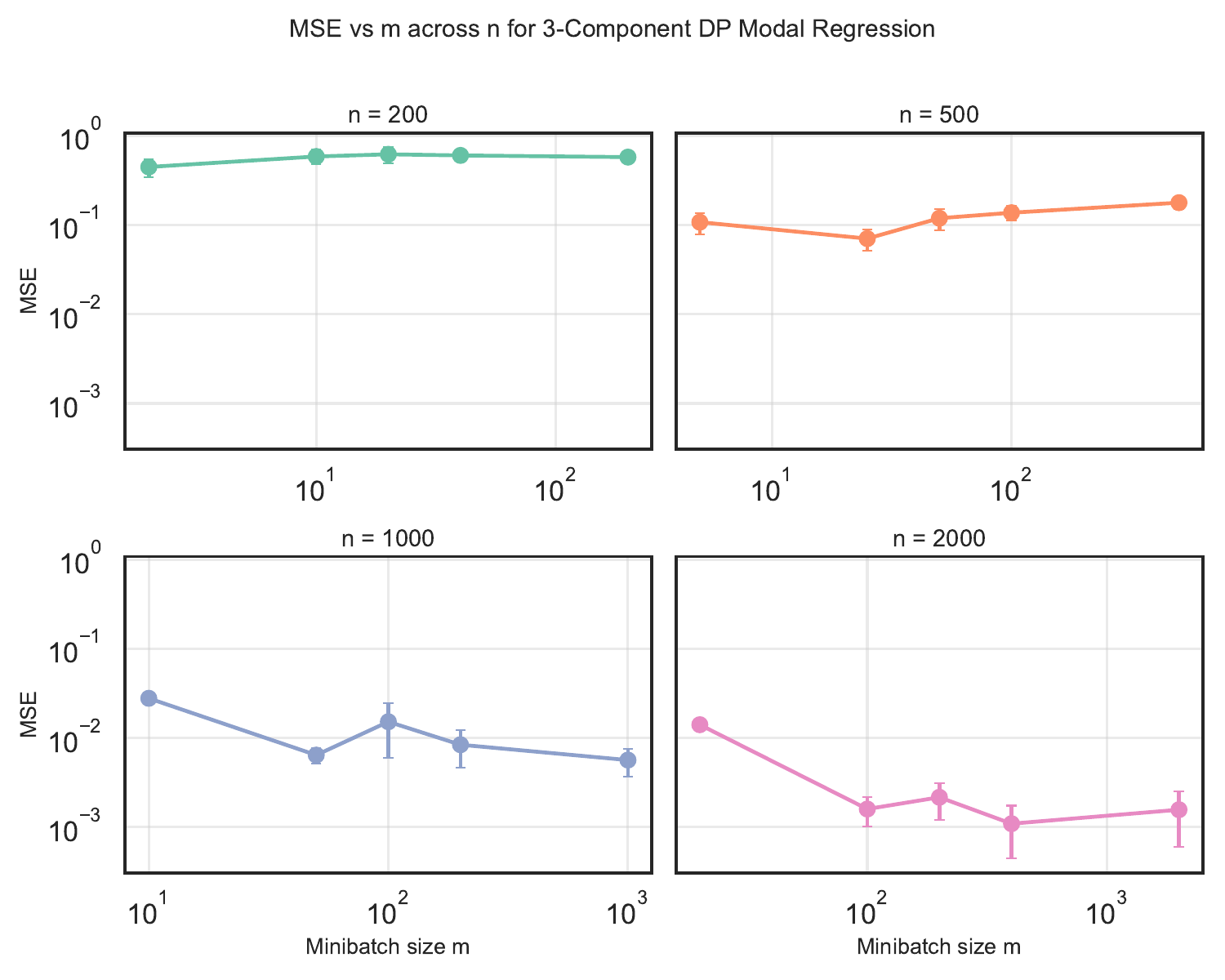}
        \caption{MSE vs.\ minibatch size $m$ across $n$.}
    \end{subfigure}
        \caption{\textbf{Hyperparameter sensitivity for \textsc{DP-PMS} on the three-component mixture.}
    (a) Effect of the clipping multiplier \texttt{clip\_multiplier} on oracle MSE for \(n\in\{200,500,1000,2000\}\) at fixed \(\varepsilon=1\).
    (b) Effect of minibatch size \(m\) on oracle MSE across the same sample sizes.
    The sweeps do not show sharp degradation near the selected defaults.}
    \label{fig:modal_hparam_threecomp}
\end{figure}

\begin{table}[!htb]
\centering
\caption{Private modal regression on the three-component mixture: oracle DP-MSE, PMS-MSE, and runtime summaries (mean \(\pm\) SE) for \textsc{DP-PMS} and PMS over the same 20 runs.}
\label{tab:pmr-3mix}
{\small\resizebox{\textwidth}{!}{%
\begin{tabular}{c c c c c c}
\toprule
\(n\) & \(\varepsilon\) & DP-MSE & PMS-MSE & PMS-runtime & DP-runtime \\
\midrule
200  & 0.1 & $1.3397 \pm 0.2489$ & $0.0020 \pm 0.0005$ & $0.0125 \pm 0.0007$ & $0.0129 \pm 0.0007$ \\
200  & 0.2 & $0.9565 \pm 0.1265$ & $0.0030 \pm 0.0005$ & $0.0131 \pm 0.0009$ & $0.0133 \pm 0.0007$ \\
200  & 0.5 & $0.7605 \pm 0.0666$ & $0.0026 \pm 0.0006$ & $0.0123 \pm 0.0006$ & $0.0127 \pm 0.0008$ \\
200  & 1.0 & $0.3653 \pm 0.0547$ & $0.0028 \pm 0.0007$ & $0.0126 \pm 0.0008$ & $0.0137 \pm 0.0010$ \\
\midrule
500  & 0.1 & $0.7823 \pm 0.0798$ & $0.0017 \pm 0.0004$ & $0.0409 \pm 0.0011$ & $0.0144 \pm 0.0009$ \\
500  & 0.2 & $0.8034 \pm 0.0558$ & $0.0021 \pm 0.0004$ & $0.0423 \pm 0.0014$ & $0.0135 \pm 0.0006$ \\
500  & 0.5 & $0.4760 \pm 0.0541$ & $0.0013 \pm 0.0003$ & $0.0412 \pm 0.0022$ & $0.0150 \pm 0.0012$ \\
500  & 1.0 & $0.1402 \pm 0.0351$ & $0.0015 \pm 0.0004$ & $0.0419 \pm 0.0014$ & $0.0166 \pm 0.0011$ \\
\midrule
1000 & 0.1 & $0.7579 \pm 0.0648$ & $0.0016 \pm 0.0002$ & $0.1070 \pm 0.0039$ & $0.0166 \pm 0.0010$ \\
1000 & 0.2 & $0.7279 \pm 0.0578$ & $0.0013 \pm 0.0003$ & $0.1059 \pm 0.0040$ & $0.0183 \pm 0.0013$ \\
1000 & 0.5 & $0.1743 \pm 0.0335$ & $0.0015 \pm 0.0002$ & $0.1214 \pm 0.0037$ & $0.0193 \pm 0.0011$ \\
1000 & 1.0 & $0.0066 \pm 0.0033$ & $0.0016 \pm 0.0003$ & $0.1053 \pm 0.0026$ & $0.0170 \pm 0.0008$ \\
\midrule
2000 & 0.1 & $0.6380 \pm 0.0570$ & $0.0016 \pm 0.0002$ & $0.3441 \pm 0.0098$ & $0.0225 \pm 0.0009$ \\
2000 & 0.2 & $0.2213 \pm 0.0356$ & $0.0015 \pm 0.0002$ & $0.4284 \pm 0.0163$ & $0.0261 \pm 0.0019$ \\
2000 & 0.5 & $0.0130 \pm 0.0076$ & $0.0012 \pm 0.0002$ & $0.3860 \pm 0.0127$ & $0.0270 \pm 0.0020$ \\
2000 & 1.0 & $0.000174 \pm 0.000026$ & $0.0014 \pm 0.0002$ & $0.3767 \pm 0.0085$ & $0.0246 \pm 0.0012$ \\
\bottomrule
\end{tabular}}}
\end{table}

Figure~\ref{fig:modal_reg_threepanel_appendix} shows that PMS separates the piecewise modal branches, whereas LOWESS averages across them and misses the conditional multimodality. At \(\varepsilon=1\), \textsc{DP-PMS} tracks the branch structure after private response-direction ascent. Table~\ref{tab:pmr-3mix} shows a steep privacy--utility transition: for example, at \(n=2000\), DP-MSE drops from \(0.6380\) at \(\varepsilon=0.1\) to \(0.000174\) at \(\varepsilon=1\). Figure~\ref{fig:modal_hparam_threecomp} shows that performance does not sharply degrade near the selected clipping and minibatch defaults. These diagnostics indicate that the binned conditional-DAP initializer is useful for this support-limited, piecewise-constant modal-regression design.

\paragraph{Sinusoidal 2-mixture.}

We extend the sinusoidal two-component experiment in Section~\ref{subsec:private-modal-regression} by reporting the full privacy--utility table and the clipping and minibatch sensitivity sweeps. The data-generating mechanism, privacy grid, and oracle loss are the same as in the main-text experiment.

\begin{table}[!htb]
\centering
\caption{Private modal regression on sinusoidal two-component mixture data: oracle DP-MSE, PMS-MSE, and runtime summaries (mean \(\pm\) SE) for \textsc{DP-PMS} and PMS over the same 20 runs.}
\label{tab:pmr-2mix}
{\small\resizebox{\textwidth}{!}{%
\begin{tabular}{c c c c c c}
\toprule
\(n\) & \(\varepsilon\) & DP-MSE & PMS-MSE & PMS-runtime & DP-runtime \\
\midrule
200  & 0.1 & $3.5963 \pm 0.2478$ & $0.0313 \pm 0.0014$ & $0.0128 \pm 0.0008$ & $0.0324 \pm 0.0018$ \\
200  & 0.2 & $0.5476 \pm 0.0423$ & $0.0305 \pm 0.0017$ & $0.0146 \pm 0.0010$ & $0.0394 \pm 0.0015$ \\
200  & 0.5 & $0.1388 \pm 0.0069$ & $0.0326 \pm 0.0013$ & $0.0127 \pm 0.0007$ & $0.0358 \pm 0.0013$ \\
200  & 1.0 & $0.1102 \pm 0.0061$ & $0.0339 \pm 0.0018$ & $0.0136 \pm 0.0009$ & $0.0379 \pm 0.0019$ \\
\midrule
600  & 0.1 & $0.1642 \pm 0.0062$ & $0.0203 \pm 0.0003$ & $0.0683 \pm 0.0041$ & $0.0633 \pm 0.0073$ \\
600  & 0.2 & $0.1020 \pm 0.0057$ & $0.0213 \pm 0.0005$ & $0.0677 \pm 0.0036$ & $0.0507 \pm 0.0029$ \\
600  & 0.5 & $0.0642 \pm 0.0042$ & $0.0218 \pm 0.0006$ & $0.0578 \pm 0.0025$ & $0.0494 \pm 0.0032$ \\
600  & 1.0 & $0.0551 \pm 0.0033$ & $0.0209 \pm 0.0005$ & $0.0568 \pm 0.0023$ & $0.0484 \pm 0.0021$ \\
\midrule
1200 & 0.1 & $0.0975 \pm 0.0046$ & $0.0186 \pm 0.0002$ & $0.1629 \pm 0.0028$ & $0.0640 \pm 0.0028$ \\
1200 & 0.2 & $0.0571 \pm 0.0041$ & $0.0179 \pm 0.0002$ & $0.1569 \pm 0.0033$ & $0.0658 \pm 0.0027$ \\
1200 & 0.5 & $0.0469 \pm 0.0028$ & $0.0187 \pm 0.0002$ & $0.1560 \pm 0.0035$ & $0.0629 \pm 0.0023$ \\
1200 & 1.0 & $0.0438 \pm 0.0018$ & $0.0183 \pm 0.0002$ & $0.1572 \pm 0.0040$ & $0.0642 \pm 0.0031$ \\
\midrule
2400 & 0.1 & $0.0453 \pm 0.0025$ & $0.0165 \pm 0.0002$ & $0.4593 \pm 0.0128$ & $0.0758 \pm 0.0029$ \\
2400 & 0.2 & $0.0429 \pm 0.0017$ & $0.0165 \pm 0.0001$ & $0.4567 \pm 0.0119$ & $0.0792 \pm 0.0030$ \\
2400 & 0.5 & $0.0426 \pm 0.0026$ & $0.0165 \pm 0.0002$ & $0.4713 \pm 0.0132$ & $0.0738 \pm 0.0025$ \\
2400 & 1.0 & $0.0417 \pm 0.0019$ & $0.0161 \pm 0.0002$ & $0.4747 \pm 0.0128$ & $0.0800 \pm 0.0029$ \\
\bottomrule
\end{tabular}}}
\end{table}

\begin{figure}[!htb]
    \centering
    \begin{subfigure}{0.48\linewidth}
        \centering
        \includegraphics[width=\linewidth]{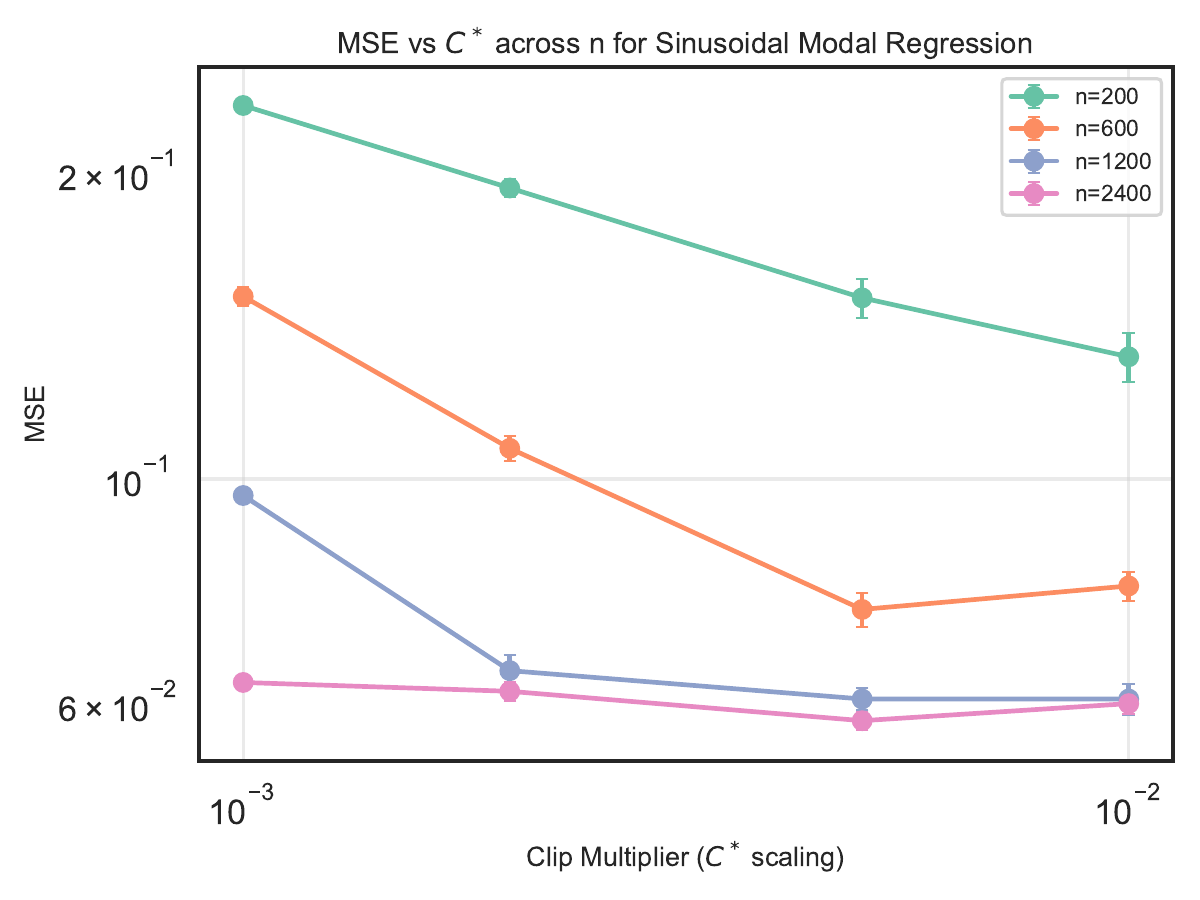}
        \caption{MSE vs.\ clipping multiplier \texttt{clip\_multiplier} across $n$.}
    \end{subfigure}\hfill
    \begin{subfigure}{0.48\linewidth}
        \centering
        \includegraphics[width=\linewidth]{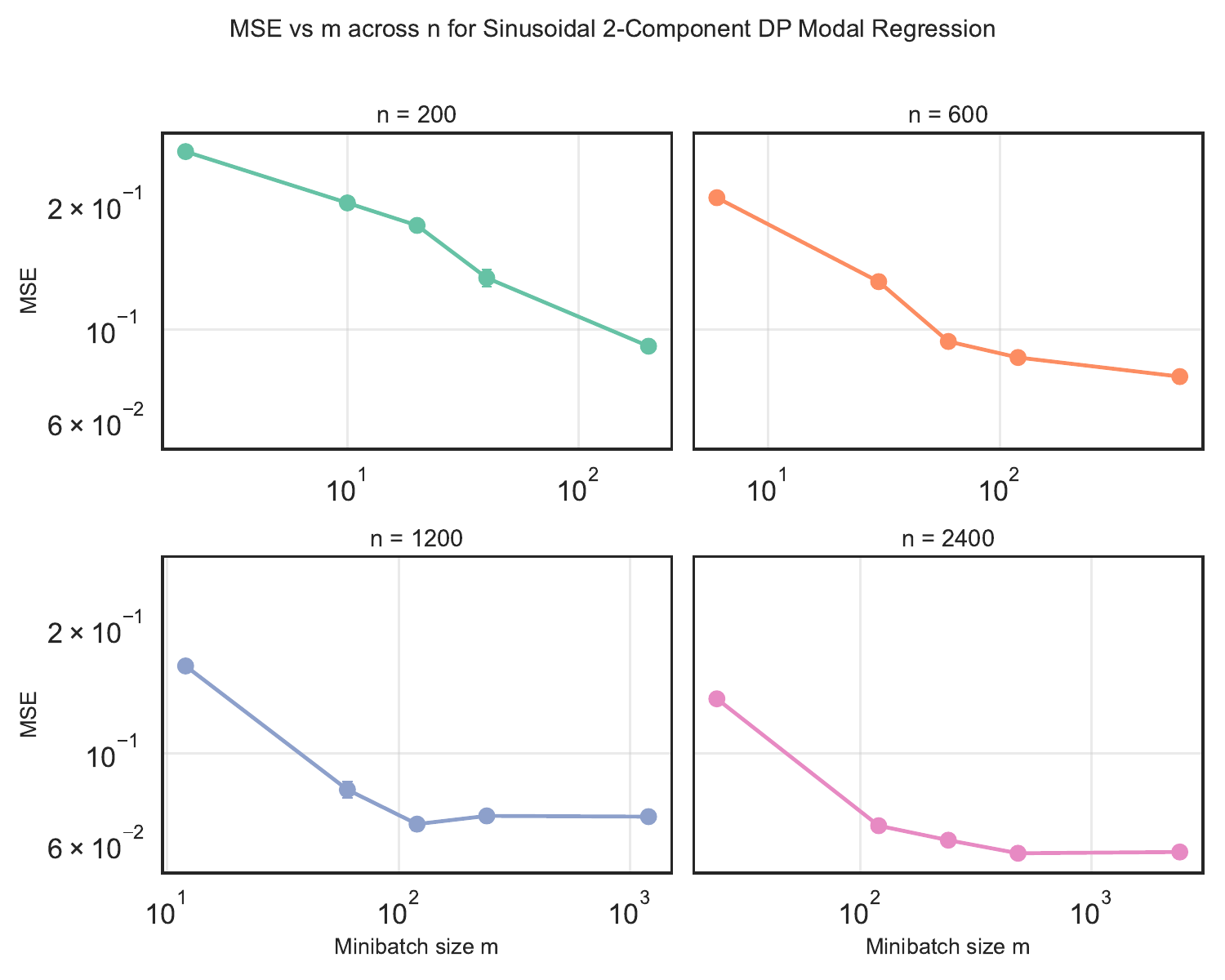}
        \caption{MSE vs.\ minibatch size $m$ across $n$.}
    \end{subfigure}
        \caption{\textbf{Hyperparameter sensitivity for \textsc{DP-PMS} on sinusoidal two-component mixture data.}
    (a) Effect of the clipping multiplier \texttt{clip\_multiplier} on oracle MSE for \(n\in\{200,600,1200,2400\}\) at fixed \(\varepsilon=1\).
    (b) Effect of minibatch size \(m\) on oracle MSE across the same sample sizes.
    The sweeps do not show sharp degradation near the selected defaults.}
    \label{fig:sin_modal_hparam}
\end{figure}

Table~\ref{tab:pmr-2mix} summarizes oracle DP-MSE, PMS-MSE, and runtime across the sinusoidal privacy--utility grid. The table supports the main trend in Figure~\ref{fig:modal_reg_sin_threepanel}: the largest DP-MSE reductions occur between the smallest privacy budgets and moderate \(\varepsilon\), while for larger \(n\) further privacy-budget increases yield smaller improvements and the private error moves closer to the PMS-MSE scale. Figure~\ref{fig:sin_modal_hparam} shows that the sinusoidal experiment is not sharply sensitive to moderate changes in the clipping multiplier or minibatch size near the selected defaults.

Together, the piecewise-constant and sinusoidal designs show that \textsc{DP-PMS} can recover both support-limited branches and smooth nonlinear modal curves under the fixed-design response-privacy setup. The two examples differ in difficulty: the three-component design shows a sharper drop in DP-MSE at larger \(n\) and \(\varepsilon\), while the sinusoidal design retains a visible gap from PMS-MSE even at the largest privacy budgets. In both cases, the clipping and minibatch sweeps indicate that the reported results are not driven by a narrow tuning choice.

\subsection{Private Clustering on Simulated Data}
\label{app:private-clustering}

\begin{table}[!htb]
\centering
\caption{Blobs (simulated): privacy--utility summary for \textsc{DP-GRAMS-C} and DP-\(k\)-Means across \((n,\varepsilon)\). We report ARI, NMI, centroid MSE, and runtime as mean \(\pm\) SE.}
\label{tab:blobs_privacy_utility_table}
\scriptsize
\setlength{\tabcolsep}{3pt}
\resizebox{\textwidth}{!}{%
\begin{tabular}{c c c c c c c c c c}
\toprule
\multirow{2}{*}{$n$} & \multirow{2}{*}{$\varepsilon$} &
\multicolumn{4}{c}{\textsc{DP-GRAMS-C}} & \multicolumn{4}{c}{DP-\(k\)-Means} \\
\cmidrule(lr){3-6}\cmidrule(lr){7-10}
& & ARI & NMI & MSE & Time (s) & ARI & NMI & MSE & Time (s) \\
\midrule
\multirow{5}{*}{700}  & 0.1 & $0.383 \pm 0.025$ & $0.476 \pm 0.018$ & $2.047 \pm 0.254$ & $0.00317 \pm 0.00018$
                      & $0.376 \pm 0.027$ & $0.466 \pm 0.020$ & $1.699 \pm 0.152$ & $0.00722 \pm 0.00002$ \\
                     & 0.2 & $0.569 \pm 0.022$ & $0.597 \pm 0.014$ & $0.605 \pm 0.121$ & $0.00252 \pm 0.00001$
                      & $0.377 \pm 0.018$ & $0.465 \pm 0.013$ & $1.557 \pm 0.145$ & $0.00704 \pm 0.00003$ \\
                     & 0.5 & $0.744 \pm 0.005$ & $0.707 \pm 0.003$ & $0.0985 \pm 0.0136$ & $0.00250 \pm 0.00001$
                      & $0.421 \pm 0.014$ & $0.506 \pm 0.010$ & $0.981 \pm 0.0553$ & $0.00719 \pm 0.00008$ \\
                     & 1.0 & $0.759 \pm 0.003$ & $0.715 \pm 0.002$ & $0.0589 \pm 0.00590$ & $0.00259 \pm 0.00002$
                      & $0.500 \pm 0.022$ & $0.545 \pm 0.016$ & $0.771 \pm 0.0845$ & $0.00703 \pm 0.00006$ \\
                     & 5.0 & $0.758 \pm 0.002$ & $0.715 \pm 0.002$ & $0.0500 \pm 0.00450$ & $0.00359 \pm 0.00023$
                      & $0.722 \pm 0.008$ & $0.693 \pm 0.005$ & $0.126 \pm 0.0180$ & $0.01230 \pm 0.00009$ \\
\midrule
\multirow{5}{*}{1000} & 0.1 & $0.486 \pm 0.027$ & $0.539 \pm 0.018$ & $1.197 \pm 0.204$ & $0.00276 \pm 0.00001$
                      & $0.387 \pm 0.025$ & $0.477 \pm 0.017$ & $1.877 \pm 0.190$ & $0.00774 \pm 0.00011$ \\
                     & 0.2 & $0.643 \pm 0.018$ & $0.647 \pm 0.012$ & $0.248 \pm 0.0406$ & $0.00281 \pm 0.00002$
                      & $0.360 \pm 0.021$ & $0.454 \pm 0.015$ & $1.877 \pm 0.168$ & $0.00729 \pm 0.00003$ \\
                     & 0.5 & $0.711 \pm 0.005$ & $0.689 \pm 0.004$ & $0.0791 \pm 0.00747$ & $0.00275 \pm 0.00001$
                      & $0.466 \pm 0.019$ & $0.528 \pm 0.013$ & $1.034 \pm 0.107$ & $0.00703 \pm 0.00003$ \\
                     & 1.0 & $0.729 \pm 0.003$ & $0.701 \pm 0.002$ & $0.0567 \pm 0.00701$ & $0.00283 \pm 0.00003$
                      & $0.521 \pm 0.025$ & $0.566 \pm 0.017$ & $0.616 \pm 0.0861$ & $0.00719 \pm 0.00006$ \\
                     & 5.0 & $0.740 \pm 0.002$ & $0.707 \pm 0.002$ & $0.0426 \pm 0.00318$ & $0.00290 \pm 0.00001$
                      & $0.689 \pm 0.013$ & $0.675 \pm 0.008$ & $0.167 \pm 0.0390$ & $0.01883 \pm 0.00011$ \\
\midrule
\multirow{5}{*}{2000} & 0.1 & $0.593 \pm 0.021$ & $0.606 \pm 0.014$ & $0.348 \pm 0.0589$ & $0.00445 \pm 0.00012$
                      & $0.378 \pm 0.014$ & $0.457 \pm 0.011$ & $1.582 \pm 0.145$ & $0.00801 \pm 0.00005$ \\
                     & 0.2 & $0.713 \pm 0.005$ & $0.682 \pm 0.003$ & $0.0923 \pm 0.00986$ & $0.00428 \pm 0.00002$
                      & $0.362 \pm 0.025$ & $0.448 \pm 0.016$ & $1.394 \pm 0.109$ & $0.00815 \pm 0.00008$ \\
                     & 0.5 & $0.728 \pm 0.002$ & $0.693 \pm 0.002$ & $0.0611 \pm 0.00472$ & $0.00446 \pm 0.00004$
                      & $0.508 \pm 0.025$ & $0.553 \pm 0.017$ & $0.669 \pm 0.0900$ & $0.00796 \pm 0.00006$ \\
                     & 1.0 & $0.733 \pm 0.003$ & $0.695 \pm 0.002$ & $0.0539 \pm 0.00495$ & $0.00410 \pm 0.00003$
                      & $0.554 \pm 0.023$ & $0.585 \pm 0.014$ & $0.439 \pm 0.0746$ & $0.00767 \pm 0.00005$ \\
                     & 5.0 & $0.739 \pm 0.002$ & $0.700 \pm 0.002$ & $0.0479 \pm 0.00639$ & $0.00410 \pm 0.00001$
                      & $0.715 \pm 0.015$ & $0.686 \pm 0.009$ & $0.0827 \pm 0.0419$ & $0.02321 \pm 0.00005$ \\
\midrule
\multirow{5}{*}{5000} & 0.1 & $0.717 \pm 0.007$ & $0.687 \pm 0.005$ & $0.0907 \pm 0.0132$ & $0.00744 \pm 0.00004$
                      & $0.405 \pm 0.024$ & $0.483 \pm 0.016$ & $1.186 \pm 0.123$ & $0.00962 \pm 0.00004$ \\
                     & 0.2 & $0.740 \pm 0.002$ & $0.702 \pm 0.001$ & $0.0342 \pm 0.00325$ & $0.00740 \pm 0.00002$
                      & $0.495 \pm 0.024$ & $0.546 \pm 0.016$ & $0.763 \pm 0.0914$ & $0.00973 \pm 0.00004$ \\
                     & 0.5 & $0.743 \pm 0.002$ & $0.702 \pm 0.001$ & $0.0308 \pm 0.00382$ & $0.00735 \pm 0.00001$
                      & $0.554 \pm 0.024$ & $0.588 \pm 0.015$ & $0.494 \pm 0.0791$ & $0.01361 \pm 0.00006$ \\
                     & 1.0 & $0.738 \pm 0.002$ & $0.701 \pm 0.001$ & $0.0456 \pm 0.00413$ & $0.00747 \pm 0.00002$
                      & $0.639 \pm 0.024$ & $0.638 \pm 0.016$ & $0.255 \pm 0.0615$ & $0.02519 \pm 0.00008$ \\
                     & 5.0 & $0.731 \pm 0.003$ & $0.697 \pm 0.002$ & $0.0682 \pm 0.00548$ & $0.00736 \pm 0.00001$
                      & $0.729 \pm 0.013$ & $0.694 \pm 0.008$ & $0.0370 \pm 0.0161$ & $0.02862 \pm 0.00007$ \\
\bottomrule
\end{tabular}}
\vspace{0.25em}
\end{table}

This subsection complements the blobs clustering experiment in Section~\ref{subsec:private-clustering-sim} by providing additional numerical summaries and hyperparameter sweeps for \textsc{DP-GRAMS-C}. Figure~\ref{fig:blobs_privacy_utility_all} reports privacy--utility curves in ARI, NMI, and centroid MSE versus $\varepsilon$ across $n\in\{700,1000,2000,5000\}$. Table~\ref{tab:blobs_privacy_utility_table} shows that \textsc{DP-GRAMS-C} improves sharply from \(\varepsilon=0.1\) to moderate privacy budgets and then largely stabilizes in ARI and NMI, while centroid MSE remains small once \(\varepsilon\) is moderate. DP-\(k\)-Means also improves with \(\varepsilon\), but is generally weaker in ARI and NMI except at the loosest privacy budgets and largest sample sizes. We then fix \(\varepsilon=1\) and study sensitivity to the clipping threshold \(C_*\) and minibatch size \(m\).

\noindent\textbf{Effects of minibatch size $m$ and clipping threshold $C_*$.} Figure~\ref{fig:blobs_subsample_clip} shows how clustering quality varies with \(C_*\) across the sample-size grid. Figure~\ref{fig:blobs_subsample_m} shows how performance varies with minibatch size across the same sample-size grid.

\begin{figure}[!htb]
    \centering
    \setlength{\tabcolsep}{0pt}
    \begin{tabular}{@{}c@{\hspace{0.02\textwidth}}c@{\hspace{0.02\textwidth}}c@{}}
    \subcaptionbox{ARI vs.\ $C_*$.}[0.32\textwidth]{%
        \includegraphics[width=0.32\textwidth]{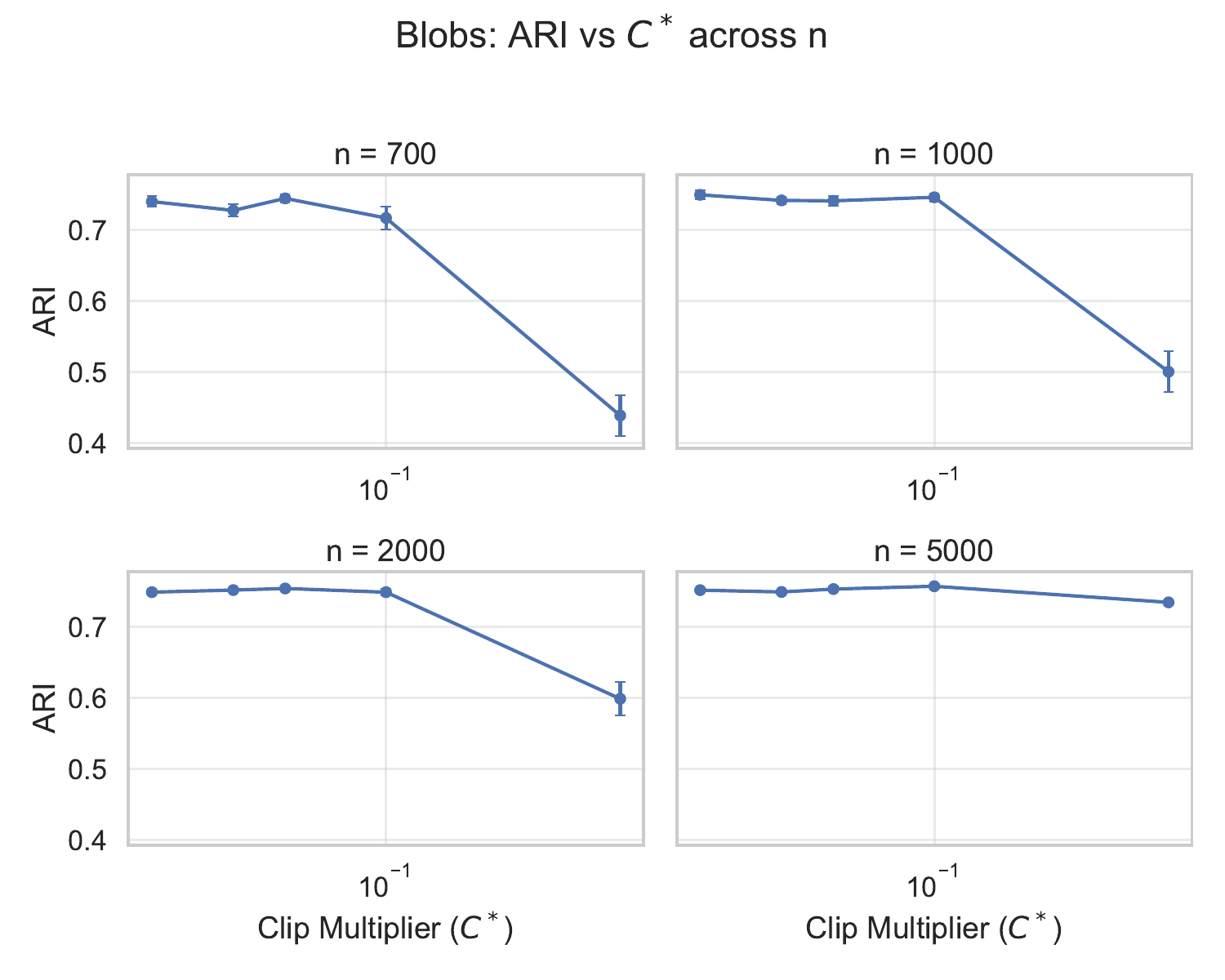}%
    } &
    \subcaptionbox{NMI vs.\ $C_*$.}[0.32\textwidth]{%
        \includegraphics[width=0.32\textwidth]{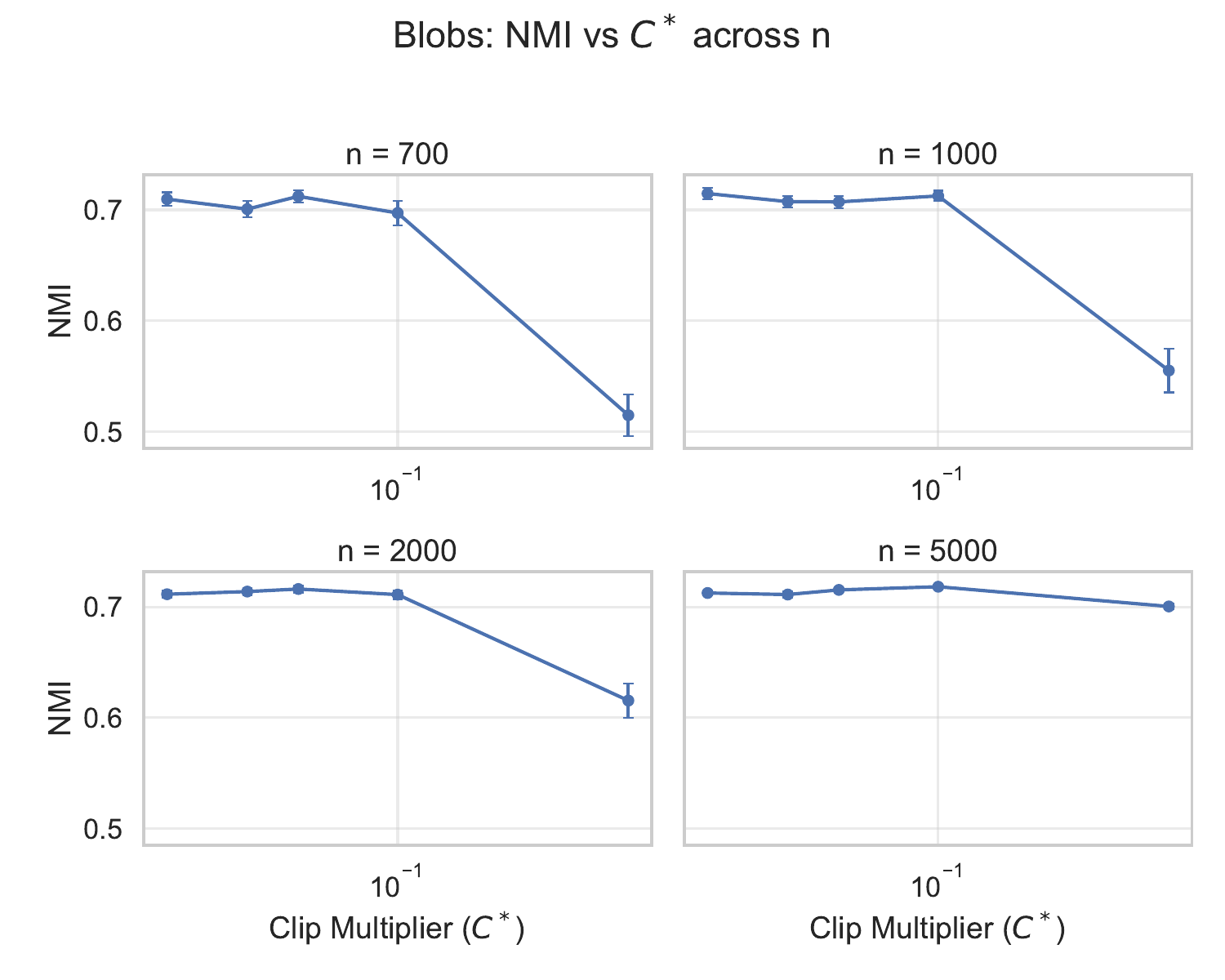}%
    } &
    \subcaptionbox{Centroid MSE vs.\ $C_*$.}[0.32\textwidth]{%
        \includegraphics[width=0.32\textwidth]{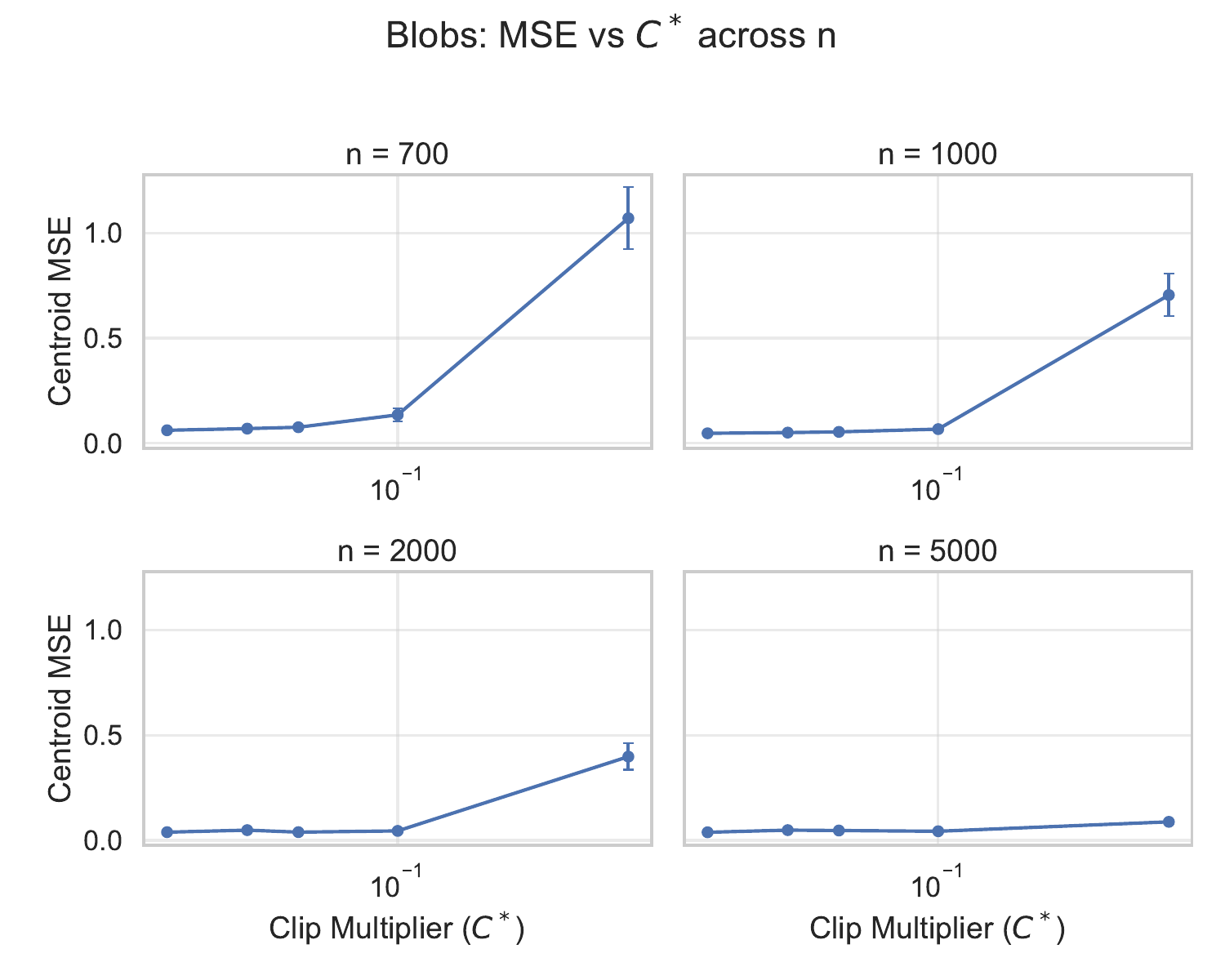}%
    }
    \end{tabular}
    \caption{\textbf{Effect of clipping on \textsc{DP-GRAMS-C} for blobs.}
    Subsampling-effect study for $C_*\in\{0.01,0.1,0.5,1.0,2.0\}$ at $\varepsilon=1$ across $n\in\{700,1000,2000,5000\}$.
    ARI, NMI, and centroid MSE vary smoothly with $C_*$, with a broad range of clipping multipliers (including the default $C_*=1$) yielding near-optimal performance.}
    \label{fig:blobs_subsample_clip}
\end{figure}

\begin{figure}[!htb]
    \centering
    \setlength{\tabcolsep}{0pt}
    \begin{tabular}{@{}c@{\hspace{0.02\textwidth}}c@{\hspace{0.02\textwidth}}c@{}}
    \subcaptionbox{ARI vs.\ minibatch size \(m\).}[0.32\textwidth]{%
        \includegraphics[width=0.32\textwidth]{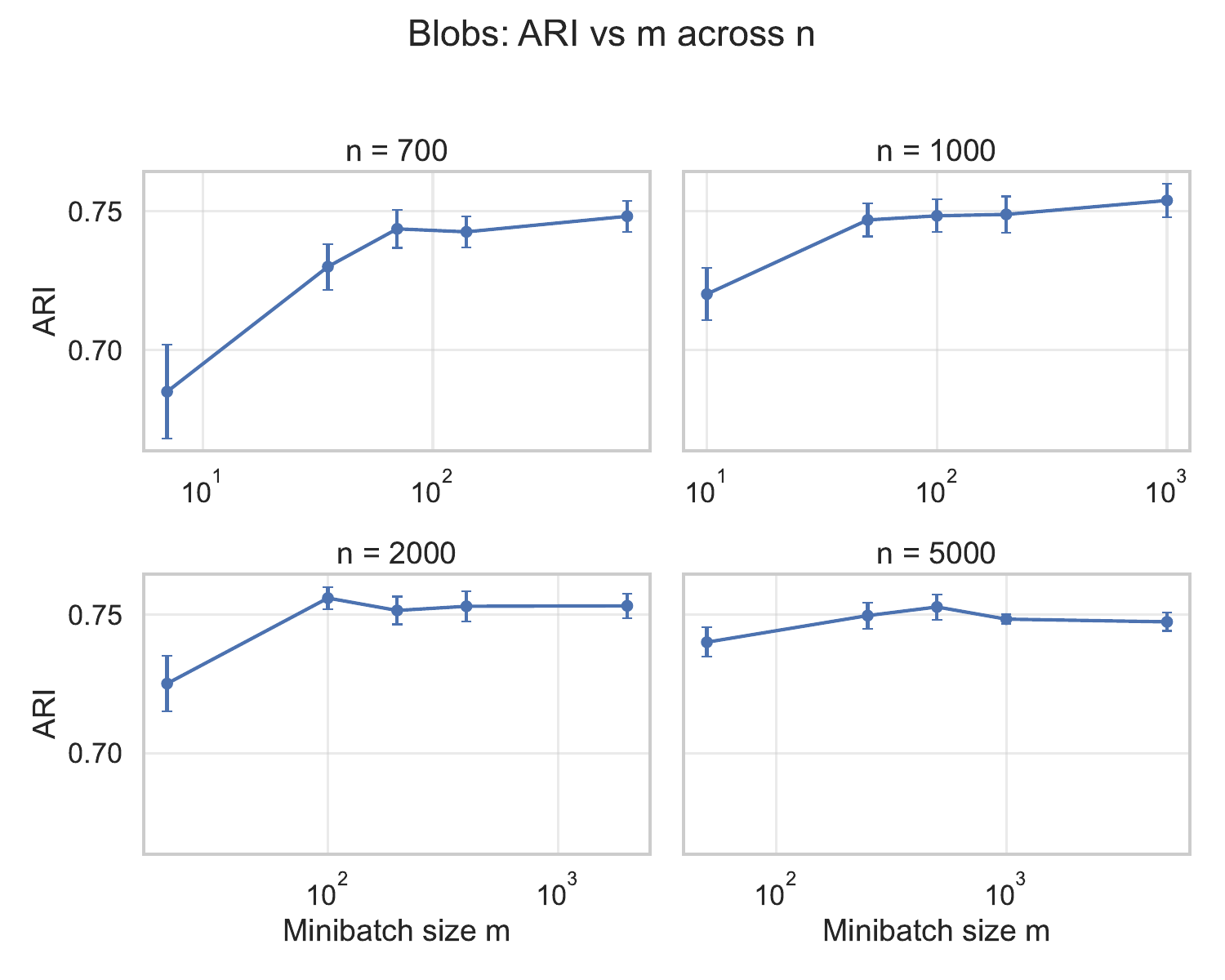}%
    } &
    \subcaptionbox{NMI vs.\ minibatch size \(m\).}[0.32\textwidth]{%
        \includegraphics[width=0.32\textwidth]{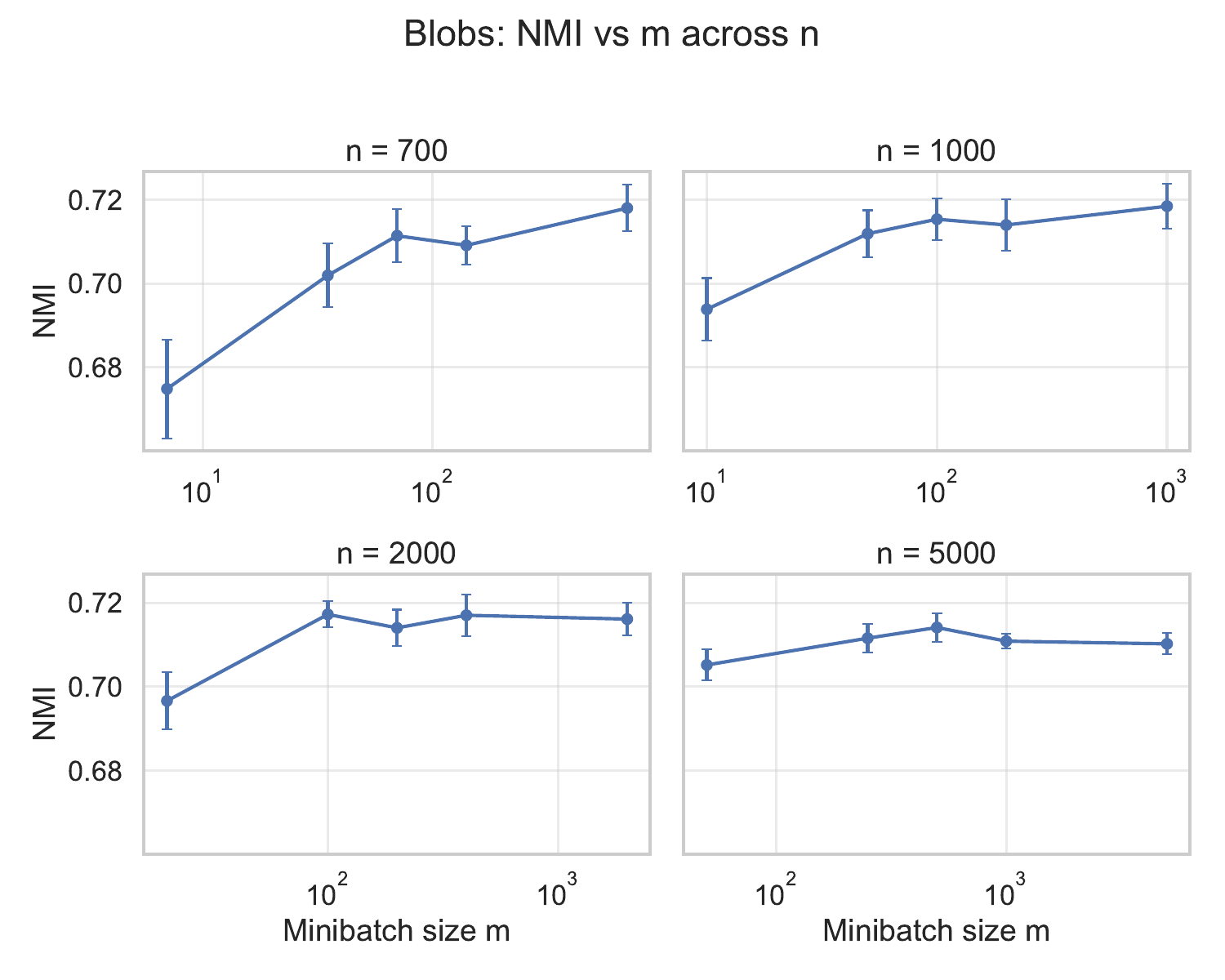}%
    } &
    \subcaptionbox{Centroid MSE vs.\ minibatch size \(m\).}[0.32\textwidth]{%
        \includegraphics[width=0.32\textwidth]{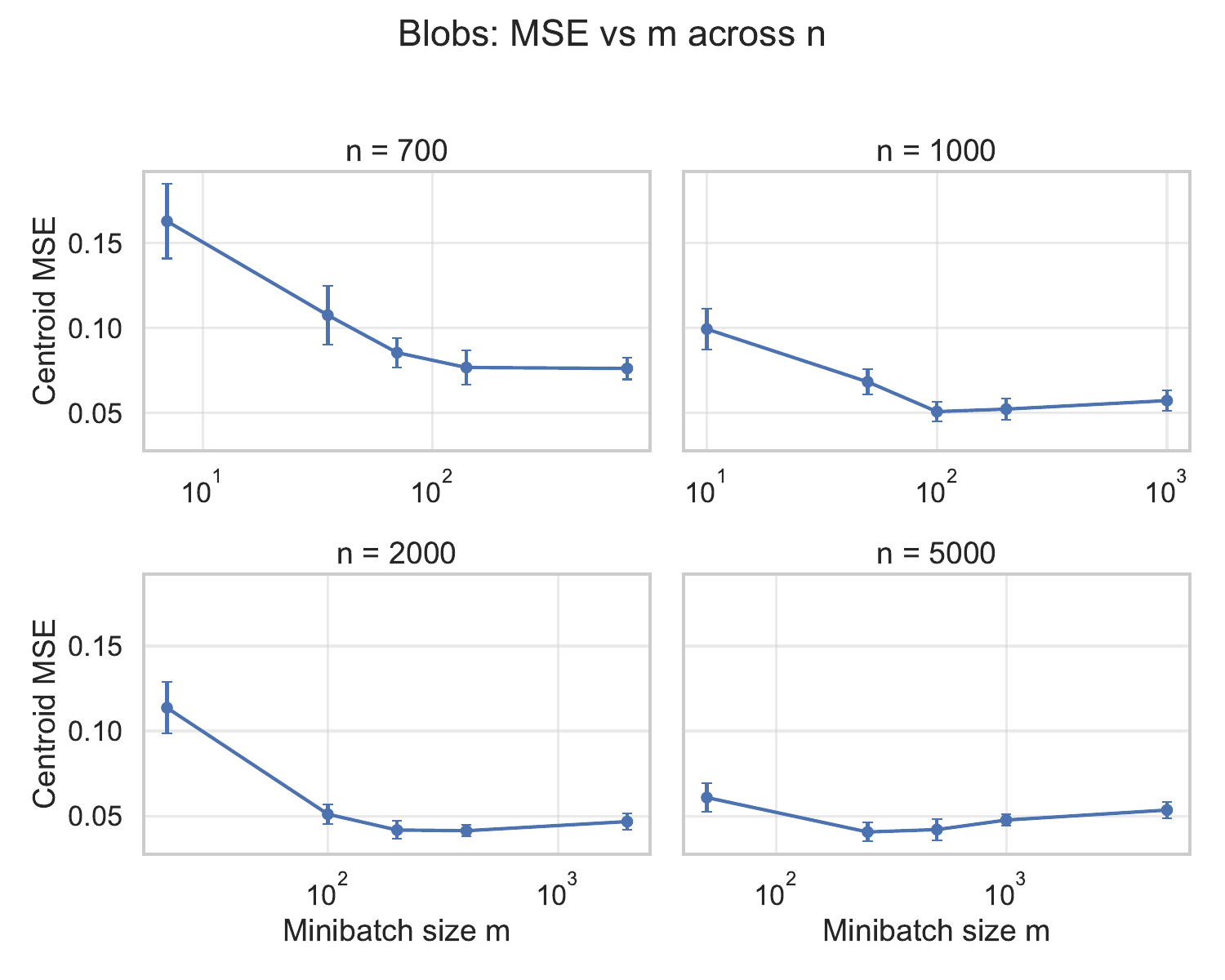}%
    }
    \end{tabular}
    \caption{\textbf{Effect of minibatch size on \textsc{DP-GRAMS-C} for blobs.}
    Sample-size grid for \(m\in\{0.01n,0.05n,0.1n,0.2n,n\}\) at \(\varepsilon=1\) and \(C_*=1\) across \(n\in\{700,1000,2000,5000\}\).
    The metrics do not show sharp deterioration near the selected minibatch default.}
    \label{fig:blobs_subsample_m}
\end{figure}

\subsection{Private Clustering on Real Datasets}
\label{subsec:private-clustering-real-appendix}

This subsection supplements the real-data clustering experiments in Section~\ref{subsec:private-clustering-main}. We report additional Digits results, together with MNIST privacy--utility summaries and sensitivity analyses. Digits uses the default public DAP grid in a six-dimensional PCA representation, whereas MNIST uses public auxiliary candidates in a whitened five-dimensional PCA representation to avoid the high-dimensional DAP-grid bottleneck. The following subsection reports the corresponding Cancer RNA-Seq diagnostics.

\begin{figure}[!htb]
    \centering
    \includegraphics[width=0.95\textwidth]{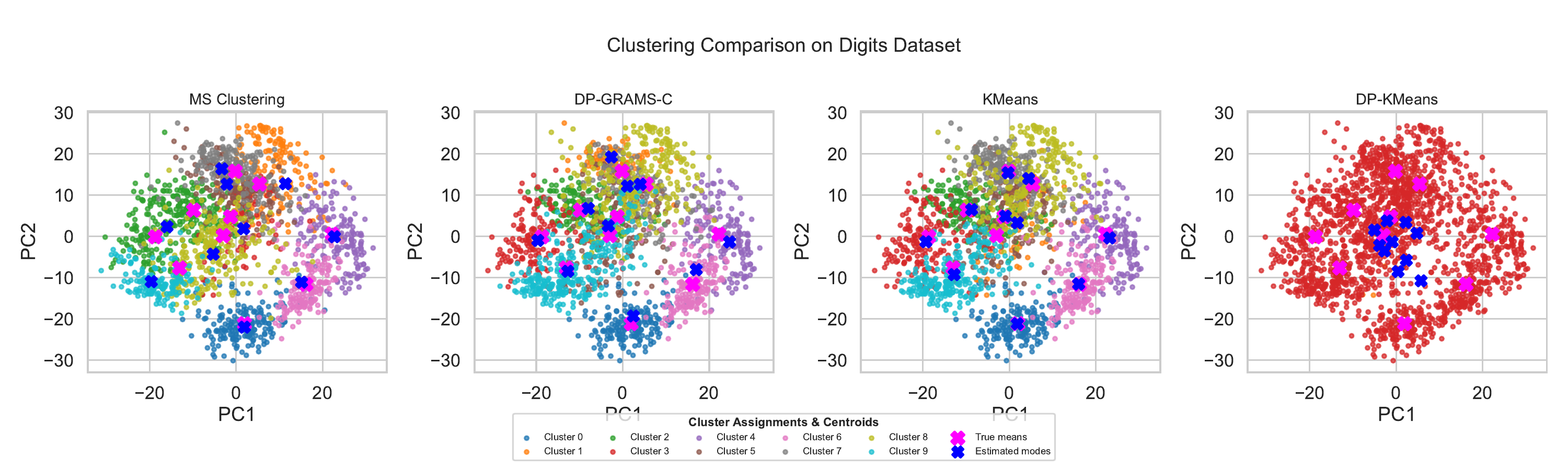}
   \caption{\small Digits dataset. Two-dimensional visualization of the six-dimensional PCA clustering representation, comparing non-private mean shift, \textsc{DP-GRAMS-C}, \(k\)-means, and DP-\(k\)-Means, with private methods run at \(\varepsilon=1\). True class centroids and estimated centroids are overlaid.}
    \label{fig:realworld_digits}
\end{figure}

\begin{figure}[!htb]
    \centering
    \begin{minipage}{0.32\textwidth}
        \includegraphics[width=\linewidth]{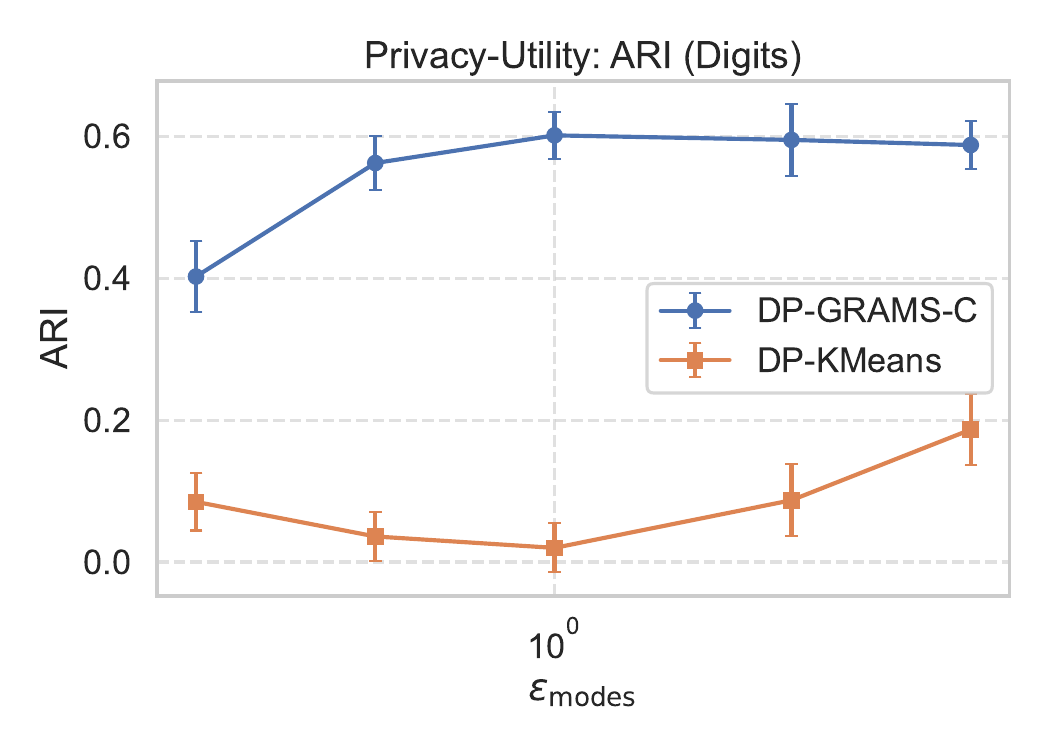}
    \end{minipage}\hfill
    \begin{minipage}{0.32\textwidth}
        \includegraphics[width=\linewidth]{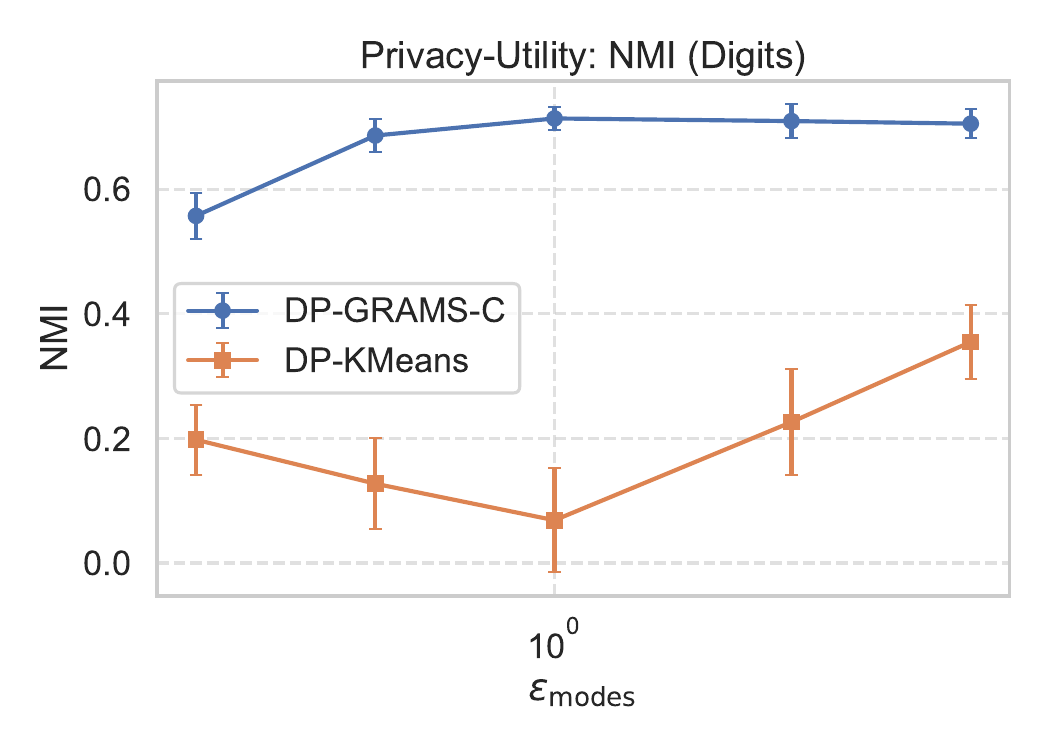}
    \end{minipage}\hfill
    \begin{minipage}{0.32\textwidth}
        \includegraphics[width=\linewidth]{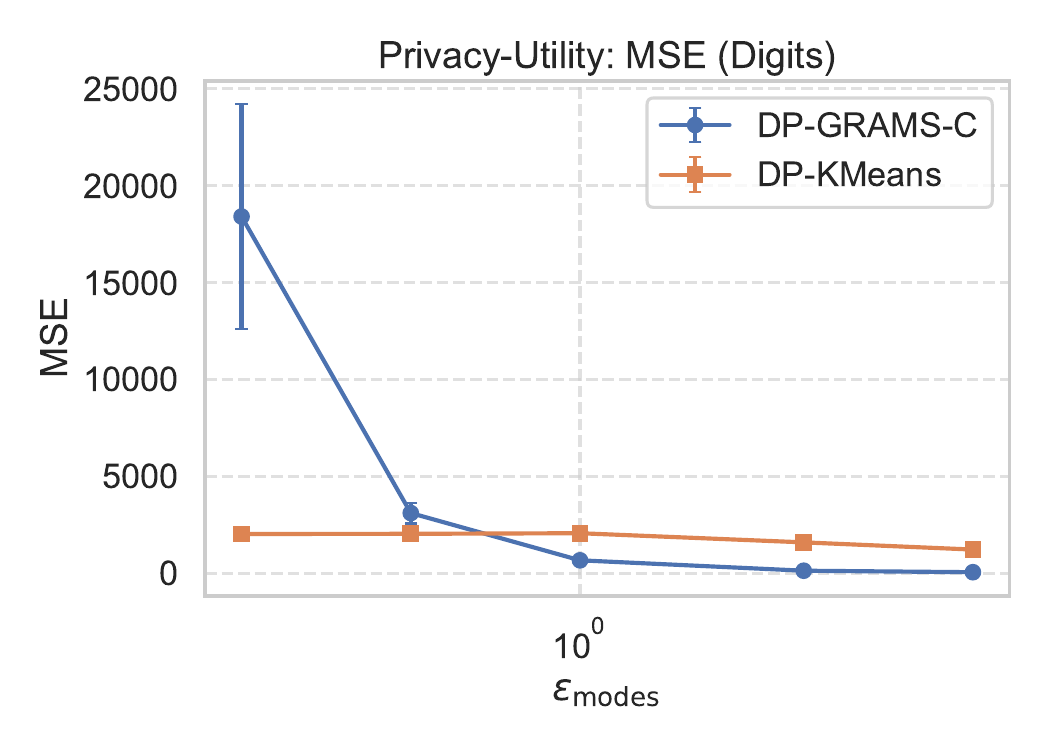}
    \end{minipage}
    \caption{\small Privacy--utility on Digits: ARI, NMI, and centroid MSE versus \(\varepsilon\) on a log scale for \textsc{DP-GRAMS-C} and DP-\(k\)-Means, with \(\varepsilon\in\{0.25,0.5,1,2.5,5\}\). Points show averages over \(20\) runs with standard-error bars.}
    \label{fig:privacy_utility_digits_all}
\end{figure}

\begin{table}[!htb]
\centering
\caption{Digits: privacy--utility summary for \textsc{DP-GRAMS-C} and DP-\(k\)-Means across \(\varepsilon\). Reported metrics are ARI, NMI, centroid MSE, and runtime, summarized as mean \(\pm\) SE over \(20\) runs.}
\label{tab:digits_privacy_utility_table}
\resizebox{\textwidth}{!}{%
\begin{tabular}{c c c c c c}
\toprule
\(\varepsilon\) & Algorithm & ARI & NMI & Centroid MSE & Runtime (s) \\
\midrule
0.25 & \textsc{DP-GRAMS-C} & $0.384 \pm 0.009$ & $0.535 \pm 0.007$ & $28.93 \pm 2.47$ & $0.1465 \pm 0.0022$ \\
0.25 & DP-\(k\)-Means      & $0.036 \pm 0.010$ & $0.134 \pm 0.021$ & $253.96 \pm 12.88$ & $0.0335 \pm 0.0009$ \\
\midrule
0.5  & \textsc{DP-GRAMS-C} & $0.374 \pm 0.011$ & $0.527 \pm 0.009$ & $31.60 \pm 3.51$ & $0.1476 \pm 0.0020$ \\
0.5  & DP-\(k\)-Means      & $0.029 \pm 0.007$ & $0.118 \pm 0.019$ & $260.87 \pm 14.63$ & $0.0336 \pm 0.0009$ \\
\midrule
1.0  & \textsc{DP-GRAMS-C} & $0.376 \pm 0.011$ & $0.532 \pm 0.008$ & $27.69 \pm 2.11$ & $0.1401 \pm 0.0009$ \\
1.0  & DP-\(k\)-Means      & $0.046 \pm 0.010$ & $0.153 \pm 0.018$ & $217.67 \pm 11.49$ & $0.0321 \pm 0.0011$ \\
\midrule
2.5  & \textsc{DP-GRAMS-C} & $0.392 \pm 0.009$ & $0.548 \pm 0.008$ & $28.58 \pm 2.03$ & $0.1395 \pm 0.0008$ \\
2.5  & DP-\(k\)-Means      & $0.115 \pm 0.016$ & $0.262 \pm 0.023$ & $204.23 \pm 12.83$ & $0.0314 \pm 0.0009$ \\
\midrule
5.0  & \textsc{DP-GRAMS-C} & $0.399 \pm 0.010$ & $0.549 \pm 0.007$ & $20.96 \pm 1.86$ & $0.1408 \pm 0.0012$ \\
5.0  & DP-\(k\)-Means      & $0.136 \pm 0.017$ & $0.285 \pm 0.020$ & $179.29 \pm 12.14$ & $0.0294 \pm 0.0007$ \\
\bottomrule
\end{tabular}}
\end{table}

Figure~\ref{fig:realworld_digits} shows that \textsc{DP-GRAMS-C} preserves much of the cluster structure seen under non-private mean shift in the PCA visualization. Figure~\ref{fig:privacy_utility_digits_all} and Table~\ref{tab:digits_privacy_utility_table} show a stable advantage over DP-\(k\)-Means on this dataset: \textsc{DP-GRAMS-C} has substantially higher ARI and NMI and lower centroid MSE throughout the privacy grid. The gains are not strictly monotone in every metric, but centroid MSE is lowest at the largest privacy budget and the ARI and NMI values remain consistently separated from the DP-\(k\)-Means baseline.

\paragraph{Hyperparameter sweeps on Digits.}

\begin{figure}[!htb]
    \centering
    \begin{minipage}{0.32\textwidth}
        \includegraphics[width=\linewidth]{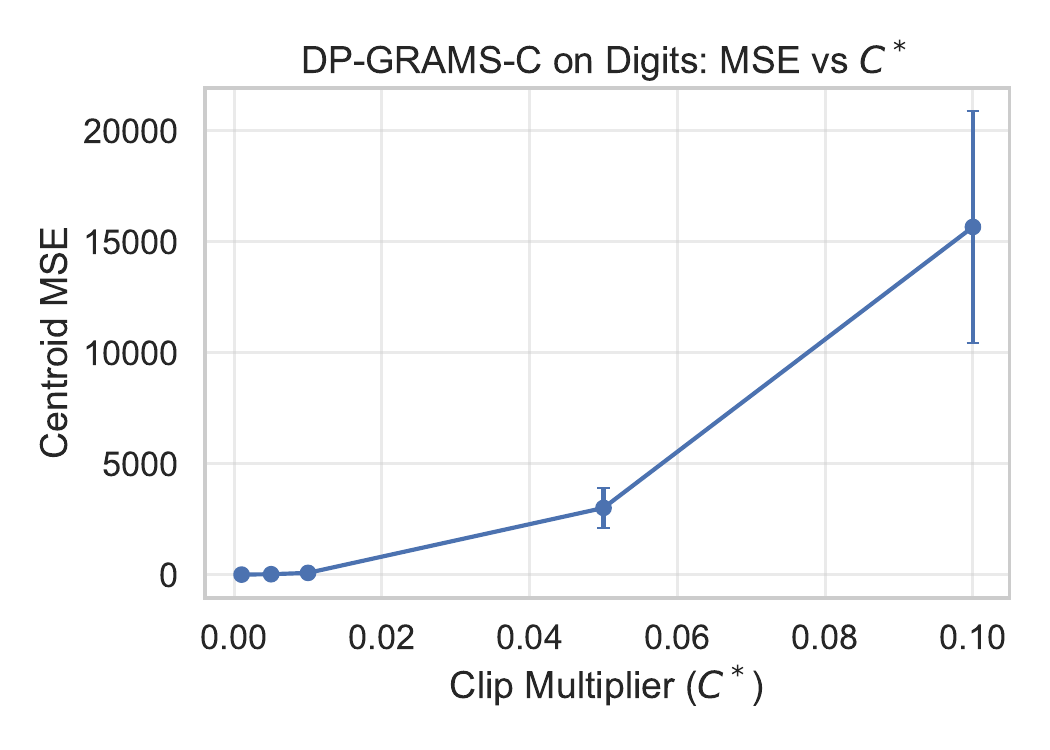}
    \end{minipage}\hfill
    \begin{minipage}{0.32\textwidth}
        \includegraphics[width=\linewidth]{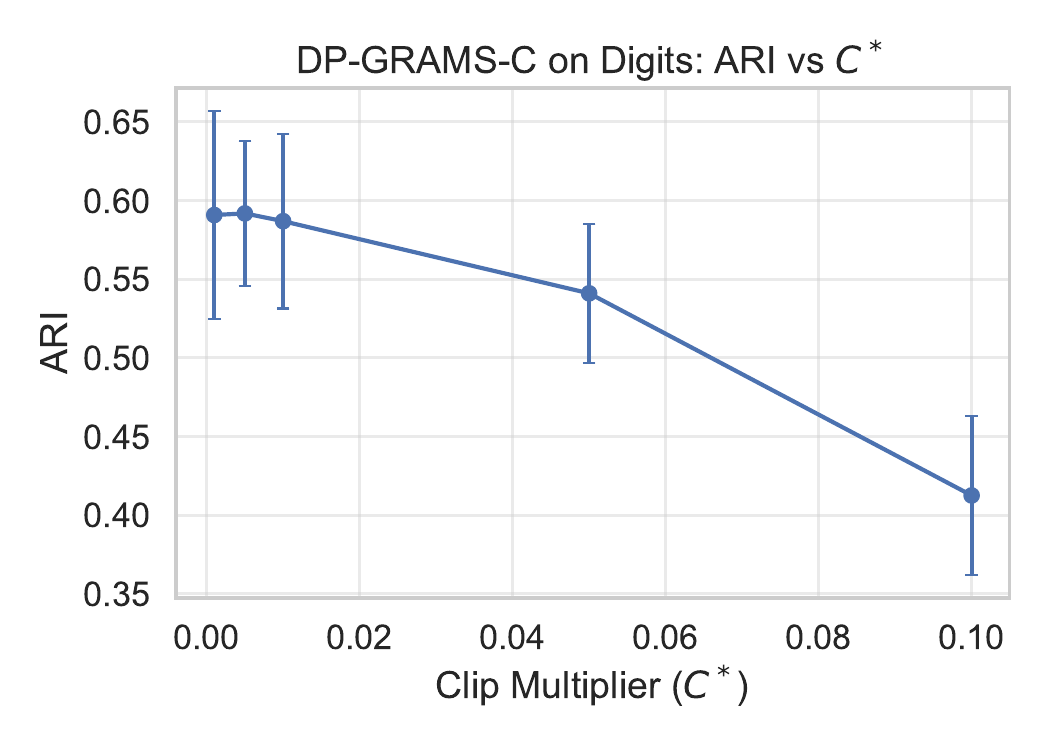}
    \end{minipage}\hfill
    \begin{minipage}{0.32\textwidth}
        \includegraphics[width=\linewidth]{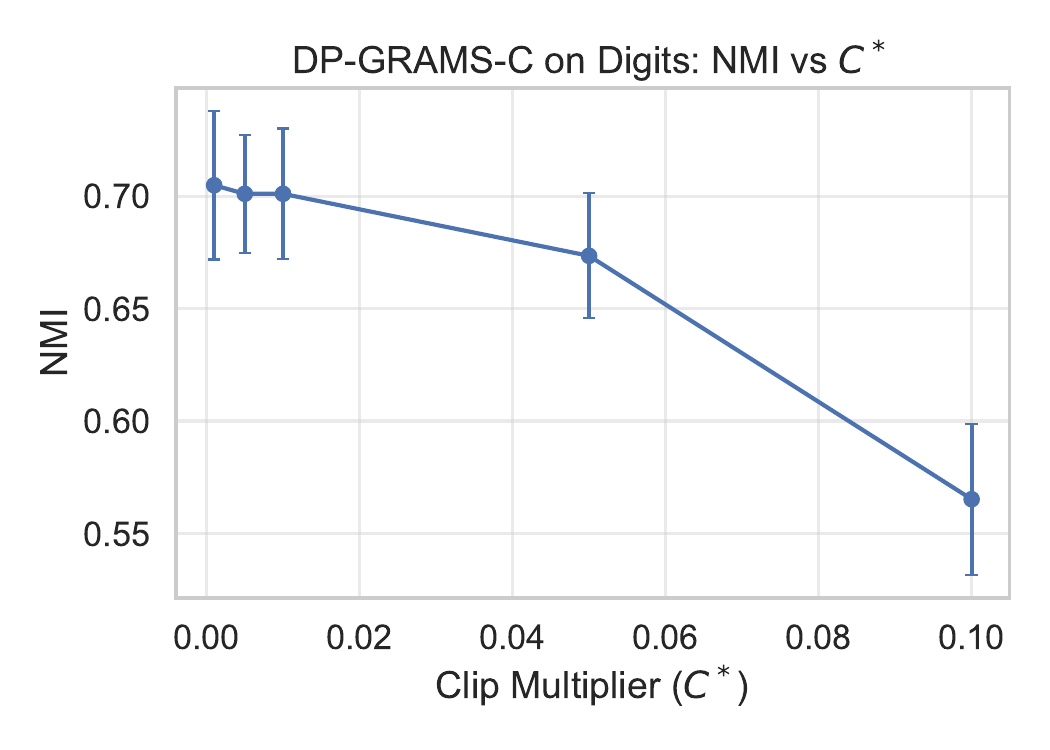}
    \end{minipage}
    \caption{\small Digits, \textsc{DP-GRAMS-C}: centroid MSE, ARI, and NMI versus clipping multiplier \texttt{clip\_multiplier} at \((\varepsilon,\delta)=(1,10^{-5})\). Points show averages over \(20\) runs with standard-error bars.}
    \label{fig:digits_hparam_clip}
\end{figure}

\begin{figure}[!htb]
    \centering
    \begin{minipage}{0.32\textwidth}
        \includegraphics[width=\linewidth]{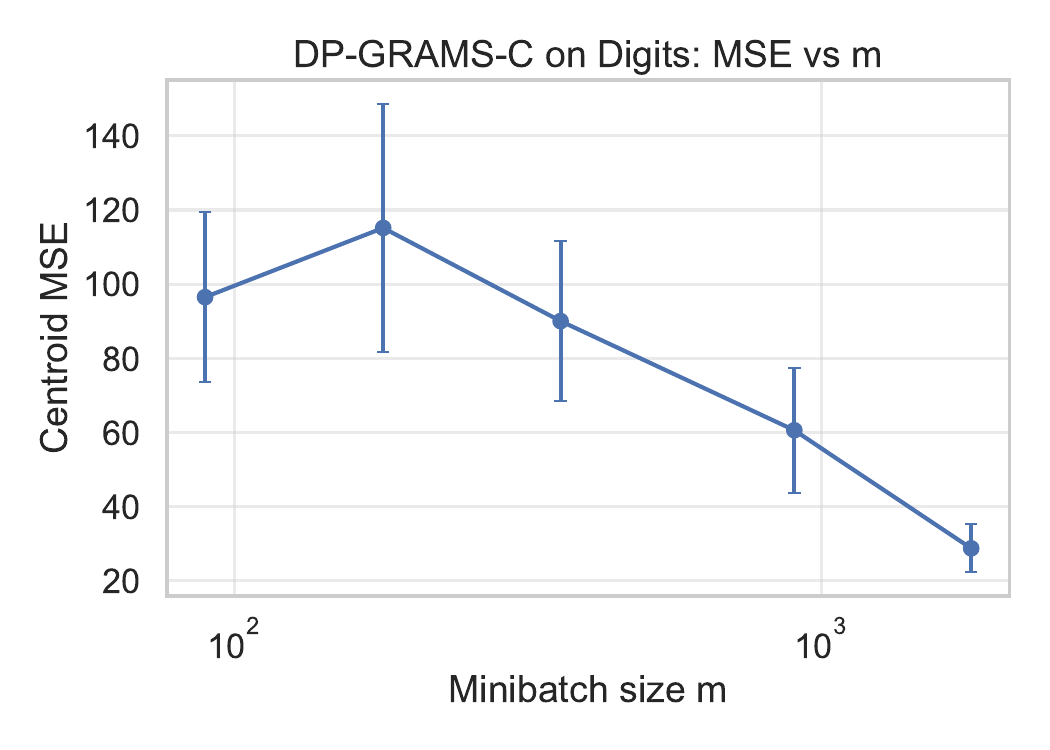}
    \end{minipage}\hfill
    \begin{minipage}{0.32\textwidth}
        \includegraphics[width=\linewidth]{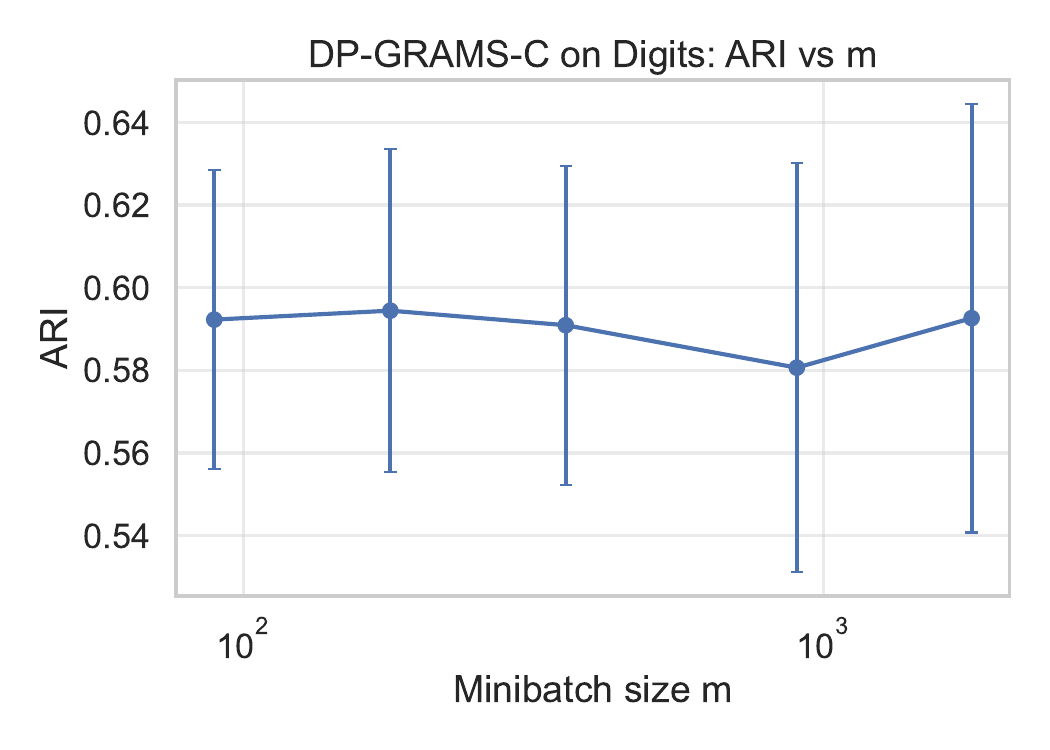}
    \end{minipage}\hfill
    \begin{minipage}{0.32\textwidth}
        \includegraphics[width=\linewidth]{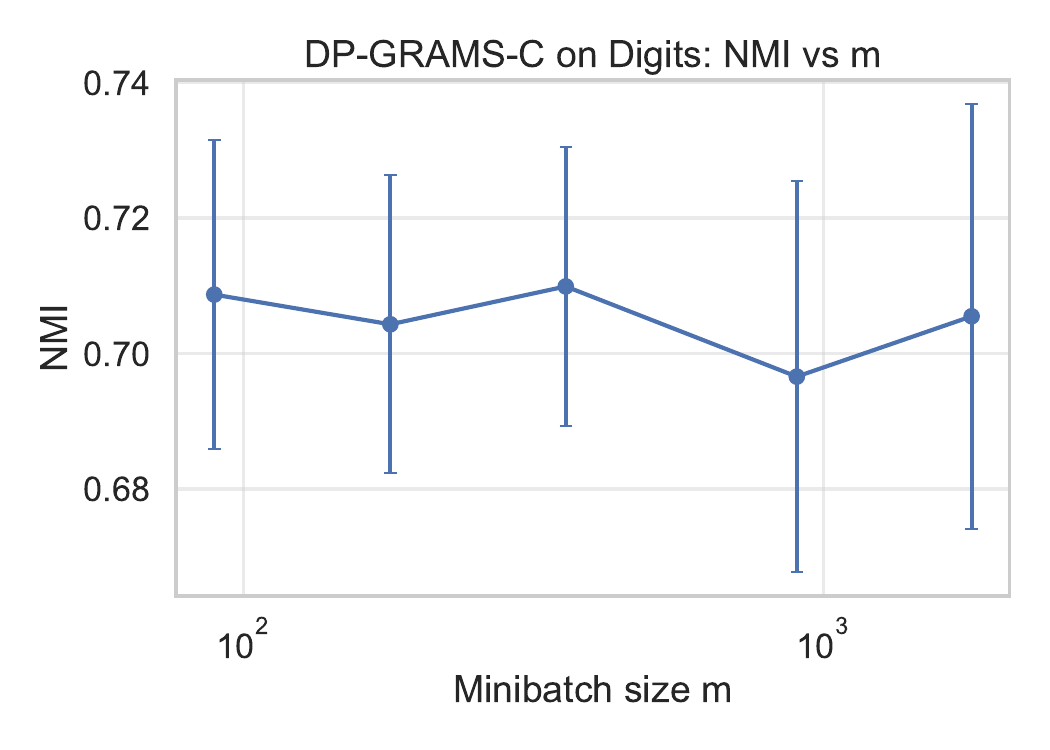}
    \end{minipage}
    \caption{\small Digits, \textsc{DP-GRAMS-C}: centroid MSE, ARI, and NMI versus minibatch size \(m\) on a log scale at \((\varepsilon,\delta)=(1,10^{-5})\). Points show averages over \(20\) runs with standard-error bars.}
    \label{fig:digits_hparam_minibatch}
\end{figure}

For Digits, we sweep \(C_* \in \{0.1,0.5,1,5,10\}\) and minibatch fractions \(m/n\in\{0.1,0.2,0.5,1.0\}\) at fixed \((\varepsilon,\delta)=(1,10^{-5})\). For each configuration, we record ARI, NMI, centroid MSE, and runtime across repeated runs. Figures~\ref{fig:digits_hparam_clip}--\ref{fig:digits_hparam_minibatch} show no sharp deterioration near the selected default settings.

\paragraph{MNIST public-candidate sensitivity.}
For MNIST, the main experiment uses a stratified public auxiliary candidate set of \(1000\) images, projects both the public candidates and the private experimental data into a whitened five-dimensional PCA representation, and runs clustering in that shared representation. No full DAP lattice grid is built for MNIST. We sweep the clipping multiplier and minibatch size at fixed \((\varepsilon,\delta)=(1,10^{-5})\), using the same public candidate construction as in Section~\ref{subsec:private-clustering-mnist}. Table~\ref{tab:mnist_privacy_utility_table} summarizes the MNIST privacy--utility statistics for \textsc{DP-GRAMS-C} and DP-\(k\)-Means.

\begin{figure}[!htb]
    \centering
    \begin{minipage}{0.32\textwidth}
        \includegraphics[width=\linewidth]{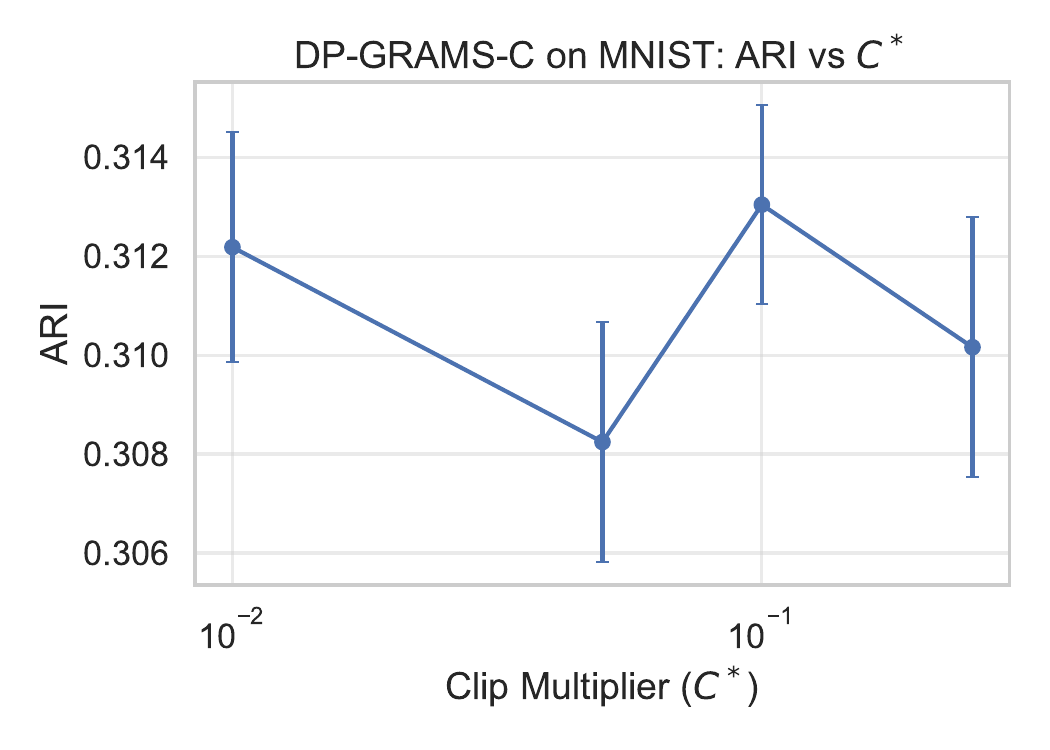}
    \end{minipage}\hfill
    \begin{minipage}{0.32\textwidth}
        \includegraphics[width=\linewidth]{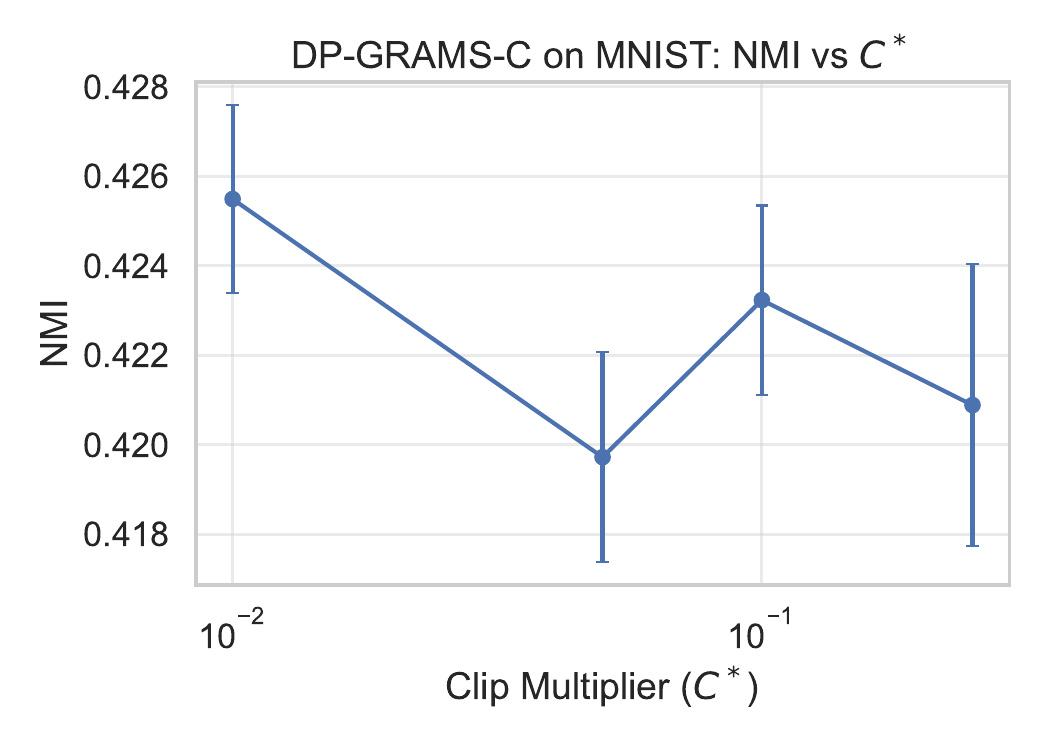}
    \end{minipage}\hfill
    \begin{minipage}{0.32\textwidth}
        \includegraphics[width=\linewidth]{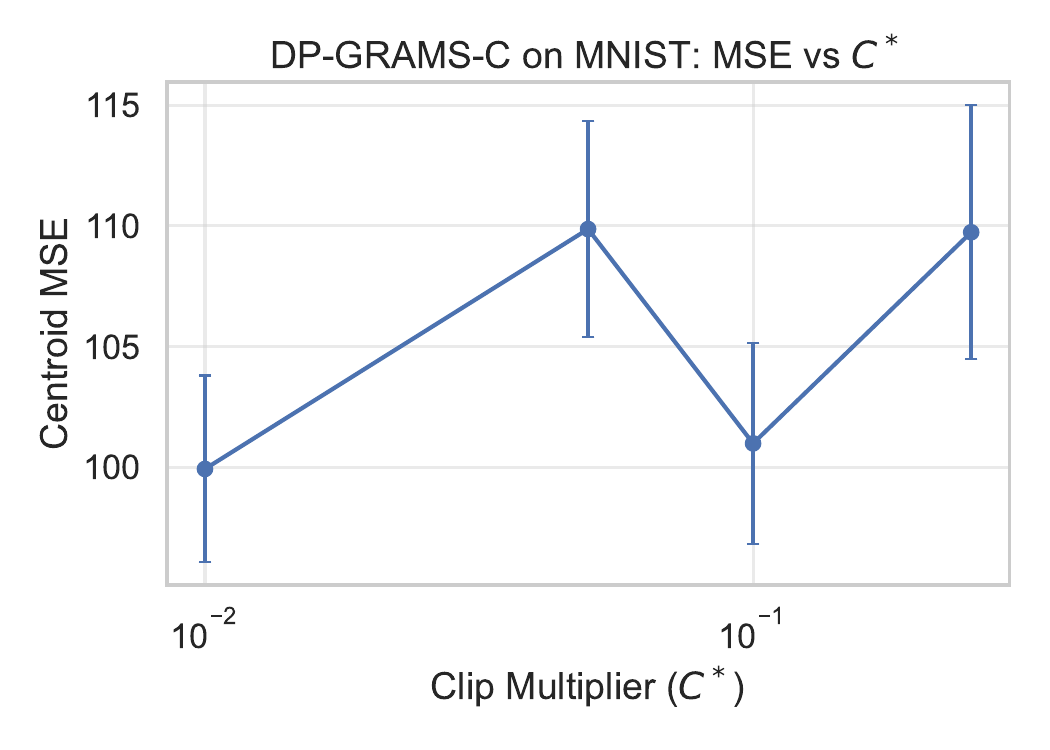}
    \end{minipage}
   \caption{\small MNIST, \textsc{DP-GRAMS-C}: ARI, NMI, and centroid MSE versus clipping multiplier \texttt{clip\_multiplier} at \((\varepsilon,\delta)=(1,10^{-5})\). Points show averages over \(20\) runs with standard-error bars.}
    \label{fig:mnist_subsample_clip_grid}
\end{figure}

\begin{figure}[!htb]
    \centering
    \begin{minipage}{0.32\textwidth}
        \includegraphics[width=\linewidth]{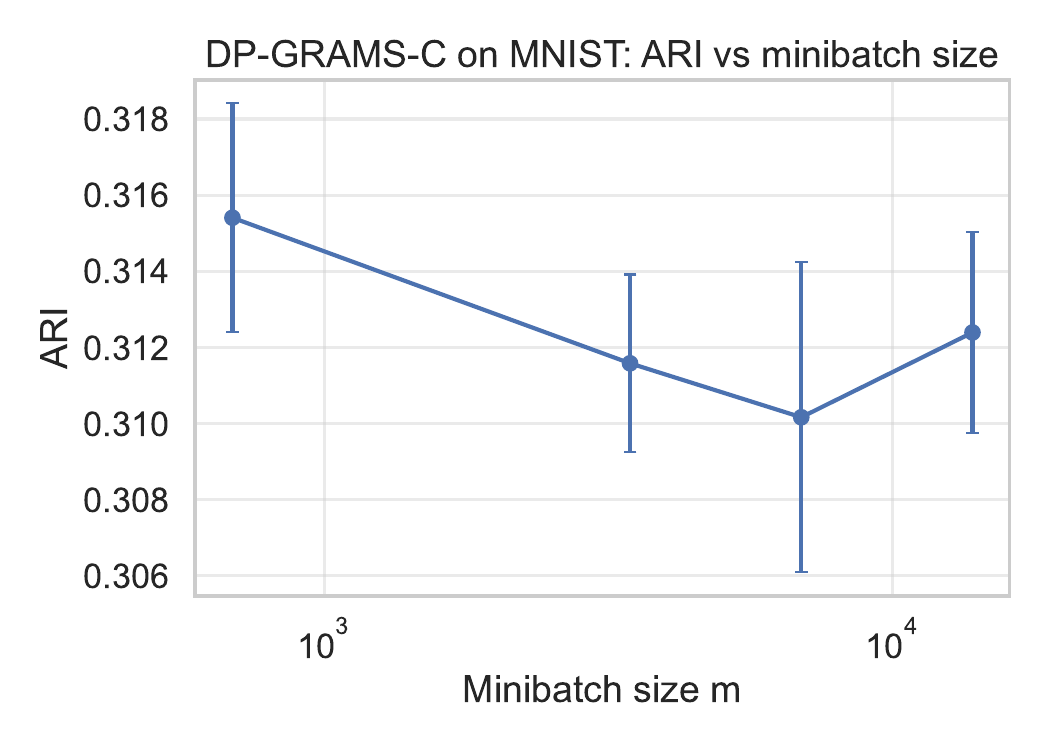}
    \end{minipage}\hfill
    \begin{minipage}{0.32\textwidth}
        \includegraphics[width=\linewidth]{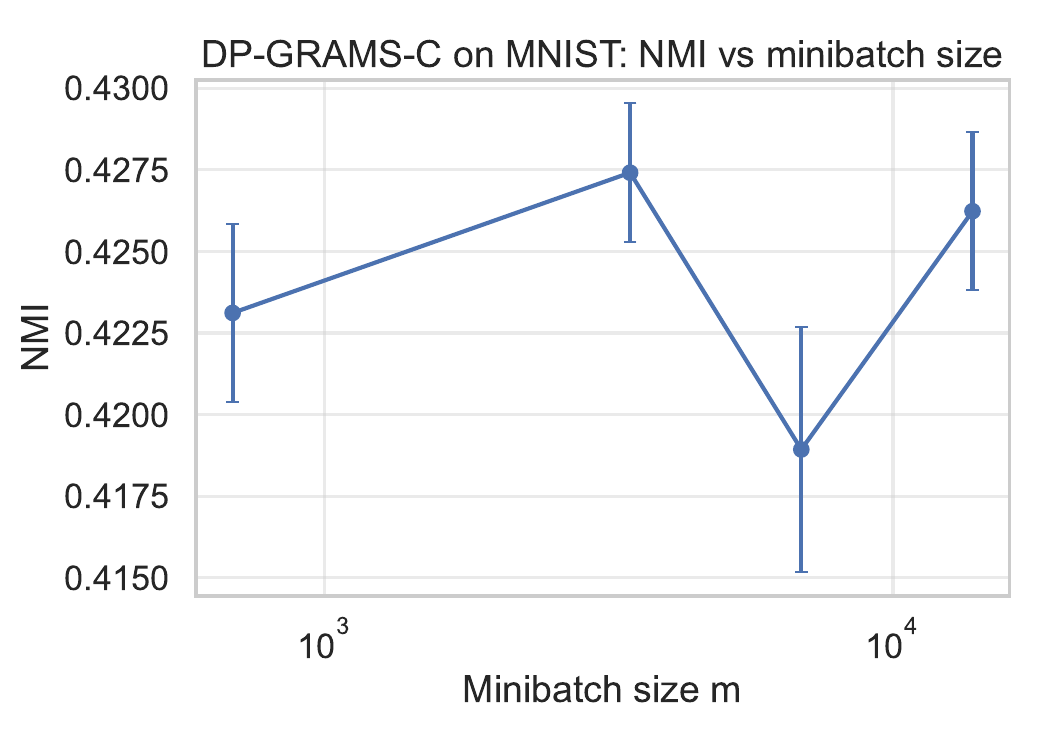}
    \end{minipage}\hfill
    \begin{minipage}{0.32\textwidth}
        \includegraphics[width=\linewidth]{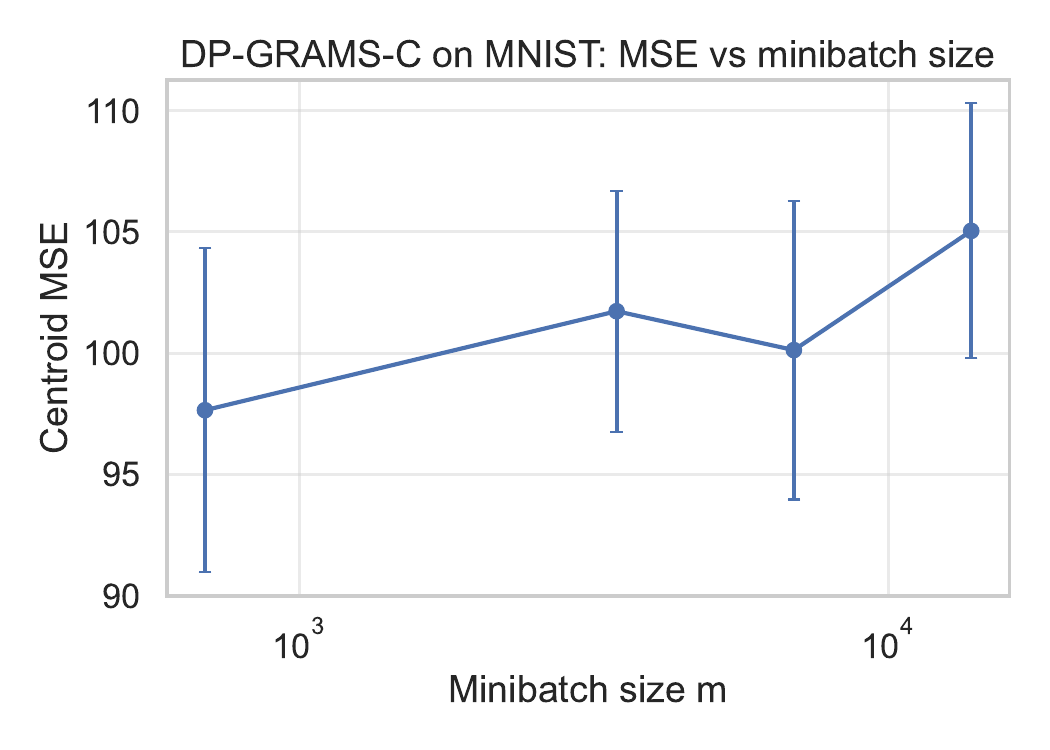}
    \end{minipage}
   \caption{\small MNIST, \textsc{DP-GRAMS-C}: ARI, NMI, and centroid MSE versus minibatch size \(m\) on a log scale at \((\varepsilon,\delta)=(1,10^{-5})\). Points show averages over \(20\) runs with standard-error bars.}
    \label{fig:mnist_subsample_m_grid}
\end{figure}

\begin{table}[h]
\centering
\caption{MNIST: privacy--utility summary for \textsc{DP-GRAMS-C} and DP-\(k\)-Means across \(\varepsilon\). Reported metrics are ARI, NMI, centroid MSE, and runtime, summarized as mean \(\pm\) SE over \(20\) runs.}
\label{tab:mnist_privacy_utility_table}
\resizebox{\textwidth}{!}{%
\begin{tabular}{c c c c c c}
\toprule
\(\varepsilon\) & Algorithm & ARI & NMI & Centroid MSE & Runtime (s) \\
\midrule
0.05 & \textsc{DP-GRAMS-C} & $0.205 \pm 0.010$ & $0.339 \pm 0.007$ & $160.12 \pm 7.06$ & $0.5064 \pm 0.0555$ \\
0.05 & DP-\(k\)-Means      & $0.148 \pm 0.008$ & $0.277 \pm 0.008$ & $348.03 \pm 16.49$ & $0.1450 \pm 0.0009$ \\
\midrule
0.1  & \textsc{DP-GRAMS-C} & $0.272 \pm 0.007$ & $0.389 \pm 0.006$ & $107.56 \pm 5.08$ & $0.4765 \pm 0.0335$ \\
0.1  & DP-\(k\)-Means      & $0.202 \pm 0.009$ & $0.326 \pm 0.007$ & $285.23 \pm 15.39$ & $0.1450 \pm 0.0007$ \\
\midrule
0.2  & \textsc{DP-GRAMS-C} & $0.303 \pm 0.003$ & $0.417 \pm 0.003$ & $99.54 \pm 4.22$ & $0.5078 \pm 0.0483$ \\
0.2  & DP-\(k\)-Means      & $0.220 \pm 0.009$ & $0.346 \pm 0.007$ & $219.99 \pm 17.62$ & $0.1466 \pm 0.0017$ \\
\midrule
0.5  & \textsc{DP-GRAMS-C} & $0.309 \pm 0.003$ & $0.422 \pm 0.003$ & $104.70 \pm 5.53$ & $0.5120 \pm 0.0417$ \\
0.5  & DP-\(k\)-Means      & $0.246 \pm 0.006$ & $0.369 \pm 0.004$ & $181.13 \pm 14.82$ & $0.1978 \pm 0.0014$ \\
\midrule
1.0  & \textsc{DP-GRAMS-C} & $0.311 \pm 0.003$ & $0.424 \pm 0.003$ & $108.92 \pm 5.23$ & $0.3755 \pm 0.0015$ \\
1.0  & DP-\(k\)-Means      & $0.279 \pm 0.003$ & $0.395 \pm 0.002$ & $148.61 \pm 14.47$ & $0.3750 \pm 0.0017$ \\
\bottomrule
\end{tabular}}
\end{table}

Together, the Digits and MNIST results provide two complementary real-image checks for the private prototype-release clustering procedure. Digits uses the default public DAP grid in a six-dimensional PCA representation and shows a clear separation between \textsc{DP-GRAMS-C} and DP-\(k\)-Means across the privacy grid. MNIST uses public auxiliary candidates in a whitened five-dimensional PCA representation; there the gap is smaller, but \textsc{DP-GRAMS-C} improves from \(\varepsilon=0.05\) to moderate privacy budgets and remains competitive at \(\varepsilon=1\). These diagnostics support the main clustering findings in Section~\ref{subsec:private-clustering-main} and indicate that the conclusions are not driven by a narrow choice of clipping threshold or minibatch size.

\subsection{Private Clustering on Cancer Gene Expression (RNA-Seq)}
\label{subsec:private-clustering-cancer-appendix}

This subsection supplements the Cancer RNA-Seq clustering experiment in Section~\ref{subsec:private-clustering-cancer}. The data are standardized gene-expression profiles, clustered in a whitened six-dimensional PCA representation. \textsc{DP-GRAMS-C} uses the default public DAP grid in this reduced space, with no data points supplied as DAP candidates; the public DAP box is chosen using the robust \(R_{90}\) rule from the implementation. Centroid MSE is computed after inverse-PCA back-projection of estimated centers to standardized gene space.

\begin{table}[h]
\centering
\caption{\small Cancer RNA-Seq after gene-wise standardization and projection to a whitened six-dimensional PCA representation: privacy--utility summary for \textsc{DP-GRAMS-C} and DP-\(k\)-Means across \(\varepsilon\). Reported metrics are ARI, NMI, centroid MSE, and runtime, summarized as mean \(\pm\) SE over \(20\) runs.}
\label{tab:gene50_privacy_utility_table}
\small
\setlength{\tabcolsep}{5pt}
\renewcommand{\arraystretch}{1.15}
\begin{tabular}{c c c c c c}
\toprule
\(\varepsilon\) & Algorithm & ARI & NMI & Centroid MSE & Runtime (s) \\
\midrule
0.5 & \textsc{DP-GRAMS-C} & $0.572 \pm 0.028$ & $0.668 \pm 0.020$ & $3401.42 \pm 422.59$ & $2.445 \pm 0.105$ \\
0.5 & DP-\(k\)-Means      & $0.151 \pm 0.026$ & $0.270 \pm 0.027$ & $13813.81 \pm 774.31$ & $0.0200 \pm 0.0003$ \\
\midrule
1.0 & \textsc{DP-GRAMS-C} & $0.773 \pm 0.017$ & $0.839 \pm 0.009$ & $1115.66 \pm 173.85$ & $2.158 \pm 0.068$ \\
1.0 & DP-\(k\)-Means      & $0.139 \pm 0.033$ & $0.246 \pm 0.031$ & $14138.50 \pm 854.17$ & $0.0196 \pm 0.0002$ \\
\midrule
2.0 & \textsc{DP-GRAMS-C} & $0.794 \pm 0.013$ & $0.846 \pm 0.008$ & $1146.84 \pm 277.41$ & $2.430 \pm 0.103$ \\
2.0 & DP-\(k\)-Means      & $0.260 \pm 0.029$ & $0.366 \pm 0.025$ & $10613.41 \pm 543.80$ & $0.0193 \pm 0.0002$ \\
\midrule
5.0 & \textsc{DP-GRAMS-C} & $0.808 \pm 0.009$ & $0.859 \pm 0.004$ & $827.88 \pm 245.25$ & $2.644 \pm 0.077$ \\
5.0 & DP-\(k\)-Means      & $0.430 \pm 0.027$ & $0.520 \pm 0.022$ & $6986.08 \pm 536.68$ & $0.0193 \pm 0.0001$ \\
\midrule
10.0 & \textsc{DP-GRAMS-C} & $0.811 \pm 0.009$ & $0.859 \pm 0.004$ & $800.60 \pm 247.91$ & $2.335 \pm 0.062$ \\
10.0 & DP-\(k\)-Means      & $0.458 \pm 0.027$ & $0.560 \pm 0.023$ & $5721.49 \pm 534.16$ & $0.0188 \pm 0.0001$ \\
\bottomrule
\end{tabular}
\end{table}

Table~\ref{tab:gene50_privacy_utility_table} summarizes privacy--utility behavior across \(\varepsilon\in\{0.5,1,2,5,10\}\) for \textsc{DP-GRAMS-C} and a DP-\(k\)-Means baseline. The improvement occurs between \(\varepsilon=0.5\) and \(\varepsilon=1\), where \textsc{DP-GRAMS-C} moves from moderate clustering quality to ARI and NMI near the non-private range; for larger \(\varepsilon\), ARI and NMI mostly stabilize while centroid MSE continues to decrease. DP-\(k\)-Means improves with \(\varepsilon\) but remains evidently worse in ARI, NMI, and centroid MSE throughout the grid.

\noindent\textbf{Hyperparameter sensitivity:} We evaluate sensitivity of \textsc{DP-GRAMS-C} to the clipping multiplier \texttt{clip\_multiplier} and minibatch size \(m\) at fixed \((\varepsilon,\delta)=(1,10^{-5})\). We sweep \(C_*\) in the grid \(\{0.01,0.02,0.05,0.1\}\) and minibatch fractions \(m/n\in\{0.05,0.1,0.2,0.5,1.0\}\), averaging each configuration over \(20\) runs. Figures~\ref{fig:gene50_hparam_clip} and~\ref{fig:gene50_hparam_minibatch} show no sharp degradation near the selected defaults.

\begin{figure}[!htb]
    \centering
    \begin{minipage}{0.32\textwidth}
        \includegraphics[width=\linewidth]{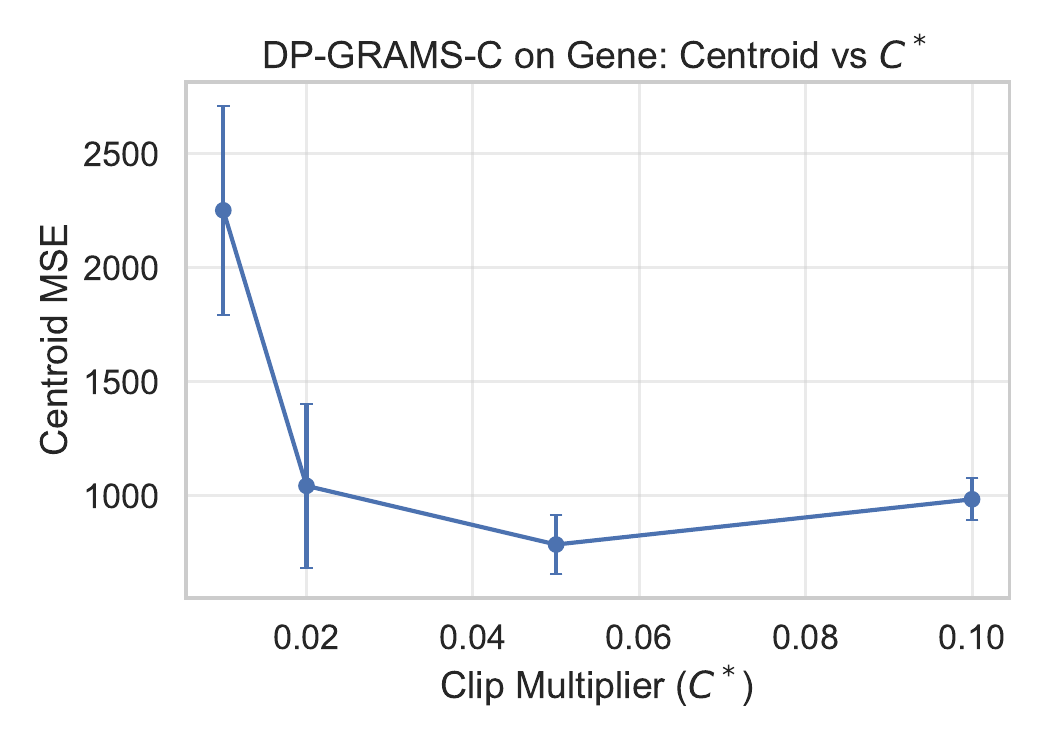}
    \end{minipage}\hfill
    \begin{minipage}{0.32\textwidth}
        \includegraphics[width=\linewidth]{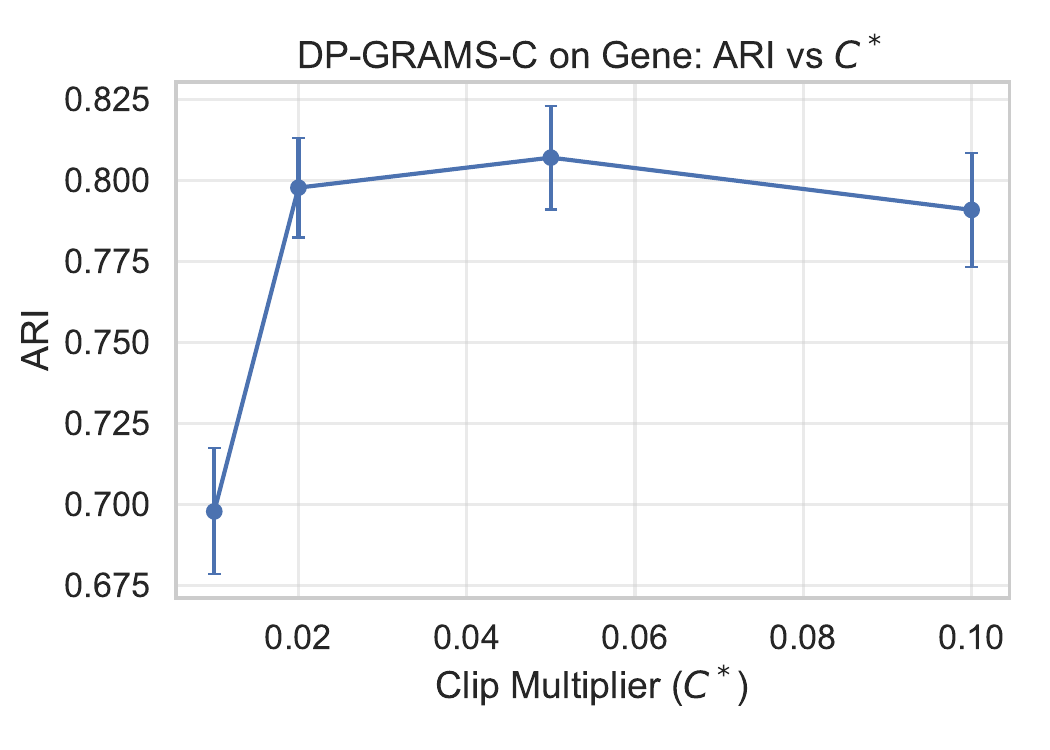}
    \end{minipage}\hfill
    \begin{minipage}{0.32\textwidth}
        \includegraphics[width=\linewidth]{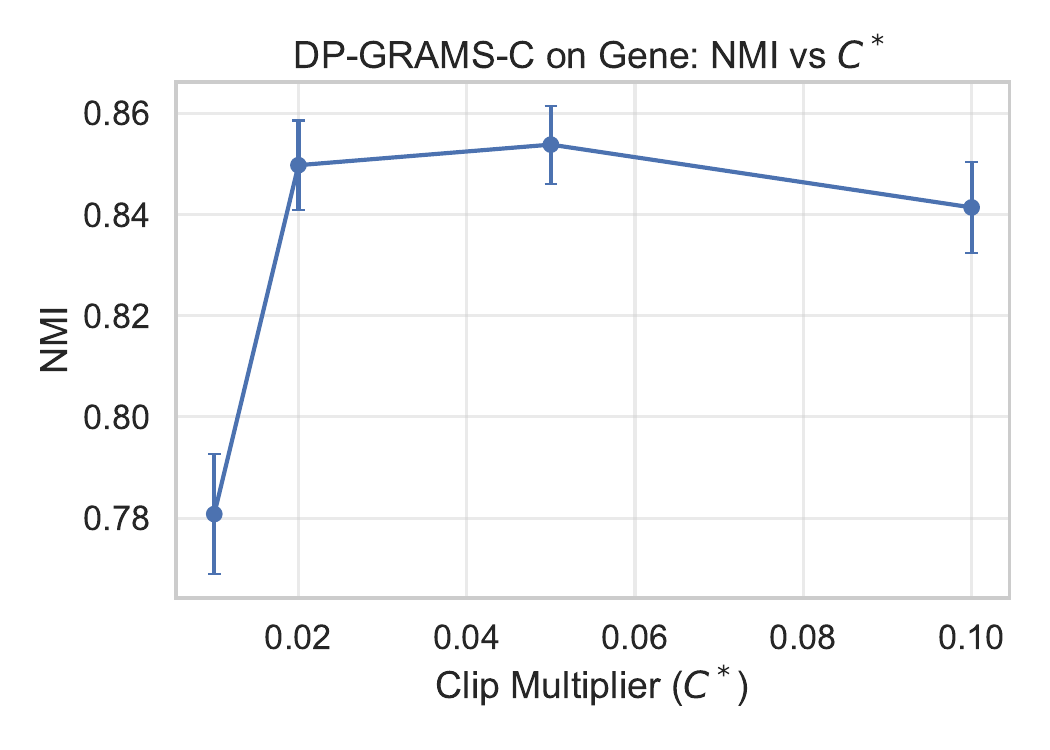}
    \end{minipage}
    \caption{\small Cancer RNA-Seq, \textsc{DP-GRAMS-C}: centroid MSE, ARI, and NMI versus clipping multiplier \texttt{clip\_multiplier} at \((\varepsilon,\delta)=(1,10^{-5})\). Points show averages over \(20\) runs with standard-error bars.}
    \label{fig:gene50_hparam_clip}
\end{figure}

\begin{figure}[!htb]
    \centering
    \begin{minipage}{0.32\textwidth}
        \includegraphics[width=\linewidth]{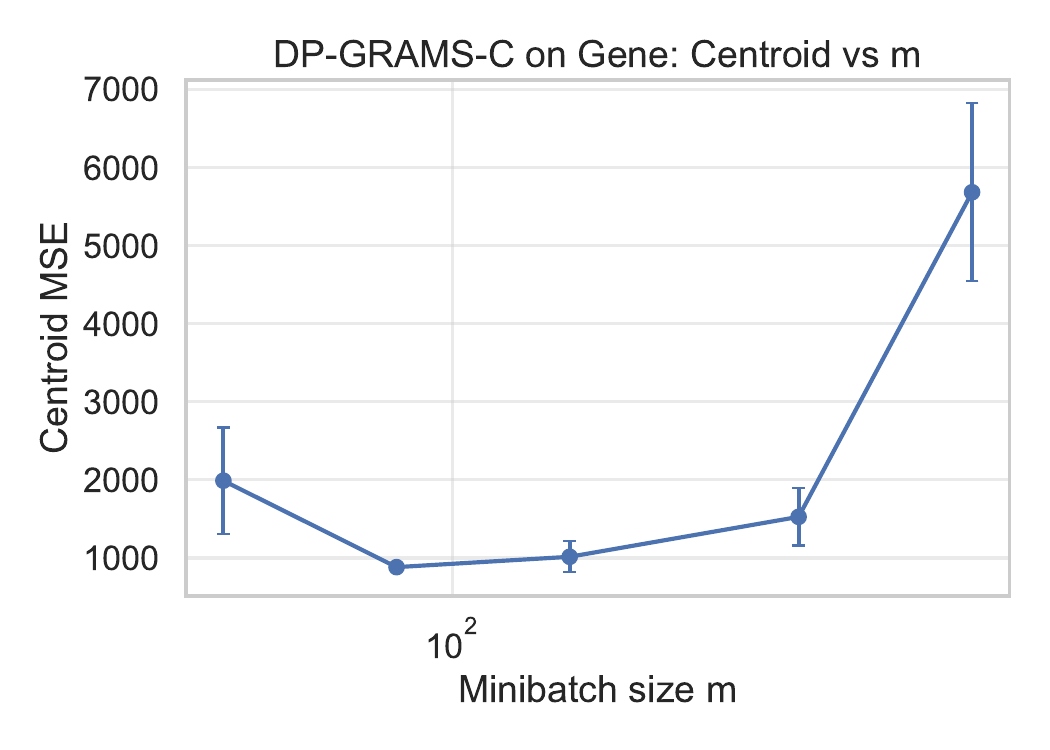}
    \end{minipage}\hfill
    \begin{minipage}{0.32\textwidth}
        \includegraphics[width=\linewidth]{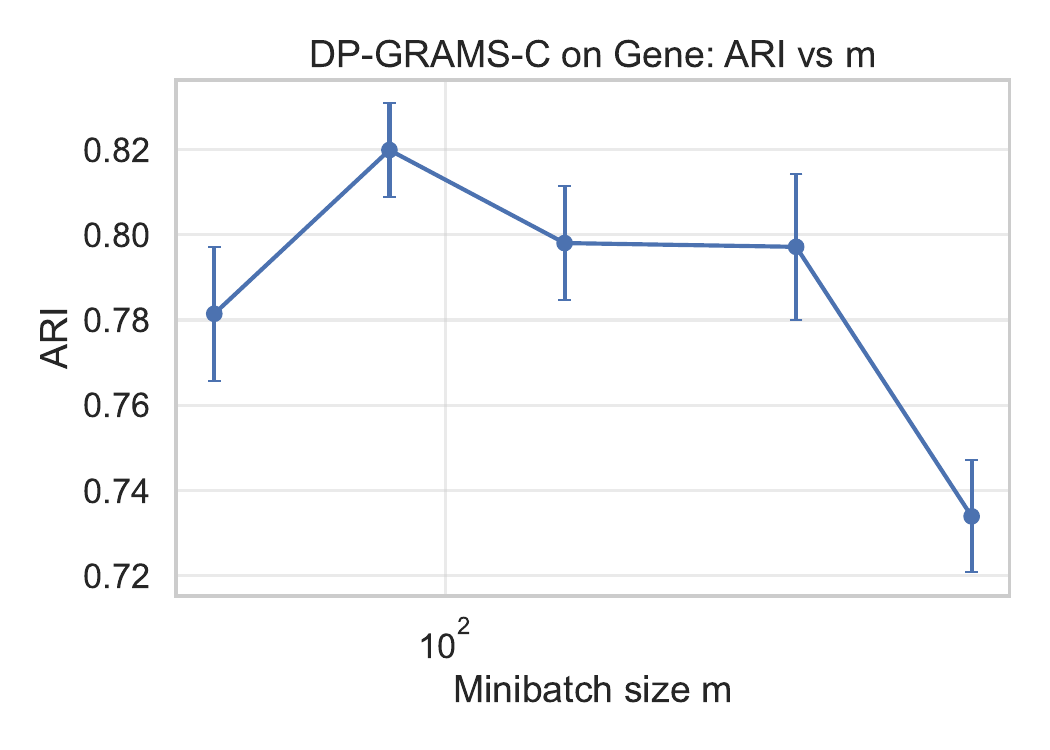}
    \end{minipage}\hfill
    \begin{minipage}{0.32\textwidth}
        \includegraphics[width=\linewidth]{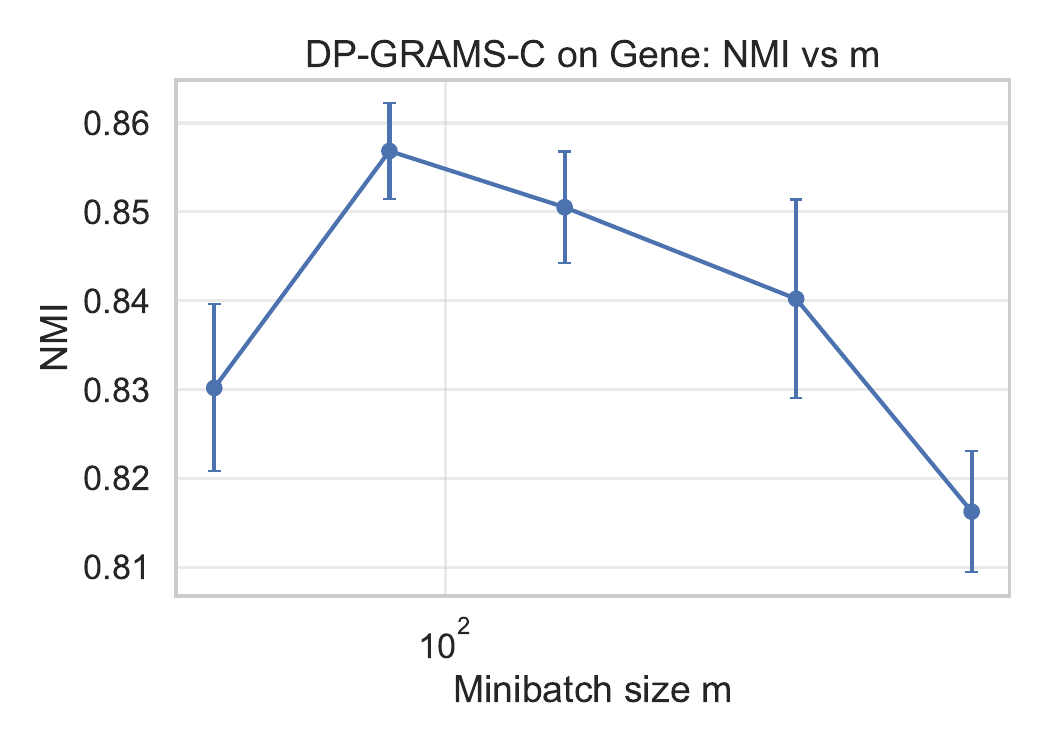}
    \end{minipage}
   \caption{\small Cancer RNA-Seq, \textsc{DP-GRAMS-C}: centroid MSE, ARI, and NMI versus minibatch size \(m\) on a log scale at \((\varepsilon,\delta)=(1,10^{-5})\). Points show averages over \(20\) runs with standard-error bars.}
    \label{fig:gene50_hparam_minibatch}
\end{figure}

Together with the main Cancer RNA-Seq results in Figures~\ref{fig:gene50_cluster_comparison} and~\ref{fig:gene50_privacy_utility_all}, these appendix diagnostics show that \textsc{DP-GRAMS-C} remains effective in a high-dimensional gene-expression task after PCA reduction. The method reaches ARI and NMI near the non-private range by moderate privacy budgets, while centroid MSE continues to improve at larger \(\varepsilon\). The sensitivity plots indicate that this behavior is not driven by a narrow choice of clipping threshold or minibatch size.


\section{Algorithms}\label{sec:apndx-algo}

\begin{algorithm}[!htbp]\small
\caption{DP-GRAMS-C: DP-GRAMS based Clustering}
\label{alg:dpgrams-c}
\SetKwInOut{Input}{Input}
\SetKwInOut{Output}{Output}

\Input{
Private data \(S=\{X_i\}_{i=1}^n\subset\mathbb R^d\); privacy parameters \((\varepsilon,\delta)\); optional number of clusters \(k_{\mathrm{est}}\); initialization privacy fraction \(p_0\); bandwidth multiplier; clipping multiplier; optional minibatch size \(m\); optional public candidate set \(\mathcal Z_{\mathrm{pub}}\).
}
\Output{Private cluster centers \(\mathcal M\); optional deterministic assignments \(y=(y_1,\dots,y_n)\) obtained from the released centers.}

Choose bandwidth \(h\) from the sample size \(n\), dimension \(d\), and bandwidth multiplier\;

Run \textsc{DP-GRAMS} on \(S\) with parameters \((\varepsilon,\delta,p_0,h,m)\), using \(\mathcal Z_{\mathrm{pub}}\) if supplied and otherwise constructing the public DAP grid internally, to obtain private candidate modes \(\mathcal M_{\mathrm{raw}}=\{\mu_1^{\mathrm{raw}},\dots,\mu_r^{\mathrm{raw}}\}\)\;

\eIf{\(k_{\mathrm{est}}\) is specified}{
    Merge \(\mathcal M_{\mathrm{raw}}\) into at most \(k_{\mathrm{est}}\) clusters using agglomerative clustering, and replace each cluster by its mean to obtain \(\mathcal M\)\;
}{
    Merge nearby points in \(\mathcal M_{\mathrm{raw}}\) using the default distance-threshold rule to obtain \(\mathcal M\)\;
}

\For{\(i=1,\dots,n\)}{
    Assign each point deterministically to its nearest released private center \(y_i=\arg\min_j\|X_i-\mathcal M_j\|_2\)\;
}

\Return \(\mathcal M\), and \(y\) when assignments are requested for post-processing or evaluation\;
\end{algorithm}

\begin{algorithm}[!htbp]
\caption{DP-PMS: Differentially Private Partial Mean Shift}
\label{alg:dp-pms}
\SetKwInOut{Input}{Input}
\SetKwInOut{Output}{Output}

\Input{
Data \(\{(X_i,Y_i)\}_{i=1}^n\), with predictor locations treated as fixed and public; predictor evaluation grid \(\mathcal X=\{x^{(1)},\dots,x^{(G)}\}\); privacy parameters \((\varepsilon,\delta)\); minibatch size \(m\); number of iterations \(T\); ascent bandwidth \(h\);
DAP bandwidth \(h_{\mathrm{score}}\); 
clipping multiplier \(c_{\mathrm{clip}}\); initialization privacy fraction \(p_0\); sparse-start multiplier \(\kappa_{\mathrm{init}}\); stepsize \(\eta\); public response candidate grid \(\mathcal Y_{\mathrm{pub}}\).
}
\Output{Private conditional mode estimates \(\{\widehat{\mathcal M}(x^{(g)})\}_{g=1}^G\).}

Set \(\varepsilon_{\mathrm{init}}=p_0\varepsilon\), \(\varepsilon_{\mathrm{asc}}=(1-p_0)\varepsilon\), \(\varepsilon_{\mathrm{draw}}=\varepsilon_{\mathrm{init}}/k\), 
clipping threshold \(C=c_{\mathrm{clip}}/h\)
;

Select 
predictor locations \(x_1,\dots,x_k\) from the predictor design, with \(k=\lceil \kappa_{\mathrm{init}}\log n\rceil\)\;


\For{\(j=1,\dots,k\)}{
    Compute the public local predictor count \(N_j=\sum_{i=1}^n \mathbf 1\{|X_i-x_j|\le h_{\mathrm{score}}\}\)\;

    For each response candidate \(z\in\mathcal Y_{\mathrm{pub}}\), compute the conditional local-mass utility
    \[
    u_j(z)
    =
    \frac{
    \sum_{i=1}^n
    \mathbf 1\{|X_i-x_j|\le h_{\mathrm{score}},\ |Y_i-z|\le h_{\mathrm{score}}\}
    }{
    \max\{N_j,1\}
    }.
    \]

    Sample \(y_j^{(0)}\in\mathcal Y_{\mathrm{pub}}\) using 
    weights proportional to
\(\exp\{(\varepsilon_{\mathrm{draw}}N_j/2)u_j(z)\}\), \(z\in\mathcal Y_{\mathrm{pub}}\)\;
}


Calibrate the Gaussian noise scale \(\sigma\) for \(T\) 
ascent steps using privacy budget \((\varepsilon_{\mathrm{asc}},\delta)\)\;

\For{\(t=0,\dots,T-1\)}{
    Sample one minibatch \(\mathcal B_t\subset[n]\) uniformly without replacement, with \(|\mathcal B_t|=m\)\;

    \For{\(j=1,\dots,k\)}{
        For each \(i\in\mathcal B_t\), compute \(w_{ij}^{(t)}=\exp\!\left(-\{(X_i-x_j)^2+(Y_i-y_j^{(t)})^2\}/(2h^2)\right)\)\;

        Compute localized scalar contributions 
        \[
        q_i^{(t)}(x_j,y_j^{(t)})
        =
        \mathbf{1}\left(\sum_{\ell\in\mathcal B_t}w_{\ell j}^{(t)}\neq 0\right)
        \frac{
        w_{ij}^{(t)}(Y_i-y_j^{(t)})
        }{
        \sum_{\ell\in\mathcal B_t}w_{\ell j}^{(t)}
        },
        \qquad i\in\mathcal B_t,
        \]

        Form the clipped update direction \(\Delta_j^{(t)}=\sum_{i\in\mathcal B_t}\max\{-C,\min\{q_i^{(t)}(x_j,y_j^{(t)}),C\}\}\)\;
    }

    Form the correlation matrix \(K_t\in\mathbb R^{k\times k}\) over current joint states \((x_j,y_j^{(t)})\):
    \[
    K_t(j,\ell)
    =
    \exp\!\left(
    -\frac{
    \left\|(x_j,y_j^{(t)})-(x_\ell,y_\ell^{(t)})\right\|_2
    }{h}
    \right).
    \]

    Draw correlated Gaussian noise \(\xi^{(t)}\sim\mathcal N(0,\sigma^2K_t)\)\;

    \For{\(j=1,\dots,k\)}{
        Update only the response coordinate:
        \(
        y_j^{(t+1)}
        =
        y_j^{(t)}
        +
        \eta\left(\Delta_j^{(t)}+\xi_j^{(t)}\right).
        \)
    }
}

For each \(x^{(g)}\in\mathcal X\), collect terminal values \(y_j^{(T)}\) whose fixed predictor locations \(x_j\) lie near \(x^{(g)}\), and merge nearby response values by post-processing to obtain \(\widehat{\mathcal M}(x^{(g)})\)\;

\Return \(\{\widehat{\mathcal M}(x^{(g)})\}_{g=1}^G\)\;
\end{algorithm}

\end{document}